\documentclass[a4paper,12pt]{amsbook}
\usepackage{amsmath,amsthm,amsfonts,graphicx,psfrag,amssymb,dsfont,mathrsfs}
\usepackage{color,minitoc}

\usepackage{hyperref}
\usepackage[alphabetic,initials]{amsrefs}

\psfrag{CP2}{$\CC P^2$}
\psfrag{CP1}{$\CC P^1$}
\psfrag{p0}{$x_0$}
\psfrag{u}{$u$}
\psfrag{Sigmadot}{$\dot{\Sigma}$}
\psfrag{inftyM}{$\{+\infty\} \times M$}
\psfrag{-inftyM}{$\{-\infty\} \times M$}
\psfrag{theta}{$\theta$}
\psfrag{rho}{$\rho$}

\theoremstyle{plain}
\newtheorem{thm}{Theorem}[chapter]
\newtheorem{prop}[thm]{Proposition}
\newtheorem{conj}[thm]{Conjecture}
\newtheorem{cor}[thm]{Corollary}
\newtheorem{lemma}[thm]{Lemma}
\newtheorem*{question}{Question}
\newtheorem*{thmu}{Theorem}
\newtheorem*{SardSmale}{Sard-Smale theorem}

\theoremstyle{definition}
\newtheorem{example}[thm]{Example}
\newtheorem{exercise}[thm]{Exercise}
\newtheorem{defn}[thm]{Definition}
\newtheorem{defs}[thm]{Definitions}
\newtheorem*{notation}{Notation}

\theoremstyle{remark}
\newtheorem{remark}[thm]{Remark}

\newcommand{\Area}{\operatorname{Area}}
\newcommand{\Aut}{\operatorname{Aut}}
\newcommand{\aut}{\mathfrak{aut}}
\newcommand{\codim}{\operatorname{codim}}
\newcommand{\coker}{\operatorname{coker}}
\newcommand{\Crit}{\operatorname{Crit}}
\newcommand{\crit}{{\operatorname{crit}}}
\newcommand{\Diff}{\operatorname{Diff}}
\newcommand{\dist}{\operatorname{dist}}
\newcommand{\dR}{{\operatorname{dR}}}
\newcommand{\End}{\operatorname{End}}
\newcommand{\ev}{\operatorname{ev}}

\newcommand{\Fred}{\operatorname{Fred}}

\newcommand{\Hom}{\operatorname{Hom}}

\newcommand{\Id}{{\operatorname{Id}}}

\renewcommand{\Im}{\operatorname{Im}}
\newcommand{\im}{\operatorname{im}}

\newcommand{\ind}{\operatorname{ind}}

\newcommand{\Inj}{\operatorname{Inj}}
\newcommand{\inter}{\iota}

\newcommand{\loc}{{\operatorname{loc}}}
\newcommand{\Lie}{{\mathcal{L}}}

\newcommand{\pr}{\operatorname{pr}}
\renewcommand{\Re}{\operatorname{Re}}
\newcommand{\reg}{{\operatorname{reg}}}

\newcommand{\std}{_{\operatorname{std}}}

\newcommand{\Vectors}{\operatorname{Vec}}
\newcommand{\virdim}{\operatorname{vir-dim}}

\newcommand{\GL}{\operatorname{GL}}

\newcommand{\PSL}{\operatorname{PSL}}

\newcommand{\OT}{{\operatorname{OT}}}
\newcommand{\U}{\operatorname{U}}

\newcommand{\SL}{\operatorname{SL}}

\newcommand{\CC}{{\mathbb C}}
\newcommand{\DD}{{\mathbb D}}

\newcommand{\HH}{{\mathbb H}}
\newcommand{\JJ}{{\mathbb J}}
\newcommand{\NN}{{\mathbb N}}

\newcommand{\RR}{{\mathbb R}}
\newcommand{\ZZ}{{\mathbb Z}}

\newcommand{\bB}{{\mathcal B}}

\newcommand{\dD}{{\mathcal D}}
\newcommand{\eE}{{\mathcal E}}

\newcommand{\jJ}{{\mathcal J}}
\newcommand{\lL}{{\mathcal L}}
\newcommand{\mM}{{\mathcal M}}

\newcommand{\oO}{{\mathcal O}}
\newcommand{\pP}{{\mathcal P}}
\newcommand{\tT}{{\mathcal T}}
\newcommand{\uU}{{\mathcal U}}
\newcommand{\vV}{{\mathcal V}}
\newcommand{\zZ}{{\mathcal Z}}
\newcommand{\Jfix}{J_{\operatorname{fix}}}
\newcommand{\Jref}{J_{\operatorname{ref}}}
\newcommand{\metrics}{{\mathfrak M}}

\newcommand{\1}{\mathds{1}}
\newcommand{\p}{\partial}
\renewcommand{\dbar}{\bar{\partial}}
\newcommand{\Cinftyloc}{C^\infty_{\loc}}
\newcommand{\defin}[1]{\textbf{#1}}

\newcommand{\Mod}{{\mathcal M}}
\newcommand{\univ}{{\mathscr U}}

\numberwithin{equation}{chapter}
\numberwithin{section}{chapter}

\definecolor{blue}{rgb}{0,0,1}
\definecolor{red}{rgb}{1,0,0}
\definecolor{green}{rgb}{0,.7,0}

\title{Lectures on 
Holomorphic Curves in Symplectic and Contact Geometry \\
{\large (Work in progress---Version 3.2) \\
\today}}
\author{Chris Wendl}
\address{Department of Mathematics, University College London}
\email{c.wendl@ucl.ac.uk}

\begin{document}
\frontmatter

\maketitle

\pagestyle{empty}             
\begin{center}
\copyright 2014 by Chris Wendl \\
Paper or electronic copies for noncommercial use may be made freely without
explicit permission from the author.  All other rights reserved.
\end{center}
\pagebreak

\dominitoc
\tableofcontents

\chapter*{Preface}

The present book-in-progress began as a set of lecture notes
written at a furious pace to accompany a graduate course on
holomorphic curves that I taught at ETH Z\"urich in Spring 2009, and
repeated at the Humboldt-Universit\"at zu Berlin in the 2009-10 
Winter semester.  In both iterations of the course,
it quickly became clear that my
conceived objectives for the notes were not really attainable within the
length of the semester, but the project nonetheless took on a life of its own.
I have written these notes with the following specific goals in mind:
\begin{enumerate}
\item 
To give a solid but readable presentation of the analytical foundations of
closed holomorphic curves from a modern perspective;
\item
To use the above foundation to explain a few of the classic
applications to symplectic topology, such as Gromov's nonsqueezing theorem
\cite{Gromov} and McDuff's results on rational and ruled symplectic
$4$-manifolds \cite{McDuff:rationalRuled};
\item
To use the aforementioned ``modern perspective'' to generalize
everything as cleanly as possible to the case of punctured holomorphic curves, 
and then explain some applications to contact geometry such as the 
Weinstein conjecture \cite{Hofer:weinstein} and obstructions 
to symplectic fillings \cite{Wendl:fillable}.
\end{enumerate}
The choice of topics covered and their presentation is partly a function of 
my own preferences, as well as my perception of which gaps in the existing
literature seemed most in need of filling.  In particular, I have devoted
special attention to a few topics that seem fundamental but are not covered
in the standard book on this subject by McDuff and Salamon 
\cite{McDuffSalamon:Jhol}, e.g.~the structure of Teichm\"uller space
and of the moduli space of unparametrized holomorphic curves of
arbitrary genus, existence results for local $J$-holomorphic curves,
and regularity for moduli spaces with constrained derivatives.
My choice of applications is biased toward those
which I personally find the most beautiful and which admit proofs with
a very geometric flavor.  For most such results, there are important abstract
invariants lurking in the background, but one need not develop them fully
in order to understand the proofs, and for that reason I have left out
topics such as gluing analysis and Gromov-Witten theory, on which I would in 
any case have nothing to add to the superb coverage in \cite{McDuffSalamon:Jhol}.
In order to save space and energy, I have also included nothing about
holomorphic curves with boundary, but aimed to make up for this by devoting
the last third of the book to punctured holomorphic curves, a topic on which
there are still very few available expositions aimed at graduate students.

My personal attitude toward technical details is essentially that of a 
non-analyst who finds analysis important: what this means is
that I've tried very hard to create an accessible presentation that
is as complete as possible without boring readers who don't enjoy
analysis for its own sake.  In contrast to \cite{McDuffSalamon:Jhol}, 
I have not put the discussion of elliptic regularity in an appendix but
rather integrated it into the main exposition, where it is (I hope) less
likely to be ignored.  On the other hand, I have presented such details
in less generality than would be theoretically possible, in most places
only as much as seems essential for the geometric applications.
One example of this is the discussion in Chapter~\ref{chapter:local}
of a local representation formula that is both \emph{weaker} and
\emph{easier to prove} than the famous result of Micallef and
White \cite{MicallefWhite}, but still suffices for crucial applications
such as positivity of intersections.
If some hardcore analysts find this approach lazy, my hope is that 
at least as many hardcore topologists may benefit from it.

\subsection*{About the current version}

Versions~1 and~2 of these notes were the versions written to accompany
the actual lecture course, and version~3 is the first major revision in book 
form, available on the arXiv at \url{http://arxiv.org/abs/1011.1690}.
Version~3.2 (v2 on the arXiv) has no new chapters that were not in
version~3.1, but it has a few substantial new sections on topics that were 
either not covered or only briefly mentioned in the previous version, including
the contractibility of the space of tame
almost complex structures (\S\ref{sec:compatible}), positivity of 
intersections (complete proofs of the local results underlying the
adjunction formula now appear in \S\ref{sec:positivity}),
transversality of the evaluation map (\S\ref{sec:evaluation}), and a proof
that ``generic holomorphic curves are immersed'' (\S\ref{sec:immersed}).
I have also made several small improvements to the exposition from
version~3.1, especially in Chapters~\ref{chapter:local} and~\ref{chapter:moduli},
and have corrected a number of misprints and minor errors that were
detected by diligent readers.

In all, this version contains a little over half of what I hope to
include in the finished product: there is not yet any serious material on contact
geometry (only a few main ideas sketched in the introduction), but the
development of the technical apparatus for closed holomorphic curves
is mostly complete.  The main thing still missing from this technical
development is Gromov's compactness theorem, though a simple case of it
is covered in Chapter~\ref{chapter:nonsqueezing} in order to prove
the nonsqueezing theorem.  I hope to add the chapter on Gromov
compactness in the next revision, along with further chapters covering
the special analytical properties of closed holomorphic curves in
dimension four, and applications to symplectic $4$-manifolds.

It should be mentioned that since the last revision nearly
four years ago, a substantial portion of the material that I eventually
plan to include in later chapters has appeared in other (shorter) sets
of lecture notes that were written for various minicourses.  In particular,
a comprehensive exposition of my perspective on 
McDuff's characterization of symplectic rational and ruled surfaces
now appears in \cite{Wendl:rationalRuled}, and 
some of the extensions
of these ideas to punctured holomorphic curves and contact $3$-manifolds
are covered in \cite{Wendl:Durham}.  Both are written with similar
target audiences in mind and should be readable by anyone who has
made it through the existing chapters of this book---in fact they
assume less technical background, but provide brief reviews of analytical
material that is treated here in much more detail.
It remains a long-term goal that the main topics covered in 
\cites{Wendl:rationalRuled,Wendl:Durham} should eventually be integrated
into the present manuscript in some form.

\subsection*{Acknowledgments}

I'd like to thank a number of people who have contributed useful comments,
ideas, explanations and encouragement on this project, including 
Peter Albers, Jonny Evans,
Joel Fish, Paolo Ghiggini, Janko Latschev, Klaus Mohnke,
Dietmar Salamon, and Sam Lisi.  I would also especially like to thank
Patrick Massot and Urs Fuchs for careful reading which led to some
important improvements and corrections.

A very large portion of what I know about this subject was
originally imparted to me by Helmut Hofer, whose unpublished manuscript
with Casim Abbas \cite{AbbasHofer} has also been an invaluable resource
for me.  Other invaluable resources worth mentioning include of course
\cite{McDuffSalamon:Jhol}, as well as the expository article \cite{Sikorav}
by Sikorav.

\subsection*{Request}

As should by now be obvious, these notes are work in progress,
and as such I welcome comments, questions, suggestions and corrections
from anyone making the effort to read them.  These may be sent to
\texttt{c.wendl@ucl.ac.uk}.

\adjustmtc

\chapter*{A Note on Terminology}

Unless otherwise specified, whenever we deal with objects such as manifolds 
and vector or fiber bundles that differential geometers normally assume to be
smooth and/or finite dimensional, the reader may assume that they are both.
When infinite-dimensional objects arise, we will either state explicitly that
they are infinite dimensional, or use standard functional analytic terms such
as \emph{Banach manifold} and
\emph{Banach space bundle}.  Similarly, maps on manifolds and sections of
bundles (including e.g.~complex and symplectic structures) should normally 
be assumed smooth unless otherwise specified, with the notation
$\Gamma(E)$ used to denote the space of sections of a bundle~$E$.

\adjustmtc

\mainmatter
\pagestyle{myheadings}

\chapter{Introduction}
\label{chapter:intro}

\minitoc
\vspace{12pt}

\section{Warm up: Holomorphic curves in $\CC^n$}
\label{sec:warmup}

The main subject of these notes is a certain interplay between
\emph{symplectic} structures and \emph{complex} (or rather \emph{almost
complex}) structures on smooth manifolds.  To illustrate the connection, 
we consider first the special case of holomorphic curves in~$\CC^n$.

If $\uU \subset \CC^m$ is an open subset and $u : \uU \to \CC^n$ is a smooth
map, we say that $u$ is \defin{holomorphic} if its partial derivatives
$\frac{\p u}{\p z_j}$ all exist for $i=j,\ldots,m$, i.e.~the limits
$$
\frac{\p u}{\p z_j} = \lim_{h \to 0} \frac{u(z_1,\ldots,z_{j-1},z_j + h,
z_{j+1},\ldots,z_m) - u(z_1,\ldots,z_m)}{h}
$$
exist, where $h$ is complex.  This is the obvious generalization of the
notion of an analytic function of one complex variable, and leads to an
obvious generalization of the usual Cauchy-Riemann equations.

We will find the following equivalent formulation quite useful.  Let us
identify $\CC^n = \RR^{2n}$ by regarding $(z_1,\ldots,z_n) \in \CC^n$
as the real vector 
$$
(p_1,q_1,\ldots,p_n,q_n) \in \RR^{2n},
$$
where $z_j = p_j + i q_j$ for $j = 1,\ldots,n$.  
Then at every point
$z \in \uU \subset \CC^m$, our smooth map
$u : \uU \to \CC^n$ has a differential $du(z) : \CC^m \to \CC^n$, which is
in general a \emph{real}-linear map $\RR^{2m} \to \RR^{2n}$.  
Observe also that for any number
$\lambda \in \CC$, the complex scalar multiplication
$$
\CC^n \to \CC^n : z \mapsto \lambda z
$$
defines a real-linear map from $\RR^{2n}$ to itself.
It turns out that
$u$ is holomorphic if and only if its differential at every point is
also \emph{complex}-linear: in particular it must satisfy
$du(z) \lambda V = \lambda \cdot du(z) V$ for every $V \in \CC^m$ and
$\lambda \in \CC$.  Since $du(z)$ is already real-linear, it suffices
to check that $du(z)$ behaves appropriately with respect to multiplication
by~$i$,~i.e.
\begin{equation}
\label{eqn:CRintegrable}
du(z) \circ i = i \circ du(z),
\end{equation}
where we regard multiplication by $i$ as a linear map on
$\RR^{2m}$ or $\RR^{2n}$.

\begin{exercise}
Show that \eqref{eqn:CRintegrable} is equivalent to the usual Cauchy-Riemann
equations for smooth maps $u : \uU \to \CC^n$.
\end{exercise}

If $m=1$, so $\uU$ is an open subset of $\CC$, we refer to holomorphic maps
$u : \uU \to \CC^n$ as \defin{holomorphic curves} in~$\CC^n$.  The choice of
wording is slightly unfortunate if you like to think in terms of \emph{real}
geometry---after all, the image of $u$ looks more like a surface than
a curve.  But we call $u$ a ``curve'' because, in complex terms, it
is a one-dimensional object.

That said, let us think of holomorphic curves for the moment as 
real $2$-dimensional
objects and ask a distinctly real $2$-dimensional question: what is the
area traced out by $u : \uU \to \CC^n$?  Denote points in $\uU$ by
$s + it$ and think of $u$ as a function of the two real variables~$(s,t)$,
with values in $\RR^{2n}$.  In these coordinates, the action of $i$ on
vectors in $\CC = \RR^2$ can be expressed succinctly by the relation
$$
i \p_s = \p_t.
$$
We first have to compute the area of the parallelogram in $\RR^{2n}$ spanned
by $\p_s u(s,t)$ and $\p_t u(s,t)$.  The Cauchy-Riemann equation
\eqref{eqn:CRintegrable} makes this easy, because 
$$
\p_t u(s,t) = du(s,t) \p_t = du(s,t) i \p_s = i\,du(s,t) \p_s =
i\,\p_s u(s,t),
$$
which implies that $\p_s u(s,t)$ and $\p_t u(s,t)$ are orthogonal vectors of
the same length.  Thus the area of $u$ is
$$
\Area(u) = \int_{\uU} |\p_s u| |\p_t u|\,ds\,dt =
\frac{1}{2} \int_{\uU} \left( |\p_s u|^2 + |\p_t u|^2 \right) \,ds \,dt,
$$
where we've used the fact that $|\p_s u| = |\p_t u|$ to write things slightly
more symmetrically.  Notice that the right hand side is really an 
\emph{analytical} quantity: up to a constant it is the square of the
$L^2$-norm of the first derivative of~$u$.  

Let us now write this area in a slightly different, more topological way.
If $\langle\ ,\ \rangle$ denotes the standard Hermitian inner product on
$\CC^n$, notice that one can define a differential $2$-form on $\RR^{2n}$
by the expression
$$
\omega\std(X,Y) = \Re \langle i X, Y \rangle.
$$
Writing points in $\CC^n$ via the coordinates 
$(p_1 + i q_1,\ldots,p_n + i q_n)$, one can show that $\omega\std$ in these
coordinates takes the form
\begin{equation}
\label{eqn:omega0}
\omega\std = \sum_{j=1}^n dp_j \wedge dq_j.
\end{equation}
\begin{exercise}
\label{ex:omega0}
Prove \eqref{eqn:omega0}, and then show that $\omega\std$ has the following
three properties:
\begin{enumerate}
\item It is \emph{nondegenerate}: $\omega\std(V,\cdot) = 0$ for some vector
$V$ if and only if $V=0$.  Equivalently, for each $z \in \RR^{2n}$, the
map $T_z \RR^{2n} \to T_z^* \RR^{2n} : V \mapsto \omega\std(V,\cdot)$ is an
isomorphism.
\item It is \emph{closed}: $d\omega\std = 0$.
\item The $n$-fold product $\omega\std^n = \omega\std \wedge \ldots \wedge
\omega\std$ is a constant multiple of the natural volume form on~$\RR^{2n}$.
\end{enumerate}
\end{exercise}
\begin{exercise}
\label{ex:volume}
Show that a $2$-form $\omega$ on $\RR^{2n}$ (and hence on any $2n$-dimensional
manifold) is nondegenerate if and only if $\omega^n$ is a volume form.
\end{exercise}

Using $\omega\std$, we see that the area of the parallelogram above is also
$$
|\p_s u| \cdot |\p_t u| = |\p_t u|^2 = \Re \langle \p_t u, \p_t u \rangle =
\Re \langle i \p_s u , \p_t u \rangle = \omega\std(\p_s u,\p_t u),
$$
thus
\begin{equation}
\label{eqn:area}
\Area(u) = \| du \|_{L^2}^2 = \int_{\uU} u^*\omega\std.
\end{equation}
This is the first appearance of symplectic geometry in our study of
holomorphic curves; we call $\omega\std$ the \defin{standard symplectic form}
on $\RR^{2n}$.  The point is that the expression on the right hand side of
\eqref{eqn:area} is essentially topological: it depends only on the evaluation
of a certain closed $2$-form on the $2$-chain defined by $u(\uU)$.  The
present example is trivial because we're only working in $\RR^{2n}$, but as
we'll see later in more interesting examples, one can often find an easy
topological bound on this integral, which by \eqref{eqn:area} implies a
bound on the analytical quantity $\| du \|_{L^2}^2$.  One can use
this to derive compactness results for spaces of holomorphic curves,
which then encode symplectic topological information about the space in
which these curves live.  We'll come back to this theme again and again.

\section{Hamiltonian systems and symplectic manifolds}
\label{sec:Hamiltonian}

To motivate the study of symplectic manifolds in general, let us see how
symplectic structures arise naturally in classical mechanics.  We shall
only sketch the main ideas here; a good comprehensive introduction may be
found in \cite{Arnold}.

Consider
a mechanical system with ``$n$ degrees of freedom'' moving under the
influence of a Newtonian potential~$V$.  This means there are $n$
``position'' variables $q = (q_1,\ldots,q_n) \in \RR^n$, which are functions
of time~$t$ that satisfy the second order differential equation
\begin{equation}
\label{eqn:Newton}
m_i \ddot{q}_i = - \frac{\p V}{\p q_i},
\end{equation}
where $m_i > 0$ are constants representing the masses
of the various particles, and $V : \RR^n \to \RR$ is a smooth function,
the ``potential''.  The space $\RR^n$, through which the vector $q(t)$
moves, is called the \defin{configuration space} of the system.
The basic idea of Hamiltonian mechanics is to turn
this 2nd order system into a 1st order system by introducing an extra set of
``momentum'' variables $p = (p_1,\ldots,p_n) \in \RR^n$, where 
$p_i = m_i\dot{q}_i$.  The space $\RR^{2n}$ with coordinates $(p,q)$
is then called \defin{phase space}, and we define a real-valued function on
phase space called the \defin{Hamiltonian}, by
$$
H : \RR^{2n} \to \RR : (p,q) \mapsto \frac{1}{2} \sum_{i=1}^n \frac{p_i^2}{m_i}
+ V(q).
$$
Physicists will recognize this as the ``total energy'' of the system, but
its main significance in the present context is that the combination of the
second order system \eqref{eqn:Newton} with our definition of $p$ 
is now equivalent to the $2n$ first order equations,
\begin{equation}
\label{eqn:Hamilton}
\dot{q}_i = \frac{\p H}{\p p_i},
\qquad
\dot{p}_i = -\frac{\p H}{\p q_i}.
\end{equation}
These are \defin{Hamilton's equations} for motion in phase space.

The motion of $x(t) := (p(t),q(t))$ in $\RR^{2n}$ can be described in more
geometric terms: it is an orbit of the vector field
\begin{equation}
\label{eqn:XHexplicit}
X_H(p,q) = \sum_{i=1}^n \left( \frac{\p H}{\p p_i} \frac{\p}{\p q_i} -
\frac{\p H}{\p q_i} \frac{\p}{\p p_i} \right).
\end{equation}
As we'll see in a moment, vector fields of this form have some important 
properties that have nothing
to do with our particular choice of the function~$H$, thus it is sensible
to call any vector field defined by this formula (for an arbitrary smooth
function $H : \RR^{2n} \to \RR$) a \defin{Hamiltonian vector field}.
This is where the symplectic structure enters the story.

\begin{exercise}
\label{ex:characterize}
Show that the vector field $X_H$ of \eqref{eqn:XHexplicit} 
can be characterized as
the unique vector field on $\RR^{2n}$ that satisfies 
$\omega\std(X_H,\cdot) = - dH$.
\end{exercise}

The above exercise shows that the symplectic structure makes it 
possible to write down a much
simplified definition of the Hamiltonian vector field.  Now we can
already prove something slightly impressive.

\begin{prop}
\label{prop:flowSymplectic}
The flow $\varphi_H^t$ of $X_H$ satisfies
$(\varphi_H^t)^*\omega\std = \omega\std$ for all~$t$.
\end{prop}
\begin{proof}
Using Cartan's formula for the Lie derivative of a form, together with
the characterization of $X_H$ in Exercise~\ref{ex:characterize} and
the fact that $\omega\std$ is closed, we compute
$\Lie_{X_H}\omega\std = d \iota_{X_H}\omega\std + \iota_{X_H} d\omega\std = -d^2 H = 0$.
\end{proof}

By Exercise~\ref{ex:omega0}, one can compute volumes on
$\RR^{2n}$ by integrating the $n$-fold product $\omega\std \wedge \ldots
\wedge \omega\std$, thus an immediate consequence of 
Prop.~\ref{prop:flowSymplectic} is the following:

\begin{cor}[Liouville's theorem]
The flow of $X_H$ is volume preserving.
\end{cor}

Notice that in most of this discussion we've not used our precise knowledge
of the $2$-form $\omega\std$ or function~$H$.  Rather, we've used the fact
that $\omega\std$ is nondegenerate (to characterize $X_H$ via $\omega\std$
in Exercise~\ref{ex:characterize}), and
the fact that it's closed (in the proof of Prop.~\ref{prop:flowSymplectic}).
It is therefore natural to generalize as follows.

\begin{defs}
\label{defn:symplectic}
A \defin{symplectic form} on a $2n$-dimensional manifold $M$ is a smooth
differential $2$-form $\omega$ that is both closed and nondegenerate.
The pair $(M,\omega)$ is then called a \defin{symplectic manifold}.
Given a smooth function $H : M \to \RR$, the corresponding \defin{Hamiltonian
vector field} is defined to be the unique vector field
$X_H \in \Vectors(M)$ such that\footnote{Some sources in the literature
define $X_H$ by $\omega(X_H,\cdot) = dH$, in which case
one must choose different sign conventions for the orientation of phase space
and definition of~$\omega\std$.  One must always be careful not to mix sign
conventions from different sources---that way you could prove anything!}
\begin{equation}
\label{eqn:XH}
\omega(X_H,\cdot) = -dH.
\end{equation}
For two symplectic manifolds $(M_1,\omega_1)$ and $(M_2,\omega_2)$, a smooth
map $\varphi : M_1 \to M_2$ is called \defin{symplectic} if 
$\varphi^*\omega_2 = \omega_1$.  If $\varphi$ is a symplectic embedding,
then we say that $\varphi(M_1)$ is a \defin{symplectic submanifold} of
$(M_2,\omega_2)$.  If $\varphi$ is symplectic and is also a
diffeomorphism, it is called a \defin{symplectomorphism}, and we then say
that $(M_1,\omega_1)$ and $(M_2,\omega_2)$ are \defin{symplectomorphic}.
\end{defs}

Repeating verbatim the argument of Prop.~\ref{prop:flowSymplectic}, we see
now that any Hamiltonian vector field on a symplectic manifold $(M,\omega)$
defines a smooth $1$-parameter family of symplectomorphisms.  If we define
volumes on $M$ by integrating the $2n$-form $\omega^n$
(see Exercise~\ref{ex:volume}), then
all symplectomorphisms are volume preserving---in particular this applies
to the flow of~$X_H$.

\begin{remark}
An odd-dimensional manifold can never admit a nondegenerate $2$-form.
(Why not?)
\end{remark}

\section{Some favorite examples}
\label{sec:examples}

We now give a few examples of symplectic manifolds (other than
$(\RR^{2n},\omega\std)$) which will be useful to have in mind.

\begin{example}
Suppose $N$ is any smooth $n$-manifold and $(q_1,\ldots,q_n)$ are a
choice of coordinates on an open subset $\uU \subset N$.  These naturally
define coordinates $(p_1,\ldots,p_n,q_1,\ldots,q_n)$
on the cotangent bundle $T^*\uU \subset T^*N$, where an arbitrary
cotangent vector at $q \in \uU$ is expressed as
$$
p_1 \,dq_1 + \ldots + p_n \,dq_n.
$$
Interpreted differently, this expression also defines a smooth $1$-form
on $T^*\uU$; we abbreviate it by $p\,dq$.
\begin{exercise}
Show that the $1$-form $p\,dq$ doesn't actually 
depend on the choice of coordinates $(q_1,\ldots,q_n)$.
\end{exercise}
What the above exercise reveals is that $T^*N$ globally
admits a \emph{canonical} $1$-form $\lambda$, whose expression in 
the local coordinates $(p,q)$ always looks
like $p\,dq$.  Moreover, $d\lambda$ is clearly a symplectic form, as it
looks exactly like \eqref{eqn:omega0} in coordinates.  We call this the
canonical symplectic form on $T^*N$.  Using this symplectic structure, the
cotangent bundle can be thought of as the ``phase space'' of a smooth
manifold, and is a natural setting for studying Hamiltonian systems when
the configuration space is something other than a Euclidean vector 
space (e.g.~a ``constrained'' mechanical system).
\end{example}

\begin{example}
On any oriented surface $\Sigma$, a $2$-form $\omega$ is symplectic if 
and only if it is
an area form, and the symplectomorphisms are precisely the area-preserving 
diffeomorphisms.  Observe that one can always find area-preserving 
diffeomorphisms between small open subsets of $(\RR^2,\omega\std)$
and $(\Sigma,\omega)$, thus every point in $\Sigma$ has a neighborhood
admitting local coordinates $(p,q)$ in which $\omega = dp \wedge dq$.
\end{example}

\begin{example}
\label{ex:CPn}
A more interesting example of a closed symplectic manifold is the
$n$-dimensional complex projective space $\CC P^n$.  This is both a
real $2n$-dimensional symplectic manifold and a complex $n$-dimensional
manifold, as we will now show.  By definition, $\CC P^n$ is the space of
complex lines in $\CC^{n+1}$, which we can express in two equivalent
ways as follows:
$$
\CC P^n = ( \CC^{n+1} \setminus \{0\}) / \CC^* = S^{2n + 1} / S^1.
$$
In the first case, we divide out the natural free action (by scalar
multiplication) of the multiplicative group
$\CC^* := \CC \setminus\{0\}$ on $\CC^{n+1} \setminus \{0\}$, and the
second case is the same thing but restricting to the unit sphere
$S^{2n + 1} \subset \CC^{n+1} = \RR^{2n+2}$ and unit circle
$S^1 \subset \CC = \RR^2$.  To define a symplectic form, consider first
the $1$-form $\lambda$ on $S^{2n+1}$ defined for $z \in S^{2n+1} \subset
\CC^{n+1}$ and $X \in T_z S^{2n+1} \subset \CC^{n+1}$ by
$$
\lambda_z(X) = \langle iz , X \rangle,
$$
where $\langle\ ,\ \rangle$ is the standard Hermitian inner product on
$\CC^{n+1}$.  (Take a moment to convince yourself that this expression is
always real.)  Since $\lambda$ is clearly invariant under the $S^1$-action
on $S^{2n + 1}$, the same is true for the closed $2$-form $d\lambda$,
which therefore descends to a closed $2$-form $\omega\std$ on $\CC P^n$.
\begin{exercise}
Show that $\omega\std$ as defined above is symplectic.
\end{exercise}

The complex manifold structure of $\CC P^n$ can be seen explicitly by
thinking of points in $\CC P^n$ as equivalence classes of vectors
$(z_0,\ldots,z_n) \in \CC^{n+1} \setminus \{0\}$, 
with two vectors equivalent if they are complex multiples of each other.
We will always write the equivalence class represented 
by $(z_0,\ldots,z_n) \in \CC^{n+1}\setminus\{0\}$ as
$$
[ z_0 : \ldots : z_n ] \in \CC P^n.
$$
Then for each $k = 0,\ldots,n$, there is an embedding
\begin{equation}
\label{eqn:complexChart}
\iota_k : \CC^n \hookrightarrow \CC P^n : (z_1,\ldots,z_n) \mapsto
[ z_1 : \ldots,z_{k-1} : 1 : z_k : \ldots : z_n ],
\end{equation}
whose image is the complement of the subset
$$
\CC P^{n-1} \cong \{ [ z_1 : \ldots : z_{k-1} : 0 : z_k : \ldots : z_n ] \in 
\CC P^n \ |\ (z_1,\ldots,z_n) \in \CC^n \}.
$$
\begin{exercise}
Show that if the maps $\iota_k^{-1}$ are thought of as complex coordinate
charts on open subsets of $\CC P^n$, then the transition maps
$\iota_k^{-1} \circ \iota_j$ are all holomorphic.
\end{exercise}
By the exercise, $\CC P^n$ naturally carries the structure of a complex 
manifold such that the embeddings $\iota_k : \CC^n \to \CC P^n$ are 
holomorphic.  Each of these embeddings also 
defines a decomposition of $\CC P^n$ into
$\CC^n \cup \CC P^{n-1}$, where $\CC P^{n-1}$ is a complex submanifold of
(complex) codimension one.
The case $n=1$ is particularly enlightening,
as here the decomposition becomes $\CC P^1 = \CC \cup \{ \text{point} \}
\cong S^2$; this is simply the Riemann sphere with its natural complex
structure, where the ``point at infinity'' is $\CC P^0$.  
In the case $n=2$, we have $\CC P^2 \cong \CC^2 \cup \CC P^1$, and we'll
occasionally refer to the complex submanifold $\CC P^1 \subset \CC P^2$
as the ``sphere at infinity''.
\end{example}

We continue for a moment with the example of $\CC P^n$ in order to observe
that it contains an abundance of holomorphic spheres.  Take for instance
the case $n=2$: then for any $\zeta \in \CC$, we claim that the holomorphic
embedding
$$
u_\zeta : \CC \to \CC^2 : z \mapsto (z,\zeta)
$$
extends naturally to a holomorphic embedding of $\CC P^1$ in $\CC P^2$.
Indeed, using $\iota_2$ to include $\CC^2$ in $\CC P^2$,
$u_\zeta(z)$ becomes the point $[z : \zeta : 1 ] = [ 1 : \zeta/z : 1/z ]$, and
as $z \to \infty$, this converges to the point $x_0 := [ 1 : 0 : 0 ]$ 
in the sphere
at infinity.  One can check using alternate charts that this extension is
indeed a holomorphic map.  The collection of all these embeddings
$u_\zeta : \CC P^1 \to \CC P^2$ thus gives a very nice decomposition of
$\CC P^2$: together with the sphere at infinity, they foliate the region
$\CC P^2 \setminus \{ x_0 \}$, but all intersect precisely at~$x_0$
(see Figure~\ref{fig:CP2}).
This decomposition will turn out to be crucial in the proof of
Theorem~\ref{thm:CP2}, stated below.

\begin{figure}
\includegraphics{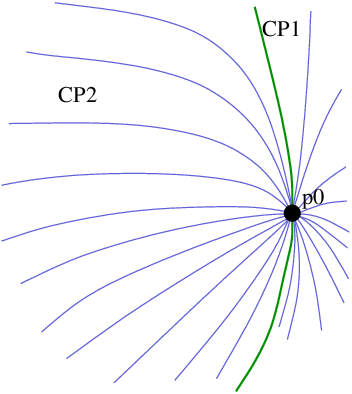}
\caption{\label{fig:CP2} $\CC P^2 \setminus \{x_0\}$ is foliated by
holomorphic spheres that all intersect at~$x_0$.}
\end{figure}

\section{Darboux's theorem and the Moser deformation trick}
\label{sec:Moser}

In Riemannian geometry, two Riemannian manifolds of the same dimension with
different metrics can have quite different local structures: there can be
no isometries between them, not even \emph{locally}, unless they have the
same curvature.  The following basic result of symplectic geometry
shows that in the symplectic world, things are quite different.  We will
give a proof using the beautiful \emph{Moser deformation trick}, which
has several important applications throughout symplectic and
contact geometry, as we'll soon see.\footnote{An alternative approach
to Darboux's theorem may be found in \cite{Arnold}.}

\begin{thm}[Darboux's theorem]
Near every point in a symplectic manifold $(M,\omega)$, there are local
coordinates
$(p_1,\ldots,p_n,q_1,\ldots,q_n)$ in which $\omega = \sum_i dp_i \wedge dq_i$.
\end{thm}
\begin{proof}
Denote by $(p_1,\ldots,p_n,q_1,\ldots,q_n)$ the standard coordinates
on $\RR^{2n}$ and define the standard symplectic form $\omega\std$
by \eqref{eqn:omega0}; this is the exterior derivative of the $1$-form
$$
\lambda\std = \sum_j p_j\,dq_j.
$$
Since the statement in the theorem is purely
local, we can assume (by choosing local coordinates) that
$M$ is an open neighborhood of the origin in $\RR^{2n}$, on which
$\omega$ is any closed, nondegenerate $2$-form.  Then it will suffice
to find two open neighborhoods $\uU, \uU_0 \subset \RR^{2n}$ of~$0$,
and a diffeomorphism
$$
\varphi : \uU_0 \to \uU
$$
preserving~$0$
such that $\varphi^*\omega = \omega\std$.  Using Exercise~\ref{EX:sympLinAlg}
below (the ``linear Darboux's theorem''), we can also assume after a linear
change of coordinates that $\varphi^*\omega$ and $\omega\std$ match
at the origin.

The idea behind the Moser trick is now the following bit of optimism:
we assume that the desired diffeomorphism $\varphi$ is the time~$1$ flow
of a time-dependent vector field defined near~$0$, and derive conditions
that this vector field must satisfy.  In fact, we will be a bit more
ambitious: consider the smooth $1$-parameter family of $2$-forms
$$
\omega_t = t \omega + (1 - t) \omega\std,
\qquad
t \in [0,1]
$$
which interpolate between $\omega\std$ and $\omega$.  These are all
obviously closed, and if we restrict to a sufficiently small neighborhood
of the origin then they are near $\omega\std$ and thus nondegenerate.
Our goal is to find a time-dependent vector field $Y_t$ on some neighborhood
of~$0$, for $t \in [0,1]$, whose flow $\varphi_t$ is well defined on some 
smaller neighborhood of~$0$ and satisfies
$$
\varphi_t^*\omega_t = \omega\std
$$
for all $t \in [0,1]$.  Differentiating this expression with respect to~$t$
and writing $\dot{\omega}_t := \frac{\p}{\p t}\omega_t$, we find
$$
\varphi_t^* \Lie_{Y_t}\omega_t + \varphi_t^*\dot{\omega}_t = 0,
$$
which by Cartan's formula and the fact that $\omega_t$ is closed
and $\varphi_t$ is a diffeomorphism, implies
\begin{equation}
\label{eqn:Moser}
d \iota_{Y_t}\omega_t + \dot{\omega}_t = 0.
\end{equation}
At this point it's useful to observe that if we restrict to a contractible
neighborhood of the origin, $\omega$ (and hence also $\omega_t$) is exact:
let us write
$$
\omega = d\lambda.
$$
Moreover, by adding a constant $1$-form, we can choose $\lambda$ so that
it matches $\lambda\std$ at the origin.  Now if $\lambda_t := t\lambda +
(1-t) \lambda\std$, we have $d\lambda_t = \omega_t$, and
$\dot{\lambda}_t := \frac{\p}{\p t}\lambda_t = \lambda - \lambda\std$
vanishes at the origin.  Plugging this into \eqref{eqn:Moser}, we see now
that it suffices to find a vector field $Y_t$ satisfying
\begin{equation}
\label{eqn:Moser2}
\omega_t(Y_t,\cdot) = - \dot{\lambda}_t.
\end{equation}
Since $\omega_t$ is nondegenerate, this equation can be solved and determines
a unique vector field $Y_t$, which vanishes at the origin since 
$\dot{\lambda}_t$ does.  
The flow $\varphi_t$ therefore exists for all $t \in [0,1]$ on a
sufficiently small neighborhood of the origin, and $\varphi_1$ is the
desired diffeomorphism.
\end{proof}
\begin{exercise}
\label{EX:sympLinAlg}
The following linear version of Darboux's theorem is an easy exercise
in linear algebra and was the first step in the proof above:
show that if $\Omega$ is any nondegenerate, antisymmetric bilinear
form on $\RR^{2n}$, then there exists a basis
$(X_1,\ldots,X_n,Y_1,\ldots,Y_n)$ such that
$$
\Omega(X_i,Y_i) = 1
$$
and $\Omega$ vanishes on all other pairs of basis vectors.
This is equivalent to the statement that $\RR^{2n}$ admits a linear
change of coordinates in which $\Omega$ looks like the standard
symplectic form~$\omega\std$.
\end{exercise}

It's worth pointing out the crucial role played in the above proof by the 
relation \eqref{eqn:Moser2}, which is almost the same as the relation
used to define Hamiltonian vector fields \eqref{eqn:XH}.  The latter,
together with the argument of Prop.~\ref{prop:flowSymplectic},
tells us that the group of symplectomorphisms on a symplectic manifold
is fantastically large, as it contains all the flows of Hamiltonian
vector fields, which are determined by arbitrary smooth real-valued
functions.  For much the same reason, one can also always find an
abundance of symplectic local coordinate charts
(usually called \emph{Darboux coordinates}).  Contrast this with the
situation on a Riemannian manifold, where the group of isometries is
generally finite dimensional, and different metrics are usually not
locally equivalent, but are distinguished by their curvature.

In light of Darboux's theorem, we can now give the following 
equivalent definition of a symplectic manifold:

\begin{defn}
\label{defn:sympMfd}
A \defin{symplectic manifold} is a $2n$-dimensional manifold $M$ together with
an atlas of coordinate charts whose transition maps are symplectic
(with respect to the standard symplectic structure of $\RR^{2n}$).
\end{defn}

In physicists' language, a symplectic manifold is thus a manifold that
can be identified locally with Hamiltonian phase space, in the sense that
all coordinate changes leave the form of Hamilton's equations unaltered.

Let us state one more important application of the Moser trick, this time
of a more global nature.  Recall that two symplectic manifolds 
$(M,\omega)$ and $(M',\omega')$ are called \emph{symplectomorphic}
if there exists a symplectomorphism between them, i.e.~a diffeomorphism
$\varphi : M \to M'$ such that $\varphi^*\omega' = \omega$.  
Working on a single manifold $M$, we say similarly that two symplectic
structures $\omega$ and $\omega'$ are 
symplectomorphic\footnote{The words ``isomorphic'' and ``diffeomorphic''
can also be used here as synonyms.} if $(M,\omega)$ and $(M,\omega')$
are symplectomorphic.  This is the most obvious notion of equivalence for
symplectic structures, but there are others that are also worth considering.

\begin{defn}
\label{defn:isotopic}
Two symplectic structures $\omega$ and $\omega'$ on $M$ are called
\defin{isotopic} if there is a symplectomorphism $(M,\omega) \to (M,\omega')$
that is isotopic to the identity.
\end{defn}

\begin{defn}
\label{defn:deformation}
Two symplectic structures $\omega$ and $\omega'$ on $M$ are called
\defin{deformation equivalent} if $M$ admits a \defin{symplectic deformation}
between them, i.e.~a smooth family of
symplectic forms $\{ \omega_t \}_{t \in [0,1]}$ such that
$\omega_0 = \omega$ and $\omega_1 = \omega'$.  Similarly, two
symplectic manifolds
$(M,\omega)$ and $(M',\omega')$ are deformation equivalent if there
exists a diffeomorphism $\varphi : M \to M'$ such that $\omega$ and
$\varphi^*\omega'$ are deformation equivalent.
\end{defn}

It is clear that if two symplectic forms are isotopic then they are also
both symplectomorphic
and deformation equivalent.  It is not true, however, that a symplectic
deformation always gives rise to an isotopy: one should not expect this,
as isotopic symplectic forms on $M$ must always represent the same
cohomology class in $H^2_\dR(M)$, whereas the cohomology class can
obviously vary under general deformations.  The remarkable fact is that
this necessary condition is also sufficient!

\begin{thm}[Moser's stability theorem]
Suppose $M$ is a closed manifold with a smooth $1$-parameter family of
symplectic forms $\{ \omega_t \}_{[t \in [0,1]}$ which all represent the
same cohomology class in $H^2_\dR(M)$.  Then there exists a smooth isotopy
$\{ \varphi_t : M \to M \}_{t \in [0,1]}$, with $\varphi_0 = \Id$ and
$\varphi_t^*\omega_t = \omega_0$.
\end{thm}

\begin{exercise}
Use the Moser isotopy trick to prove the theorem.  \textsl{Hint:
In the proof of Darboux's theorem, we had to use the fact that
symplectic forms are locally exact in order to get from
\eqref{eqn:Moser} to \eqref{eqn:Moser2}.  Here you will find the
cohomological hypothesis helpful for the same reason.  If you get stuck,
see \cite{McDuffSalamon:ST}.}
\end{exercise}

\begin{exercise}
\label{EX:CPndeformation}
Show that if $\omega$ and $\omega'$ are two deformation equivalent symplectic
forms on $\CC P^n$, then $\omega$ is isotopic to $c \omega'$ for some
constant $c > 0$.
\end{exercise}

\section{From symplectic geometry to symplectic topology}
\label{sec:sympTop}

As a consequence of Darboux's theorem, symplectic manifolds
have no local invariants---there is no ``local symplectic geometry''.
Globally things are different, and here there are a number 
of interesting questions one can ask, all of which fall under the heading of
\emph{symplectic topology}.  (The word ``topology'' is used to indicate
the importance of global rather than local phenomena.)

The most basic such question concerns the classification of symplectic
structures.  One can ask, for example, whether there exists a symplectic
manifold $(M,\omega)$ that is diffeomorphic to $\RR^4$ but not
symplectomorphic to $(\RR^4,\omega\std)$, i.e.~an ``exotic'' 
symplectic~$\RR^4$.  The answer turns out to be 
yes---exotic $\RR^{2n}$'s exist in fact for all~$n$,
see \cite{AudinLalondePolterovich}---but it changes if
we prescribe the behavior of $\omega$ at infinity.  The following result
says that $(\RR^{2n},\omega\std)$ is actually the only aspherical 
symplectic manifold that is ``standard at infinity''.

\begin{thm}[Gromov \cite{Gromov}]
Suppose $(M,\omega)$ is a symplectic $4$-manifold with $\pi_2(M) = 0$, 
and there are compact subsets $K \subset M$ and $\Omega \subset \RR^4$
such that $(M\setminus K,\omega)$ and $(\RR^4\setminus\Omega,\omega\std)$
are symplectomorphic.  Then $(M,\omega)$ is symplectomorphic to
$(\RR^4,\omega\std)$.
\end{thm}

In a later chapter we will be able to prove a stronger version of 
this statement, as a corollary of some classification results for
symplectic fillings of contact manifolds (cf.~Theorem~\ref{thm:fillingsS3}).

Another interesting question is the following:
suppose $(M_1,\omega_1)$ and $(M_2,\omega_2)$ are symplectic manifolds of
the same dimension~$2n$, possibly with boundary, such that there exists a
smooth embedding $M_1 \hookrightarrow M_2$.  Can one also find a
\emph{symplectic} embedding $(M_1,\omega_1) \hookrightarrow (M_2,\omega_2)$?
What phenomena related to the symplectic structures can prevent this?
There's one obstruction that jumps out immediately: there can be no such
embedding unless
$$
\int_{M_1} \omega_1^n \le \int_{M_2} \omega_2^n,
$$
i.e.~$M_1$ has no more volume than~$M_2$.  In dimension two there's nothing
more to say, because symplectic and area-preserving maps are the same thing.
But in dimension~$2n$ for $n \ge 2$, it was not known for a long time
whether there are obstructions to symplectic embeddings other than the
volume.  A good thought experiment along these lines is the
``squeezing'' question: denote by $B_r^{2n}$ the ball of radius $r$ about
the origin in $\RR^{2n}$.  Then it's fairly obvious that for any
$r, R > 0$ one can always find a volume-preserving embedding
$$
B_r^{2n} \hookrightarrow B_R^2 \times \RR^{2n - 2},
$$
even if $r > R$, for then one can ``squeeze'' the first two dimensions of
$B_r^{2n}$ into $B_R^2$ but make up for it by spreading out further
in $\RR^{2n - 2}$.  But can one do this \emph{symplectically}?
The answer was provided by the following groundbreaking result:

\begin{thm}[Gromov's ``nonsqueezing'' theorem \cite{Gromov}]
There exists a symplectic embedding of $(B_r^{2n},\omega\std)$ into
$(B_R^2 \times \RR^{2n - 2},\omega\std)$ if and only if $r \le R$.
\end{thm}

This theorem was one of the first important applications of pseudoholomorphic
curves.  We will prove it in Chapter~\ref{chapter:nonsqueezing}, and
will spend a great deal of time in the next few chapters
learning the technical machinery that is needed to understand the proof.

We will close this brief introduction to symplectic topology by sketching
the proof of a result that was introduced in \cite{Gromov} and later 
generalized by McDuff, and provides us with a good excuse to introduce
$J$-holomorphic curves.  Recall from \S\ref{sec:examples} that
$\CC P^2$ admits a singular foliation by embedded spheres
that all intersect each other at one point, and all can be parametrized
by holomorphic maps $\CC P^1 \to \CC P^2$.  One can check that these
spheres are also symplectic submanifolds with respect to the standard
symplectic structure $\omega\std$ introduced in Example~\ref{ex:CPn};
moreover, they intersect each other
positively, so their self-intersection numbers are always~$1$.
The following result essentially says that the existence of such a 
symplectically
embedded sphere is a rare phenomenon: it can only occur in a very specific
set of symplectic $4$-manifolds, of which $(\CC P^2,\omega\std)$ is the simplest.
It also illustrates an important feature of symplectic topology
specifically in four dimensions: once you find a single holomorphic curve
with sufficiently nice local properties, it can sometimes fully
determine the manifold in which it lives.

\begin{thm}[M.~Gromov \cite{Gromov} and D.~McDuff \cite{McDuff:rationalRuled}]
\label{thm:CP2}
Suppose $(M,\omega)$ is a closed and connected 
symplectic $4$-manifold containing
a symplectically embedded $2$-sphere $C \subset M$ with self-intersection
$C \cdot C = 1$, but no symplectically embedded $2$-sphere with
self-intersection~$-1$.  Then $(M,\omega)$ is symplectomorphic to
$(\CC P^2,c \omega\std)$, where $c > 0$ is a constant and $\omega\std$ is
the standard symplectic form on $\CC P^2$.
\end{thm}

The idea of the proof is to choose appropriate data so that the symplectic
submanifold $C \subset M$ can be regarded in some sense as a holomorphic
curve, and then analyze the global structure of the space of 
holomorphic curves to which
it belongs.  It turns out that for a combination of analytical and
topological reasons, this space will contain a smooth family of embedded
holomorphic spheres that fill all of~$M$ and all intersect each other at
one point, thus reproducing the singular foliation of
Figure~\ref{fig:CP2}.  This type of decomposition is a well-known object
in algebraic geometry and has more recently become quite popular in
symplectic topology as well: it's called a \emph{Lefschetz pencil}.
As we'll see when we generalize Theorem~\ref{thm:CP2} in a later chapter, 
there is an intimate connection between isotopy classes of Lefschetz
pencils and deformation classes of symplectic structures: in the present
case, the existence of this Lefschetz pencil implies that $(M,\omega)$
is symplectically deformation equivalent to $(\CC P^2,\omega\std)$, and thus
also symplectomorphic due to the Moser stability theorem
(see Exercise~\ref{EX:CPndeformation}).

The truly nontrivial part of the proof is the analysis of the moduli
space of holomorphic curves, and this is what we'll concentrate on for
the next several chapters.  As a point of departure, consider the
formulation \eqref{eqn:CRintegrable} of the Cauchy-Riemann equations at
the beginning of this chapter.  Here $u$ was a map
from an open subset of $\CC^m$ into $\CC^n$, but one can also make sense
of \eqref{eqn:CRintegrable} when $u$ is a map between two complex 
manifolds.  In such a situation, $u$ is called holomorphic if and only if 
it looks holomorphic in any choice of holomorphic local coordinates.
To put this in coordinate-free language, the tangent spaces of any complex
manifold $X$ are naturally complex vector spaces, on which multiplication
by $i$ makes sense, thus defining a natural bundle endomorphism
$$
i : TX \to TX
$$
that satisfies $i^2 = -\1$.  Then \eqref{eqn:CRintegrable}
makes sense globally and is the equation defining holomorphic maps between
any two complex manifolds.

In the present situation, we're interested in smooth maps
$u : \CC P^1 \to M$.  The domain is thus a complex manifold, but the target
might not be, which means we lack an ingredient needed to write down the
right hand side of \eqref{eqn:CRintegrable}.  It turns out that 
in any symplectic manifold, one
can always find an object to fill this role, i.e.~a fiberwise linear map
$J : TM \to TM$ with the following properties:
\begin{itemize}
\item $J^2 = -\1$,
\item $\omega(\cdot,J\cdot)$ defines a Riemannian metric on $M$.
\end{itemize}
The first condition allows us to interpret $J$ as ``multiplication by~$i$'',
thus turning the tangent spaces of $M$ into complex vector spaces.
The second reproduces the relation between $i$ and $\omega\std$ that
exists in $\RR^{2n}$, thus generalizing the important interaction between
\emph{symplectic} and \emph{complex} that we illustrated in \S\ref{sec:warmup}:
complex subspaces of $TM$ are also symplectic, and their areas can be
computed in terms of~$\omega$.  These conditions make $J$ into a
\emph{compatible almost complex structure} on $(M,\omega)$; we will prove
the fundamental existence result for these
by fairly elementary methods in \S\ref{sec:compatible}.
Now, the fact that $C$ is embedded in $M$
\emph{symplectically} also allows us to arrange the following additional
condition:
\begin{itemize}
\item the tangent spaces $TC \subset TM$ are invariant under~$J$.
\end{itemize}

We are thus ready to introduce the following generalization
of the Cauchy-Riemann equation: consider smooth maps
$u : \CC P^1 \to M$ whose differential is a complex-linear map at every
point,~i.e.
\begin{equation}
\label{eqn:nonlinearCR0}
Tu \circ i = J \circ Tu.
\end{equation}
Solutions to \eqref{eqn:nonlinearCR0} are called \emph{pseudoholomorphic},
or more specifically,
\emph{$J$-holomorphic} spheres in~$M$.  Now pick a point $x_0 \in C$ and
consider the following space of $J$-holomorphic spheres,
\begin{equation*}
\begin{split}
\Mod := \{ u \in C^\infty(\CC P^1,M) \ |\ & Tu \circ i = J \circ Tu,\\
& u_*[\CC P_1] = [C] \in H_2(M),\\
& u(0) = x_0 \} / \sim,
\end{split}
\end{equation*}
where $u \sim u'$ if there is a holomorphic diffeomorphism $\varphi :
\CC P^1 \to \CC P^1$ such that $u' = u \circ \varphi$ and $\varphi(0) = 0$.
We assign to $\Mod$ the natural topology defined by $C^\infty$-convergence
of smooth maps $\CC P^1 \to M$.

\begin{lemma}
$\Mod$ is not empty: in particular it contains an embedded 
$J$-ho\-lo\-mor\-phic sphere whose image is~$C$.
\end{lemma}
\begin{proof}
Since $C$ has $J$-invariant tangent spaces, any diffeomorphism
$u_0 : \CC P^1 \to C$ with $u_0(0) = x_0$ allows us to pull back $J$ to 
an almost complex structure $j := u_0^*J$ on~$\CC P^1$.  As we'll review
in Chapter~\ref{chapter:moduli}, 
the uniqueness of complex structures on $S^2$ then
allows us to find a diffeomorphism $\varphi : \CC P^1 \to \CC P^1$ such
that $\varphi(0) = 0$ and $\varphi^*j = i$, thus the desired curve is
$u := u_0 \circ \varphi$.
\end{proof}

The rest of the work is done by the following rather powerful lemma,
which describes the global structure of $\Mod$.
Its proof requires a substantial volume of analytical machinery which we will
develop in the coming chapters; note that since $M$ is not a complex
manifold, the methods of complex analysis play only a minor role in this
machinery, and are subsumed in particular by the theory of nonlinear
elliptic PDEs.  This is the point where we need the technical
assumptions that $C \cdot C = 1$ and $M$ contains no symplectic spheres
of self-intersection~$-1$,\footnote{As we'll see, the assumption of no
symplectic spheres with self-intersection~$-1$ is a surprisingly weak one:
it can always be attained by modifying $(M,\omega)$ in a standard way known
as ``blowing down''.} as such topological conditions figure into the
index computations that determine the local structure of~$\Mod$.

\begin{lemma}
$\Mod$ is compact and admits the structure of a smooth $2$-dimensional
manifold.  Moreover, the curves in $\Mod$ are all embeddings that do not
intersect each other except at the point $x_0$; in particular, they
foliate $M \setminus \{x_0\}$.
\end{lemma}

By this result, the curves in $\Mod$ form the fibers of a symplectic
Lefschetz pencil on $(M,\omega)$, so that the latter's diffeomorphism and
symplectomorphism type are completely determined by the moduli space of
holomorphic curves.

\section{Contact geometry and the Weinstein conjecture}
\label{sec:Weinstein}

Contact geometry is often called the ``odd-dimensional cousin'' of
symplectic geometry, and one context in which it arises naturally is in the
study of Hamiltonian dynamics.  Again we shall only sketch the main ideas;
the book \cite{HoferZehnder} is recommended for a more detailed account.

Consider a $2n$-dimensional 
symplectic manifold $(M,\omega)$ with a Hamiltonian
$H : M \to \RR$.  By the definition of the Hamiltonian vector field,
$dH(X_H) = -\omega(X_H,X_H) = 0$, thus the flow of $X_H$ preserves the
level sets 
$$
S_c := H^{-1}(c)
$$
for $c \in \RR$.  If $c$ is a regular value of $H$ then $S_c$ is a smooth
manifold of dimension $2n-1$, called a \defin{regular energy surface}, 
and $X_H$ restricts to a nowhere zero vector field on~$S_c$.

\begin{exercise}
\label{ex:charLine}
If $S_c \subset M$ is a regular energy surface, show that the direction of
$X_H$ is uniquely determined by the condition $\omega(X_H,\cdot)|_{TS_c} = 0$.
\end{exercise}
The directions in Exercise~\ref{ex:charLine} define the so-called 
\defin{characteristic line field} on~$S_c$: its existence implies that the
paths traced out on $S_c$ by orbits of $X_H$ depend only on $S_c$ and on
the symplectic structure, not on~$H$ itself.  In particular, a closed
orbit of $X_H$ on $S_c$ is merely a closed integral curve of the characteristic
line field.  It is thus meaningful to ask the following question:

\begin{question}
Given a symplectic manifold $(M,\omega)$ and a smooth hypersurface 
$S \subset M$, does the characteristic line field on $S$ have any closed
integral curves?
\end{question}

We shall often refer to closed integral curves of the characteristic line
field on $S \subset M$ simply as \emph{closed orbits on~$S$}.
There are examples of Hamiltonian systems that have no closed orbits
at all, cf.~\cite{HoferZehnder}*{\S 4.5}.
However, the following result (and the related result of
A.~Weinstein \cite{Weinstein:convex} for convex energy surfaces)
singles out a special class of hypersurfaces
for which the answer is always yes:

\begin{thm}[P.~Rabinowitz \cite{Rabinowitz:starshaped}]
\label{thm:starshaped}
Every star-shaped hypersurface in the standard symplectic $\RR^{2n}$ 
admits a closed orbit.
\end{thm}

Recall that a hypersurface $S \subset \RR^{2n}$ is called \defin{star-shaped}
if it doesn't intersect the origin and the projection
$\RR^{2n} \setminus \{0\} \to S^{2n-1} : z \mapsto z / |z|$ restricts to
a diffeomorphism $S \to S^{2n-1}$ (see Figure~\ref{fig:starshaped}).  In
particular, $S$ is then transverse to the radial vector field
\begin{equation}
\label{eqn:LiouvilleR}
V\std := \frac{1}{2} \sum_{i=1}^n \left( p_i \frac{\p}{\p p_i} + q_i \frac{\p}
{\p q_i} \right).
\end{equation}

\begin{figure}
\includegraphics{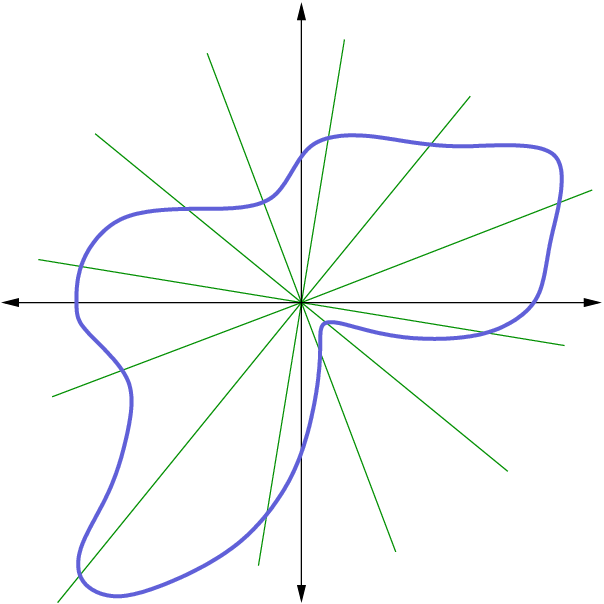}
\caption{\label{fig:starshaped} A star-shaped hypersurface in $\RR^2$.}
\end{figure}

\begin{exercise}
\label{ex:Liouville}
Show that the vector field $V\std$ of \eqref{eqn:LiouvilleR} satisfies
$\Lie_{V\std}\omega\std = \omega\std$.
\end{exercise}

\begin{defn}
A vector field $V$ on a symplectic manifold $(M,\omega)$ is called a
\defin{Liouville vector field} if it satisfies $\Lie_V \omega = \omega$.
\end{defn}

By Exercise~\ref{ex:Liouville}, star-shaped hypersurfaces in $\RR^{2n}$ are
always transverse to a Liouville vector field, and this turns out to be a
very special property.

\begin{defn}
A hypersurface $S$ in a symplectic manifold $(M,\omega)$ is said to be of
\defin{contact type} if some neighborhood of $S$ admits a Liouville vector
field that is transverse to~$S$.
\end{defn}

Given a closed contact type hypersurface $S \subset (M,\omega)$, 
one can use the flow of the Liouville vector field $V$ to produce a very nice
local picture of $(M,\omega)$ near~$S$.  Define a $1$-form on $S$ by
$$
\alpha = \iota_V\omega|_S,
$$
and choose $\epsilon > 0$ sufficiently small so that
$$
\Phi : (-\epsilon,\epsilon) \times S \to M : (t,x) \mapsto \varphi^t_V(x)
$$
is an embedding, where $\varphi_V^t$ denotes the flow of~$V$.
\begin{exercise} \ 
\label{ex:symplectization}
\begin{enumerate}
\renewcommand{\labelenumi}{(\alph{enumi})}
\item
Show that the flow of $V$ ``dilates'' the symplectic form, 
i.e.~$(\varphi_V^t)^*\omega = e^t\omega$.
\item
Show that $\Phi^*\omega = d(e^t \alpha)$, where we define $\alpha$ 
as a $1$-form on $(-\epsilon,\epsilon) \times S$ by pulling it back through 
the natural projection to~$S$.
\textsl{Hint: Show first that if $\lambda := \iota_V\omega$, then
$\Phi^*\lambda = e^t\alpha$, and notice that $d\lambda = \omega$ by the
definition of a Liouville vector field.}
\item
Show that $d\alpha$ restricts to a nondegenerate skew-symmetric $2$-form 
on the hyperplane field $\xi := \ker\alpha$ over $S$.  As
a consequence, $\xi$ is transverse to a smooth line field $\ell$ on $S$
characterized by the property that $X \in \ell$ if and only if
$d\alpha(X,\cdot) = 0$.
\item
Show that on each of the hypersurfaces $\{c\} \times S$ for
$c \in (-\epsilon,\epsilon)$, the line field $\ell$ defined above is
the characteristic line field with respect to
the symplectic form $d(e^t\alpha)$.
\end{enumerate}
\end{exercise}

Several interesting consequences follow from Exercise~\ref{ex:symplectization}.
In particular, the use of a Liouville vector field to identify a neighborhood
of $S$ with $(-\epsilon,\epsilon) \times S$ gives us a smooth family of
hypersurfaces $S_c := \{c\} \times S$ whose characteristic line fields all
have exactly the same dynamics.  This provides some intuitive motivation
to believe Theorem~\ref{thm:starshaped}: it's sufficient to find one 
hypersurface in the family $S_c$ that admits a periodic orbit, for then
they all do.  As it turns out, one can prove a variety of ``almost existence''
results in $1$-parameter families of hypersurfaces, e.g.~in 
$(\RR^{2n},\omega\std)$, a result of Hofer-Zehnder \cite{HoferZehnder:capacity}
and Struwe \cite{Struwe:almostExistence} implies that for any smooth
$1$-parameter family of hypersurfaces, almost every (in a measure theoretic
sense) hypersurface in the family admits a closed orbit.
This gives a proof of the following
generalization of Theorem~\ref{thm:starshaped}:

\begin{thm}[C.~Viterbo \cite{Viterbo:weinstein}]
\label{thm:WeinsteinR2n}
Every contact type hypersurface in $(\RR^{2n},\omega\std)$ admits a closed orbit.
\end{thm}

Having generalized this far, it's natural to wonder whether the crucial
properties of a contact hypersurface can be considered independently of its
embedding into a symplectic manifold.  The answer comes from the
$1$-form $\alpha$ and hyperplane distribution $\xi = \ker\alpha \subset TS$ 
in Exercise~\ref{ex:symplectization}.

\begin{defn}
A \defin{contact form} on a $(2n-1)$-dimensional manifold is a smooth
$1$-form $\alpha$ such that $d\alpha$ is always nondegenerate on
$\xi := \ker\alpha$.  The hyperplane distribution $\xi$ is then called a
\defin{contact structure}.
\end{defn}

\begin{exercise}
\label{ex:contact}
Show that the condition of $d\alpha$ being nondegenerate on $\xi = \ker\alpha$
is equivalent to $\alpha \wedge (d\alpha)^{n-1}$ being a volume form on~$S$,
and that $\xi$ is nowhere integrable if this is satisfied.
\end{exercise}

Given an orientation of $S$, we call the contact structure 
$\xi = \ker\alpha$ \defin{positive} if the orientation induced by
$\alpha \wedge (d\alpha)^{n-1}$ agrees with the given orientation.
One can show that if $S \subset (M,\omega)$ is a contact type hypersurface 
with the natural orientation induced from $M$ and a transverse
Liouville vector field, then the induced contact structure is always positive.

Note that Liouville vector fields are far from unique, in fact:
\begin{exercise}
Show that if $V$ is a Liouville vector field on $(M,\omega)$ and $X_H$ is
any Hamiltonian vector field, then $V + X_H$ is also a Liouville vector field.
\end{exercise}

Thus the contact form $\alpha = \iota_V\omega|_S$ induced on a contact type
hypersurface should not be considered an intrinsic property of the 
hypersurface.  As the next result indicates, the contact \emph{structure}
is the more meaningful object.

\begin{prop}
\label{prop:sameXi}
Up to isotopy, the contact structure $\xi = \ker \alpha$ induced on a
contact type hypersurface $S \subset (M,\omega)$ by $\alpha = \iota_V\omega|_S$
is independent of the choice of~$V$.
\end{prop}
The proof of this is a fairly easy exercise using a
standard fundamental result of contact geometry:
\begin{thm}[Gray's stability theorem]
\label{thm:Gray}
If $S$ is a closed $(2n-1)$-dimensional manifold and $\{\xi_t\}_{t \in [0,1]}$
is a smooth $1$-parameter family of contact structures on $S$, 
then there exists a smooth $1$-parameter family of diffeomorphisms
$\{ \varphi_t \}_{t\in[0,1]}$ such that $\varphi_0 = \Id$ and
$(\varphi_t)_*\xi_0 = \xi_t$.
\end{thm}
This is yet another application of the Moser deformation trick; we'll explain
the proof at the end of this section.  Note that the theorem provides an 
isotopy between any two deformation
equivalent contact \emph{structures}, but there is no such result for contact
\emph{forms}---that's one of the reaons why contact structures are considered 
to be more geometrically natural objects.

By now we hopefully have sufficient motivation to study odd-dimensional
manifolds with contact structures.  The pair $(S,\xi)$ is called a
\defin{contact manifold}, and for two contact manifolds $(S_1,\xi_1)$ and
$(S_2,\xi_2)$ of the same dimension, 
a smooth embedding $\varphi : S_1 \hookrightarrow
S_2$ is called a \defin{contact embedding} 
$$
(S_1,\xi_1) \hookrightarrow (S_2,\xi_2)
$$
if $\varphi_*\xi_1 = \xi_2$.  If $\varphi$ is also a diffeomorphism, then
we call it a \defin{contactomorphism}.  One of the main questions in contact
topology is how to distinguish closed contact manifolds that aren't
contactomorphic.  We'll touch upon this subject in the next section.

But first there is more to say about Hamiltonian dynamics.
We saw in Exercise~\ref{ex:symplectization} that the characteristic line
field on a contact type hypersurface $S \subset (M,\omega)$ 
can be described in terms of a
contact form $\alpha$: it is the unique line field containing all vectors
$X$ such that $d\alpha(X,\cdot) = 0$, and is necessarily transverse to the
contact structure.  The latter implies that $\alpha$ is nonzero in this
direction, so we can use it to choose a normalization, leading to the following
definition.

\begin{defn}
\label{defn:Reeb}
Given a contact form $\alpha$ on a $(2n-1)$-dimensional manifold $S$, the
\defin{Reeb vector field} is the unique vector field $R_\alpha$ satisfying
$$
d\alpha(R_\alpha,\cdot) = 0,
\qquad\text{ and }\qquad
\alpha(R_\alpha) = 1.
$$
\end{defn}

Thus closed integral curves 
on contact hypersurfaces can be identified with closed
orbits of their Reeb vector fields.\footnote{Note that since Liouville vector
fields are not unique, the Reeb vector field on a contact hypersurface is
not uniquely determined, but its \emph{direction} is.}  The ``intrinsic''
version of Theorems~\ref{thm:starshaped} and~\ref{thm:WeinsteinR2n} is
then the following famous conjecture.

\begin{conj}[Weinstein conjecture]
For every closed odd-dimensional manifold $M$ with a contact form $\alpha$, 
$R_\alpha$ has a closed orbit.
\end{conj}

The Weinstein conjecture is still open in general, though a proof in
dimension three was produced recently by C.~Taubes \cite{Taubes:weinstein},
using Seiberg-Witten theory.  Before this, there was a long history of
partial results using the theory of pseudoholomorphic curves, such as the
following (see Definition~\ref{defn:overtwisted} below for the definition
of ``overtwisted''):

\begin{thm}[Hofer \cite{Hofer:weinstein}]
\label{thm:HoferOvertwisted}
Every Reeb vector field on a closed $3$-dimensional overtwisted 
contact manifold admits a contractible periodic orbit.
\end{thm}

The key idea introduced in \cite{Hofer:weinstein} was to look at
$J$-holomorphic curves for a suitable class of almost complex structures
$J$ in the so-called \emph{symplectization} $(\RR\times M, d(e^t\alpha))$
of a manifold $M$ with contact form~$\alpha$.  Since the symplectic form is
now exact, it's no longer useful to consider \emph{closed} holomorphic curves,
e.g.~a minor generalization of \eqref{eqn:area} shows that all
$J$-holomorphic spheres $u : \CC P^1 \to \RR\times M$ are constant:
$$
\Area(u) = \| du \|_{L^2}^2 = \int_{\CC P^1} u^* d(e^t\alpha) =
\int_{\p \CC P^1} u^*(e^t\alpha) = 0.
$$
Instead, one considers $J$-holomorphic maps
$$
u : \dot{\Sigma} \to \RR\times M,
$$
where $\dot{\Sigma}$ denotes a closed Riemann surface with finitely
many \emph{punctures}.  It turns out that under suitable conditions, 
the image of $u$ near each puncture approaches $\{\pm\infty\} \times M$ and
becomes asymptotically close to a cylinder of the form
$\RR \times \gamma$, where $\gamma$ is a closed orbit of~$R_\alpha$
(see Figure~\ref{fig:symplectization}).
Thus an existence result for punctured 
holomorphic curves in $\RR\times M$ implies
the Weinstein conjecture on~$M$.

\begin{figure}
\includegraphics{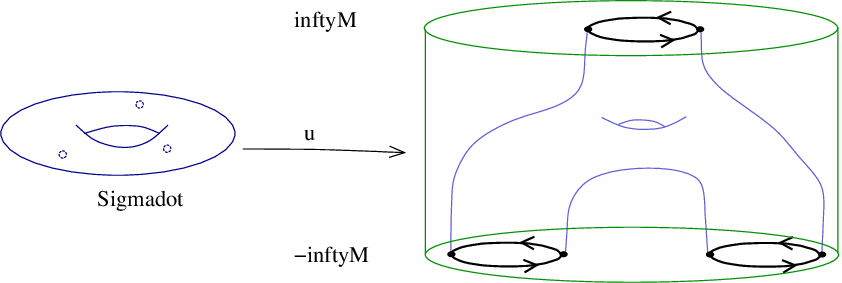}
\caption{\label{fig:symplectization} A three-punctured pseudoholomorphic
torus in the symplectization of a contact manifold.}
\end{figure}

To tie up a loose end, here's the proof of Gray's stability theorem,
followed by another important contact application of the Moser trick.
\begin{proof}[Proof of Theorem~\ref{thm:Gray}] 
Assume $S$ is a closed manifold with a smooth family of contact forms
$\{\alpha_t\}_{t\in [0,1]}$ defining contact structures $\xi_t = \ker\alpha_t$.
We want to find a time-dependent vector field $Y_t$ whose flow $\varphi_t$ 
satisfies
\begin{equation}
\label{eqn:ft}
\varphi_t^*\alpha_t = f_t \alpha_0
\end{equation}
for some (arbitrary) smooth $1$-parameter family of functions
$f_t : S \to \RR$.  Differentiating this expression and writing $\dot{f}_t :=
\frac{\p}{\p t} f_t$ and $\dot{\alpha}_t := \frac{\p}{\p t} \alpha_t$, we have
$$
\varphi_t^*\left(\dot{\alpha}_t + \Lie_{Y_t}\alpha_t \right) = \dot{f}_t \alpha_0
= \frac{\dot{f}_t}{f_t} \varphi_t^*\alpha_t,
$$
and thus
\begin{equation}
\label{eqn:Gray}
\dot{\alpha}_t + d \iota_{Y_t}\alpha_t + \iota_{Y_t} d\alpha_t =
g_t \alpha_t,
\end{equation}
where we define a new family of functions $g_t : S \to \RR$ via the relation
\begin{equation}
\label{eqn:gt}
g_t \circ \varphi_t = \frac{\dot{f}_t}{f_t} = 
\frac{\p}{\p t} \log f_t.
\end{equation}
Now to make life a bit simpler, we assume (optimistically!) that
$Y_t$ is always tangent to $\xi_t$, hence $\alpha_t(Y_t) = 0$ and the
second term in \eqref{eqn:Gray} vanishes.  We therefore need to find a
vector field $Y_t$ and function $g_t$ such that
\begin{equation}
\label{eqn:Gray2}
d\alpha_t(Y_t,\cdot) = -\dot{\alpha}_t + g_t \alpha_t.
\end{equation}
Plugging in the Reeb vector field $R_{\alpha_t}$ on both sides, we find
$$
0 = -\dot{\alpha}_t\left( R_{\alpha_t} \right) + g_t,
$$
which determines the function~$g_t$.  Now restricting both sides of
\eqref{eqn:Gray2} to $\xi_t$, there is a unique solution for $Y_t$
since $d\alpha_t|_{\xi_t}$ is nondegenerate.  We can then integrate
this vector field to obtain a family of diffeomorphisms $\varphi_t$,
and integrate \eqref{eqn:gt} to obtain $f_t$ so that
\eqref{eqn:ft} is satisfied.
\end{proof}

\begin{exercise}
Try to adapt the above argument to construct an isotopy such that
$\varphi_t^*\alpha_t = \alpha_0$ for any two deformation equivalent
contact forms.  But don't try very hard.
\end{exercise}

Finally, just as there is no local symplectic geometry, there is no local
contact geometry either:

\begin{thm}[Darboux's theorem for contact manifolds]
Near every point in a $(2n+1)$-dimensional manifold $S$ with contact form 
$\alpha$, there are local coordinates
$(p_1,\ldots,p_n,q_1,\ldots,q_n,z)$ in which 
$\alpha = dz + \sum_i p_i\,dq_i$.
\end{thm}
\begin{exercise}
Prove the theorem using a Moser argument.  If you get stuck, see
\cite{Geiges:book}.
\end{exercise}

\section{Symplectic fillings of contact manifolds}
\label{sec:fillings}

In the previous section, contact manifolds were introduced as objects that
occur naturally as hypersurfaces in symplectic manifolds.  In particular,
every contact manifold $(M,\xi)$ with contact form $\alpha$ is obviously
a contact type hypersurface in its own symplectization
$(\RR\times M, d(e^t\alpha))$, though this example is in some sense trivial.
By contrast, it is far from obvious whether any given contact manifold
can occur as a contact hypersurface in a \emph{closed} symplectic manifold,
or relatedly, if it is a ``contact type boundary'' of some compact 
symplectic manifold.

\begin{defn}
\label{defn:convexBoundary}
A compact symplectic manifold $(W,\omega)$ with boundary is said to have 
\defin{convex boundary} if there exists a Liouville vector field in a
neighborhood of $\p W$ that points transversely out of~$\p W$.
\end{defn}

\begin{defn}
\label{defn:strongFilling}
A \defin{strong symplectic filling} (also called a \defin{convex filling})
of a closed contact manifold $(M,\xi)$ is a compact symplectic manifold
$(W,\omega)$ with convex boundary, such that $\p W$ with the contact
structure induced by a Liouville vector field is contactomorphic to
$(M,\xi)$.
\end{defn}

Since we're now considering symplectic manifolds that are not closed,
it's also possible for $\omega$ to be exact.  Observe that a primitive
$\lambda$ of $\omega$ always gives rise to a Liouville vector field,
since the unique vector field $V$ defined by $\iota_V\omega = \lambda$
then satisfies
$$
\Lie_V\omega = d\iota_V\omega = d\lambda = \omega.
$$

\begin{defn}
\label{defn:exactFilling}
A strong filling $(W,\omega)$ of $(M,\xi)$ is called an \defin{exact filling}
if $\omega = d\lambda$ for some $1$-form $\lambda$ such that the vector
field $V$ defined by $\iota_V\omega = \lambda$ points transversely
out of~$\p W$.
\end{defn}

\begin{exercise}
Show that if $(W,\omega)$ is a compact symplectic manifold with boundary,
$V$ is a Liouville vector field defined near $\p W$ and 
$\lambda = \iota_V\omega$, then $V$ is positively transverse to $\p W$
if and only if $\lambda|_{\p W}$ is a positive contact form.
\end{exercise}

The exercise makes possible the following alternative formulations of
the above definitions:
\begin{enumerate}
\item
A compact symplectic manifold $(W,\omega)$ with boundary is a strong filling
if $\p W$ admits a contact form that extends to a primitive of $\omega$ on
a neighborhood of $\p W$.
\item
A strong filling is exact if the primitive mentioned above can be extended
globally over~$W$.
\item
A strong filling is exact if it has a transverse outward pointing Liouville 
vector field near $\p W$ that can be extended globally over~$W$.
\end{enumerate}

By now you're surely wondering what a ``weak'' filling is.
Observe that for any strong filling $(W,\omega)$ with Liouville vector
field $V$ and induced contact structure $\xi = \ker \iota_V\omega$ on the
boundary, $\omega$ has a nondegenerate restriction to~$\xi$ (see
Exercise~\ref{ex:symplectization}).  The latter condition can be expressed
without mentioning a Liouville vector field, hence:

\begin{defn}
\label{defn:weakFilling}
A \defin{weak symplectic filling} of a closed contact manifold $(M,\xi)$
is a compact symplectic manifold $(W,\omega)$ with boundary, such that
there exists a diffeomorphism $\varphi : \p W \to M$ and $\omega$ has
a nondegenerate restriction to $\varphi^*\xi$.
\end{defn}

\begin{remark}
One important definition that we are leaving out of the present discussion
is that of a \emph{Stein filling}: this is a certain type of complex
manifold with contact boundary, which is also an exact symplectic filling.
The results we'll prove in these notes for strong and exact fillings 
apply to Stein fillings as well, but we will usually not make specific
mention of this since the Stein condition itself has no impact on our general
setup.  Much more on Stein manifolds can be found in the monographs
\cite{OzbagciStipsicz} and \cite{CieliebakEliashberg}.
\end{remark}

A contact manifold is called exactly/strongly/weakly fillable if it admits an
exact/strong/weak filling.  Recall that in the smooth
category, every $3$-manifold is the boundary of some $4$-manifold; by
contrast, we will see that many contact $3$-manifolds are not
symplectically fillable.

The unit ball in $(\RR^4,\omega\std)$ obviously has convex boundary: the
contact structure induced on $S^3$ is called the \defin{standard} contact
structure~$\xi\std$.  But there are other contact structures on $S^3$
not contactomorphic to~$\xi\std$, and one way to see this is to show that
they are not fillable.  Indeed, it is easy (via ``Lutz twists'', see
\cite{Geiges:book} or \cite{Geiges:contact}) to produce a contact 
structure on $S^3$ that is \emph{overtwisted}.  Note that the following is
not the standard definition\footnote{It is standard to call a contact
$3$-manifold $(M,\xi)$ overtwisted if it contains an embedded \emph{overtwisted
disk}, which is a disk $\dD \subset M$ such that $T(\p \dD) \subset \xi$
but $T\dD|_{\p\dD} \ne \xi|_{\p\dD}$.}
of this term, but is equivalent due to a deep
result of Eliashberg \cite{Eliashberg:overtwisted}.

\begin{defn}
\label{defn:overtwisted}
A contact $3$-manifold $(M,\xi)$ is \defin{overtwisted} if it admits
a contact embedding of $(S^1 \times \DD,\xi_{\OT})$,
where $\DD \subset \RR^2$ is the closed unit disk and $\xi_{\OT}$ is a
contact structure of the form
$$
\xi_{\OT} = \ker \left[ f(\rho)\,d\theta + g(\rho)\,d\phi \right]
$$
with $\theta \in S^1$, $(\rho,\phi)$ denoting polar coordinates on~$\DD$,
and $(f,g) : [0,1] \to \RR^2 \setminus \{0\}$ a smooth path that begins
at $(1,0)$ and winds counterclockwise around the origin, making at least
one half turn.
\end{defn}

\begin{figure}
\includegraphics{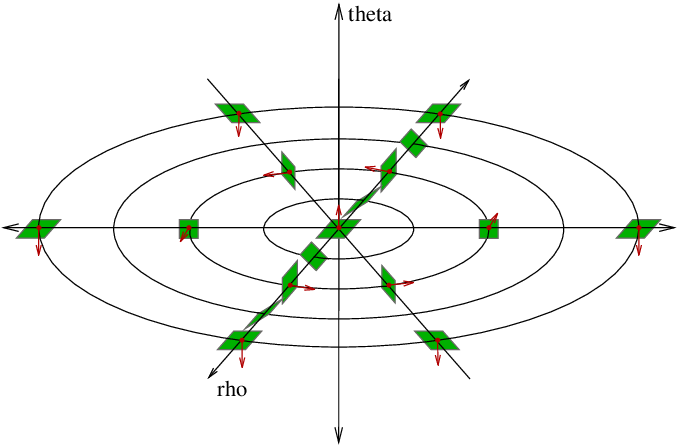}
\caption{\label{fig:overtwisted} An overtwisted contact structure.}
\end{figure}

For visualization, a portion of the domain $(S^1 \times \DD,\xi_{\OT})$ is shown
in Figure~\ref{fig:overtwisted}.  One of the earliest applications of
holomorphic curves in contact topology was the following nonfillability 
result.

\begin{thm}[M.~Gromov \cite{Gromov} and 
Ya.~Eliashberg \cite{Eliashberg:diskFilling}]
If $(M,\xi)$ is closed and overtwisted, then it is not weakly fillable.
\end{thm}

The Gromov-Eliashberg proof worked by assuming a weak filling $(W,\omega)$
of $(M,\xi)$ exists, then constructing a family of 
$J$-holomorphic disks in $W$ with boundaries on a totally real submanifold
in~$M$ and showing that this
family leads to a contradiction if $(M,\xi)$ contains an overtwisted disk.  
We will later
present a proof that is similar in spirit but uses slightly different
techniques: instead of dealing with boundary conditions for holomorphic
disks, we will adopt Hofer's methods and consider punctured holomorphic
curves in a noncompact symplectic manifold obtained by gluing a cylindrical
end to $\p W$.  The advantage of this approach is that it generalizes
nicely to prove the following related result on Giroux torsion, 
which is much more recent.
Previous proofs due to D.~Gay and Ghiggini and Honda required the large
machinery of gauge theory
and Heegaard Floer homology respectively, but we will only use
punctured holomorphic curves.

\begin{thm}[D.~Gay \cite{Gay:GirouxTorsion}, P.~Ghiggini and K.~Honda \cite{GhigginiHonda:twisted}]
\label{thm:torsion}
Suppose $(M,\xi)$ is a closed contact $3$-manifold that admits a contact
embedding of $(T^2 \times [0,1],\xi_T)$, where $\xi_T$ is the contact structure
defined in coordinates $(\theta,\phi,r) \in S^1 \times S^1 \times [0,1]$ by
$$
\xi_T = \ker \left[ \cos(2\pi r) \,d\theta + \sin(2\pi r)\,d\phi \right].
$$
Then $(M,\xi)$ is not strongly fillable.  Moreover if the embedded torus
$T^2 \times \{0\}$ separates $M$, then $(M,\xi)$ is also not weakly fillable.
\end{thm}

A contact $3$-manifold that admits a contact embedding of
$(T^2 \times [0,1],\xi_T)$ as defined above is said to have
\defin{Giroux torsion}.

\begin{example}
\label{ex:T3}
Using coordinates $(\theta,\phi,\eta) \in S^1 \times S^1 \times S^1 = T^3$,
one can define for each $N \in \NN$ a contact structure $\xi_N = \ker\alpha_N$,
where
$$
\alpha_N = \cos(2\pi N \eta)\,d\theta + \sin(2\pi N\eta)\,d\phi.
$$
Choosing the natural flat metric on $T^2 = S^1 \times S^1$, it's easy to
show that the unit circle bundle in $T^*T^2$ is a contact type hypersurface
contactomorphic to $(T^3,\xi_1)$, thus this is strongly (and even
exactly) fillable.  Giroux
\cite{Giroux:plusOuMoins} and Eliashberg \cite{Eliashberg:fillableTorus}
have shown that $(T^3,\xi_N)$ is in fact weakly fillable for all~$N$,
but Theorem~\ref{thm:torsion} implies that it is not strongly 
fillable for $N \ge 2$ (a result originally proved by Eliashberg 
\cite{Eliashberg:fillableTorus}).  Unlike the case of $S^3$, none of these
contact structures are overtwisted---one can see this easily from
Theorem~\ref{thm:HoferOvertwisted} and the exercise below.
\end{example}

\begin{exercise}
Derive expressions for the Reeb vector fields $R_{\alpha_N}$ on $T^3$ and
show that none of them admit any contractible periodic orbits.
\end{exercise}

Finally, we mention one case of a fillable contact manifold in which all
the symplectic fillings can be described quite explicitly.  Earlier we
defined the standard contact structure $\xi\std$ on $S^3$ to be the one that is
induced on the convex boundary of a round ball in $(\RR^4,\omega\std)$.
By looking at isotopies of convex boundaries and using Gray's stability 
theorem, you should easily be able to convince yourself that every
star-shaped hypersurface in $(\RR^4,\omega\std)$ has an induced contact structure
isotopic to~$\xi\std$.  Thus the regions bounded by these hypersurfaces,
the ``star-shaped domains'' in $(\RR^4,\omega\std)$, can all be regarded as
convex fillings of $(S^3,\xi\std)$.  Are there any others?  Well\ldots

\begin{thm}[Eliashberg \cite{Eliashberg:diskFilling}]
\label{thm:fillingsS3}
Every exact filling of $(S^3,\xi\std)$ is symplectomorphic to a
star-shaped domain in $(\RR^4,\omega\std)$.
\end{thm}

In fact we will just as easily be able to classify all the weak fillings
of $(S^3,\xi\std)$ up to symplectic deformation equivalence.
Again, our proof will differ from Eliashberg's in using punctured
holomorphic curves asymptotic to Reeb orbits instead of compact curves
with totally real boundary conditions.  But in either case,
the proof has much philosophically in common 
with the proof of Theorem~\ref{thm:CP2} that we already sketched: one
first finds a single holomorphic curve, in this case near the boundary
of the filling, and then lets the moduli space of such curves ``spread out''
until it yields a geometric decomposition of the filling.


\chapter{Fundamentals}
\label{chapter:local}

\minitoc
\vspace{12pt}

\section{Almost complex manifolds and $J$-holomorphic curves}
\label{sec:intro}

We now begin the study of $J$-holomorphic curves in earnest by defining
the nonlinear Cauchy-Riemann equation in its most natural setting, and
then examining the analytical properties of its solutions.  This will be
the focus for the next few chapters.

Given a $2n$-dimensional real vector space, we define a 
\defin{complex structure} on $V$ to be any linear map $J : V \to V$ such
that $J^2 = -\1$.  It's easy to see that a complex structure always
exists when $\dim V$ is even, as one can choose a basis to identify
$V$ with $\RR^{2n}$ and identify this in turn with $\CC^n$, so that the
natural ``multiplication by $i$'' on $\CC^n$ becomes a linear map on~$V$.
In the chosen basis, this linear map is represented by the matrix
$$
\JJ\std := 
\begin{pmatrix}
0 & -1 &        &   & \\
1 &  0 &        &   & \\
  &    & \ddots &   & \\
  &    &        & 0 & -1 \\
  &    &        & 1 &  0 
\end{pmatrix}.
$$
We call this the \defin{standard} complex structure on~$\RR^{2n}$, and will
alternately denote it by $\JJ\std$ or~$i$, depending on the context.  A
complex structure $J$ on~$V$ allows us to view $V$ as a complex 
$n$-dimensional
vector space, in that we identify 
the scalar multiplication by any complex number
$a + ib \in \CC$ with the linear map $a\1 + b J$.  A real-linear map on $V$
is then also complex linear in this sense if and only if it commutes with~$J$.
Similarly, we call a real-linear map $A : V \to V$ 
\defin{complex antilinear} if it anticommutes with $J$, i.e.~$A J = - J A$.
This is equivalent to the requirement that $A$ preserve vector addition
but satisfy $A (\lambda v) = \bar{\lambda} Av$ for all $v \in V$ and
complex scalars $\lambda \in \CC$.

\begin{exercise}
\label{EX:Jstandard} \ 
\begin{enumerate}
\renewcommand{\labelenumi}{(\alph{enumi})}
\item
Show that for every even-dimensional
vector space $V$ with complex structure $J$, there
exists a basis in which $J$ takes the form of the standard complex
structure~$\JJ\std$.
\item
Show that if~$V$ is an odd-dimensional vector space, then there is no
linear map $J : V \to V$ satisfying $J^2 = -\1$.
\item
Show that all real-linear maps on $\RR^{2n}$ that commute with~$\JJ\std$ have
positive determinant.
\end{enumerate}
\end{exercise}

Note that due to the above exercise, a complex structure~$J$
on a $2n$-dimensional vector space~$V$ induces a natural \emph{orientation}
on~$V$, namely by defining any basis of the form
$(v_1,J v_1,\ldots,v_n, J v_n)$ to be positively oriented.  This is equivalent
to the statement that every finite-dimensional complex vector space has
a natural orientation as a \emph{real} vector space.

The above notions can easily be generalized from spaces to
bundles: if~$M$ is a topological space and $E \to M$ is a real vector
bundle of even rank, then a \defin{complex structure} on $E \to M$ is
a continuous family of complex structures on the fibers of~$E$, i.e.~a
section $J \in \Gamma(\End(E))$ of the bundle $\End(E)$ of fiber-preserving
linear maps $E \to E$, such that $J^2 = -\1$.  If $E \to M$ is a smooth
vector bundle, then we will always assume that~$J$ is \emph{smooth} unless some
other differentiability class is specifically indicated.
A complex structure gives
$E \to M$ the structure of a complex vector bundle, due to the following
variation on Exercise~\ref{EX:Jstandard} above.

\begin{exercise}
\label{EX:Jtriv} \ 
\begin{enumerate}
\renewcommand{\labelenumi}{(\alph{enumi})}
\item Show that whenever $E \to M$ is a real vector bundle of even rank
with a complex structure~$J$, every point $p \in M$ lies in a
neighborhood on which
$E$ admits a trivialization such that~$J$ takes the form of the standard
complex structure~$\JJ\std$.
\item Show that for any two trivializations having the above property,
the transition map relating them is fiberwise complex linear (using the
natural identification $\RR^{2n} = \CC^n$).
\end{enumerate}
\end{exercise}

For this reason, it is often convenient to denote complex vector bundles
of rank~$n$ as pairs $(E,J)$, where $E$ is a real bundle of rank~$2n$
and~$J$ is a complex structure on~$E$.
Note that not every real vector bundle of even rank admits a complex
structure: the above discussion shows that such bundles must always be
orientable, and this condition is not even generally
sufficient except for the case of rank two.

For a smooth $2n$-dimensional manifold~$M$, we refer to any complex 
structure~$J$ on the tangent bundle~$TM$ as an
\defin{almost complex structure}
on~$M$, and the pair $(M,J)$ is then an \defin{almost complex manifold}.  
The reason for the word ``almost''
will be explained in a moment.

\begin{example}
\label{ex:integrable}
Suppose $M$ is a complex manifold of complex dimension $n$, i.e.~there exist
local charts covering $M$ that identify subsets of $M$ with subsets of
$\CC^n$ such that all transition maps are holomorphic.  Any choice
of holomorphic local coordinates on a subset $\uU \subset M$ then identifies
the tangent spaces $T_p \uU$ with $\CC^n$.  If we use this identification to
assign the standard complex structure $i$ to each tangent space $T_p \uU$,
then the fact that transition maps are holomorphic implies that this
assignment doesn't depend on the choice of coordinates (prove this!).
Thus $M$ has a natural almost complex structure $J$ that looks like the
standard complex structure in any holomorphic coordinate chart.
\end{example}

An almost complex structure is called \defin{integrable} if it arises in
the above manner from a system of holomorphic coordinate charts; in this
case we drop the word ``almost'' and simply call $J$ a \defin{complex 
structure} on~$M$.  
By definition, then, a real manifold $M$ admits a complex structure
(i.e.~an integrable almost complex structure) if and only if it also admits
coordinate charts that make it into a complex manifold.  In contrast to
Exercise~\ref{EX:Jtriv}, which applies to trivializations on vector bundles,
one cannot always find a coordinate chart that makes a given almost complex
structure look standard on a neighborhood.  The following
standard (but hard) result of complex analysis characterizes integrable
complex structures; we include it here for informational purposes, but will
not make essential use of it in the following.

\begin{thm}
\label{thm:Nijenhuis}
The almost complex structure $J$ on $M$ is integrable if and only if the
tensor $N_J$ vanishes identically, where $N_J$ is defined on two vector fields
$X$ and $Y$ by
\begin{equation}
\label{eqn:Nijenhuis}
N_J(X,Y) = [JX,JY] - J[JX,Y] - J[X,JY] - [X,Y].
\end{equation}
\end{thm}
The tensor \eqref{eqn:Nijenhuis} is called the \defin{Nijenhuis tensor}.

\begin{exercise} \ 
\label{EX:Nijenhuis}
\begin{enumerate}
\renewcommand{\labelenumi}{(\alph{enumi})}
\item Verify that \eqref{eqn:Nijenhuis} defines a tensor.
\item Show that $N_J$ always vanishes if $\dim M = 2$.
\item Prove one direction of Theorem~\ref{thm:Nijenhuis}: if $J$ is
integrable, then $N_J$ vanishes.
\end{enumerate}
\end{exercise}
The converse direction is much harder to prove, see for instance
\cite{DonaldsonKronheimer}*{Chapter~2}.  But if you believe this, then
Exercise~\ref{EX:Nijenhuis} has the following nice consequence:

\begin{thm}
\label{thm:RiemannSurfaces}
Every almost complex structure on a surface is integrable.
\end{thm}

In other words, complex $1$-dimensional
manifolds are the same thing as almost complex manifolds of real
dimension two.  This theorem follows from an existence result for local
pseudoholomorphic curves which we'll prove in \S\ref{sec:nonlinear}.
Actually, that existence result can be thought of as the first step
in the proof of Theorem~\ref{thm:Nijenhuis}.  Complex manifolds in the
lowest dimension have a special status, and deserve a special name:

\begin{defn}
A \defin{Riemann surface} is a complex manifold of complex dimension one.
\end{defn}

By Theorem~\ref{thm:RiemannSurfaces}, a Riemann surface can equivalently
be regarded as a surface $\Sigma$ with an almost complex structure $j$,
and we will thus typically denote Riemann surfaces as pairs $(\Sigma,j)$.

Surfaces are the easy special case; in dimensions four and higher,
\eqref{eqn:Nijenhuis} does not usually vanish, in fact it is \emph{generically}
nonzero, which shows that, in some sense, ``generic'' almost complex structures
are not integrable.  Thus in higher dimensions, integrable complex structures
are very rigid objects---too rigid for our purposes, as it will turn out.
For instance, there are real manifolds that do not admit complex structures
but do admit almost complex structures.  It will be most important for our
purposes
to observe that \emph{symplectic} manifolds always admit almost complex
structures that are ``compatible'' with the symplectic form in a certain
geometric sense.  We'll come back to this in \S\ref{sec:compatible} and make
considerable use of it in later applications, but for most of the present
chapter, we will focus only on
the \emph{local} properties of $J$-holomorphic curves and thus be content
to work in the more general context of almost complex manifolds.

\begin{defn}
Suppose $(\Sigma,j)$ is a Riemann surface and $(M,J)$ is an almost complex
manifold.  A smooth map $u : \Sigma \to M$ is called
\defin{$J$-holomorphic} (or \defin{pseudoholomorphic}) if its differential
at every point is complex-linear,~i.e.
\begin{equation}
\label{eqn:nonlinearCR}
Tu \circ j = J \circ Tu.
\end{equation}
\end{defn}

Note that in general, the equation \eqref{eqn:nonlinearCR} makes sense if
$u$ is only of class $C^1$ (or more generally, of Sobolev class $W^{1,p}$) 
rather than smooth, but it will turn out to
follow from elliptic regularity (see \S\ref{sec:estimates} and
\S\ref{sec:nonlinear}) that 
$J$-holomorphic curves are \emph{always} smooth if $J$ is smooth---we will
therefore assume smoothness whenever convenient.
Equation~\eqref{eqn:nonlinearCR} is a nonlinear first-order PDE, 
often called the \defin{nonlinear Cauchy-Riemann equation}.  If you are
not accustomed to PDEs expressed in geometric notation, you may prefer to view
it as follows: choose holomorphic local coordinates $s + it$ on a subset
of $\Sigma$, so $j \p_s = \p_t$ and $j \p_t = -\p_s$ (note that we're
assuming the integrability of~$j$).  Then \eqref{eqn:nonlinearCR} is
locally equivalent to the equation
\begin{equation}
\label{eqn:nonlinearCR2}
\p_s u + J(u)\, \p_t u = 0.
\end{equation}

\begin{notation}
We will sometimes write $u : (\Sigma,j) \to (M,J)$ to mean that
$u : \Sigma \to M$ is a map satisfying \eqref{eqn:nonlinearCR}.  
When the domain is
the open unit ball $B \subset \CC$ (or any other open subset of~$\CC$)
and we say $u : B \to M$ is $J$-holomorphic without specifying the complex
structure of the domain, then the standard complex structure is implied,
i.e.~$u$ is a pseudoholomorphic map $(B,i) \to (M,J)$ and thus satisfies
\eqref{eqn:nonlinearCR2}.  The symbol $B_r$ for $r > 0$ will be used to
denote the open ball of radius $r$ in $(\CC,i)$.
\end{notation}

Note that the standard Cauchy-Riemann equation for maps $u : \CC \to \CC^n$
can be written as $\p_s u + i\, \p_t u = 0$, thus \eqref{eqn:nonlinearCR2} can
be viewed as a perturbation of this.  In fact, due to 
Exercise~\ref{EX:Jstandard}, one can always choose
coordinates near a point $p \in M$ so that $J(p)$ is identified with the
standard complex structure; then in a sufficiently small neighborhood of~$p$,
\eqref{eqn:nonlinearCR2} really is a \emph{small} perturbation of the
usual Cauchy-Riemann equation.  We'll make considerable
use of this perspective in the
following.  Here is a summary of the most important results we aim to
prove in this chapter.

\begin{thmu}
Assume $(M,J)$ is a smooth almost complex manifold.  Then:
\begin{itemize}
\item \textsc{(regularity)} Every map $u : \Sigma \to M$ of class~$C^1$ 
solving the nonlinear Cauchy-Riemann equation \eqref{eqn:nonlinearCR} is smooth
(cf.~Theorem~\ref{thm:regularity}).
\item \textsc{(local existence)} For any $p \in M$ and $X \in T_p M$, there 
exists a neighborhood $\uU \subset \CC$ of
the origin and a $J$-holomorphic map $u : \uU \to M$
such that $u(0) = p$ and $\p_s u(0) = X$ in standard coodinates
$s + it \in \uU$ (cf.~Theorem~\ref{thm:localExistence3}).
\item \textsc{(critical points)}
If $u : \Sigma \to M$ is a nonconstant $J$-holomorphic curve with a critical
point $z \in \Sigma$, then there is a neighborhood $\uU \subset \Sigma$ of $z$
such that $u|_{\uU\setminus\{z\}}$ is a $k$-to-$1$ immersion for some
$k \in \NN$
(cf.~Corollary~\ref{cor:critical} and Theorem~\ref{thm:injective}).
\item \textsc{(intersections)} Suppose $u_1 : \Sigma_1 \to M$ and 
$u_2 : \Sigma_2 \to M$ are two nonconstant
$J$-holomorphic curves with an intersection $u_1(z_1) = u_2(z_2)$.  Then
there exist neighborhoods $z_1 \in \uU_1 \subset \Sigma_1$ and
$z_2 \in \uU_2 \subset \Sigma_2$ such that the images
$u_1(\uU_1\setminus\{z_1\})$ and $u_2(\uU_2\setminus\{z_2\})$ 
are either identical or disjoint (cf.~Theorem~\ref{thm:intersections}).
In the latter case, if $\dim M = 4$, then the intersection has positive
local intersection index, which equals~$1$ if and only if the intersection
is transverse (cf.~Theorem~\ref{thm:positivity}).
\end{itemize}
\end{thmu}

This theorem amounts to the statement that locally, $J$-holomorphic
curves behave much the same way as holomorphic curves, i.e.~the same as in
the integrable case.  But since $J$ is usually not integrable, the methods
of complex analysis cannot be applied here, and we will instead need to employ 
techniques from the theory of elliptic PDEs.  As preparation, we'll derive 
the natural linearization of \eqref{eqn:nonlinearCR} and introduce the 
theory of linear Cauchy-Riemann operators, as well as some fundamental ideas
of global analysis, all of which will be useful in the chapters to come.

\section{Compatible and tame almost complex structures}
\label{sec:compatible}

For any given even-dimensional manifold~$M$, it is not always immediately
clear whether an almost complex structure exists.  If $\dim M = 2$
for instance, then this is true if and only if~$M$ is orientable, and in
higher dimensions the question is more delicate.  We will not
address this question in full generality, but merely show in the present
section that for the cases we are most interested in, namely for
\emph{symplectic} manifolds, the answer is exactly as we might hope.
The results of this section are mostly independent of the rest of the
chapter, but they will become crucial once we discuss compactness results and
applications, from Chapter~\ref{chapter:nonsqueezing} onwards.

Given a manifold $M$ and a smooth vector bundle 
$E \to M$ of even rank, denote by
$\jJ(E)$ the space of all (smooth) complex structures on~$E$.  We shall
regard this as a topological space with the $\Cinftyloc$-topology,\footnote{Also
known as the \emph{weak} or \emph{compact-open $C^\infty$-topology}, see
e.g.~\cite{Hirsch}*{Chapter~2}.}
i.e.~a sequence $J_k \in \jJ(E)$ converges if and only if it is
$C^\infty$-convergent on all compact subsets.  As explained in
\S\ref{sec:intro} above, any choice of $J \in \jJ(E)$ makes $(E,J)$ into a
complex vector bundle.

\begin{notation}
We shall denote by $\End_\RR(\CC^n)$ the space of real-linear endomorphisms
of $\CC^n$, i.e.~$\End_\RR(\CC^n) = \End(\RR^{2n})$ under the usual
identification of $\CC^n$ with $\RR^{2n}$.  The spaces of complex-linear
and complex-antilinear endomorphisms of $\CC^n$ will be denoted by
$\End_\CC(\CC^n)$ and $\overline{\End}_\CC(\CC^n)$ respectively, 
or $\End_\CC(\CC^n,J)$ and $\overline{\End}_\CC(\CC^n,J)$ whenever an
alternative complex structure $J$ on $\CC^n$ is specified.
For a complex vector bundle $(E,J)$, we will analogously denote the various 
vector bundles of fiber-preserving linear maps on $E$ by
$\End_\RR(E)$, $\End_\CC(E,J)$ and $\overline{\End}_\CC(E,J)$.  The open
subsets
\begin{equation*}
\begin{split}
\Aut_\RR(E) &:= \left\{ A \in \End_\RR(E) \ |\ \text{$A$ is invertible} \right\} \\
\Aut_\CC(E,J) &:= \left\{ A \in \End_\CC(E,J) \ |\ \text{$A$ is invertible} \right\} 
\end{split}
\end{equation*}
are then smooth fiber bundles.
Let $\jJ(\CC^n) \subset \End_\RR(\CC^n)$ denote the space of all
complex structures on the vector space $\RR^{2n} = \CC^n$.
\end{notation}

\begin{exercise}
\label{EX:homogeneous}
Consider the smooth map
$$
\Phi : \GL(2n,\RR) \to \GL(2n,\RR) : A \mapsto A i A^{-1},
$$
where $i$ is identified with the standard complex structure
on $\RR^{2n} = \CC^n$.
Show that if $\GL(n,\CC)$ is regarded as the subgroup of all matrices
in $\GL(2n,\RR)$ that commute with $i$, then $\Phi$ descends to an
embedding of the homogeneous space $\GL(2n,\RR) / \GL(n,\CC)$ into
$\GL(2n,\RR)$, whose image is precisely~$\jJ(\CC^n)$.  Deduce that
$\jJ(\CC^n)$ is a noncompact $2n^2$-dimensional smooth submanifold of
$\End_\RR(\CC^n)$, and show that its tangent space at any 
$J \in \jJ(\CC^n)$ is
$$
T_J\jJ(\CC^n) = \overline{\End}_\CC(\CC^n,J) \subset
\End_\RR(\CC^n).
$$
\end{exercise}

\begin{exercise}
\label{EX:J}
Use Exercise~\ref{EX:homogeneous} to show that for any smooth complex vector bundle
$(E,J_0) \to M$, the space $\jJ(E)$ of complex structures on $E$ can be
identified with the space of smooth sections of the fiber bundle
$\Aut_\RR(E) / \Aut_\CC(E,J_0) \to M$.
\end{exercise}

The map $\Phi : \GL(2n,\RR) \to \jJ(\CC^n)$ of Exercise~\ref{EX:homogeneous} also 
yields a natural way to construct smooth local charts on~$\jJ(\CC^n)$.  
For instance, the standard structure $i \in \jJ(\CC^n)$ is $\Phi(\1)$, and 
on $T_\1 \GL(2n,\RR) = \End_\RR(\CC^n)$ we have a natural splitting
$$
\End_\RR(\CC^n) = \End_\CC(\CC^n) \oplus \overline{\End}_\CC(\CC^n) =
T_\1 \GL(n,\CC) \oplus T_i \jJ(\CC^n),
$$
so that matrices of the form $\1 + Y$ for $Y \in \overline{\End}_\CC(\CC^n)$
near~$0$ form a local slice parametrizing a neighborhood of $[\1]$ in
$\GL(2n,\RR) / \GL(n,\CC)$.  Consequently, $\jJ(\CC^n)$ is parametrized
near~$i$ by matrices of the form $\Phi(\1 + Y) = (\1 + Y) i (\1 + Y)^{-1}$
for $Y \in \overline{\End}_\CC(\CC^n)$.
It will be convenient to modify this parametrization by a linear
transformation on $\overline{\End}_\CC(\CC^n)$: consider the map
\begin{equation}
\label{eqn:Cayley0}
Y \mapsto J_Y := \left( \1 + \frac{1}{2} i Y \right) i 
\left( \1 + \frac{1}{2} i Y \right)^{-1}.
\end{equation}
This identifies a neighborhood of~$0$ in $T_i \jJ(\CC^n) = 
\overline{\End}_\CC(\CC^n)$ with a neighborhood of~$i$ in
$\jJ(\CC^n)$, and the following exercise shows that it can
be thought of informally as a kind of ``exponential map'' on~$\jJ(\CC^n)$.

\begin{exercise}
\label{EX:Cayley0}
Show that the derivative of the map \eqref{eqn:Cayley0} at~$0$ is
the identity transformation on $\overline{\End}_\CC(\CC^n)$.
\end{exercise}

\begin{remark}
\label{remark:Cayley0}
Since all complex structures on $\CC^n$ are equivalent up to a change of
basis, the above discussion also shows that a neighborhood of any 
$J_0 \in \jJ(\CC^n)$ can be identified with a neighborhood of~$0$ in 
$\overline{\End}_\CC(\CC^n,J_0)$ via the map
$Y \mapsto J := \left( \1 + \frac{1}{2} J_0 Y \right) J_0
\left( \1 + \frac{1}{2} J_0 Y \right)^{-1}$.
\end{remark}

Suppose next that $(E,\omega)$ is a 
\defin{symplectic vector bundle}, 
i.e.~a vector bundle
whose fibers are equipped with a nondegenerate skew-symmetric bilinear 
$2$-form $\omega$ that varies smoothly.
It is straightforward to show that such a bundle admits local trivializations
that identify every fiber symplectically with $(\RR^{2n},\omega\std)$;
see \cite{McDuffSalamon:ST}.  On $(E,\omega)$, we will consider two special
subspaces of~$\jJ(E)$:
\begin{equation*}
\begin{split}
\jJ^\tau(E,\omega) &:= \{ J \in \jJ(E)\ |\ 
\omega(v,Jv) > 0 \text{ for all } v \ne 0 \}, \\
\jJ(E,\omega) &:= \{ J \in \jJ(E)\ |\ 
g_J(v,w) := \omega(v,Jw) \text{ is a Euclidean bundle metric} \}.
\end{split}
\end{equation*}
We say that $J$ is \defin{tamed by $\omega$} if $J \in \jJ^\tau(E,\omega)$,
and it is \defin{compatible with} (some authors also say \emph{callibrated by}) 
$\omega$ if $J \in \jJ(E,\omega)$.
Clearly $\jJ(E,\omega) \subset \jJ^\tau(E,\omega)$.  
The taming condition is weaker than compatibility because we do not require
the bilinear form $(v,w) \mapsto \omega(v,Jw)$ to be symmetric, but one
can still symmetrize it to define a bundle metric,
\begin{equation}
\label{eqn:gJ}
g_J(v,w) := \frac{1}{2} \left[ \omega(v,Jw) + \omega(w,Jv) \right],
\end{equation}
which is identical to the above definition in the case $J \in \jJ(E,\omega)$.

\begin{exercise}
\label{EX:Jinvariant}
Show that a tamed complex structure $J \in \jJ^\tau(E,\omega)$ 
is also $\omega$-compatible if and only if $\omega$ is $J$-invariant, 
i.e.~$\omega(Jv,Jw) = \omega(v,w)$ for all~$v,w \in E$.
\end{exercise}

\begin{exercise}
\label{EX:subbundles}
Suppose $(E,\omega)$ is a symplectic vector bundle and $F \subset E$
is a symplectic subbundle, i.e.~a smooth subbundle such that
$\omega|_{F}$ is also nondegenerate.  Denote its \defin{symplectic
complement} by
$$
F^{\perp\omega} = \{ v \in E\ |\ \omega(v,\cdot)|_F = 0 \},
$$
and recall that $\omega|_{F^{\perp\omega}}$ is necessarily also
nondegenerate, and $E = F \oplus F^{\perp\omega}$ (see e.g.~\cite{McDuffSalamon:ST}).
Show that if $j$ and $j'$ are tame/compatible complex structures on
$(F,\omega)$ and $(F^{\perp\omega},\omega)$ respectively, 
then $j \oplus j'$ defines
a tame/compatible complex structure on~$(E,\omega)$.  
\end{exercise}

\begin{exercise}
\label{EX:Un}
Show that for any symplectic vector bundle $(E,\omega)$, a complex structure
$J \in \jJ(E)$ is compatible with $\omega$ if and only if
there exists a system of local trivializations that simultaneously
identify $\omega$ and $J$ with the standard symplectic and complex 
structures $\omega\std$ and $i$ respectively on~$\RR^{2n} = \CC^n$.
\textsl{Hint: If $J$ is $\omega$-compatible, then the pairing
$\langle v,w \rangle := \omega(v,Jw) + i \omega(v,w) \in \CC$ defines a
Hermitian bundle metric on $(E,J)$.}
\end{exercise}

The main result of this section is the following.

\begin{thm}
\label{thm:contractible}
For any finite rank symplectic vector bundle $(E,\omega) \to M$, the spaces
$\jJ(E,\omega)$ and $\jJ^\tau(E,\omega)$ are both nonempty and
contractible.
\end{thm}

\begin{exercise}
\label{EX:sympConvex}
The following is a converse of sorts to Theorem~\ref{thm:contractible},
but is much easier.  Given a smooth vector bundle $E \to M$, 
define the space of \defin{symplectic vector bundle structures}
$\Omega(E)$ as the space of smoothly varying nondegenerate skew-symmetric
bilinear $2$-forms $\omega$ on the fibers of~$E$, and assign to this
space the natural $\Cinftyloc$-topology.  
Show that on any complex vector bundle $(E,J)$, the
spaces
\begin{equation*}
\begin{split}
\Omega^\tau(E,J) &:= \left\{ \omega \in \Omega(E)\ |\ J \in \jJ^\tau(E,\omega) \right\},\\
\Omega(E,J) &:= \left\{ \omega \in \Omega(E)\ |\ J \in \jJ(E,\omega) \right\}
\end{split}
\end{equation*}
are each nonempty convex subsets of vector spaces and are thus contractible.
\textsl{Hint: To show nonemptiness, choose a Hermitian metric and consider
its imaginary part.}
\end{exercise}

Before proving the theorem, let us give some initial indications of the
role that tameness plays in the theory of $J$-holomorphic curves.
We will usually assume $(E,\omega) := (TM,\omega)$ for some symplectic manifold
$(M,\omega)$, and in this case use the notation
$$
\jJ(M) := \jJ(TM),\quad \jJ^\tau(M,\omega) := \jJ^\tau(TM,\omega),\quad
\jJ(M,\omega) := \jJ(TM,\omega).
$$
Most simple examples of almost complex structures one can write down on
symplectic manifolds are compatible: e.g.~this is true for the standard
(integrable) complex structures on $(\CC^n = \RR^{2n},\omega\std)$ and
$(\CC P^n,\omega\std)$, and for any
complex structure compatible with the canonical orientation on a 
$2$-dimensional symplectic manifold.
Since every almost complex structure looks like the standard one
at a point in appropriate coordinates, it is easy to see that every $J$ is
\emph{locally} tamed by some symplectic structure: namely, if
$J$ is any almost complex structure on a neighborhood of the origin
in $\RR^{2n}$ with $J(0) = i$, then $J$ is tamed by~$\omega\std$ on a
possibly smaller neighborhood of~$0$, since tameness
is an open condition.

The key property of a
tame almost complex structure on a symplectic manifold is that every complex 
line in a tangent space is also a symplectic subspace, hence every embedded 
$J$-holomorphic curve parametrizes a symplectic submanifold.  
At the beginning of Chapter~\ref{chapter:intro}, 
we showed that holomorphic curves
in the standard $\CC^n$ have the important property that the area they
trace out can be computed by integrating the standard symplectic structure.
It is precisely this relation between symplectic structures and tame
almost complex structures that makes the compactness theory of
$J$-holomorphic curves possible.  The original computation generalizes as follows:
assume $(M,\omega)$ is a symplectic manifold,
$J \in \jJ^\tau(M,\omega)$, and let $g_J$ be the Riemannian metric
defined in \eqref{eqn:gJ}.  If $u : (\Sigma,j) \to (M,J)$ is a $J$-holomorphic
curve and we choose holomorphic local coordinates $(s,t)$ on a subset
of $\Sigma$, then $\p_t u = J \p_s u$ implies that with respect to the metric
$g_J$, $\p_s u$ and $\p_t u$ are orthogonal vectors of the same length.
Thus the geometric area of the parallelogram spanned by these two vectors
is simply
$$
| \p_s u |_{g_J} \cdot | \p_t u |_{g_J} = | \p_s u |_{g_J}^2 =
\omega(\p_s u, J \p_s u) = \omega(\p_s u, \p_t u),
$$
hence
\begin{equation}
\label{eqn:area2}
\Area_{g_J}(u) = \int_\Sigma u^*\omega.
\end{equation}

\begin{defn}
\label{defn:energy}
For any symplectic manifold $(M,\omega)$ and tame almost complex structure
$J \in \jJ^\tau(M,\omega)$, we define the \defin{energy} of a
$J$-holomorphic curve $u : (\Sigma,j) \to (M,J)$ by
$$
E(u) = \int_\Sigma u^*\omega.
$$
\end{defn}

The following is an immediate consequence of \eqref{eqn:area2}.

\begin{prop}
If $J \in \jJ^\tau(M,\omega)$ then for every $J$-holomorphic
curve $u : (\Sigma,j) \to (M,J)$,
$E(u) \ge 0$, with equality if and only if $u$ is locally 
constant.\footnote{The term \defin{locally constant} means that
the restriction of~$u$ to each connected component of its
domain~$\Sigma$ is constant.}
\end{prop}

The energy as defined above is especially important in the case where
the domain~$\Sigma$ is a closed surface.  Then $u : \Sigma \to M$
represents a homology class $[u] := u_*[\Sigma] \in H_2(M)$, and the quantity
$E(u)$ is not only nonnegative but also \emph{topological}:
it can be computed via the pairing
$\langle [\omega] , [u] \rangle$, and thus depends only on 
$[u] \in H_2(M)$ and $[\omega] \in H^2_\dR(M)$.
This implies an a priori
energy bound for $J$-holomorphic curves in a fixed homology class, which
we'll make considerable use of in applications.

For the next result, we can drop the assumption that $M$ is a symplectic
manifold, though the proof does make use of a (locally defined)
symplectic structure.  The result can be summarized by saying that for
any reasonable moduli space of $J$-holomorphic curves, the \emph{constant}
curves form an open subset.

\begin{prop}
\label{prop:seqConst}
Suppose $\Sigma$ is a closed surface,
$J_k \in \jJ(M)$ is a sequence of almost complex structures that
converge in $C^\infty$ to $J \in \jJ(M)$, and
$u_k : (\Sigma,j_k) \to (M,J_k)$ is a sequence of non-constant
pseudoholomorphic curves
converging in $C^\infty$ to a pseudoholomorphic curve
$u : (\Sigma,j) \to (M,J)$.  Then $u$ is also not constant.
\end{prop}
\begin{proof}
Assume $u$ is constant and its image is $p \in M$.  Choosing coordinates
near~$p$, we can assume without loss of generality that $p$ is the origin
in $\CC^n$ and $u_k$ maps into a neighborhood of the origin, with
almost complex structures $J_k$ on $\CC^n$ converging to~$J$ such that 
$J(0) = i$.  Then for sufficiently large~$k$, the standard symplectic 
form $\omega\std$ tames each $J_k$ in a 
sufficiently small neighborhood of the origin, and
$[u_k] = [u] = 0 \in H_2(M)$, implying $E(u_k) = \langle [\omega\std], [u_k]
\rangle = 0$, thus $u_k$ is also constant.
\end{proof}

The remainder of this section is devoted to proving Theorem~\ref{thm:contractible}.
We will explain two quite different proofs.  In the first, which is due to
Gromov \cite{Gromov}, the spaces $\jJ(E,\omega)$ and $\jJ^\tau(E,\omega)$
must be handled by separate arguments, and the former is easier---it is also 
the space
that is most commonly needed in applications, so we shall explain this part
first.

\begin{proof}[Proof of Theorem~\ref{thm:contractible} for $\jJ(E,\omega)$]
Let $\metrics(E)$ denote the space of smooth bundle metrics on~$E \to M$, also
with the $\Cinftyloc$-topology.  There
is then a natural continuous map
$$
\jJ(E,\omega) \to \metrics(E) : J \mapsto g_J,
$$
where $g_J := \omega(\cdot,J\cdot)$.  We shall construct a continuous
left inverse to this map, i.e.~a continuous map
$$
\Phi : \metrics(E) \to \jJ(E,\omega)
$$
such that $\Phi(g_J) = J$ for every $J \in \jJ(E,\omega)$.  Then since
$\metrics(E)$ is a nonempty convex subset of a vector space and hence
contractible, the identity map $J \mapsto \Phi(g_J)$ can be contracted
to a point by contracting~$\metrics(E)$.

To construct the map~$\Phi$, observe that if $g \in \metrics(E)$ happens
to be of the form $g_J$ for some $J \in \jJ(E,\omega)$, then it is
related to~$J$ by $\omega \equiv g(J\cdot,\cdot)$.  For more
general metrics~$g$, this relation still determines $J$ as a linear
bundle map on~$E$, and the latter will not necessarily be a complex 
structure, but we will see that it is not hard to derive one from it.  
Thus as a first step, define a continuous map 
$$
\metrics(E) \to \Gamma(\End(E)) : g \mapsto A
$$
via the relation
$$
\omega \equiv g(A\cdot,\cdot).
$$
As is easy to check, the skew-symmetry of $\omega$ now implies that the
fiberwise adjoint of~$A$ with respect to the bundle metric~$g$ is
$$
A^* = -A,
$$
so in particular $A$ is a fiberwise \emph{normal} operator, i.e.~it commutes
with its adjoint.  Since $A^*A$ is a positive definite symmetric form
(again with respect to~$g$), it has a well-defined square root, and there
is thus a continuous map $\Gamma(\End(E)) \to \Gamma(\End(E))$ that 
sends~$A$ to
$$
J_g := A \sqrt{A^* A}^{-1}.
$$
Now since $A$ is normal, it also commutes with $\sqrt{A^*A}^{-1}$, and
then $A^*A = -A^2$ implies $J_g^2 = -\1$.
It is similarly straightforward to check that $J_g$ is compatible
with~$\omega$, and $J_g = J$ whenever $g = g_J$, hence the desired map
is $\Phi(g) = J_g$.
\end{proof}

The above implies that $\jJ^\tau(E,\omega)$ is also nonempty, since it
contains $\jJ(E,\omega)$.  Gromov's proof concludes by using certain
abstract topological principles to show that once
$\jJ(E,\omega)$ is known to be contractible, this forces
$\jJ^\tau(E,\omega)$ to be contractible as well.
The abstract principles in question come from homotopy theory---in
particular, one needs to be familiar with the notion of a \emph{Serre
fibration} and the \emph{homotopy exact sequence}
(see e.g.~\cite{Hatcher}*{Theorem~4.41}), which has the following useful
corollary:

\begin{lemma}
\label{lemma:LES}
Suppose $\pi : X \to B$ is a Serre fibration with path-connected base.
Then the fibers $\pi^{-1}(*)$ are weakly contractible if and only if
$\pi$ is a weak homotopy equivalence.  \qed
\end{lemma}
Recall that a map $f : X \to Y$ is said to be a \defin{weak homotopy 
equivalence} whenever the induced maps $f_* : \pi_k(X) \to \pi_k(Y)$ are
isomorphisms for all~$k$, and $X$ is \defin{weakly contractible} if
$\pi_k(X) = 0$ for all~$k$.  Whitehead's theorem \cite{Hatcher}*{Theorem~4.5}
implies that whenever $X$ is a connected smooth manifold, contractibility
and weak contractibility are equivalent.

We will find it convenient at this point to dispense with the
vector bundle $E \to M$
and restrict attention to a single fiber.  Recall that by Exercise~\ref{EX:J},
$\jJ(E)$ can be regarded as the space of smooth sections of a locally
trivial fiber bundle over~$M$.  We claim that the same is true of 
$\jJ^\tau(E,\omega)$:\footnote{The same is also true of $\jJ(E,\omega)$
and can be deduced from Proposition~\ref{prop:convex} or
Corollary~\ref{cor:compatible} below, but this is not
needed for the present discussion.}
indeed, pick a compatible
structure $J_0 \in \jJ(E,\omega)$, whose existence is guaranteed by the
above proof.  Then by Exercise~\ref{EX:Un}, $E \to M$ admits local
trivializations that identify $\omega$ and $J$ simultaneously with the standard
structures $\omega\std$ and~$i$, and in such a trivialization, any
$J \in \jJ^\tau(E,\omega)$ is identified locally with a
smooth map into a fixed open subset of the manifold
$\GL(2n,\RR) / \GL(n,\CC)$; see Exercise~\ref{EX:homogeneous}.
The following standard topological lemma will thus allow us to restrict 
attention to the various spaces of complex structures on the vector space~$\CC^n$.

\begin{lemma}
\label{lemma:contractibleFibers}
Suppose $\pi : E \to M$ is a smooth locally trivial fiber bundle over a 
manifold~$M$, and the fibers are contractible.
Then the space $\Gamma(E)$ of smooth sections is nonempty and contractible
(in the $\Cinftyloc$-topology).
\end{lemma}
\begin{proof}
It suffices to construct a smooth section $s_0 \in \Gamma(E)$ and a smooth
map $r : [0,1] \times E \to E$ such that $r(\tau,\cdot) : E \to E$ is
fiber preserving for all $\tau \in [0,1]$, $r(1,\cdot)$ is the identity
and $r(0,\cdot) = s_0 \circ \pi$.  Note that any such map can also be viewed
as a section of a fiber bundle, namely of 
$(\pi \circ \pr_2)^*E \to [0,1] \times E$,
where $\pr_2 : [0,1] \times E \to E$ denotes the natural projection, and
$r$ is required to match a fixed section over the closed subset
$\{0,1\} \times E$.  Then since continuous sections can always be approximated
by smooth ones \cite{Steenrod}*{\S 6.7}, it suffices to construct a
\emph{continuous} map $r$ with the above properties.

Let us therefore work in the topological category: assume $\pi : E \to M$ is
a topological fiber bundle with contractible fiber $F$, 
and $M$ is a finite-dimensional CW-complex.\footnote{The assumption 
that the CW-complex is finite dimensional is inessential, but lifting
it involves some logical subtleties, and we are anyway most interested in
the case where $M$ is a smooth finite-dimensional manifold.}
There is a standard procedure for constructing sections by induction over
the skeleta of~$M$, see \cite{Steenrod}.  Since $E$ is necessarily trivial
over each cell, it suffices to consider the closed $k$-disk $\DD^k \subset \RR^k$
for each $k \in \NN$ and the trivial bundle $\DD^k \times F \to \DD^k$: the 
key inductive step is then to show that any continuous maps 
$s_0 : \p \DD^k \to F$ and 
$r : [0,1] \times \p \DD^k \times F \to F$ satisfying $r(0,b,p) = s_0(b)$ 
and $r(1,b,p) = p$ for all $(b,p) \in \p\DD^k \times F$ can be
extended with these properties continuously over $\DD^k$ and $[0,1] \times \DD^k
\times F$ respectively.  Let us first extend $s_0$: this is clearly possible
since $\pi_{k-1}(F) = 0$.  We then require any extension of $r$ to satisfy
$r(0,b,p) = s_0(b)$ and $r(1,b,p) = p$ for all $(b,p) \in \DD^k \times F$, thus
the problem is to extend a map defined on
$$
\left( \{0,1\} \times \DD^k \times F \right) \cup \left( [0,1] \times \p\DD^k \times F \right)
= \p\left( [0,1] \times \DD^k \right) \times F
$$
over the interior of $[0,1] \times \DD^k \times F \cong \DD^{k+1} \times F$.  
This can be done using a contraction of~$F$.
\end{proof}

With Lemma~\ref{lemma:contractibleFibers} in hand, the proof of
Theorem~\ref{thm:contractible} will be complete if we can show
that the space $\jJ^\tau(\CC^n,\omega\std)$ of linear complex structures
on $\CC^n$ tamed by the standard symplectic form is contractible.

\begin{proof}[Proof that $\jJ^\tau(\CC^n,\omega\std)$ is contractible (Gromov)]
Let $\Omega(\CC^n)$ denote the space of nondegenerate skew-symmetric bilinear
forms on $\CC^n$, i.e.~linear symplectic structures.  We then
define the topological spaces
\begin{equation*}
\begin{split}
X(\CC^n) &= \{ (\omega,J) \in \Omega(\CC^n) \times \jJ(\CC^n) \ |\ 
J \in \jJ(\CC^n,\omega) \},\\
X^\tau(\CC^n) &= \{ (\omega,J) \in \Omega(\CC^n) \times \jJ(\CC^n) \ |\ 
J \in \jJ^\tau(\CC^n,\omega) \}.
\end{split}
\end{equation*}
Observe that for any fixed $J \in \jJ(\CC^n)$, the set of all $\omega \in \Omega(\CC^n)$
that tame~$J$ is convex, and thus contractible; the same is true for the set
of all $\omega \in \Omega(\CC^n)$ for which $J$ is $\omega$-compatible.
Thus the projection maps $\pr_2 : X(\CC^n) \to \jJ(\CC^n)$ and $\pr_2 : X^\tau(\CC^n) \to
\jJ(\CC^n)$ both have contractible fibers; one can show moreover that both are
Serre fibrations, and both are therefore weak homotopy equivalences by
Lemma~\ref{lemma:LES}.  This implies that the inclusion
$X(\CC^n) \hookrightarrow X^\tau(\CC^n)$ is also a weak homotopy equivalence.
Since the fibers $\jJ(\CC^n,\omega)$ of the projection
$\pr_1 : X(\CC^n) \to \Omega(\CC^n)$ are also contractible, the latter is also
a weak homotopy equivalence, and by commuting diagrams, we see that
$\pr_1 : X^\tau(\CC^n) \to \Omega(\CC^n)$ is therefore a weak homotopy equivalence,
whose fibers $\jJ^\tau(\CC^n,\omega)$ must then be contractible.
\end{proof}

\begin{exercise}
\label{EX:homotEquiv}
Show that for any vector bundle $E$ of even rank, there is a natural weak homotopy
equivalence between the space of complex structures $\jJ(E)$ and the
space of symplectic vector bundle structures $\Omega(E)$
(cf.~Exercise~\ref{EX:sympConvex}).
\end{exercise}

\begin{remark}
\label{remark:Achtung}
Exercise~\ref{EX:homotEquiv} does not immediately imply any correspondence between
the space of symplectic forms on a manifold $M$ and the space of almost
complex structures $\jJ(M)$, as a symplectic vector
bundle structure on $TM \to M$ is in general a nondegenerate $2$-form which
need not be closed.  Such a correspondence does exist however if $M$ is
\emph{open}, by a deep ``flexibility'' result of Gromov, 
see e.g.~\cite{EliashbergMishachev} or \cite{Geiges:hPrinciple}.
\end{remark}

We next give a more direct proof of Theorem~\ref{thm:contractible} 
using a variation on 
an argument due to S\'evennec (cf.~\cite{Audin:almostComplex}*{Corollary~1.1.7}), 
which can be applied somewhat more generally.
The starting point is the observation 
that for any choice of ``reference'' complex structure 
$J_0 \in \jJ(\CC^n)$, the map
\begin{equation}
\label{eqn:Cayley1}
Y \mapsto J_Y := \left( \1 + \frac{1}{2} J_0 Y \right) J_0 
\left( \1 + \frac{1}{2} J_0 Y \right)^{-1}.
\end{equation}
identifies a neighborhood of~$0$ in $\overline{\End}_\CC(\CC^n,J_0)$
smoothly with a neighborhood of $J_0$ in~$\jJ(\CC^n)$, 
and can thus be regarded as the inverse of a local chart on the 
smooth submanifold $\jJ(\CC^n) \subset \End_\RR(\CC^n)$;
cf.~Remark~\ref{remark:Cayley0} and the discussion that precedes it.
In fact, \eqref{eqn:Cayley1} is well defined for all $Y$ in the
open subset of $\overline{\End}_\CC(\CC^n)$ for which 
$\1 + \frac{1}{2} J_0 Y \in \GL(2n,\RR)$,
which turns out to be a large enough domain to cover the entirety of
$\jJ^\tau(\CC^n,\omega_0)$!  In the following statement, we say that a
subset $\uU \subset E$ in a vector bundle $E$ is \defin{fiberwise convex} if 
its intersection with every fiber is convex, and we denote by
$\Gamma(\uU)$ the space of (smooth) sections of~$E$ that are everywhere 
contained in~$\uU$.

\begin{prop}
\label{prop:convex}
Suppose $(E,\omega) \to M$ is a symplectic vector bundle and 
$J_0 \in \jJ^\tau(E,\omega)$.  Then there exists an open and 
fiberwise convex subset $\uU^{\omega,J_0} \subset \overline{\End}_{\CC}(E,J_0)$ such that
$$
\jJ^\tau(E,\omega) = \left\{ J_Y\ |\ Y \in \Gamma\left(\uU^{\omega,J_0}\right) \right\},
$$
where $J_Y$ is defined via \eqref{eqn:Cayley1}.  Moreover, if
$J_0 \in \jJ(E,\omega)$, let $\End_\RR^S(E,\omega,J_0) \subset \End_\RR(E)$
denote the subbundle of linear maps that are symmetric with respect to the
bundle metric $\omega(\cdot,J_0 \cdot)$.  Then
$$
\jJ(E,\omega) = \left\{ J_Y\ |\ Y \in 
\Gamma\left(\uU^{\omega,J_0} \cap \End_\RR^S(E,\omega,J_0)\right) \right\}.
$$
\end{prop}

The next exercise is a lemma needed for the proof of
Proposition~\ref{prop:convex}.

\begin{exercise}
\label{EX:Cayley}
Show that for any $J_0 \in \jJ(\CC^n)$, the map
\eqref{eqn:Cayley1} defines a bijection
$$
\left\{ Y \in \overline{\End}_\CC(\CC^n,J_0)\ \Big|\ 
\1 + \frac{1}{2} J_0 Y \in \GL(2n,\RR) \right\} \to
\left\{ J \in \jJ(\CC^n)\ |\ J_0 + J \in \GL(2n,\RR) \right\},
$$
with inverse $J \mapsto 2 J_0 ( J + J_0 )^{-1} ( J - J_0 )$.
\textsl{Hint: The identities
$(J \pm J_0) J_0 = J (J_0 \pm J)$ and $J_0 (J \pm J_0) = (J_0 \pm J)J$ hold for any
$J_0, J \in \jJ(\CC^n)$.  For some additional perspective on this exercise,
see Exercise~\ref{EX:CayleyInverse} and Remark~\ref{remark:explainCayley}.}
\end{exercise}

\begin{proof}[Proof of Proposition~\ref{prop:convex}]
Suppose $J_0$ and $J$ are two $\omega$-tame complex structures on some
fiber $E_x \subset E$ for $x \in M$.  Then $J_0 + J$ is invertible: indeed,
for any nontrivial $v \in E_x$ we have
$$
\omega(v , (J_0 + J) v) = \omega(v, J_0 v) + \omega(v, Jv) > 0,
$$
thus $J_0 + J$ has trivial kernel.  It follows by 
Exercise~\ref{EX:Cayley} that $J = J_Y$ for a unique
$Y \in \overline{\End}_\CC(E_x,J_0)$.  Denote by
$\uU_x^{\omega,J_0}$ the set of complex-antilinear 
maps $Y : E_x \to E_x$ that arise in this way.

To show that $\uU_x^{\omega,J_0}$ is convex, observe that the
condition $Y \in \uU_x^{\omega,J_0}$ means
$$
\omega\left( v , \left(\1 + \frac{1}{2} J_0 Y\right) J_0 
\left( \1 + \frac{1}{2} J_0 Y\right)^{-1} v \right) > 0
\quad\text{ for all $v \in E_x \setminus \{0\}$},
$$
which is equivalent to
$$
\omega\left( \left(\1 + \frac{1}{2} J_0 Y\right) v ,
\left(\1 + \frac{1}{2} J_0 Y\right) J_0 v \right) > 0
\quad\text{ for all $v \in E_x \setminus \{0\}$}.
$$
Given $Y_0, Y_1 \in \uU_x^{\omega,J_0}$, let
$Y_t = t Y_1 + (1-t) Y_0$ for $t \in [0,1]$, fix a nontrivial
vector $v \in E_x$ and consider the function
$$
P_v(t) := 
\omega\left( \left(\1 + \frac{1}{2} J_0 Y_t\right) v ,
\left(\1 + \frac{1}{2} J_0 Y_t \right) J_0 v \right) \in \RR.
$$
This function is of the form $P_v(t) = a t^2 + bt + c$, and using the fact 
that $J_0$ anticommutes with both $Y_0$ and $Y_1$, we find that its quadratic
coefficient is
\begin{equation*}
\begin{split}
a &= \omega\left( \frac{1}{2} J_0 (Y_1 - Y_0) v, \frac{1}{2} J_0
(Y_1 - Y_0) J_0 v \right) \\
&= - \omega\left( \frac{1}{2} J_0 (Y_1 - Y_0) v, J_0 \left[ \frac{1}{2} J_0
(Y_1 - Y_0) v \right] \right) \le 0
\end{split}
\end{equation*}
since $J_0$ is tamed by~$\omega$.  This implies that $P_v$ is a concave 
function, and since $P_v(0)$ and $P_v(1)$ are both positive, we conclude
$P_v(t) > 0$ and hence $Y_t \in \uU_x^{\omega,J_0}$ for all $t \in [0,1]$.

Finally, if $J_0$ is $\omega$-compatible, we will show that $J_Y$ is also
compatible if and only if $Y$ satisfies $\langle v,Yw \rangle =
\langle Yv,w \rangle$ for all $v,w \in E_x$, where $\langle v ,w \rangle
:= \omega(v,J_0 w)$.  Recall that by Exercise~\ref{EX:Jinvariant}, an
$\omega$-tame complex structure $J$ is $\omega$-compatible if and only if
$\omega$ is $J$-invariant, i.e.~$\omega(v,w) = \omega(Jv,Jw)$ for all $v,w$.
Plugging in $J=J_Y$ and replacing $v$ and $w$ by 
$\left(\1 + \frac{1}{2} J_0 Y\right) v$
and $\left(\1 + \frac{1}{2} J_0 Y\right) w$ respectively, this condition
is equivalent to
$$
\omega\left( \left(\1 + \frac{1}{2} J_0 Y\right) v, \left(\1 + \frac{1}{2} J_0 Y\right) w \right) =
\omega\left( \left(\1 + \frac{1}{2} J_0 Y\right) J_0 v, \left(\1 + \frac{1}{2} J_0 Y\right) J_0 w \right)
$$
for all $v,w \in E_x$.
Expanding both sides, using the fact that $\omega$ is also $J_0$-invariant and then
cancelling everything that can be cancelled, one derives from this the condition
$$
-\omega(Yv,J_0 w) + \omega(v,J_0 Yw) = 0 \quad\text{ for all $v,w \in E_x$},
$$
which means $- \langle Yv,w \rangle + \langle v,Yw \rangle = 0$.
\end{proof}

As an easy corollary, we have:

\begin{proof}[Alternative proof of Theorem~\ref{thm:contractible} (after S\'evennec)]
\ \\
Using Proposition~\ref{prop:convex}, each of the spaces 
$\jJ^\tau(E,\omega)$ and $\jJ(E,\omega)$ is contractible if it is nonempty,
as it can then be identified via \eqref{eqn:Cayley1} with a convex subset
of a vector space.  Nonemptiness follows from this almost immediately: indeed,
Proposition~\ref{prop:convex} also implies that both $\jJ^\tau(E,\omega)$ and
$\jJ(E,\omega)$ can be regarded as the spaces of sections of certain smooth
fiber bundles with contractible fibers; the fibers are each obviously nonempty
since $i \in \jJ(\CC^n,\omega\std)$.  Existence of sections then follows from
Lemma~\ref{lemma:contractibleFibers}.
\end{proof}

\begin{exercise}
\label{EX:Jextension}
Prove the following generalization of Theorem~\ref{thm:contractible} for
extensions: given a symplectic vector bundle $(E,\omega) \to M$, 
a closed subset
$A \subset M$ and a compatible/tame complex structure $J$ defined
on $E$ over a neighborhood of~$A$, the space of compatible/tame complex
structures on $(E,\omega)$ that match $J$ near~$A$ is nonempty and contractible.
\end{exercise}
\begin{exercise}
\label{EX:JextensionSub}
In the setting of the previous exercise, suppose additionally that we
are given a submanifold $\Sigma \subset M$ and a symplectic subbundle
$F \subset E|_{\Sigma}$.  Show that if $J$ is a
compatible/tame complex structure that is defined on a neighborhood of~$A$
and preserves $F$ over a neighborhood of $\Sigma \cap A$ in $\Sigma$,
and $j$ is a compatible/tame complex structure on $F$ that matches
$J|_{F}$ near $\Sigma \cap A$, then the space of all compatible/tame
complex structures on $E$ that match $J$ near~$A$ and restrict to
$j$ on $F$ is also nonempty and contractible.
\textsl{Hint: It may help to recall Exercise~\ref{EX:subbundles}.}
\end{exercise}

Proposition~\ref{prop:convex} also implies
the following useful description of $\jJ(\CC^n,\omega\std)$,
which we will need in Chapter~\ref{chapter:moduli}:

\begin{cor}
\label{cor:compatible}
The space $\jJ(\CC^n,\omega\std)$ is a smooth submanifold of
$\End_\RR(\CC^n)$, with tangent space at~$i \in \jJ(\CC^n,\omega\std)$ 
given by
$$
T_i \jJ(\CC^n,\omega\std) = \{ Y \in \overline{\End}_\CC(\CC^n) \ |\ 
\text{$Y$ is symmetric} \}.
$$
Moreover, the map $Y \mapsto J_Y$ of \eqref{eqn:Cayley0} identifies
a neighborhood of $0$ in $T_i \jJ(\CC^n,\omega\std)$ smoothly with a
neighborhood of $i$ in $\jJ(\CC^n,\omega\std)$.
\end{cor}

\begin{remark}
The above argument can also be used to show that for any collection
$\Omega$ of symplectic structures on a given bundle $E \to M$, the spaces of
complex structures that are simultaneously either
tamed by or compatible with \emph{every} $\omega \in \Omega$
are contractible whenever they are nonempty, see 
\cite{MassotNiederkruegerWendl}*{Appendix~A.1}.  Of course, such spaces
may indeed be empty if $\Omega$ has more than one element.
\end{remark}

As an aside, it is worth mentioning an alternative way 
to understand Proposition~\ref{prop:convex} in terms of the classical
\emph{Cayley transform}; this was the original viewpoint of S\'evennec as
presented in \cite{Audin:almostComplex}.
The Cayley transform on $\CC$ is the linear fractional
transformation
$$
\varphi(z) = \frac{z - i}{z + i},
$$
which maps $\CC \setminus \{ -i \}$ conformally and bijectively 
to $\CC \setminus \{1\}$, sending $\{ \Im z > 0 \}$ to
$\{ |z| < 1 \}$ and $i$ to~$0$.  Its inverse is
$\varphi^{-1}(w) = -i \frac{w+1}{w-1}$.

Notice that if we identify $\CC$ with the subspace of
$\End_\RR(\CC^n)$ consisting of complex multiples of the
identity, then $\varphi$ is the restriction of the map
\begin{equation}
\label{eqn:CayleyMatrix}
\Phi(J) := (J + i)^{-1} (J - i),
\end{equation}
defined for all $J \in \End_\RR(\CC^n)$ such that $J + i \in
\GL(2n,\RR)$, with $i$ now denoting the standard complex
structure on~$\CC^n$.\footnote{Since $\End_\RR(\CC^n)$
is not commutative, 
there are actually two obvious extensions of $\varphi$
to $\End_\RR(\CC^n)$, the other being $\Phi(J) :=
(J - i) (J + i)^{-1}$.  One could carry out this entire
discussion with the alternative choice and prove equivalent
results.}

\begin{exercise}
\label{EX:CayleyInverse}
Show that \eqref{eqn:CayleyMatrix} defines a diffeomorphism
$$
\{ J \in \End_\RR(\CC^n)\ |\ J + i \in \GL(2n,\RR) \} \to
\{ Y \in \End_\RR(\CC^n)\ |\ Y - \1 \in \GL(2n,\RR) \},
$$
with inverse $\Phi^{-1}(Y) = -i (Y + \1) (Y - \1)^{-1}$.
\end{exercise}
\begin{exercise}
Denote the natural inclusion of $\CC \hookrightarrow \End_\RR(\CC^n)$ as
described above by $z \mapsto J_z$.  If $\omega\std$ is the standard
symplectic form on~$\CC^n = \RR^{2n}$, show that
$\omega\std(v, J_z v) > 0$ holds for all nontrivial $v \in \CC^n$
if and only if $z$ lies in the open upper half-plane.
\end{exercise}

With the previous exercise in mind,
the fact that $\varphi$ maps the upper half-plane to the unit disk
in $\CC$ now generalizes as follows.  Let $\|\cdot\|$ denote the
operator norm on $\End_\RR(\CC^n)$ defined via the standard
Euclidean metric $\langle \cdot,\cdot \rangle = \omega\std(\cdot, i\cdot)$.

\begin{lemma}
\label{lemma:unitBall}
Every $J \in \End_\RR(\CC^n)$ that satisfies $\omega\std(v,Jv) > 0$ 
for all nontrivial $v \in \CC^n$ is in the domain of $\Phi$,
and every $Y \in \End_\RR(\CC^n)$ with $\| Y \| < 1$ is in the
domain of $\Phi^{-1}$.  Moreover, a given $J$ in the domain of $\Phi$
satisfies the above condtion with respect to $\omega\std$
if and only if $\| \Phi(J) \| < 1$.
\end{lemma}
\begin{exercise}
Show that if $J \in \End_\RR(\CC^n)$ satisfies $\omega\std(v,Jv) > 0$ 
for all $v \ne 0$, then $J + i$ is invertible, thus $J$ is in the
domain of $\Phi$.  It follows via Exercise~\ref{EX:CayleyInverse}
that for $Y := \Phi(J)$, $Y - \1$ is invertible and
$J = -i (Y + \1) (Y - \1)^{-1}$.  Now given $v \in \CC^n$, write
$w = (Y - \1)^{-1} v$ and show that
$\omega\std(v,Jv) = |w|^2 - |Yw|^2$.
Use this to prove Lemma~\ref{lemma:unitBall}.
\end{exercise}

\begin{exercise}
\label{EX:complexCayley}
If $Y = \Phi(J)$, show that $J \in \jJ(\CC^n)$ if and only if
$Y \in \overline{\End}_\CC(\CC^n)$.
\textsl{Hint: Notice that when $\Phi(J) = Y \in \overline{\End}_\CC(\CC^n)$, we have
\begin{equation}
\label{eqn:whenYisAntilinear}
J = \Phi^{-1}(Y) = -i (Y + \1) (Y - \1)^{-1} = (Y - \1) i (Y - \1)^{-1}.
\end{equation}
As in Exercise~\ref{EX:Cayley}, the identities
$(J \pm i) J = -1 \pm iJ = i (i \pm J)$ and
$(J \pm i) i = Ji \mp \1 = J (i \pm J)$ hold if $J \in \jJ(\CC^n)$.}
\end{exercise}

\begin{remark}
\label{remark:explainCayley}
In light of \eqref{eqn:whenYisAntilinear} above, one can now express the
map $Y \mapsto J_Y$ from \eqref{eqn:Cayley0} as the composition of
$\Phi^{-1}$ with the linear isomorphism
$$
\overline{\End}_\CC(\CC^n) \to \overline{\End}_\CC(\CC^n) : Y \mapsto
-\frac{1}{2} i Y.
$$
\end{remark}

Together with our characterization of the compatible case in the proof
of Proposition~\ref{prop:convex}, the results of
Lemma~\ref{lemma:unitBall} and Exercise~\ref{EX:complexCayley} can now
be summarized as follows.

\begin{thm}[S\'evennec]
\label{thm:Sevennec}
The Cayley transform $J \mapsto (J + i)^{-1} (J - i)$ defines
diffeomorphisms
\begin{equation*}
\begin{split}
\jJ^\tau(\CC^n,\omega\std) &\to \{ Y \in \overline{\End}_\CC(\CC^n)\ |\ 
\| Y \| < 1 \},\\
\jJ(\CC^n,\omega\std) &\to \{ Y \in \overline{\End}_\CC(\CC^n)\ |\ 
\text{$\| Y \| < 1$ and $Y$ is symmetric }\}.
\end{split}
\end{equation*}  \qed
\end{thm}

\begin{remark}
Theorem~\ref{thm:Sevennec} could be stated a bit more generally by
replacing $\omega\std$ and $i$ with different symplectic and complex
structures $\omega$ and $J_0$ respectively, but in this form, it
does require the assumption that $J_0$ be \emph{compatible} with
$\omega$, not just tame.  Our alternative proof of
Theorem~\ref{thm:contractible} had the slight advantage of not 
requiring this extra condition, and this relaxation is important in
certain applications, cf.~\cite{MassotNiederkruegerWendl}*{Appendix~A.1}.
\end{remark}

\section{Linear Cauchy-Riemann type operators}

Many important results about solutions to the nonlinear Cauchy-Riemann
equation can be reduced to statements about solutions of corresponding
linearized equations, thus it is important to understand the linearized
equations first.  Consider a Riemann surface $(\Sigma,j)$ and a
complex vector bundle $(E,J) \to (\Sigma,j)$ of (complex) rank~$n$: this
means that $E \to \Sigma$ is a real vector bundle of rank~$2n$ and $J$ is
a complex structure on the bundle.  We say that the bundle admits
a \defin{holomorphic structure} if $\Sigma$ has an open covering 
$\{ \uU_\alpha \}$ with complex-linear local trivializations 
$E|_{\uU_\alpha} \to \uU_\alpha \times \CC^n$
whose transition maps are holomorphic functions from open subsets of
$\Sigma$ to $\GL(n,\CC)$.

On the space $C^\infty(\Sigma,\CC)$ of 
smooth complex-valued functions, there are natural first-order differential
operators
\begin{equation}
\label{eqn:dbar}
\dbar : f \mapsto df + i \,df \circ j
\end{equation}
and
\begin{equation}
\label{eqn:d}
\p : f \mapsto df - i \,df \circ j.
\end{equation}
We can regard $\dbar$ as a linear map
$C^\infty(\Sigma) \to \Gamma(\overline{\Hom}_\CC(T\Sigma,\CC))$, where the
latter denotes the space of smooth sections of the bundle
$\overline{\Hom}_\CC(T\Sigma,\CC)$ of complex-antilinear
maps $T\Sigma \to \CC$; similarly, $\p$ maps $C^\infty(\Sigma)$ to
$\Gamma(\Hom_\CC(T\Sigma,\CC))$.\footnote{Many authors prefer to write 
the spaces of sections of $\Hom_\CC(T\Sigma,\CC)$ and
$\overline{\Hom}_\CC(T\Sigma,\CC)$ as $\Omega^{1,0}(\Sigma)$ and
$\Omega^{0,1}(\Sigma)$ respectively, calling these sections
``$(1,0)$-forms'' and ``$(0,1)$-forms.''}  
Observe that the holomorphic functions
$f : \Sigma \to \CC$ are precisely those which satisfy $\dbar f \equiv 0$;
the solutions of $\p f \equiv 0$ are called \defin{antiholomorphic}.

If $(E,J) \to (\Sigma,j)$ has a holomorphic structure, one can likewise
define a natural operator on the space of sections $\Gamma(E)$,
$$
\dbar : \Gamma(E) \to \Gamma(\overline{\Hom}_\CC(T\Sigma,E)),
$$
which is defined the same as \eqref{eqn:dbar} on any section written in a
local holomorphic trivialization.  We then call a section $v \in \Gamma(E)$
holomorphic if $\dbar v \equiv 0$, which is equivalent to the condition
that it look holomorphic in all holomorphic local trivializations.

\begin{exercise}
Check that the above definition of $\dbar : \Gamma(E) \to
\Gamma(\overline{\Hom}_\CC(T\Sigma,E))$ doesn't depend on the trivialization
if all transition maps are holomorphic.  You may find 
Exercise~\ref{EX:Leibnitz} helpful.  (Note that the operator
$\p f := df - i \,df \circ j$ is \emph{not} similarly well defined on
a holomorphic bundle---it does depend on the trivialization in general.)
\end{exercise}
\begin{exercise}
\label{EX:Leibnitz}
Show that the $\dbar$-operator on a holomorphic vector bundle satisfies
the following Leibnitz identity: for any $v \in \Gamma(E)$ and
$f \in C^\infty(\Sigma,\CC)$, $\dbar(f v) = (\dbar f) v + f (\dbar v)$.
\end{exercise}

\begin{defn}
A \defin{complex-linear Cauchy-Riemann type operator} on a complex vector bundle
$(E,J) \to (\Sigma,j)$ is a complex-linear map 
$$
D : \Gamma(E) \to \Gamma(\overline{\Hom}_\CC(T\Sigma,E))
$$ 
that satisfies the Leibnitz rule
\begin{equation}
\label{eqn:Leibnitz}
D(f v) = (\dbar f) v + f (D v)
\end{equation}
for all $f \in C^\infty(\Sigma,\CC)$ and
$v \in \Gamma(E)$.
\end{defn}
One can think of this definition as analogous to the simplest modern 
definition of a connection on a vector bundle; in fact it turns out that
every complex Cauchy-Riemann type operator is the complex-linear part of
some connection (see Proposition~\ref{prop:connection} below).
The following is then the Cauchy-Riemann version of the existence
of the Christoffel symbols.

\begin{exercise}
\label{EX:Christoffel}
Fix a complex vector bundle $(E,J) \to (\Sigma,j)$.
\begin{enumerate}
\renewcommand{\labelenumi}{(\alph{enumi})}
\item
Show that if $D$ and $D'$ are two complex-linear Cauchy-Riemann type 
operators on $(E,J)$, then there exists a smooth
complex-linear bundle map $A : E \to \overline{\Hom}_\CC(T\Sigma,E)$
such that $D'v = Dv + A v$ for all $v \in \Gamma(E)$.
\item
Show that in any local trivialization on a subset $\uU \subset \Sigma$,
every complex-linear Cauchy-Riemann type operator $D$ can be written in
the form
$$
D v = \dbar v + A v,
$$
for some smooth map $A : \uU \to \End_\CC(\CC^n)$.
\end{enumerate}
\end{exercise}

\begin{exercise}
\label{EX:nabla}
Show that if $\nabla$ is any complex connection on $E$,\footnote{By ``complex
connection'' we mean that the parallel transport isomorphisms are 
complex-linear.  This is equivalent to the requirement that $\nabla :
\Gamma(E) \to \Gamma(\Hom_\RR(T\Sigma,E))$ be a complex-linear map.} then
$\nabla + J \circ \nabla \circ j$ is a complex-linear Cauchy-Riemann type 
operator.
\end{exercise}

\begin{prop}
\label{prop:connection}
For any Hermitian vector bundle $(E,J) \to (\Sigma,j)$ with a
complex-linear Cauchy-Riemann type operator $D : \Gamma(E) \to
\Gamma(\overline{\Hom}_\CC(T\Sigma,E))$, there exists a unique
Hermitian connection $\nabla$ such that $D = \nabla + J \circ \nabla \circ j$.
\end{prop}
\begin{proof}
Denote the Hermitian bundle metric by $\langle\ ,\ \rangle$, and for any
choice of connection $\nabla$, denote
$$
\nabla^{1,0} := \nabla - J \circ \nabla \circ j \quad\text{ and }\quad
\nabla^{0,1} := \nabla + J \circ \nabla \circ j.
$$
Any Hermitian connection satisfies
\begin{equation}
\label{eqn:this2}
d \langle \xi,\eta \rangle = \langle \nabla \xi, \eta \rangle +
\langle \xi , \nabla \eta \rangle,
\end{equation}
for $\xi, \eta \in \Gamma(E)$, where both sides are to be interpreted as
complex-valued $1$-forms.  Then applying $\p = d - i \circ d \circ j$ and
$\dbar = d + i \circ d \circ j$ to the function in \eqref{eqn:this2} leads to 
the two relations
\begin{equation*}
\begin{split}
\p \langle \xi, \eta \rangle &= \langle \nabla^{0,1} \xi, \eta \rangle
+ \langle \xi, \nabla^{1,0} \eta \rangle, \\
\dbar \langle \xi, \eta \rangle &= \langle \nabla^{1,0} \xi, \eta \rangle
+ \langle \xi, \nabla^{0,1} \eta \rangle.
\end{split}
\end{equation*}
Now if we require $\nabla^{0,1} = D$, the rest of $\nabla$ is uniquely
determined by the relation
$$
\langle \nabla^{1,0} \xi, \eta \rangle = \dbar \langle \xi, \eta \rangle -
\langle \xi , D \eta \rangle.
$$
Indeed, taking this as a \emph{definition} of $\nabla^{1,0}$ and writing
$\nabla := \frac{1}{2}(\nabla^{1,0} + D)$, it is straightforward 
to verify that $\nabla$ is now a Hermitian connection.
\end{proof}

Since connections exist in abundance on any vector bundle, there is always
a Cauchy-Riemann type operator, even if $(E,J)$ doesn't come equipped with
a holomorphic structure.  We now have the following analogue of
Theorem~\ref{thm:RiemannSurfaces} for bundles:

\begin{thm}
\label{thm:holomorphicBundles}
For any complex-linear Cauchy-Riemann type operator $D$ on a complex
vector bundle $(E,J)$ over a Riemann surface $(\Sigma,j)$, there is a 
unique holomorphic structure on $(E,J)$ such that the naturally induced
$\dbar$-operator is~$D$.
\end{thm}
The proof can easily be reduced to the following local existence lemma,
which is a special case of an analytical result that we'll prove in
\S\ref{sec:localSections} (see Theorem~\ref{thm:linearExistence}):
\begin{lemma}
\label{lemma:localExistence}
Suppose $D$ is a complex-linear Cauchy-Riemann type operator on
$(E,J) \to (\Sigma,j)$.  Then for any $z \in \Sigma$ and $v_0 \in E_z$,
there is a neighborhood $\uU \subset\Sigma$ of $z$ and a smooth section
$v \in \Gamma(E|_{\uU})$ such that $D v = 0$ and $v(z) = v_0$.
\end{lemma}
\begin{exercise}
Prove Theorem~\ref{thm:holomorphicBundles}, assuming 
Lemma~\ref{lemma:localExistence}.
\end{exercise}

As we'll see in the next section, it's also quite useful to consider 
Cauchy-Riemann type operators
that are only \emph{real}-linear, rather than complex.

\begin{defn}
\label{defn:realCR}
A \defin{real-linear Cauchy-Riemann type operator} on a complex vector bundle
$(E,J) \to (\Sigma,j)$ is a real-linear map $D : \Gamma(E) \to 
\Gamma(\overline{\Hom}_\CC(T\Sigma,E))$ 
such that \eqref{eqn:Leibnitz} is satisfied
for all $f \in C^\infty(\Sigma,\RR)$ and $v \in \Gamma(E)$.
\end{defn}
\begin{remark}
To understand Definition~\ref{defn:realCR}, it is important to note that
when $f$ is a real-valued function on $\Sigma$, the $1$-form
$\dbar f$ is still \emph{complex}-valued, so multiplication of
$\dbar f$ by sections of $E$ involves the complex structure.
\end{remark}
The following is now an addendum to Exercise~\ref{EX:Christoffel}.
\begin{exercise}
\label{EX:ChristoffelReal}
Show that in any local trivialization on a subset $\uU \subset \Sigma$,
every real-linear Cauchy-Riemann type operator $D$ can be written in
the form
$$
D v = \dbar v + A v,
$$
for some smooth map $A : \uU \to \End_\RR(\CC^n)$, where 
$\End_\RR(\CC^n)$ denotes
the space of \emph{real}-linear maps on $\CC^n = \RR^{2n}$.
\end{exercise}

\section{The linearization of $\dbar_J$ and critical points}
\label{sec:linearization}

We shall now see how linear Cauchy-Riemann type operators arise naturally
from the nonlinear Cauchy-Riemann equation.  
Theorem~\ref{thm:holomorphicBundles} will then allow us already to prove
something quite nontrivial: nonconstant $J$-holomorphic curves have 
only isolated critical points!  It turns out that one can reduce this
result to the corresponding statement about zeroes of holomorphic functions, 
a well-known fact from complex analysis.

For the next few paragraphs, we will be doing a very informal version of
``infinite-dimensional differential geometry,'' in which we assume that
various spaces of smooth maps can sensibly be regarded as infinite-dimensional
smooth manifolds and vector bundles.  For now this is purely for motivational
purposes, thus we can avoid worrying about the technical details; when it
comes time later to prove something using these ideas, we'll have to
replace the spaces of smooth maps with Banach spaces, which will have to
contain nonsmooth maps in order to attain completeness.

So, morally speaking, if $(\Sigma,j)$ is a Riemann surface and
$(M,J)$ is an almost complex manifold, then
the space of smooth maps $\bB := C^\infty(\Sigma,M)$ is an infinite-dimensional
smooth manifold, and there is a vector bundle $\eE \to \bB$ whose fiber
$\eE_u$ at $u \in \bB$ is the space of smooth sections,
$$
\eE_u = \Gamma(\overline{\Hom}_\CC(T\Sigma,u^*TM)),
$$
where we pull back $J$ to define a complex bundle structure on $u^*TM \to
\Sigma$.  The tangent vectors at a point $u \in \bB$ are simply vector fields
along $u$, thus
$$
T_u \bB = \Gamma(u^*TM).
$$
Now we define a section $\dbar_J : \bB \to \eE$ by
$$
\dbar_J u = Tu + J \circ Tu \circ j.
$$
This section is called the nonlinear Cauchy-Riemann operator, and its
zeroes are precisely the $J$-holomorphic maps from
$\Sigma$ to~$M$.  Recall now that zero sets of smooth sections on 
bundles generically have a very nice structure---this follows from the
implicit function theorem, of which we'll later use an infinite-dimensional
version.  For motivational purposes only, we state here a finite-dimensional
version with geometric character.  
Recall that any section of a bundle can be regarded as an
embedding of the base into the total space, thus we can always ask whether
two sections are ``transverse'' when they intersect.

\begin{thm}[Finite dimensional implicit function theorem]
Suppose $E \to B$ is a smooth vector bundle of real rank~$k$ over an 
$n$-dimensional manifold and $s : B \to E$ is a smooth section that is
everywhere transverse to the zero section.  Then $s^{-1}(0) \subset B$ is a
smooth submanifold of dimension $n - k$.
\end{thm}
The transversality assumption can easily be restated in terms of the
\emph{linearization} of the section $s$ at a zero.  The easiest way to
define this is by choosing a connection $\nabla$ on $E \to B$, as one can
easily show that the linear map $\nabla s : T_p B \to E_p$ is independent
of this choice at any point $p$ where $s(p) = 0$; this follows from the
fact that $TE$ along the zero section has a canonical splitting into 
horizontal and vertical subspaces.  Let us therefore denote
the linearization at $p \in s^{-1}(0)$ by
$$
Ds(p) : T_p B \to E_p.
$$
Then the intersections of $s$ with the zero section are precisely the set
$s^{-1}(0)$, and these intersections are transverse if and only if
$Ds(p)$ is a surjective map for all $p \in s^{-1}(0)$.

In later chapters we will devote considerable effort to finding ways of
showing that the
linearization of $\dbar_J$ at any $u \in \dbar_J^{-1}(0)$ is a 
surjective operator in the appropriate Banach space setting.  With this as
motivation, let us now deduce a formula for the linearization itself.  
It will be slightly easier to do this if we regard $\dbar_J$ as a section
of the larger vector bundle $\widehat{\eE}$ with fibers
$$
\widehat{\eE}_u = \Gamma(\Hom_\RR(T\Sigma,u^*TM)).
$$
To choose a ``connection'' on $\widehat{\eE}$, choose first a
connection $\nabla$ on $M$ and assume that for any smoothly parametrized path
$\tau \mapsto u_\tau \in \bB$ and a section 
$\ell_\tau \in \widehat{\eE}_{u_\tau} = \Gamma(\Hom_\RR(T\Sigma,u_\tau^*TM))$
along the path,
the covariant derivative $\nabla_\tau \ell_\tau \in \widehat{\eE}_{u_\tau}$ 
should take the form
$$
(\nabla_\tau \ell_\tau)X = \nabla_\tau\left( \ell_\tau(X) \right) \in
(u^*TM)_z = T_{u(z)}M
$$
for $z \in\Sigma$, $X \in T_z\Sigma$.  Then $\nabla_\tau \ell_\tau$ doesn't depend on
the choice of $\nabla$ at any value of $\tau$ for which $\ell_\tau=0$.

Now given $u \in \dbar_J^{-1}(0)$, consider a smooth family of maps
$\{ u_\tau \}_{\tau \in (-1,1)}$ with $u_0 = u$, and write
$\p_\tau u_\tau|_{\tau=0} =: \eta \in \Gamma(u^*TM)$.  By definition,
the linearization
$$
D\dbar_J(u) : \Gamma(u^*TM) \to \Gamma(\Hom_\RR(T\Sigma,u^*TM))
$$
will be the unique linear map such that
$$
D\dbar_J(u) \eta = \left.\nabla_\tau \left( \dbar_J u_\tau 
\right)\right|_{\tau=0} =
\left. \nabla_\tau \left[ T u_\tau + J(u_\tau) \circ T u_\tau \circ j \right]
\right|_{\tau=0}.
$$
To simplify this expression, choose holomorphic local coordinates
$s + it$ near the point $z \in \Sigma$ and consider the action of the above
expression on the vector $\p_s$: this gives
$$
\left. \nabla_\tau \left[ \p_s u_\tau + J(u_\tau) \p_t u_\tau \right]
\right|_{\tau=0}.
$$
The expression simplifies further if we assume $\nabla$ is a
\emph{symmetric} connection on~$M$; this is allowed since the end result
will not depend on the choice of connection.  In this case
$\nabla_\tau \p_s u_\tau |_{\tau=0} = \nabla_s \p_\tau u_\tau |_{\tau=0}
= \nabla_s \eta$ and similarly for the derivative by~$t$, thus the above
becomes
$$
\nabla_s \eta + J(u) \nabla_t \eta + (\nabla_\eta J) \p_t u.
$$
Taking the coordinates back out, we're led to the following expression for
the linearization of $\dbar_J$:
\begin{equation}
\label{eqn:linearization}
D\dbar_J(u) \eta = \nabla \eta + J(u) \circ \nabla\eta \circ j +
(\nabla_\eta J) \circ Tu \circ j.
\end{equation}
Though it may seem non-obvious from looking at the formula, it turns
out that the right hand side of
\eqref{eqn:linearization} belongs not only to $\widehat{\eE}_u$ but
also to $\eE_u$, i.e.~it is a complex 
antilinear bundle map $T\Sigma \to u^*TM$.
\begin{exercise}
Verify that if $u \in \dbar_J^{-1}(0)$, then 
for any $\eta \in \Gamma(u^*TM)$, the bundle map 
$T\Sigma \to u^*TM$ defined by the right hand side of
\eqref{eqn:linearization} is complex-antilinear.  \textsl{Hint: Show first
that $\nabla_X J$ always anticommutes with $J$ for any vector $X$.}
\end{exercise}

To move back into the realm of solid mathematics, let us now regard
\eqref{eqn:linearization} as a definition, i.e.~to any smooth $J$-holomorphic
map $u : \Sigma \to M$ we associate the operator
$$
\mathbf{D}_u := D\dbar_J(u),
$$
which is a real-linear map taking sections of $u^*TM$ to sections of
$\overline{\Hom}_\CC(T\Sigma,u^*TM)$.  The following exercise is
straightforward but important.

\begin{exercise}
Show that $\mathbf{D}_u$ is a real-linear Cauchy-Riemann type operator on
$u^*TM$.
\end{exercise}

With this and Theorem~\ref{thm:holomorphicBundles} to work with, it is
already quite easy to prove that $J$-holomorphic curves have isolated
critical points.  The key idea, due to Ivashkovich and Shevchishin
\cite{IvashkovichShevchishin}, is to use the linearized operator $\mathbf{D}_u$
to define a holomorphic structure on $\Hom_\CC(T\Sigma,u^*TM)$ so that
$du$ becomes a holomorphic section.  Observe first that since
$(\Sigma,j)$ is a complex manifold, the bundle $T\Sigma \to \Sigma$ has
a natural holomorphic structure, so one can speak of holomorphic vector
fields on~$\Sigma$.  In general such vector fields will be defined only
locally, but this is sufficient for our purposes.

\begin{exercise}
\label{EX:holX}
A map $\varphi : (\Sigma,j) \to (\Sigma,j)$ is holomorphic if and only 
if it satisfies
the low-dimensional case of the nonlinear Cauchy-Riemann equation,
$\dbar_j\varphi = 0$.  The simplest example of such a map is the 
identity~$\Id : \Sigma \to \Sigma$, 
and the linearization $\mathbf{D}_{\Id}$ gives an operator
$\Gamma(T\Sigma) \to \Gamma(\overline{\Hom}_\CC(T\Sigma,T\Sigma))$.
Show that $\mathbf{D}_{\Id}$ is complex-linear, and in fact it is the natural
Cauchy-Riemann operator determined by the holomorphic structure of~$T\Sigma$.
\textsl{Hint: In holomorphic local coordinates this is almost obvious.}
\end{exercise}

\begin{lemma}
\label{lemma:holomorphicFlow}
Suppose $X$ is a holomorphic vector field on some open subset
$\uU \subset \Sigma$, $\uU' \subset \uU$ is another open subset and
$\epsilon > 0$ a number such that the flow $\varphi^t_X : \uU' \to \Sigma$
is well defined for all $t \in (-\epsilon,\epsilon)$.  Then the maps
$\varphi^t_X$ are holomorphic.
\end{lemma}
\begin{proof}
Working in local holomorphic coordinates, this reduces to the following
claim: if $\uU \subset \CC$ is an open subset containing a smaller open
set $\uU' \subset \uU$, $X : \uU \to \CC$ is a
holomorphic function and $\varphi^\tau : \uU' \to \CC$ satisfies
\begin{equation}
\label{eqn:this}
\begin{split}
\p_\tau \varphi^\tau(z) &= X(\varphi^\tau(z)),\\
\varphi^0(z) &= z
\end{split}
\end{equation}
for $\tau \in (-\epsilon,\epsilon)$, then $\varphi^\tau$ is holomorphic
for every~$\tau$.  To see this, apply the operator $\dbar := \p_s + i\p_t$
to both sides of \eqref{eqn:this} and exchange the order of partial
derivatives: this gives
$$
\frac{\p}{\p \tau} \dbar \varphi^\tau(z) = X'(\varphi^\tau(z)) \cdot
\dbar\varphi^\tau(z).
$$
For any fixed $z \in \uU'$, this is a linear differential equation for
the complex-valued path $\tau \mapsto \dbar\varphi^\tau(z)$.
Since it begins at zero, uniqueness of solutions implies that it is
identically zero.
\end{proof}

\begin{lemma}
\label{lemma:holX}
For any holomorphic vector field $X$ defined on an open subset
$\uU \subset \Sigma$, $\mathbf{D}_u \left[ Tu(X) \right] = 0$ on~$\uU$.
\end{lemma}
\begin{proof}
By shrinking $\uU$ if necessary, we can assume that the flow
$\varphi_X^t : \uU \to \Sigma$ is well defined for sufficiently small
$|t|$, and by Lemma~\ref{lemma:holomorphicFlow} it is holomorphic,
hence the maps $u \circ \varphi_X^t$ are also $J$-holomorphic.  Then
$\dbar_J(u\circ \varphi_X^t) = 0$ and
$$
\mathbf{D}_u\left[ Tu(X)\right] = \left.\nabla_{t} \left[ \dbar_J (u\circ\varphi_X^t)
\right] \right|_{t=0} = 0.
$$
\end{proof}

The Cauchy-Riemann type operator $\mathbf{D}_u$ is real-linear, but one can easily
define a complex-linear operator by projecting out the antilinear part:
$$
\mathbf{D}_u^\CC = \frac{1}{2}\left( \mathbf{D}_u - J \circ \mathbf{D}_u \circ J\right).
$$
This defines a complex-linear map $\Gamma(u^*TM) \to
\Gamma(\overline{\Hom}_\CC(T\Sigma,u^*TM))$.
\begin{exercise}
\label{EX:DuC}
Show that $\mathbf{D}_u^\CC$ is a complex-linear Cauchy-Riemann type operator.
\end{exercise}
In light of Exercise~\ref{EX:DuC} and Theorem~\ref{thm:holomorphicBundles},
the induced bundle $u^*TM \to \Sigma$ for any smooth $J$-holomorphic curve
$u : \Sigma \to M$ admits a holomorphic structure for which holomorphic
sections satisfy $\mathbf{D}_u^\CC \eta = 0$.  Moreover, Lemma~\ref{lemma:holX}
implies that for any local holomorphic vector field $X$ on~$\Sigma$,
$$
\mathbf{D}_u^\CC\left[ Tu(X)\right] = \frac{1}{2}\mathbf{D}_u\left[Tu(X)\right] -
\frac{1}{2} J \mathbf{D}_u\left[ J \circ Tu(X) \right] =
\frac{1}{2} J \mathbf{D}_u\left[ Tu(jX) \right] = 0,
$$
where we've used the nonlinear Cauchy-Riemann equation for $u$ and the
fact that $jX$ is also holomorphic.  Thus $Tu(X)$ is a holomorphic
section on $u^*TM$ whenever $X$ is holomorphic on~$T\Sigma$.
Put another way, the holomorphic bundle structures on $T\Sigma$ and
$u^*TM$ naturally induce a holomorphic structure on $\Hom_\CC(T\Sigma,u^*TM)$,
and the section $du \in \Gamma(\Hom_\CC(T\Sigma,u^*TM))$ is then holomorphic.
We've proved:

\begin{thm}
\label{thm:IShol}
For any smooth 
$J$-holomorphic map $u : \Sigma \to M$, the complex-linear part of
the linearization $\mathbf{D}_u$ induces on $\Hom_\CC(T\Sigma,u^*TM)$ a holomorphic
structure such that $du$ is a holomorphic section.
\end{thm}
\begin{cor}
\label{cor:critical}
If $u : \Sigma \to M$ is smooth, $J$-holomorphic and not constant, then
the set $\Crit(u) := \{ z \in \Sigma\ |\ du(z) = 0 \}$ is discrete.
\end{cor}
Actually we've proved more: using a holomorphic trivialization of the bundle
$\Hom_\CC(T\Sigma,u^*TM)$ near any $z_0 \in \Crit(u)$,
one can choose holomorphic coordinates identifying $z_0$ with $0 \in \CC$
and write $du(z)$ in the trivialization as
$$
du(z) = z^k F(z),
$$
where $k \in \NN$ and $F$ is a nonzero $\CC^n$-valued holomorphic function.
This means that each critical point of $u$ has a well-defined and positive
\emph{order} (the number $k$), as well as a \emph{tangent plane}
(the complex $1$-dimensional subspace spanned by $F(0)$ in the 
trivialization).  We will see this again when we investigate intersections
in \S\ref{sec:intersections}, and it will also prove useful later when
we discuss ``automatic'' transversality.

\begin{remark}
The above results for the critical set of a $J$-holomorphic curve $u$
remain valid if we don't require smoothness but 
only assume $J \in C^1$ and $u \in C^2$:
then $u^*TM$ and $\overline{\Hom}_\CC(T\Sigma,u^*TM)$ are complex vector 
bundles of class $C^1$ and $du$ is a $C^1$-section, but turns out to
be holomorphic with respect to a system of non-smooth trivializations 
which have holomorphic (and therefore smooth!) transition functions.
One can prove this using the weak regularity assumptions in
Theorem~\ref{thm:linearExistence} below; in practice of course, the
regularity results of \S\ref{sec:nonlinear} will usually allow us to
avoid such questions altogether.
\end{remark}

\section{Linear elliptic regularity}
\label{sec:estimates}

Until now we've usually assumed that our $J$-holomorphic
maps $u : \Sigma \to M$ are smooth, but for technical reasons we'll later
want to allow maps with weaker, 
Sobolev-type regularity assumptions.  In the end it all comes to the
same thing, because if $J$ is smooth, then it turns out that all
$J$-holomorphic curves are also smooth.  In the integrable case,
one can choose coordinates in $M$ so that $J = i$ and $J$-holomorphic 
curves are honestly holomorphic, then this
smoothness statement is a well-known corollary of the Cauchy integral formula.
The nonintegrable case requires more work and makes heavy use of the
machinery of elliptic PDE theory.  We will not cover this subject in full
detail---in particular, a few estimates will have to be taken as black
boxes---but we shall give an overview of the regularity
results that we'll need and try to explain why they're true.\footnote{For
a more comprehensive treatment from a slightly different perspective,
Appendix~B of \cite{McDuffSalamon:Jhol} is indispensable.}
As consequences, in this section we will see how to prove smoothness of 
solutions to linear Cauchy-Riemann type equations and also derive an important 
surjectivity property of the $\dbar$-operator, which will later 
help in proving local existence results.  
The discussion necessarily begins 
with the \emph{linear} case, and we will address the nonlinear case in
\S\ref{sec:nonlinearReg}.  It should also be mentioned that
the estimates in this section have more than
just local consequences: they will be crucial later when we
discuss the global Fredholm and compactness theory of $J$-holomorphic curves.

Let us first look at a much simpler differential equation to illustrate 
the idea of elliptic regularity.  Suppose $F : \RR^n \to \RR^n$ is a function
of class $C^k$ and we have a $C^1$ solution to the nonlinear ODE,
\begin{equation}
\label{eqn:ODE}
\dot{x} = F(x).
\end{equation}
Then if $k \ge 1$, the right hand side is clearly of class $C^1$, thus so
is $\dot{x}$, implying that $x$ is actually~$C^2$.  If $k \ge 2$,
we can repeat the argument and find that $x$ is $C^3$ and so on;
in the end we find $x \in C^{k+1}$, i.e.~$x$ is at least one step smoother
than~$F$.  This induction is the simplest example of an ``elliptic
bootstrapping argument''.

The above argument is extremely easy because the left
hand side of \eqref{eqn:ODE} tells us everything we'd ever want to know
about the first derivative of our solution.  The situation for a
first-order PDE is no longer so simple: e.g.~consider the usual 
Cauchy-Riemann operator for functions $\CC \to \CC$,
$$
\dbar = \p_s + i\p_t,
$$
and the associated linear inhomogeneous equation 
$$
\dbar u = f.
$$
Now the left hand side carries \emph{part}, but not \emph{all} of the
information one could want to have about $du$: one can say that $\p_s u +
i \p_t u$ is at least as smooth as $f$, but this doesn't immediately imply
the same statement for each of $\p_s u$ and~$\p_t u$.  What we need is a way 
to estimate $du$ (in some suitable norm) in terms of $u$ and $\dbar u$, and this
turns out to be possible precisely because $\dbar$ is an
\emph{elliptic} operator.  We will not attempt here to define
precisely what ``elliptic'' 
means;\footnote{The definition of a second-order elliptic operator is
treated in most introductory texts on PDE theory, see for example
\cite{Evans}.  A general definition for all orders may be found
in \cite{DouglisNirenberg}.} in practice, a differential operator
is called elliptic if it can be shown to satisfy a fundamental
estimate of the type stated in Theorem~\ref{thm:elliptic} below.

We briefly recall some ideas and notation from the theory of Sobolev spaces
(see e.g.~\cite{Evans} or the appendices of \cite{McDuffSalamon:Jhol}).  
If $k \in \NN$, $p \ge 1$
and $\uU \subset \CC$ is an open subset, then
$W^{0,p}(\uU) := L^p(\uU)$ is the space of (real- or complex-valued)
functions of class $L^p$ on $\uU$, and inductively $W^{k,p}(\uU)$ denotes the
space of functions in $L^p(\uU)$ that have weak derivatives
in $W^{k-1,p}(\uU)$.  Recall that a locally integrable
function $u \in L^1_\loc(\uU)$
is said to have \defin{weak derivative} $\p_s u = g \in L^1_\loc(\uU)$ if
for every smooth compactly supported function 
$\varphi \in C_0^\infty(\uU)$,
$$
\int_\uU \varphi g = - \int_\uU (\p_s\varphi) u.
$$
In other words, $g = \p_s u$ \emph{in the sense of distributions}.
One defines $\p_t u$ and higher order weak derivatives via similar
formulas based on integration by parts, and in this way one can also speak
of \defin{weak solutions} to equations such as $\dbar u = f$, where in general
$u$ and $f$ need not be more regular than~$L^1_\loc$ (or even more
generally, distributions).  We say that a function on $\uU$ is of class
$W^{k,p}_\loc$ if it is in $W^{k,p}(\uU')$ for every open subset 
$\uU'$ with compact closure $\overline{\uU}' \subset \uU$.
In the following, we will consider
Sobolev spaces of maps valued in complex vector spaces such as
$\CC^n$; we'll specify the target space by writing e.g.~$W^{k,p}(\uU,\CC^n)$ 
whenever there is danger of confusion.  The symbols $B$ and $B_r$ will as usual
denote the open balls 
in $\CC$ of radius $1$ and $r$ respectively.  

We will often make use of
the \defin{Sobolev embedding theorem}, which in the present context implies
that if $\uU \subset \CC$ is a bounded open domain with smooth boundary
and $kp > 2$, then there are natural continuous inclusions
$$
W^{k+d,p}(\uU) \hookrightarrow C^d(\uU)
$$
for each integer $d \ge 0$.  In fact, these inclusions are \emph{compact} linear
operators (cf.~Exercise~\ref{EX:Massot} below), as are the obvious inclusions
$$
W^{k,p}(\uU) \hookrightarrow W^{k-1,p}(\uU).
$$
Additionally, $W^{k,p}(\uU)$ has two related properties when $k \ge 1$ and
$p > 2$ that will be especially useful: first, it is a 
\defin{Banach algebra}, meaning that products of functions in 
$W^{k,p}(\uU,\CC)$ are also in $W^{k,p}(\uU,\CC)$ and satisfy
\begin{equation}
\label{eqn:BanachAlgebra0}
\| u v \|_{W^{k,p}} \le \| u \|_{W^{k,p}} \| v \|_{W^{k,p}}.
\end{equation}
Secondly, if $\Omega \subset \CC^n$ is an open subset and we denote by
$W^{k,p}(\uU,\Omega)$ the (open) set of functions $u \in W^{k,p}(\uU,\CC^n)$
such that $u(\uU) \subset \Omega$,
then the pairing $(f,u) \mapsto f \circ u$ defines a continuous map
\begin{equation}
\label{eqn:CkPairing}
C^k(\Omega,\RR^N) \times W^{k,p}(\uU,\Omega) \to W^{k,p}(\uU,\RR^N) :
(f,u) \mapsto f \circ u.
\end{equation}

\begin{exercise}
\label{EX:Massot}
Use H\"older's inequality to prove the following simple case of the
Sobolev embedding theorem: for every $p > 1$, 
there exists a constant $C > 0$ such that
for all smooth functions $f : (0,1) \to \RR$ with compact support,
$$
\| f \|_{C^{0,\alpha}} \le C \| f \|_{W^{1,p}},
$$
where $\alpha := 1 - 1/p$, and the two norms are defined by
$$
\| f \|_{C^{0,\alpha}} := \sup_{t \in (0,1)} | f(t) | + 
\sup_{s,t \in (0,1),\ s \ne t} \frac{|f(s) - f(t)|}{|s - t|^\alpha},
$$
and
$$
\| f \|_{W^{1,p}} := \left( \int_0^1 |f(t)|^p \, dt \right)^{1/p} +
\left( \int_0^1 |f'(t)|^p \, dt \right)^{1/p}.
$$
Conclude via the Arzel\`{a}-Ascoli theorem that any sequence
$f_k \in C_0^\infty((0,1))$ that is bounded in $W^{1,p}$ has a
$C^0$-convergent subsequence.
\end{exercise}

The following is the basic analytical result we will need.

\begin{thm}[Calder\'{o}n-Zygmund inequality]
\label{thm:elliptic}
For each $p \in (1,\infty)$, there is a constant $c > 0$ such that for every 
$u \in C_0^\infty(B,\CC^n)$,
$$
\| u \|_{W^{1,p}} \le c \| \dbar u \|_{L^{p}}.
$$
\end{thm}
\begin{exercise}
\label{EX:CZhigher}
Assuming the theorem above, differentiate the equation $\dbar u = f$ and 
argue by induction to prove the following generalization: for each
$k \in \NN$ and $p \in (1,\infty)$ there is a constant $c > 0$ such that
$$
\| u \|_{W^{k,p}} \le c \| \dbar u \|_{W^{k-1,p}}
$$
for all $u \in C^\infty_\loc(B)$.  By a density argument, show that this
also holds for all $u \in W^{k,p}_0(B)$, where the latter denotes the
closure of $C^\infty_0(B)$ in $W^{k,p}(B)$.
\end{exercise}

Theorem~\ref{thm:elliptic} follows from 
estimates of certain singular integrals carried out in
\cites{CalderonZygmund:52,CalderonZygmund:56} (see also the discussion in
Appendix~B of \cite{McDuffSalamon:Jhol}, which expresses it in terms of the
Laplace operator).  More general versions for
elliptic systems of all orders appear in \cite{DouglisNirenberg}, and
versions with boundary conditions are treated in
\cites{AgmonDouglisNirenberg:I,AgmonDouglisNirenberg:II}.
Before launching into a cursory discussion of the proof,
let us see how this estimate
can be used to prove a basic local regularity result for
the linear inhomogeneous Cauchy-Riemann equation.  We will later improve
this result to apply to weak solutions of class~$L^1_\loc$ (see
Theorem~\ref{thm:inhomWeak} below). 

\begin{prop}
\label{prop:inhomogeneousRegularity}
Suppose $u \in W^{1,p}(B,\CC^n)$ and $\dbar u \in W^{k,p}(B,\CC^n)$
for some $p \in (1,\infty)$.  
Then $u \in W^{k+1,p}(B_r,\CC^n)$ for any $r < 1$,
and there is a constant $c$, depending on $r$ and $p$ 
but not on~$u$, such that
\begin{equation}
\label{eqn:FredholmEstimate}
\| u \|_{W^{k+1,p}(B_r)} \le c \| u \|_{W^{k,p}(B)} + 
c \| \dbar u \|_{W^{k,p}(B)}.
\end{equation}
\end{prop}
\begin{cor}
If $f : B \to \CC^n$ is smooth, then every solution to $\dbar u = f$
of class $W^{1,p}$ for some $p \in (1,\infty)$
is also smooth.  Moreover, given sequences $f_k \to f$ converging in
$C^\infty(B)$ and $u_k \to u$ converging in $W^{1,p}(B)$ and
satisfying $\dbar u_k = f_k$, the sequence $u_k$ also converges in
$C^\infty_\loc$ on~$B$.
\end{cor}
\begin{exercise}
Prove the corollary.
\end{exercise}
\begin{proof}[Proof of Prop.~\ref{prop:inhomogeneousRegularity}]
Write $\dbar u = f$.
It will suffice to consider the case $k=1$, as once this is settled,
the result follows from an easy induction argument using the fact that any 
derivative $D^\alpha u$ of $u$ satisfies $\dbar D^\alpha u = D^\alpha f$.
(Here $\alpha$ is a multiindex, so $D^\alpha$ may be any
differential operator of order one or greater.)

Now assuming $u, f \in W^{1,p}(B)$, 
we'd first like to prove that $u$ is of class
$W^{2,p}$ on $B_r$ for any $r < 1$.  The idea is to show that
$\p_s u$ (and similarly $\p_t u$) is of class $W^{1,p}$
by expressing it as a limit of the difference quotients,
$$
u^h(s,t) := \frac{u(s+h,t) - u(s,t)}{h}
$$
as $h > 0$ shrinks to zero.  These functions are clearly well defined and
belong to $W^{1,p}(B_r)$ if $h$ is sufficiently small, and it is 
straightforward (e.g.~using approximation by smooth functions) 
to show that $u^h \to \p_s u$ in $L^p(B_r)$ as $h \to 0$.  The significance
of Theorem~\ref{thm:elliptic} is that it gives us a uniform 
$W^{1,p}$-bound on $u^h$ with respect to~$h$.  Indeed, pick a cutoff function
$\beta \in C_0^\infty(B)$ that equals $1$ on $B_r$.  Then
$\beta u^h \in W^{1,p}_0(B)$ and thus satisfies 
the estimate of Theorem~\ref{thm:elliptic} (cf.~Exercise~\ref{EX:CZhigher}).  
We compute
\begin{multline}
\label{eqn:diffQuotients}
\| u^h \|_{W^{1,p}(B_r)} \le \| \beta u^h \|_{W^{1,p}(B)} \le
c \| \dbar (\beta u^h) \|_{L^p(B)} \\
= c \| (\dbar\beta) u^h + \beta (\dbar u^h) \|_{L^p(B)} \le
c' \| u^h \|_{L^p(B)} + c' \| f^h \|_{L^p(B)},
\end{multline}
and observe that the right hand side is bounded as $h \to 0$ because
$u^h \to \p_s u$ and $f^h \to \p_s f$ in~$L^p$.

In light of this bound, the Banach-Alaoglu theorem implies that any
sequence $u^{h_k}$ with $h_k \to 0$ has a weakly convergent subsequence
in $W^{1,p}(B_r)$.  But since $u^h$ already converges to $\p_s u$ in 
$L^p(B_r)$, the latter must also be the weak $W^{1,p}$-limit, implying
$\p_s u \in W^{1,p}(B_r)$.  Now the estimate \eqref{eqn:FredholmEstimate}
follows from \eqref{eqn:diffQuotients}, using Exercise~\ref{EX:weakNorm}
below to bound the $W^{1,p}$-norm of the derivative of $u$ in terms
its difference quotients:
\begin{multline*}
\| \p_s u \|_{W^{1,p}(B_r)} \le \liminf_{h \to 0} \| u^h \|_{W^{1,p}(B_r)} \\
\le c \| \p_s u \|_{L^p(B)} + c \| \p_s f \|_{L^p(B)} \le
c \| u \|_{W^{1,p}(B)} + c \| f \|_{W^{1,p}(B)}.
\end{multline*}
\end{proof}

\begin{exercise}
\label{EX:weakNorm}
If $X$ is a Banach space and $x_n \in X$ converges weakly to $x$, show that
$\| x \| \le \liminf \| x_n \|$.
\textsl{Hint: The natural inclusion of $X$ into $(X^*)^*$ is isometric, so
$\| x \| = \sup_{\lambda \in X^* \setminus \{0\}} \frac{| \lambda(x) |}{\| \lambda \|}$.}
\end{exercise}

\begin{exercise}
\label{EX:smoothSections}
Use Proposition~\ref{prop:inhomogeneousRegularity} 
to show that for any real-linear Cauchy-Riemann type
operator $D$ on a vector bundle $(E,J) \to (\Sigma,j)$, continuously
differentiable solutions of
$D\eta = 0$ are always smooth.  \textsl{Note: due to
Exercise~\ref{EX:ChristoffelReal}, this reduces to showing that solutions
$u \in W^{1,p}(B,\CC^n)$ of $(\dbar + A)u = 0$ are smooth if
$A : B \to \End_\RR(\CC^n)$ is smooth.}
\end{exercise}

The proof of Theorem~\ref{thm:elliptic} contains several easy steps and one
that is hard.  For the easy part: observe first that by the
Poincar\'e inequality, $\| u \|_{L^p}$ can be bounded in terms of
$\| du \|_{L^p}$ for any $u \in C_0^\infty(B)$, thus it will suffice to
bound the first derivatives in terms of $\dbar u$.  For this it is natural
to consider the conjugate of the $\dbar$-operator,
$$
\p := \p_s - i\p_t,
$$
as the expressions $\dbar u$ and $\p u$ together can produce both
$\p_s u$ and $\p_t u$ by linear combinations.  Thus we are done if we can
show that $\| \p u \|_{L^p}$ is bounded in terms of $\| \dbar u \|_{L^p}$.

It is easy to see why this is true in the case $p=2$:
the following ``symplectic'' argument is borrowed from \cite{Sikorav}.
It suffices to set $n=1$ and consider compactly supported smooth functions
$u : B \to \CC$.  Using the coordinate $z = s + it$, define the 
differential operators
$$
\p_z = \frac{\p}{\p z} = \frac{1}{2} ( \p_s - i \p_t )
\qquad
\p_{\bar{z}} = \frac{\p}{\p \bar{z}} = \frac{1}{2} (\p_s + i \p_t)
$$
and corresponding complex-valued $1$-forms
$$
dz = d(s + it) = ds + i\,dt
\qquad
d\bar{z} = d(s - it) = ds - i\,dt.
$$
Observe that $\p_z$ and $\p_{\bar{z}}$ are the same as $\p$ and $\dbar$
respectively up to a factor of 
two,\footnote{It is in some sense more natural to define the operators 
$\dbar$ and $\p$ with the factor of $1/2$ included, but we have dropped
this for the sake of notational convenience.} and we now have
$du = \p_z u\,dz + \p_{\bar{z}} u \,d\bar{z}$.  
The complex-valued $1$-form $u\, d\bar{u}$ then has compact support
in the interior of the unit disk~$\DD \subset \CC$,
so applying Stokes' theorem to $d(u \,d\bar{u}) = du \wedge
d\bar{u}$ gives
\begin{equation*}
\begin{split}
0 &= \int_{\p \DD} u\,d\bar{u} = \int_{\DD} du \wedge d\bar{u} = 
\int_{\DD} \left( \p_z u \,dz + \p_{\bar{z}} u\,d\bar{z}\right) \wedge
\left( \p_z \bar{u} \,dz + \p_{\bar{z}} \bar{u} \,d\bar{z} \right) \\
&= \frac{1}{4}\int_{\DD} \left( |\p u|^2 - |\dbar u|^2 \right) \,dz \wedge
d\bar{z},
\end{split}
\end{equation*}
hence $\| \p u \|_{L^2} = \| \dbar u \|_{L^2}$.

For $p \ne 2$, a bound on $\|\p u\|_{L^p}$ can be found by rephrasing the
equation $\dbar u = f$ in terms of fundamental solutions.  We recall the
basic idea: a fundamental solution to the $\dbar$-equation is an 
$L^1_\loc$-function
$K : \CC \to \CC$ that satisfies $\dbar K = \delta$ in the sense of
distributions, where $\delta$ is the Dirac delta function, i.e.~the 
distribution whose action on any test function $\varphi \in C_0^\infty$
is $\delta(\varphi) = \varphi(0)$.  
Then weak solutions of $\dbar u = f$ can be expressed as
convolutions $u = K * f$, since $\dbar(K * f) = \dbar K * f = \delta * f = f$.
To make all this precise, define the function $K \in L^1_\loc(\CC,\CC)$ to be
$$
K(z) = \frac{1}{2\pi z},
$$
and for any $f \in C_0^\infty(\CC,\CC^n)$, define $T f : \CC \to \CC^n$
to be the convolution
\begin{equation}
\label{eqn:T}
T f(z) = K * f(z) = \int_{\CC} K(z - \zeta) f(\zeta) \frac{d\zeta \wedge
d\bar{\zeta}}{-2 i} =
\frac{1}{2\pi} \int_{\CC} \frac{f(\zeta)}{z - \zeta} \,
\frac{d\zeta \wedge d\bar{\zeta}}{-2 i},
\end{equation}
where we use the notation $\frac{d\zeta \wedge d\bar{\zeta}}{-2i}$ to 
abbreviate the standard Lebesgue measure on~$\CC$ with complex 
coordinate~$\zeta$.  This integral is clearly well defined if
$f \in C_0^\infty$, and it is not too hard to show that it gives a
smooth function $T f \in C^\infty(\CC,\CC^n)$ which decays to zero at
infinity and satisfies $\dbar T f = f$
(see for example \cite{HoferZehnder}*{Appendix~A4}).  It follows that
for $u \in C_0^\infty(\CC,\CC^n)$, $\dbar u = f$ if and only if
$u = T f$.  Indeed, if $u$ has compact support then so does $f = \dbar u$,
thus $T f$ is well defined and we have $\dbar (u - T f) = 0$.  But
since $T f$ decays at infinity, this means $u - T f$ is a bounded
holomorphic function on $\CC$ that approaches zero at infinity, hence it is
identically zero.

By the above remarks, it suffices to show that the function
$$
\Pi f := \p T f
$$
satisfies an $L^p$-bound in terms of $\| f \|_{L^p}$ whenever 
$f \in C_0^\infty(B)$.  Naively, one can derive an expression for
$\p T f = \p(K * f)$ by differentiating the fundamental solution: it
should be the convolution of $f$ with
$$
\p K(z) = 2 \frac{\p}{\p z} \frac{1}{2\pi z} = - \frac{1}{\pi z^2}.
$$
Thus we are led to the expression
$$
\Pi f(z) = \lim_{\epsilon \to 0} \int_{|\zeta - z| \ge \epsilon}
- \frac{f(\zeta)}{\pi (z - \zeta)^2} \,\frac{d\zeta \wedge d\bar{\zeta}}{-2i},
$$
where we've attached a limit in order to define the integral since
$\p K \not\in L^1_\loc$;
this is a \emph{Cauchy principal value} integral.  One can now verify
explicitly that this limit is well defined if $f \in C_0^\infty$, 
and it computes $\p T f(z)$.  The hard work is thus reduced to the following
estimates.

\begin{lemma}
\label{lemma:integrals}
For each $p \in (1,\infty)$ there is a constant $c > 0$ 
such that for every $f \in C_0^\infty(B)$,
$$
\| T f \|_{L^p} \le c \| f \|_{L^p}
\qquad\text{ and }\qquad
\| \Pi f \|_{L^p} \le c \| f \|_{L^p}.
$$
\end{lemma}
This is quite hard work indeed; complete proofs may be found in
\cite{Vekua}*{Part~I, \S 5--9}, or in a slightly different context, 
\cite{McDuffSalamon:Jhol}*{Appendix~B}.  Taking this as a black box,
the lemma implies Theorem~\ref{thm:elliptic}.

In fact, Lemma~\ref{lemma:integrals} implies a stronger result about the
Cauchy-Riemann operator that we will find useful.  Since $C_0^\infty(B)$
is dense in $L^p(B)$, the first inequality in the lemma says that
$T$ can be extended to a bounded linear operator on $L^p(B)$, but the
second inequality implies even more: we already have a bound on
$\dbar T f$ since this is simply $f$, so a bound on $\Pi f = \p T f$
gives a bound on the whole first derivative of $T f$, and we conclude:

\begin{prop}
\label{prop:rightInverse}
For each $p \in (1,\infty)$, 
the operator $T$ of \eqref{eqn:T} extends to a bounded linear operator
$T : L^p(B) \to W^{1,p}(B)$, which is a right inverse of
$\dbar : W^{1,p}(B) \to L^p(B)$.
\end{prop}

The upshot is that the equation $\dbar u = f$ can be solved for any
$f \in L^p(B)$, and in a way that controls the first derivatives 
of the solution.  This can be improved further using the previous 
regularity results:

\begin{thm}
\label{thm:rightInverse}
For any integer $k \ge 0$ and $p \in (1,\infty)$,
the operator $\dbar : W^{k+1,p}(B) \to W^{k,p}(B)$ admits a bounded right
inverse
$$
\widehat{T} : W^{k,p}(B) \to W^{k+1,p}(B),
$$
i.e.~$\dbar \widehat{T} f = f$ for all $f \in W^{k,p}(B)$.
\end{thm}
\begin{proof}
Prop.~\ref{prop:rightInverse} proves the result for $k=0$, so we proceed
by induction, assuming the result is proven already for~$k-1$.
Pick $R > 1$, and for each $\ell$ let 
$$
W^{\ell,p}(B) \to W^{\ell,p}(B_R) : f \mapsto \hat{f}
$$ 
denote a bounded linear \emph{extension} operator, i.e.~$\hat{f}$ satisfies
$\hat{f}|_B = f$ and $\| \hat{f} \|_{W^{\ell,p}(B_R)} \le c 
\| f \|_{W^{\ell,p}(B)}$ for some $c > 0$.  
Then by assumption there is a bounded operator
$$
T_R : W^{k-1,p}(B_R) \to W^{k,p}(B_R)
$$
that is a right inverse of $\dbar$, hence $u := T_R \hat{f}$ satisfies
$\dbar u = \hat{f}$.  But then if $\hat{f} \in W^{k,p}(B_R)$, 
Prop.~\ref{prop:inhomogeneousRegularity} implies that $u \in W^{k+1,p}(B)$ and
\begin{multline*}
\| u \|_{W^{k+1,p}(B)} \le c \| u \|_{W^{k,p}(B_R)} +  
c \| \hat{f} \|_{W^{k,p}(B_R)} \\
\le \| T_R \| \cdot \| \hat{f} \|_{W^{k-1,p}(B_R)} + c_1 \| f \|_{W^{k,p}(B)}
\le c_2 \| f \|_{W^{k,p}(B)}.
\end{multline*}
\end{proof}

Now that we are guaranteed to have nice solutions of the equation
$\dbar u = f$, we can also improve the previous regularity results to apply to
more general weak solutions.  We begin with the simple fact that ``weakly'' 
holomorphic functions are actually smooth.

\begin{lemma}
\label{lemma:Weyl}
If $u \in L^1(B)$ is a weak solution of $\dbar u = 0$, then
$u$ is smooth.
\end{lemma}
\begin{proof}
By taking real and imaginary parts, it suffices to prove the same statement
for real-valued weak solutions of the Laplace equation: thus consider a 
function $u \in L^1(B,\RR)$ such that $\Delta u = 0$ in the sense of
distributions.  On $B_r$ for any $r < 1$ we
can approximate $u$ by smooth functions $u_\epsilon$ using a standard
mollifier,
$$
u_\epsilon = j_\epsilon * u,
$$
so that $u_\epsilon \to u$ in $L^1(B_r)$ as $\epsilon \to 0$.  
Moreover, $\Delta u_\epsilon =
j_\epsilon * \Delta u = 0$, thus the $u_\epsilon$ are harmonic.
This implies that they satisfy the mean value property, so for every
sufficiently small ball $B_\delta(z)$ about any point $z \in B_r$,
$$
u_\epsilon(z) = \frac{1}{\pi \delta^2} \int_{B_\delta(z)} 
u_\epsilon(s,t) \,ds\,dt.
$$
By $L^1$-convergence, this expression converges pointwise in a neighborhood
of $z$ to the map
$$
z \mapsto \frac{1}{\pi \delta^2} \int_{B_\delta(z)} u(s,t) \,ds \,dt.
$$
The latter is continuous, and must be equal to
$u$ almost everywhere, thus $u$ satisfies the mean value property and is
therefore a smooth harmonic function (see \cite{Evans}*{\S 2.2.3}).
\end{proof}

\begin{thm}
\label{thm:inhomWeak}
Suppose $f \in W^{k,p}(B,\CC^n)$ for some $p \in (1,\infty)$ 
and $u \in L^1(B,\CC^n)$ is a weak
solution of the equation $\dbar u = f$.  Then 
$u \in W^{k+1,p}(B_r,\CC^n)$ for any $r < 1$.
\end{thm}
\begin{proof}
By Theorem~\ref{thm:rightInverse}, there is a solution
$\eta \in W^{k+1,p}(B)$ to $\dbar\eta = f$, and then
$\dbar(u - \eta) = 0$.  Lemma~\ref{lemma:Weyl} then
implies that $u - \eta$ is smooth and hence in $W^{k+1,p}(B_r)$
for all $r < 1$, thus $u$ is also in $W^{k+1,p}(B_r)$.
\end{proof}

\begin{cor}
\label{cor:weakRegularity}
Suppose $1 < p < \infty$, $k$ is a nonnegative integer,
$A \in L^\infty(B,\End_\RR(\CC^n))$, $f \in W^{k,p}(B,\CC^n)$ and 
$u \in L^p(B,\CC^n)$ is
a weak solution of the equation $\dbar u + Au = f$.
Then $u \in W^{1,p}(B_r,\CC^n)$ for any $r < 1$.  
Moreover if $A$ is smooth, then 
$u \in W^{k+1,p}(B_r,\CC^n)$, and in particular $u$ is smooth if $f$ is smooth.
\end{cor}
\begin{proof}
We have $\dbar u = - Au + f$ of class $L^p$, thus $u \in W^{1,p}(B_r)$ by
Theorem~\ref{thm:inhomWeak}.  If $A$ is also smooth and $k \ge 1$, then 
$-Au + f$ is now of class $W^{1,p}$, so $u \in W^{2,p}(B_r)$, and repeating
this argument inductively, we eventually find $u \in W^{k+1,p}(B_r)$.
\end{proof}

The invertibility results for $\dbar$ will also be useful for proving
more general local existence results, because the property of having a bounded
right inverse is preserved under small perturbations of the operator---thus
any operator close enough to $\dbar$ in the appropriate functional analytic
context is also surjective!

\begin{exercise}
\label{EX:rightInverse}
Show that if $A : X \to Y$ is a bounded linear map between Banach spaces and
$B : Y \to X$ is a bounded right inverse of $A$, then any small perturbation
of $A$ in the norm topology also has a bounded right inverse.
\textsl{Hint: Recall that any small perturbation of the identity on a Banach
space is invertible, as its inverse can be expressed as a power series.}
\end{exercise}

\begin{remark}
In most presentations (e.g.~\cite{McDuffSalamon:Jhol},
\cite{HoferZehnder}), some version of Prop.~\ref{prop:rightInverse}
and Theorem~\ref{thm:rightInverse} is proven
by ``reducing the local problem to a global problem'' so that one can
apply the Fredholm theory of the Cauchy-Riemann operator.
For instance, \cite{McDuffSalamon:Jhol} uses the fact that $\dbar$ is a
surjective Fredholm operator from $W^{1,p}$ to $L^p$ on the \emph{closed}
unit disk if suitable boundary conditions are imposed, and a related
approach is taken in \cite{HoferZehnder}*{Appendix~A.4}, which
introduces the fundamental solution $K(z)$ and defines $Tf$ as a convolution,
but then compactifies $\CC$ to a sphere in order to make $\ker(\dbar)$ 
finite dimensional.  We have chosen
instead to view the existence of a right inverse as an aspect of the 
basic local regularity theory for~$\dbar$, which is a \emph{prerequisite}
for the Fredholm theory mentioned above.  However, 
a second proof of these results will easily present
itself when we discuss the Fredholm theory in Chapter~\ref{chapter:Fredholm}.
\end{remark}

\section{Local existence of holomorphic sections}
\label{sec:localSections}

We now prove a generalization of Lemma~\ref{lemma:localExistence},
which implies the existence of holomorphic structures on complex vector
bundles with Cauchy-Riemann type operators.  The question is a purely local
one, thus we can work in the trivial bundle over the open unit ball
$B \subset \CC$ with coordinates $s + it \in B$ 
and consider operators of the form
$$
C^\infty(B,\CC^n) \to C^\infty(B,\CC^n) : u \mapsto \dbar u + A u
$$
where $\dbar$ denotes the differential operator $\p_s + i\p_t$, and
$A : B \to \End_\RR(\CC^n)$ is a family of real-linear maps
on~$\CC^n$.  For Lemma~\ref{lemma:localExistence} it suffices to assume
$A$ is smooth, but in the proof and in further applications we'll find it
convenient to assume that $A$ has much weaker regularity.  The smoothness
of our solutions will then follow from elliptic regularity.

\begin{thm}
\label{thm:linearExistence}
Assume $A \in L^p(B,\End_\RR(\CC^n))$ for some $p \in (2,\infty]$.  
Then for each finite $q \in (2,p]$, there is an
$\epsilon > 0$ such that for any $u_0 \in \CC^n$, the problem
\begin{equation*}
\begin{split}
\dbar u + A u &= 0 \\
u(0) &= u_0
\end{split}
\end{equation*}
has a solution $u \in W^{1,q}(B_\epsilon,\CC^n)$.
\end{thm}
\begin{proof}
The main idea is that if we take $\epsilon > 0$ sufficiently small,
then the restriction of $\dbar + A$ to $B_\epsilon$ can be regarded as a
small perturbation of the standard operator~$\dbar$, and we conclude
from Prop.~\ref{prop:rightInverse} and Exercise~\ref{EX:rightInverse}
that the perturbed operator is surjective.

Since $q > 2$, the Sobolev embedding theorem implies that functions
$u \in W^{1,q}$ are also continuous and bounded by $\| u \|_{W^{1,q}}$, 
thus we can define a bounded linear operator
$$
\Phi : W^{1,q}(B) \to L^q(B) \times \CC^n : u \mapsto (\dbar u, u(0)).
$$
Prop.~\ref{prop:rightInverse}
implies that this operator is also surjective and has a bounded right inverse,
namely
$$
L^q(B) \times \CC^n \to W^{1,q}(B) : (f,u_0) \mapsto
Tf - Tf(0) + u_0,
$$
where $T : L^q(B) \to W^{1,q}(B)$ is a right inverse of~$\dbar$.
Thus any operator sufficiently close to $\Phi$ in the norm topology
also has a right inverse.
Now define $\chi_\epsilon : B \to \RR$ to be the function that equals $1$ on 
$B_\epsilon$ and $0$ outside of it, and let
$$
\Phi_\epsilon : W^{1,q}(B) \to L^q(B) \times \CC^n : u \mapsto
((\dbar + \chi_\epsilon A) u , u(0)).
$$
To see that this is a bounded operator, it suffices to check that
$W^{1,q} \to L^q : u \mapsto A u$ is bounded if $A \in L^p$; indeed,
$$
\| A u \|_{L^q} \le \| A \|_{L^q} \| u \|_{C^0} \le
c \| A \|_{L^p} \| u \|_{W^{1,q}},
$$
again using the Sobolev embedding theorem and the assumption that
$q \le p$.  Now by this same trick, we find
$$
\|\Phi_\epsilon u - \Phi u\| = \| \chi_\epsilon A u \|_{L^q(B)} 
\le c \| A \|_{L^p(B_\epsilon)} \| u \|_{W^{1,q}(B)},
$$
thus $\|\Phi_\epsilon - \Phi\|$ is small if $\epsilon$ is small, and it
follows that in this case $\Phi_\epsilon$ is surjective.  Our desired solution
is therefore the restriction of any 
$u \in \Phi_\epsilon^{-1}(0,u_0)$ to~$B_\epsilon$.
\end{proof}

By Exercise~\ref{EX:smoothSections}, the local solutions found above 
are smooth if $A : B \to \End_\RR(\CC^n)$ is smooth, thus
applying this to any smooth complex vector bundle with a complex-linear
Cauchy-Riemann operator, we've completed the proof of 
Lemma~\ref{lemma:localExistence} and hence
Theorem~\ref{thm:holomorphicBundles}.

\section{The similarity principle}
\label{sec:similarity}

Another consequence of the local existence result in \S\ref{sec:localSections}
is that all solutions to
equations of the form $\dbar u + Au = 0$, even when $A$ is \emph{real}-linear,
behave like holomorphic sections in certain respects.  This will be
extremely useful in studying the local properties of $J$-holomorphic curves,
as well as global transversality issues.  In practice, we'll usually need this
result only in the case where $A$ is smooth, but we'll state it in greater
generality since the proof is not any harder.

\begin{thm}[The similarity principle]
\label{thm:similarity}
Suppose $A \in L^\infty(B,\End_\RR(\CC^n))$  and
$u \in W^{1,p}(B,\CC^n)$ for some $p > 2$ is a solution of the
equation $\dbar u + Au = 0$ with $u(0) = 0$.  Then for sufficiently 
small $\epsilon > 0$, there exist maps $\Phi \in C^0(B_\epsilon,
\End_\CC(\CC^n))$ and $f \in C^\infty(B_\epsilon,\CC^n)$ such that
$$
u(z) = \Phi(z) f(z),
\qquad
\dbar f = 0,
\qquad\text{ and }\qquad
\Phi(0) = \1.
$$
\end{thm}

The theorem says in effect that the trivial complex vector bundle 
$B \times \CC^n \to B$ admits a holomorphic structure for which the given
$u$ is a holomorphic section.  In particular, this implies that if $u$
is not identically zero, then the zero at~$0$ is isolated,
a fact that we'll often find quite useful.  
There's a subtlety here to be aware of:
the holomorphic structure in question is generally 
\emph{not compatible with the
canonical smooth structure} of the bundle, i.e.~the sections that we now
call ``holomorphic'' are not smooth in the usual sense.  They will instead be
of class $W^{1,q}$ for some $q > 2$, 
which implies they're continuous, and that's
enough to imply the above statement about $u$ having isolated zeroes.
Of course, a holomorphic structure also \emph{induces} a smooth structure
on the bundle, but it will in general be a \emph{different} smooth structure.

\begin{proof}[Proof of Theorem~\ref{thm:similarity}]
Given the solution $u \in W^{1,p}(B,\CC^n)$,
we claim that there exists a map $C \in L^\infty(B,\End_\CC(\CC^n))$
such that $C(z) u(z) = A(z) u(z)$ almost everywhere.  Indeed, whenever
$u(z) \ne 0$ it is simple enough to define
$$
C(z) \frac{u(z)}{|u(z)|} = A(z) \frac{u(z)}{|u(z)|}
$$
and extend $C(z)$ to a complex-linear map so that it satisfies a uniform
bound in~$z$ almost everywhere; 
it need not be continuous.  Now $(\dbar + C)u = 0$,
and we use Theorem~\ref{thm:linearExistence} to find 
a basis of $W^{1,p}$-smooth solutions 
to $(\dbar + C) v = 0$ on $B_\epsilon$ that define
the standard basis of $\CC^n$ at~$0$; equivalently, this is a map
$\Phi \in W^{1,p}(B_\epsilon,\End_\CC(\CC^n))$ that satisfies
$(\dbar + C)\Phi = 0$ and $\Phi(0) = \1$.  Since $p > 2$, $\Phi$ is continuous
and we can thus assume without loss of generality that $\Phi(z)$ is
invertible everywhere on~$B_\epsilon$, and the smoothness of the map
$\GL(n,\CC) \to \GL(n,\CC) : \Psi \mapsto \Psi^{-1}$ then implies via 
\eqref{eqn:CkPairing} that $\Phi^{-1}  \in W^{1,p}(B_\epsilon,\End_\CC(\CC^n))$.
Then we can define a function $f := \Phi^{-1} u : B_\epsilon \to \CC^n$,
which is of class $W^{1,p}$ since $W^{1,p}$ is a Banach algebra.
But since $u = \Phi f$, the Leibnitz rule implies $\dbar f = 0$, thus
$f$ is smooth and holomorphic.
\end{proof}

\begin{exercise}
\label{EX:Jvaries}
By a change of local trivialization, show that a minor variation on
Theorem~\ref{thm:similarity} also holds for any $u : B \to \CC^n$ satisfying
$$
\p_s u(z) + J(z) \p_t u(z) + A(z) u(z) = 0,
$$
where $J(z)$ is a smooth family of complex structures on $\CC^n$, parametrized
by $z \in B$.  In particular, $u$ has only isolated zeroes.
\end{exercise}

\begin{remark}
\label{remark:smoothness}
It will occasionally be useful to note that if
the $0$th-order term $A(z)$ is not only smooth but \emph{complex}-linear, 
then the term $\Phi(z)$ in the factorization $u(z) = \Phi(z) f(z)$
given by Theorem~\ref{thm:similarity} will also be smooth.  This is clear
by a minor simplification of the proof, since it is no longer necessary
to replace $A(z)$ by a separate complex-linear term $C(z)$ (which in our
argument above could not be assumed to be more regular than $L^\infty$),
but suffices to find a local solution of the equation 
$(\dbar + A) \Phi = 0$ with $\Phi(0) = \1$.  This exists due to 
Theorem~\ref{thm:linearExistence} and is smooth by the regularity results 
of \S\ref{sec:estimates}.  A similar remark holds in the generalized situation
treated by Exercise~\ref{EX:Jvaries},
whenever $\p_s + J(z) \p_t + A(z)$ defines a \emph{complex}-linear operator with
smooth coefficients, e.g.~it is always true if $J$ is smooth and $A \equiv 0$.
\end{remark}

We shall study a few simple applications of the similarity principle in the
next two sections.

\section{Unique continuation}
\label{sec:continuation}

The following corollary of the similarity principle will be important when 
we study the transversality question for global solutions to the linearized 
Cauchy-Riemann equation.

\begin{cor}
Suppose $u : (\Sigma,j) \to (M,J)$ is a smooth $J$-holomorphic curve and
$\eta \in \Gamma(u^*TM)$ is in the kernel of the linearization
$D\dbar_J(u)$.  Then either $\eta \equiv 0$ or 
the zero set of $\eta$ is discrete.
\end{cor}

On the local level, one can view this as a unique continuation result
for $J$-holomorphic curves.  The following is a simple special case
of such a result, which we'll generalize in a moment.

\begin{prop}
\label{prop:uniqueContin}
Suppose $J$ is a smooth almost complex structure on $\CC^n$ and
$u,v : B \to \CC^n$ are smooth $J$-holomorphic curves such that
$u(0) = v(0) = 0$ and $u$ and $v$ have matching partial derivatives 
of all orders at~$0$.  Then $u \equiv v$ on a neighborhood of~$0$.
\end{prop}
\begin{proof}
Let $h = v - u : B \to \CC^n$.  We have
\begin{equation}
\label{eqn:u}
\p_s u + J(u(z)) \p_t u = 0
\end{equation}
and
\begin{equation}
\label{eqn:v}
\begin{split}
\p_s v + J(u(z)) \p_t v &= \p_s v + J(v(z)) \p_t v + \left[
J(u(z)) - J(v(z)) \right] \p_t v \\
&= - \left[ J(u(z) + h(z)) - J(u(z)) \right] \p_t v \\
&= -\left( \int_0^1 \frac{d}{dt} J(u(z) + t h(z)) \,dt \right) \p_t v \\
&= -\left(\int_0^1 dJ(u(z) + t h(z)) \cdot h(z) \,dt \right) 
\p_t v =: -A(z) h(z),
\end{split}
\end{equation}
where the last step defines a smooth family of linear maps $A(z) \in
\End_\RR(\CC^n)$.  Subtracting \eqref{eqn:u} from \eqref{eqn:v} gives
the linear equation
$$
\p_s h(z) + J(u(z)) \p_t h(z) + A(z) h(z) = 0,
$$
thus by Theorem~\ref{thm:similarity} and Exercise~\ref{EX:Jvaries},
$h(z) = \Phi(z) f(z)$ near~$0$ for some continuous 
$\Phi(z) \in \GL(2n,\RR)$ and holomorphic $f(z) \in \CC^n$.  Now if
$h$ has vanishing derivatives of all orders at~$0$, Taylor's formula
implies 
$$
\lim_{z \to 0} \frac{| \Phi(z) f(z) |}{|z|^k} = 0
$$
for all $k \in \NN$, so $f$ must also have a zero of infinite order
and thus $f \equiv 0$.
\end{proof}

The preceding proposition is not generally as useful as one would hope,
because we'll usually want to think of pseudoholomorphic curves not as 
specific maps but as equivalence classes of maps up to parametrization,
whereas the condition that $u$ and $v$ have matching 
derivatives of all
orders at a point depends heavily on the choices of parametrizations.
We shall now prove a more powerful version of unique continuation that
doesn't have this drawback.  It will be of use to us when we study
local intersection properties in \S\ref{sec:intersections}.

\begin{thm}
\label{thm:uniqueContin}
Suppose $j_1$ and $j_2$ are smooth complex structures on $B$, $J$ 
is a smooth almost complex structure on $\CC^n$, and 
$u : (B,j_1) \to (\CC^n,J)$ and 
$v : (B,j_2) \to (\CC^n,J)$ are smooth nonconstant
pseudoholomorphic curves which satisfy $u(0) = v(0) = 0$ and have
matching partial derivatives to all orders at~$z=0$.  Then for sufficiently
small $\epsilon > 0$ there exists an embedding $\varphi : B_\epsilon \to B$
with $\varphi(0) = 0$ such that $u \equiv v \circ \varphi$ on $B_\epsilon$.
\end{thm}

A corollary is that if $u , v : (B,i) \to (\CC^n,J)$ are 
$J$-holomorphic curves that have the same $\infty$-jet at~$0$ after a
smooth reparametrization, then they are also identical up to 
parametrization.  The reparametrization may be
smooth but not necessarily holomorphic, in which case it changes $i$ on the
domain to a nonstandard complex structure $j$, so that the reparametrized curve
no longer satisfies $\p_s u + J(u) \p_t = 0$ and
Prop.~\ref{prop:uniqueContin} thus no longer applies.  We will show however
that in this situation, one can find a diffeomorphism on the domain that
not only transforms $j$ back into~$i$ but also has vanishing derivatives
of all orders at~$0$, thus producing the conditions for 
Prop.~\ref{prop:uniqueContin}.

To prepare for the next lemma, recall that if $d \in \NN$ and
$u : B \to \CC^n$ is a $C^d$-smooth map, then its degree~$d$
Taylor polynomial at $z=0$ can be expressed in terms of the variables
$z = s + it$ and $\bar{z} = s - it$ as
\begin{equation}
\label{eqn:Taylor}
\sum_{k=0}^d \sum_{j+\ell=k} \frac{1}{j! \ell!}
\p_z^j \p_{\bar{z}}^\ell u(0) z^j \bar{z}^\ell,
\end{equation}
where the differential operators 
$$
\p_z = \frac{\p}{\p z} = \frac{1}{2} (\p_s - i \p_t)
\qquad\text{ and }\qquad
\p_{\bar{z}} = \frac{\p}{\p \bar{z}} = \frac{1}{2} (\p_s + i\p_t)
$$
are defined via the formal chain rule.  If you've never seen this before,
you should take a moment to convince yourself that
\eqref{eqn:Taylor} matches the standard Taylor's formula for a 
complex-valued function of two real variables $u(s,t) = u(z)$.  
The advantage of this formalism is that it is quite easy to recognize
whether a polynomial expressed in $z$ and $\bar{z}$ is holomorphic:
the holomorphic polynomials are precisely those which only depend on powers
of~$z$, and not~$\bar{z}$.

In the following, we'll use multiindices of the form $\alpha=(j,k)$ to
denote higher order partial derivatives with respect to $z$ and $\bar{z}$
respectively, i.e.
$$
D^\alpha = \p^j_z \p^k_{\bar{z}}.
$$

\begin{lemma}
\label{lemma:higherDerivs}
Suppose $u : B \to \CC^n$ is a smooth solution to the linear
Cauchy-Riemann type equation
\begin{equation}
\label{eqn:linearCR3}
\p_s u(z) + J(z) \p_t u(z) + A(z) u(z) = 0
\end{equation}
with $u(0) = 0$, where $J, A \in C^\infty(B,\End_\RR(\CC^n))$
with $[J(z)]^2 = -\1$ and $J(0) = i$.  If there exists
$k \in \NN$ such that $\p^\ell_z u(0) = 0$ for all 
$\ell = 1,\ldots,k$, then $\p_{\bar{z}} \p^\alpha u(0) = 0$ for all multiindices
$\alpha$ with $|\alpha| \le k$.  In particular, the first $k$ derivatives
of $u$ at $z=0$ all vanish, and $\p^{k+1}_z u(0)$ is the only potentially
nonvanishing partial derivative of order~$k+1$.
\end{lemma}
\begin{proof}
Since $J(0)=i$, \eqref{eqn:linearCR3} gives $\p_{\bar{z}} u(0) = 0$, thus we
argue by induction and assume $\p_{\bar{z}} D^\alpha u(0) = 0$ for all
multiindices $\alpha$ of order up to $\ell \le k - 1$.  This implies that
the first $\ell+1$ derivatives of $u$ vanish at $z=0$.  Now for any
multiindex $\alpha$ of order $\ell+1$, applying $D^\alpha$ to both sides
of \eqref{eqn:linearCR3} and reordering the partial derivatives yields
$$
\p_s D^\alpha u(z) + J(z) \p_t D^\alpha u(z) +
\sum_{|\beta| \le \ell+1} C_\beta(z) D^\beta u(z),
$$
where $C_\beta(z)$ are smooth functions that depend on the derivatives of
$A$ and~$J$.  Evaluating at $z=0$, the term $D^\beta u(0)$ always vanishes
since $|\beta| \le \ell+1$, so we obtain $\dbar D^\alpha u(0) = 0$ as claimed.
\end{proof}

\begin{lemma}
\label{lemma:uniqueContin1}
Given the assumptions of Theorem~\ref{thm:uniqueContin}, the
complex structures $j_1$ and $j_2$ satisfy $j_1(0) = j_2(0)$ and also have
matching partial derivatives to all orders at $z=0$.
\end{lemma}
\begin{proof}
This would be obvious if $u$ and $v$ were immersed at~$0$, since then we
could write $j_1 = u^*J$ and $j_2 = v^*J$, so the complex structures and
their derivatives at $z=0$ are fully determined by those of $u$, $v$ 
and~$J$.  In general we cannot assume $u$ and $v$ are immersed, but
we shall still use this kind of argument by taking advantage of
the fact that if $u$ and $v$ are not constant, then
Prop.~\ref{prop:uniqueContin} implies that 
they must indeed have a nonvanishing derivative of \emph{some} order at~$0$.  

Write $j := j_2$ and assume without loss of 
generality that $j_1 = i$, so $u$ satisfies $\p_s u + J(u) \p_t u = 0$.
We can also assume $J(0) = i$.
Regarding the first derivative of $u$ as the smooth map
$du : B \to \Hom_\RR(\RR^2,\CC^n)$ defined by the matrix-valued function
$$
du(z) = \begin{pmatrix}
\p_s u(z) & \p_t u(z) \end{pmatrix},
$$
let $m \in \NN$ denote the smallest order for which the $m$th derivative of
$u$ at $z=0$ does not vanish.  Since $u$ also satisfies a linear Cauchy-Riemann
type equation $\p_s u + \bar{J}(z) \p_t u$ with $\bar{J}(z) := J(u(z))$,
Lemma~\ref{lemma:higherDerivs} then implies that $\p^m_z u(0)$ is the only
nonvanishing $m$th order partial derivative with respect to $z$ and $\bar{z}$.
In particular, $\p^{m-1}_z du(0)$ is then the lowest order nonvanishing
derivative of $du$ at $z=0$, and the only one of order~$m-1$.  We claim
that the matrix $\p^{m-1}_z du(0) \in \Hom_\RR(\RR^2,\CC)$ is not only 
nonzero but also nonsingular, i.e.~it defines an injective linear 
transformation.  Indeed, computing another $m$th order derivative of $u$
which must necessarily vanish,
$$
0 = \dbar \p^{m-1}_z u(0) = \p^{m-1}_z \p_s u + 
i \p^{m-1}_z \p_t u,
$$
which means that the transformation defined by $\p^{m-1}_z du(0)$ is in
fact complex-linear, implying the claim.

Let us now consider together the equations satisfied by $u$ and $v$:
\begin{equation*}
\begin{split}
du(z) i &= J(u(z)) du(z),\\
dv(z) j(z) &= J(v(z)) dv(z),
\end{split}
\end{equation*}
where the two sides of each equation are both regarded as smooth functions
$B \to \Hom_\RR(\RR^2,\CC^n)$.  By assumption, the right hand sides of both 
equations have matching partial derivatives of all orders at $z=0$, thus
so do the left hand sides.  Subtracting the second from the first, we obtain
the function
$$
du(z) i - dv(z) j(z) = du(z) \left[ i - j(z) \right] + 
\left[ du(z) - dv(z) \right] j(z),
$$
which must have vanishing derivatives of all orders at $z=0$.  For the second
term in the expression this is already obvious, so we deduce
\begin{equation}
\label{eqn:palpha}
\left. D^\alpha \left[ du \cdot ( i - j ) \right]\right|_{z=0} = 0
\end{equation}
for all multiindices $\alpha$.  Applying $\p^{m-1}_z$ in particular and
using the fact that $D^\beta du(0)$ vanishes whenever
$|\beta| < m - 1$, this implies
$$
\p^{m-1}_z du(0) \cdot \left[ i - j(0) \right] = 0,
$$
so $j(0) = i$ since $\p^{m-1}_z du(0)$ is injective.  We now argue
inductively that all higher derivatives of $i-j(z)$ must also vanish at $z=0$.
Assuming it's true for all derivatives up to order $k-1$, suppose
$\alpha$ is a multiindex of order~$k$, and plug the operator
$\p^{m-1}_z D^\alpha$ into \eqref{eqn:palpha}.  This yields
$$
\left. \p^{m-1}_z D^{\alpha} \left[ du \cdot (i - j) \right]\right|_{z=0} =
c \cdot \p^{m-1}_z du(0) \cdot \left. D^\alpha ( i - j)\right|_{z=0} = 0,
$$
where $c > 0$ is a combinatorial constant; all other ways of distributing
the operator $\p^{m-1}_z D^\alpha$ across this product kill at least one of
the two terms.
Thus using the injectivity of $\p^{m-1}_z du(0)$ once more,
$\left.D^\alpha (i - j)\right|_{z=0} = 0$.
\end{proof}

\begin{lemma}
\label{lemma:holoTaylor}
Suppose $j$ is a smooth complex structure on $\CC$ such that
$j(0) = i$ and the derivatives $D^\alpha j(0)$ vanish for all
orders $|\alpha| \ge 1$.  If $\varphi : (B_\epsilon,i) \to (\CC,j)$ is 
pseudoholomorphic with $\varphi(0) = 0$, then
the Taylor series of $\varphi$ about $z=0$ converges to a holomorphic
function on~$B_\epsilon$.
\end{lemma}
\begin{proof}
The map $\varphi : B_\epsilon \to \CC$ satisfies the linear Cauchy-Riemann
equation
\begin{equation}
\label{eqn:CRpsi}
\p_s \varphi(z) + \bar{\jmath}(z) \p_t \varphi(z) = 0,
\end{equation}
where we define $\bar{\jmath}(z) = j(\varphi(z))$.  Our conditions on
$j$ imply that $\bar{\jmath}(0) = i$ and
$\bar{\jmath}$ also has vanishing derivatives
of all orders at~$0$, thus for any multiindex $\alpha$, applying the
differential operator $D^\alpha$ to both sides of \eqref{eqn:CRpsi}
and evaluating at $z=0$ yields $\dbar D^\alpha\varphi(0) = 0$.  This implies
that all terms in the Taylor expansion of $\varphi$ about $z=0$ are
holomorphic, as the only nonvanishing partial derivatives in
\eqref{eqn:Taylor} are of the form $\p_z^k \varphi(0)$ for $k \ge 0$.

To see that this Taylor series is actually convergent, we can use a 
Cauchy integral to construct the holomorphic function to which it converges: 
for any $\delta < \epsilon$ and $z \in B_\delta$, let
$$
f(z) = \frac{1}{2\pi i} \int_{\p \bar{B}_\delta} 
\frac{ \varphi(\zeta) \,d\zeta}{\zeta - z}.
$$
This is manifestly a holomorphic function, and its derivatives at
$z=0$ are given by
\begin{equation}
\label{eqn:higherDeriv}
f^{(n)}(0) = \frac{n!}{2\pi i} \int_{\p \bar{B}_\delta} 
\frac{ \varphi(\zeta) \,d\zeta}{ \zeta^{n+1}}.
\end{equation}
Observe that this integral doesn't depend on the value of~$\delta$.
To compute it, write $\varphi$ in terms of its degree~$n$ Taylor polynomial as
$$
\varphi(z) = \sum_{k=0}^n \frac{1}{k!} \p^k_z \varphi(0) z^k +
|z|^{n+1} B(z),
$$
with $B(z)$ a bounded function.
The integral in \eqref{eqn:higherDeriv} thus expands into a sum of
$n+2$ terms, of which the first $n$ are integrals of holomorphic
functions and thus vanish, the last vanishes in the limit
$\delta \to 0$, and the only one left is
$$
f^{(n)}(0) = \frac{n!}{2\pi i} \int_{\p \bar{B}_\delta} 
\frac{ \p_z^n\varphi(0)}{n!} \,\frac{d\zeta}{\zeta} = \p_z^n\varphi(0).
$$
Thus $f$ and $\varphi$ have the same Taylor series.
\end{proof}

\begin{proof}[Proof of Theorem~\ref{thm:uniqueContin}]
Denote $j := j_2$ and without loss of generality, assume $j_1 \equiv i$
and $J(0) = i$.  Since all complex structures on $B$ are integrable,
there exists a smooth pseudoholomorphic embedding
$$
\varphi : (B,i) \to (B,j)
$$
with $\varphi(0) = 0$.  Now Lemma~\ref{lemma:uniqueContin1} implies that
$j - i$ has vanishing derivatives of all orders at $z=0$, and applying
Lemma~\ref{lemma:holoTaylor} in turn, we find a holomorphic function
$f : B \to \CC$ with $f(0) = 0$ whose derivatives at~$0$ of all 
orders match those of~$\varphi$.  In particular $f'(0) = d\varphi(0)$ is 
nonsingular, thus $f$ is a biholomorphic
diffeomorphism between open neighborhoods of~$0$, and for sufficiently
small $\epsilon > 0$, we obtain a pseudoholomorphic map
$$
\varphi \circ f^{-1} : (B_\epsilon,i) \to (B,j)
$$
whose derivatives of all orders at~$0$ match those of the identity map.
It follows that $v \circ \varphi \circ f^{-1} : B_\epsilon \to \CC^n$ is now
a $J$-holomorphic curve with the same $\infty$-jet as $u$ at $z=0$,
so Prop.~\ref{prop:uniqueContin} implies
$v \circ \varphi \circ f^{-1} \equiv u$.
\end{proof}

\section{Intersections with holomorphic hypersurfaces}
\label{sec:intersectionsHigher}

The similarity principle can also be used to prove
certain basic facts about intersections of $J$-holomorphic curves.
The following is the ``easy'' case of an important phenomenon
known as \emph{positivity of intersections}.  A much stronger
version of this result is valid in dimension four and will be
proved in \S\ref{sec:positivity}.

Let us recall the notion of the local intersection index
for an isolated intersection of two maps.
Suppose $M$ is an oriented smooth manifold of dimension $n$, $M_1$ and $M_2$ 
are oriented smooth manifolds
of dimension $n_1$ and $n_2$ with $n_1 + n_2 = n$, and
$f_1 : M_1 \to M$ and $f_2 : M_2 \to M$ are smooth maps.  We say that the
pair $(p_1,p_2) \in M_1 \times M_2$ is an \defin{isolated intersection} of
$f_1$ and $f_2$ at $p \in M$ if $f_1(p_1) = f_2(p_2) = p$ and there exist 
neighborhoods $p_1 \in \uU_1 \subset M_1$ and $p_2 \in \uU_2 \subset M_2$ such 
that
$$
f_1(\uU_1 \setminus \{p_1\}) \cap f_2(\uU_2 \setminus \{p_2\}) = \emptyset.
$$
In this case, one can define the \defin{local intersection index}
$$
\inter(f_1,p_1 ; f_2,p_2) \in \ZZ
$$
as follows.  If the intersection is transverse, we set
$\inter(f_1,p_1 ; f_2,p_2) = \pm 1$, with the sign chosen to be positive
if and only if the natural orientations defined on each side of the 
decomposition
$$
T_p M = \im d f_1(p_1) \oplus \im d f_2(p_2)
$$
match.  If the intersection is not transverse, choose two neighborhoods
$\uU_1$ and $\uU_2$ as above and make generic $C^\infty$-small perturbations
of $f_1$ and $f_2$ to maps $f_1^\epsilon$ and $f_2^\epsilon$ 
such that $f_1^\epsilon|_{\uU_1} \pitchfork f_2^\epsilon|_{\uU_2}$, then 
define
$$
\inter(f_1,p_1 ; f_2,p_2) = \sum_{(q_1,q_2)} \inter(f_1^\epsilon,q_1 ; f_2^\epsilon,q_2),
$$
where the sum ranges over all pairs $(q_1,q_2) \in \uU_1 \times \uU_2$ such
that $f_1^\epsilon(q_1) = f_2^\epsilon(q_2)$.

\begin{exercise}
\label{EX:independent}
Suppose $M_1$ and $M_2$ are compact oriented smooth manifolds with boundary, 
$M$ is an oriented smooth manifold such that $\dim M_1 + \dim M_2 = \dim M$,
and
$$
f_1^\tau : M_1 \to M, \qquad f_2^\tau : M_2 \to M, \qquad \tau \in [0,1]
$$
are smooth homotopies of maps with the property that for all $\tau \in [0,1]$,
$$
f_1^\tau(\p M_1) \cap f_2^\tau(M_2) = f_1^\tau(M_1) \cap f_2^\tau(\p M_2) = \emptyset.
$$
Show that if $f_1^\tau$ and $f_2^\tau$ have only transverse intersections for
$\tau \in \{0,1\}$, then
\begin{equation}
\label{eqn:intIndependent}
\sum_{f_1^0(p_1) = f_2^0(p_2)} \inter(f_1^0,p_1 ; f_2^0,p_2) =
\sum_{f_1^1(p_1) = f_2^1(p_2)} \inter(f_1^1,p_1 ; f_2^1,p_2).
\end{equation}
Deduce from this that the above definition of the local intersection index
for an isolated but non-transverse intersection is independent of choices.
Then, show that \eqref{eqn:intIndependent} also holds if the intersections
for $\tau \in \{0,1\}$ are assumed to be isolated but not necessarily
transverse.
\textsl{Hint: If you have never read \cite{Milnor:differentiable}, you should.}
\end{exercise}

Similarly, if $f : M_1 \to M$ is a smooth map and $N \subset M$ is an 
oriented submanifold with $\dim M_1 + \dim N = \dim M$, a point $p \in M_1$ 
with $f(p) \in N$ can be regarded as an isolated intersection of $f$ with $N$
if it defines an isolated intersection of $f_1$ with the inclusion map
$N \hookrightarrow M$, and the resulting local intersection index will
be denoted by
$$
\inter(f,p ; N) \in \ZZ.
$$

\begin{thm}
\label{thm:babyIntersections}
Suppose $(M,J)$ is an almost complex manifold of dimension $2n \ge 4$,
and $\Sigma \subset M$ is a $(2n-2)$-dimensional oriented submanifold which is
$J$-holomorphic in the sense that $J(T\Sigma) = T\Sigma$ and whose orientation
matches the canonical orientation determined by $J|_{T\Sigma}$.
Then for any smooth nonconstant $J$-holomorphic curve
$u : B \to M$ with $u(0) \in \Sigma$, either $u(B) \subset \Sigma$ or
the intersection $u(0) \in \Sigma$ is isolated.  In the latter case,
$$
\inter(u,0 ; \Sigma) \ge 1,
$$
with equality if and only if the intersection is transverse.
\end{thm}
\begin{proof}
By choosing coordinates intelligently, we can assume without loss of generality
that $\Sigma = \CC^{n-1} \times \{0\} \subset \CC^{n-1} \times \CC = M$,
$u(0) = (0,0)$, and $J$ satisfies
$$
J(w,0) = \begin{pmatrix}
\hat{J}(w) & 0 \\
0 & i
\end{pmatrix}
$$
for all $w \in \CC^{n-1}$ near~$0$, where $i$ in the lower right entry means
the standard complex structure on~$\CC$ and $\hat{J}$ is a smooth almost 
complex structure on $\CC^{n-1}$.
Write $u(z) = (\hat{u}(z),f(z)) \in \CC^{n-1} \times \CC$, so that
intersections of $u$ with $\Sigma$ correspond to zeroes of $f : B \to \CC$.
We shall use an interpolation trick as in the proof of 
Prop.~\ref{prop:uniqueContin} to show that $f$ satisfies a 
linear Cauchy-Riemann type equation.

For $t \in [0,1]$, let $u_t(z) = (\hat{u}(z), t f(z))$, so $u_1 = u$ and
$u_0 = (\hat{u},0)$.  Then since $\p_s u + J(u) \,\p_t u = 0$, we have
\begin{equation*}
\begin{split}
\p_s u + J(u_0) \, \p_t u &= \p_s u + J(u) \, \p_t u -
\left[ J(u_1) - J(u_0) \right] \p_t u \\
&= - \left( \int_0^1 \frac{d}{dt} J(\hat{u},t f) \, dt \right) \p_t u 
= - \left( \int_0^1 D_2 J(\hat{u}, t f) \cdot f \, dt \right) \p_t u \\
&=: - \tilde{A} f,
\end{split}
\end{equation*}
where the last step defines a smooth family of linear maps $\tilde{A} : B \to
\Hom_\RR(\CC,\CC^n)$.  Since $J(u_0) = J(\hat{u},0)$ preserves the factors in 
the splitting $\CC^n = \CC^{n-1} \times \CC$, we can project 
this expression to the second factor and obtain a smooth family of linear maps
$A : B \to \End_\RR(\CC,\CC)$ such that the equation 
$\p_s f + i\, \p_t f + A f$ is satisfied.

By the similarity principle, $f$ either vanishes identically near $z=0$ or 
has an isolated zero there.  The former would imply $u(B) \subset \Sigma$.
In the latter case, the isolated zero has positive order, so $f$ can be 
perturbed slightly near~$0$ to a smooth
function with only simple zeroes, where the signed count of these is positive
and matches
the signed count of transverse intersections between $\Sigma$ and the 
resulting perturbation of~$u$.  Moreover, the signed count is $1$ if and only 
if the zero at~$z=0$ is already simple, which means the unperturbed 
intersection of $u$ with $\Sigma$ is transverse.
\end{proof}

\section{Nonlinear regularity}
\label{sec:nonlinearReg}

We now extend the previous linear regularity results to the nonlinear
case.  In order to understand local questions regarding pseudoholomorphic maps
$u : (\Sigma,j) \to (M,J)$, it suffices to study $u$ in local coordinates
near any given points on the domain and target, where by 
Theorem~\ref{thm:RiemannSurfaces}, we can always take \emph{holomorphic}
coordinates on the domain.  We can therefore assume $(\Sigma,j) = (B,i)$
and $M$ is the unit ball $B^{2n} \subset \CC^n$, with an almost complex
structure $J$ that matches the standard complex structure~$i$ at the origin.
Denote by 
$$
\jJ^m(B^{2n}) = \left\{ J \in C^m(B^{2n},\End_\RR(\CC^n)) \ |\ J^2 \equiv -\1 
\right\}
$$
the space of $C^m$-smooth almost complex structures on~$B^{2n}$.

\begin{thm}
\label{thm:regularity}
Assume $p \in (2,\infty)$, $m \ge 1$ is an integer, 
$J \in \jJ^m(B^{2n})$ with $J(0) = i$ and
$u : B \to B^{2n}$ is a $J$-holomorphic curve in $W^{1,p}(B)$ with
$u(0) = 0$.  Then $u$ is also of class $W^{m+1,p}_\loc$ on~$B$.
Moreover, if $J_k \in \jJ^m(B^{2n})$ is a sequence with $J_k \to J$ in $C^m$
and $u_k \in W^{1,p}(B)$ is a sequence of $J_k$-holomorphic curves in $B^{2n}$
converging in $W^{1,p}$ to~$u$, then $u_k$ also converges in
$W^{m+1,p}_\loc$.
\end{thm}

By the Sobolev embedding theorem, this implies that if $J$ is
smooth, then every $J$-holomorphic curve is also smooth, and the topology
of $W^{1,p}_\loc$-convergence on a space of pseudoholomorphic curves is
equivalent to the topology of $C^\infty_\loc$-convergence.
This equivalence has an important consequence for the compactness theory
of holomorphic curves, arising from the fact that the hierarchy of
Sobolev spaces
$$
\ldots \subset W^{k,p} \subset W^{k-1,p} \subset \ldots 
\subset W^{1,p} \subset L^p
$$
comes with natural inclusions that are not only continuous but also
compact.  Indeed, the following result plays a fundamental role in the
proof of Gromov's compactness theorem, to be discussed later---it is
often summarized by the phrase ``gradient bounds imply $C^\infty$-bounds.''

\begin{cor}
\label{cor:gradBounds}
Assume $p \in (2,\infty)$ and $m \ge 1$, $J_k \in \jJ^m(B^{2n})$ is a 
sequence of almost complex structures converging in $C^m$ to $J \in 
\jJ^m(B^{2n})$, and $u_k : B \to B^{2n}$ is a sequence of $J_k$-holomorphic
curves satisfying a uniform bound $\| u_k \|_{W^{1,p}(B)} < C$.
Then $u_k$ has a subsequence converging in $W^{m+1,p}_\loc$ to a
$J$-holomorphic curve $u : B \to B^{2n}$.
\end{cor}
\begin{proof}
Our main task is to show that $u_k$ also satisfies a uniform bound
in $W^{m+1,p}$ on every compact subset of~$B$, as the compact embedding
$W^{m+1,p} \hookrightarrow W^{m,p}$ then gives a convergent subsequence
in $W^{m,p}_\loc$, which by Theorem~\ref{thm:regularity} must also converge
in $W^{m+1,p}_\loc$.  We begin with the observation that $u_k$ already has a
$C^0$-convergent subsequence, since $W^{1,p}(B)$ embeds compactly into
$C^0(B)$; thus assume without loss of generality that $u_k$ converges in
$C^0$ to a continuous map $u : B \to B^{2n}$, and after a change of coordinates
on the target, $u(0) = 0$ and $J(0) = i$.

Theorem~\ref{thm:regularity} can be rephrased in terms of the following
\emph{local moduli spaces}: let 
$$
\Mod^{1,p,m} \subset C^m(B^{2n},\End_\RR(\CC^n)) \times W^{1,p}(B,\CC^n)
$$
denote the space of pairs $(J,u)$ such that
$J \in \jJ^m(B^{2n})$ and $u : B \to B^{2n}$ is a $J$-holomorphic curve.
This is naturally a metric space due to its inclusion in the Banach
space above.  Similarly,
for any positive number $r < 1$, define the Banach space
$$
W^{1,p}_r(B,\CC^n) = \left\{ u \in W^{1,p}(B,\CC^n) \ \big|\ u|_{B_r} \in W^{m+1,p}(B_r) \right\},
$$
whose norm is the sum of the norms on $W^{1,p}(B)$ and $W^{m+1,p}(B_r)$,
and define the metric subspace
$$
\Mod^{1,p,m}_r
= \{ (J,u) \in \jJ^m(B^{2n}) \times W^{1,p}_r(B,\CC^n) \ |\ 
\text{$u(B) \subset B^{2n}$ and $\p_s u + J(u)\p_t u = 0$} \}.
$$
Theorem~\ref{thm:regularity} implies that the natural inclusion
\begin{equation}
\label{eqn:homeo}
\Mod^{1,p,m}_r \hookrightarrow \Mod^{1,p,m}
\end{equation}
is a homeomorphism.  Now the pairs $(J_k,u_k)$ form a bounded sequence in
$\Mod^{1,p,m}$, and we can use the following rescaling trick to replace
$(J_k,u_k)$ by a
sequence that stays within a small neighborhood of $\jJ^m(B^{2n}) \times \{0\}$.
For any $\epsilon > 0$ and $u \in W^{1,p}(B)$, define the map
$u^\epsilon : B \to \CC^n$ by
$$
u^\epsilon(z) = u(\epsilon z).
$$
We claim that for any $\delta > 0$, one can choose $\epsilon > 0$
such that $\| u_k^\epsilon \|_{W^{1,p}(B)} < \delta$ for
sufficiently large~$k$.  Indeed, integrating by change of variables,
\begin{equation*}
\begin{split}
\| u_k^\epsilon \|_{L^p(B)}^p &= \int_B | u_k^\epsilon(z) |^p \,ds\,dt =
\frac{1}{\epsilon^2} \int_{B_\epsilon} | u_k(z) |^p \,ds\,dt \le
\frac{1}{\epsilon^2} \int_{B_\epsilon} \| u_k \|_{C^0(B_\epsilon)}^p \,ds\,dt \\
&= \pi \| u_k \|_{C^0(B_\epsilon)}^p \to \pi \| u \|_{C^0(B_\epsilon)}^p,
\end{split}
\end{equation*}
where the latter is small for small $\epsilon$ since $u(0) = 0$.  Likewise,
$$
\| D u_k^\epsilon \|_{L^p(B)}^p = \int_B \| \epsilon D u_k(\epsilon z) |^p \,ds\,dt
= \epsilon^{p-2} \int_{B_\epsilon} | D u_k(z) |^p \,ds\,dt \le
\epsilon^{p-2} \| D u_k \|_{L^p(B)}^p,
$$
which is small due to the uniform bound on $\| u_k \|_{W^{1,p}(B)}$.
Thus choosing $\epsilon$ sufficiently small, 
$(J_k,u_k^\epsilon) \in \Mod^{1,p,m}$ lies in an arbitrarily small ball about
$(J,0)$ for large~$k$, and the homeomorphism \eqref{eqn:homeo} then implies that
the same is true in $\Mod^{1,p,m}_r$, thus giving a uniform bound
$$
\| u_k^\epsilon \|_{W^{m+1,p}(B_r)} < C.
$$
Rescaling again, this implies a uniform bound on 
$\| u_k \|_{W^{m+1,p}(B_{\epsilon r})}$.  Since this same argument can be
carried out on any sufficiently small ball about an interior point in $B$,
and any compact subset is covered by finitely many such balls, this
implies the desired bound in $W^{m+1,p}_\loc$ on~$B$.
\end{proof}

Theorem~\ref{thm:regularity} will be proved by induction, and the hard
part is the initial step: we need to show that if $J$ is of class $C^1$,
then the regularity of $u$ can be improved from $W^{1,p}$ to $W^{2,p}_\loc$.
Observe that it suffices to find a number $\epsilon > 0$ such that
$u \in W^{2,p}(B_\epsilon)$ and the sequence $u_k$ converges 
in $W^{2,p}(B_\epsilon)$, since any compact subset of $B$ can be covered by
finitely many such balls of arbitrarily small radius.  To obtain the
desired results on $B_\epsilon$, we will use much the same argument that was
used in Prop.~\ref{prop:inhomogeneousRegularity} for the linear case: more
bookkeeping is required since $J$ is not standard, but we'll take advantage
of the assumption $J(0)=i$, so that $J$ is \emph{nearly} standard on 
$B_\epsilon$ if $\epsilon$ is sufficiently small.

\begin{proof}[Proof of Theorem~\ref{thm:regularity} for $m=1$]
We shall use the method of difference quotients as in 
Prop.~\ref{prop:inhomogeneousRegularity} to show that
$u \in W^{2,p}(B_\epsilon)$ for small $\epsilon > 0$.\footnote{The difference
quotient argument explained here is adapted from the proof given in
\cite{AbbasHofer}*{Appendix~4}.}
For any $r < 1$ and 
$h \in \RR\setminus \{0\}$ sufficiently small, define a function
$u^h \in W^{1,p}(B_r,\CC^n)$ by
$$
u^h(s,t) = \frac{u(s+h,t) - u(s,t)}{h},
$$
so $u^h$ converges in $L^p(B_r)$ to $\p_s u$ as $h \to 0$.  Our main
goal is to find constants $\epsilon > 0$ and $C > 0$ such that
\begin{equation}
\label{eqn:diffQuotBound}
\| u^h \|_{W^{1,p}(B_\epsilon)} < C
\end{equation}
for all sufficiently small $h \ne 0$.
The Banach-Alaoglu theorem then gives a
sequence $h_j \to 0$ such that $u^{h_j}$ converges weakly in
$W^{1,p}(B_\epsilon)$, implying that its limit $\p_s u$ is also in
$W^{1,p}(B_\epsilon)$; since exactly the same argument works for $\p_t u$,
we will conclude $u \in W^{2,p}(B_\epsilon)$.

To prove the bound \eqref{eqn:diffQuotBound}, assume at first that $\epsilon$ 
is any real number with $0 < \epsilon < 1/2$; its value will be further
specified later.  Choose a smooth cutoff function
$\beta_\epsilon : B \to [0,1]$ with support in $B_{2\epsilon}$ such that
$\beta|_{B_\epsilon} \equiv 1$.
It will then suffice to show that if $\epsilon$ is
taken small enough, we can find a uniform bound on
$\| \beta_\epsilon u^h \|_{W^{1,p}(B_{2\epsilon})}$ as $h \to 0$.  
The latter has compact 
support in~$B_{2\epsilon}$, so the Calder\'{o}n-Zygmund inequality gives
$$
\| \beta_\epsilon u^h \|_{W^{1,p}(B_{2\epsilon})} 
\le c \| \dbar (\beta_\epsilon u^h) \|_{L^p(B_{2\epsilon})}.
$$
We wish to take advantage of the fact that $\dbar_J u \equiv 0$, where
we abbreviate $\dbar_J := \p_s + J(u) \p_t$.  The latter can be
regarded as the standard
Cauchy-Riemann operator on a trivial bundle with nonstandard complex
structure $J(u(z))$, so in particular it satisfies the Leibnitz rule
$\dbar_J(f v) = (\dbar_J f) v + f (\dbar_J v)$ for $f : B \to \RR$ and
$v : B \to \CC^n$.  The difference quotient also satisfies a Leibnitz
rule $(fv)^h = f^h v + f v^h$.  Now rewriting $\dbar(\beta_\epsilon u^h)$
in terms of $\dbar_J$, we have
\begin{equation}
\label{eqn:ineq1}
\dbar(\beta_\epsilon u^h) = \dbar_J(\beta_\epsilon u^h) +
\left[ i - J(u) \right] \p_t (\beta_\epsilon u^h),
\end{equation}
where the first term can be expanded as
\begin{equation}
\label{eqn:ineq2}
\begin{split}
\dbar_J(\beta_\epsilon u^h) 
&= (\dbar_J \beta_\epsilon) u^h + \beta_\epsilon \dbar_J(u^h) \\
&= (\dbar \beta_\epsilon) u^h + \left[J(u) - i \right] (\p_t \beta_\epsilon) u^h 
+ \beta_\epsilon \left( \p_s u^h + J(u) \p_t u^h \right) \\
&= (\dbar \beta_\epsilon) u^h + \left[J(u) - i \right] (\p_t \beta_\epsilon) u^h 
+ \beta_\epsilon \left( (\dbar_J u)^h - [J(u)]^h \p_t u \right) \\
&= (\dbar \beta_\epsilon) u^h + \left[J(u) - i \right] (\p_t \beta_\epsilon) u^h 
- \beta_\epsilon [J(u)]^h \p_t u.
\end{split}
\end{equation}
The last term in \eqref{eqn:ineq1} satisfies the bound
\begin{equation*}
\begin{split}
\left\| \left[ i - J(u) \right] \p_t(\beta_\epsilon u^h) 
\right\|_{L^p(B_{2\epsilon})} &\le
\| i - J(u) \|_{C^0(B_{2\epsilon})} \| \p_t (\beta_\epsilon u^h) \|_{L^p(B_{2\epsilon})} \\
&\le C_1(\epsilon) \| \beta_\epsilon u^h \|_{W^{1,p}(B_{2\epsilon})},
\end{split}
\end{equation*}
where $C_1(\epsilon) := \| i - J(u) \|_{C^0(B_{2\epsilon})}$, and the
fact that $J(u(0)) = J(0) = i$ implies that $C_1(\epsilon)$ goes to zero as
$\epsilon \to 0$.  We can find similar bounds for every term on the right
hand side of \eqref{eqn:ineq2}: the first two, $\| (\dbar\beta_\epsilon) u^h \|_{L^p}$ 
and $\| [ J(u) - i ] (\p_t\beta_\epsilon) u^h \|_{L^p}$, are both bounded
uniformly in~$h$ since $\| u^h \|_{L^p} \to \| \p_s u \|_{L^p}$ as $h \to 0$.
For the third term, we use the fact that $J \in C^1$ to find a pointwise bound
\begin{equation*}
\begin{split}
\left| [J(u)]^h(s,t) \right| &= \frac{1}{h} \left| J(u(s + h,t)) - J(u(s,t)) \right| \le
\frac{1}{h} \| J \|_{C^1} \left| u(s+h,t) - u(s,t) \right| \\ 
&= \| J \|_{C^1} \left| u^h(s,t) \right|,
\end{split}
\end{equation*}
which implies
\begin{equation*}
\begin{split}
\left\| \beta_\epsilon [J(u)]^h \p_t u \right\|_{L^p(B_{2\epsilon})} &\le
\left\| \beta_\epsilon [J(u)]^h \right\|_{C^0(B)} 
\| \p_t u \|_{L^p(B_{2\epsilon})} \\
&\le C \| \beta_\epsilon u^h \|_{C^0(B)} \| u \|_{W^{1,p}(B_{2\epsilon})} \\
&\le C_2(\epsilon) \| \beta_\epsilon u^h \|_{W^{1,p}(B)}
= C_2(\epsilon) \| \beta_\epsilon u^h \|_{W^{1,p}(B_{2\epsilon})},
\end{split}
\end{equation*}
using the continuous embedding of $W^{1,p}(B)$ into $C^0(B)$.  Here
$C_2(\epsilon)$ is a constant multiple of $\| u \|_{W^{1,p}(B_{2\epsilon})}$
and thus also decays to zero as $\epsilon \to 0$.  Putting all of this
together, we have
$$
\| \beta_\epsilon u^h \|_{W^{1,p}(B_{2\epsilon})} \le C +
C_3(\epsilon) \| \beta_\epsilon u^h \|_{W^{1,p}(B_{2\epsilon})}
$$
where $C_3(\epsilon) \to 0$ as $\epsilon \to 0$, thus taking $\epsilon$
sufficiently small, we can move the last term to the left hand side and
obtain the desired bound,
$$
\| \beta_\epsilon u^h \|_{W^{1,p}(B_{2\epsilon})} \le
\frac{C}{1 - C_3(\epsilon)}.
$$

The statement about convergent sequences follows by a similar argument:
we assume $\| u - u_k \|_{W^{1,p}(B)} \to 0$ and use 
Exercise~\ref{EX:CZhigher} to estimate
$\| u - u_k \|_{W^{2,p}(B_\epsilon)}$ via
$$
\| \beta_\epsilon (u - u_k) \|_{W^{2,p}(B_{2\epsilon})} \le
c_1 \| \dbar(\beta_\epsilon u) - \dbar(\beta_\epsilon u_k) \|_{W^{1,p}(B_{2\epsilon})}.
$$
It will be important to note that the constant $c_1 > 0$ in this relation
does not depend on the choice of $\epsilon > 0$.
Adapting the computation of \eqref{eqn:ineq1} and \eqref{eqn:ineq2} using
$\p_s u + J(u) \p_t u = \p_s u_k + J_k(u_k)\p_t u_k = 0$, we now find
\begin{equation*}
\begin{split}
\dbar(\beta_\epsilon u) - \dbar(\beta_\epsilon u_k) &= 
(\dbar\beta_\epsilon) (u - u_k) \\
&\quad + (\p_t\beta_\epsilon) [J(u) - i] (u - u_k) 
 + (\p_t\beta_\epsilon) [J(u) - J_k(u_k)] u_k \\
&\quad + [ J_k(u_k) - J(u) ] \p_t (\beta_\epsilon u) 
 + [ i - J_k(u_k) ] [ \p_t(\beta_\epsilon u) - \p_t(\beta_\epsilon u_k) ].
\end{split}
\end{equation*}
Since $W^{1,p}$ is a Banach algebra, it is easy
to see that for any fixed $\epsilon > 0$ sufficiently small,
the first three terms in this expression each decay
to zero in $W^{1,p}(B_{2\epsilon})$ as 
$\| u - u_k \|_{W^{1,p}} \to 0$;
in particular for the third term, we use the fact that $J_k \to J$ in $C^1$
to conclude $J_k(u_k) \to J(u)$ in $W^{1,p}$.  The fourth term is bounded
similarly since $\| \p_t (\beta_\epsilon u) \|_{W^{1,p}(B_{2\epsilon})} \le
\| \beta_\epsilon u \|_{W^{2,p}(B_{2\epsilon})}$, and we've already proved 
above that $u \in W^{2,p}(B_r)$ for sufficiently small~$r$.  The fifth
term is a bit trickier: using the definition of the $W^{1,p}$-norm,
we have
\begin{equation}
\label{eqn:bigMess}
\begin{split}
\big\| [ i - J_k(u_k) ] &[ \p_t(\beta_\epsilon u) - \p_t(\beta_\epsilon u_k) ]
\big\|_{W^{1,p}(B_{2\epsilon})} \\
&\le \big\| [ i - J_k(u_k) ] [ \p_t(\beta_\epsilon u) - \p_t(\beta_\epsilon u_k) ]
\big\|_{L^p(B_{2\epsilon})} \\
& \qquad + \big\| DJ_k(u_k) \cdot Du_k \cdot [ \p_t(\beta_\epsilon u) - \p_t(\beta_\epsilon u_k) ]
\big\|_{L^p(B_{2\epsilon})} \\
& \qquad + \big\| [ i - J_k(u_k) ] [ D\p_t(\beta_\epsilon u) - D\p_t(\beta_\epsilon u_k) ]
\big\|_{L^p(B_{2\epsilon})}.
\end{split}
\end{equation}
Since $u_k \to u$ and $J_k \to J$ in $C^0$ while $J(u(0)) = i$, we can fix
$\epsilon > 0$ small enough so that for all $k$ sufficiently large,
$$
\left\| i - J_k(u_k) \right\|_{C^0(B_{2\epsilon})} \le \frac{1}{3 c_1}.
$$
The first term on the right hand side of \eqref{eqn:bigMess} is then bounded
by a constant times $\| \beta_\epsilon u - \beta_\epsilon u_k \|_{W^{1,p}}$,
which goes to zero as $k \to \infty$, and the third term is bounded by
$$
\| i - J_k(u_k) \|_{C^0(B_{2\epsilon})} 
\| \beta_\epsilon u - \beta_\epsilon u_k \|_{W^{2,p}(B_{2\epsilon})}
\le \frac{1}{3 c_1} \| \beta_\epsilon u - \beta_\epsilon u_k \|_{W^{2,p}(B_{2\epsilon})}.
$$
For the second term, we use the continuous embedding 
$W^{1,p} \hookrightarrow C^0$ and obtain the bound
\begin{equation*}
\begin{split}
\| DJ_k \|_{C^0} & \| D u_k \|_{L^p(B_{2\epsilon})} 
\| \p_t (\beta_\epsilon u) - \p_t (\beta_\epsilon u_k) \|_{C^0(B)} \\
&\le
c_2 \| J_k \|_{C^1} \| u_k \|_{W^{1,p}(B_{2\epsilon})}
\| \p_t (\beta_\epsilon u) - \p_t (\beta_\epsilon u_k) \|_{W^{1,p}(B)} \\
&\le c_3 \| u \|_{W^{1,p}(B_{2\epsilon})} 
\| \beta_\epsilon u - \beta_\epsilon u_k \|_{W^{2,p}(B_{2\epsilon})},
\end{split}
\end{equation*}
where we observe that the constant $c_3 > 0$ is also independent of the
choice of $\epsilon > 0$.  We can therefore shrink $\epsilon$ if necessary
and assume
$$
\| u \|_{W^{1,p}(B_{2\epsilon})} \le \frac{1}{3 c_1 c_3}.
$$
Putting all this together, we now have a bound of the form
$$
\| \beta_\epsilon (u - u_k) \|_{W^{2,p}(B_{2\epsilon})} \le
F\left(\| u - u_k \|_{W^{1,p}}\right) + 
\frac{2}{3} \| \beta_\epsilon (u - u_k) \|_{W^{2,p}(B_{2\epsilon})}
$$
for sufficiently large~$k$, where $F(t) \to 0$ as $t \to 0$, thus 
we conclude that 
$\| \beta_\epsilon (u - u_k) \|_{W^{2,p}(B_{2\epsilon})} \to 0$ as
$k \to \infty$.
\end{proof}

To complete the proof of Theorem~\ref{thm:regularity} by induction,
we use the following simple fact: if $u$ is $J$-holomorphic, then its
$1$-jet can also be regarded as a pseudoholomorphic map.  A global
version of this statement is made precise in the appendix by
P.~Gauduchon of \cite{Audin:almostComplex}, but we will only need a local
version, which is much simpler to see.  If $J \in \jJ^m(B^{2n})$,
we can define an almost complex structure $\hat{J}$ of class $C^{m-1}$
on $B \times B^{2n} \times \CC^n$ in block form by
$$
\hat{J}(z,u,X) = \begin{pmatrix}
i & 0 & 0 \\
0 & J(u) & 0 \\
A(u,X) & 0 & J(u)
\end{pmatrix},
$$
where $A(u,X) \in \Hom_\RR(\CC,\CC^n)$ is defined by
$$
A(u,X) (x + iy) = \begin{pmatrix}
 DJ(u)X \cdot X & DJ(u)X \cdot J(u)X \end{pmatrix}
\begin{pmatrix}
x \\
y
\end{pmatrix}.
$$
Using the fact that $0 = D(J^2)(u)X = DJ(u)X \cdot J(u) + J(u) \cdot DJ(u)X$,
one can easily compute that $A(u,X) i + J A(u,X) = 0$ and thus $\hat{J}$
is indeed an almost complex structure.  Moreover, if $u : B \to B^{2n}$ 
satisfies $\p_s u + J(u) \p_t u = 0$ then 
$$
\hat{u} : B \to B \times B^{2n} \times \CC^n : z \mapsto (z,u(z),\p_s u(z))
$$
satisfies $\p_s\hat{u} + \hat{J}(\hat{u}) \p_t\hat{u} = 0$.  Indeed,
this statement amounts to a system of three PDEs, of which the first is
trivial, the second is $\p_s u + J(u) \p_t u = 0$ and the third is the
latter differentiated with respect to~$s$.

\begin{exercise}
Verify all of the above.
\end{exercise}

We can now carry out the inductive step in the proof of 
Theorem~\ref{thm:regularity}: assume the theorem is proved for almost
complex structures of class $C^{m-1}$.  Then if $J \in \jJ^m(B^{2n})$ and
$u \in W^{1,p}(B)$ is $J$-holomorphic, we have
$u \in W^{m,p}_\loc$, and $\p_s u$ is $\hat{J}$-holomorphic for an almost
complex structure $\hat{J}$ of class $C^{m-1}$, implying $\p_s u \in
W^{m,p}_\loc$ as well.  Now $\p_t u = J(u) \p_s u$ is also in
$W^{m,p}_\loc$ since $W^{m,p}$ is a Banach algebra, hence 
$u \in W^{m+1,p}_\loc$ as claimed.  The statement about converging
sequences follows by a similarly simple argument.

\section{Some tools of global analysis}
\label{sec:calculus}

To understand the structure of spaces of solutions to the nonlinear 
Cauchy-Riemann equation, and in particular to prove local existence
in the next section, we will use the generalization of the standard 
differential calculus for smooth maps between Banach spaces.  
A readable and elegant introduction to this topic may be found in the 
book of Lang \cite{Lang:analysis}; here we shall merely summarize the 
essential facts.

Most of the familiar properties of derivatives and
differentiable functions generalize nicely to maps between arbitrary normed 
linear spaces $X$ and $Y$, so long as both spaces are complete.  The 
derivative of the map $f: X \rightarrow Y$ at $x \in X$ (also often 
called its \defin{linearization}) is by definition a continuous linear operator
$$
df(x) \in {\mathcal L}(X,Y)
$$
such that for small $h \in X$,
$$
f(x + h) = f(x) + df(x)h + o(\| h \|_X ),
$$
where $o(\|h\|_X)$ denotes an arbitrary map of the form $\eta(h) \cdot \|h\|_X$
with $\lim_{h\rightarrow 0}\eta(h) = 0$.  If $df(x)$ exists for all $x\in X$, 
one has a map between Banach spaces
 $df: X \rightarrow {\mathcal L}(X,Y)$, 
which may have its own derivative, and one thus obtains the notions of higher 
order derivatives and smoothness.  Proving differentiability 
in the infinite-dimensional setting is sometimes an intricate problem, 
often requiring integral inequalities such as Sobolev or H\"older 
estimates, and it is not hard to find natural examples of maps that are 
everywhere continuous but nonsmooth on some dense set.

\begin{exercise}
\label{EX:notDifferentiable}
If $S^1 = \RR / \ZZ$, we can denote the Banach space of real-valued continuous
and $1$-periodic functions on $\RR$ by $C^0(S^1)$.
Show that the map $\Phi : \RR \times C^0(S^1)
\to C^0(S^1)$ defined by $\Phi(s,f)(t) = f(s + t)$ is continuous but
not differentiable.
\end{exercise}

Despite these complications, having defined the derivative, 
one can prove infinite-dimensional versions of the familiar differentiation rules, Taylor's 
formula and the implicit function theorem, which can become powerful tools.  
The proofs, in fact, are virtually the same as in the finite-dimensional 
case, with occasional reference to some simple tools of linear functional 
analysis such as the Hahn-Banach theorem.  Let us state the two most 
important results that we will make use of.

\begin{thm}[Inverse function theorem]
Suppose $X$ and $Y$ are Banach spaces, $\uU \subset X$ is an open subset and
$f : \uU \to Y$ is a map of class $C^k$ for $k \ge 1$ such that for some
$x_0 \in \uU$, $df(x_0) : X \to Y$ is a continuous isomorphism.
Then $f$ maps some neighborhood $\oO$ of $x_0$ bijectively to an open
neighborhood of $y_0 := f(x_0)$, and its local inverse
$f^{-1} : f(\oO) \to \oO$ is also of class $C^k$, with
$$
d(f^{-1})(y_0) = \left[ df(x_0) \right]^{-1}.
$$
\end{thm}

Note that while derivatives and notions of differentiability 
can be defined in more
general normed vector spaces, the inverse function theorem really requires 
$X$ and $Y$ to be \emph{complete}, as the proof uses Banach's fixed point 
theorem (i.e.~the ``contraction mapping principle'').  The implicit 
function theorem follows from this, though we should emphasize that it 
requires an extra hypothesis that is vacuous in the finite-dimensional case:

\begin{thm}[Implicit function theorem]
\label{thm:IFT}
Suppose $X$ and $Y$ are Banach spaces, $\uU \subset X$ is an open subset and
$f : \uU \to Y$ is a map of class $C^k$ for $k \ge 1$ such that for some
$x_0 \in \uU$, $df(x_0) : X \to Y$ is surjective and admits a bounded right inverse.
Then there exists a $C^k$-map
$$
\Phi_{x_0} : \oO_{x_0} \to X,
$$
which maps some open neighborhood $\oO_{x_0} \subset \ker df(x_0)$ of~$0$
bijectively to an open neighborhood of $x_0$ in $f^{-1}(y_0)$, where
$y_0 = f(x_0)$.
\end{thm}

Note that the existence of a bounded right inverse of $df(x_0)$ is equivalent
to the existence of a splitting
$$
X = \ker df(x_0) \oplus V,
$$
where $V \subset X$ is a closed linear subspace, so there is a bounded 
linear projection map $\pi_K : X \to \ker df(x_0)$.
One makes use of this in the proof as follows: assume without loss of 
generality that $x_0 = 0$ and consider the map
\begin{equation}
\label{eqn:Psi}
\Psi_0 : \uU \to Y \oplus \ker df(0) : x \mapsto (f(x), \pi_K(x)).
\end{equation}
Then $d\Psi_0(0) = (df(0),\pi_K) : X \to Y \oplus \ker df(0)$ is an isomorphism,
so the inverse function theorem gives a local $C^k$-smooth inverse
$\Psi_0^{-1}$, and the desired parametrization of $f^{-1}(y_0)$ can be
written as $\Phi_0(v) = \Psi_0^{-1}(f(0),v)$ for sufficiently small
$v \in \ker df(0)$.

Of course the most elegant way to state the implicit function theorem is
in terms of manifolds: a \defin{Banach manifold} of class $C^k$ is simply
a topological space that has local charts identifying neighborhoods with
open subsets of Banach spaces such that all transition maps are
$C^k$-smooth diffeomorphisms.  Then the map $\Phi_{x_0}$ in the implicit
function theorem can be regarded as the inverse of a chart, defining a
Banach manifold structure on a subset of $f^{-1}(y_0)$.  In fact, it is
not hard to see that if $x_1 , x_2 \in f^{-1}(y_0)$ are two distinct
points satisfying the hypotheses of the theorem, then the resulting
``transition maps''
$$
\Phi_{x_1}^{-1} \circ \Phi_{x_2} : \oO_{x_2} \to \oO_{x_1}
$$
are $C^k$-smooth diffeomorphisms.  Indeed, these can be defined in terms
of the $\Psi$-map of \eqref{eqn:Psi} via
$$
\Psi_{x_1} \circ \Psi_{x_2}^{-1}(y_0,v) = (y_0, \Phi_{x_1}^{-1} \circ
\Phi_{x_2}(v)),
$$
where $\Psi_{x_1}$ and $\Psi_{x_2}$ are $C^k$-smooth local diffeomorphisms.
Moreover, these charts identify the tangent space to $f^{-1}(y_0)$ at
any $x_0 \in f^{-1}(y_0)$ with $\ker df(x_0) \subset X$.
Thus we can restate the implicit function theorem as follows.

\begin{cor}
Suppose $X$ and $Y$ are Banach spaces, $\uU \subset X$ is an open subset,
$f : \uU \to Y$ is a $C^k$-smooth map for $k \ge 1$
and $y \in Y$ is a regular value
of $f$ such that for every $x \in f^{-1}(y)$, $df(x)$ has a bounded right
inverse.  Then $f^{-1}(y)$ admits the structure of a $C^k$-smooth Banach
submanifold of~$X$, whose tangent space at $x \in f^{-1}(y)$ is
$\ker df(x)$.
\end{cor}

By picking local charts, one sees that a similar statement is true if
$X$ and $Y$ are also Banach manifolds instead of linear spaces, and one
can generalize a step further to consider smooth sections of Banach
space bundles.  These results will become particularly useful when we deal 
with Fredholm maps, 
for which the linearization has finite-dimensional kernel and thus satisfies 
the bounded right inverse assumption trivially whenever it is surjective.
In this way one can prove that solution sets of certain PDEs are
finite-dimensional smooth manifolds.  In contrast, we'll see an example in 
the next section of a solution set that is an \emph{infinite-dimensional} 
smooth Banach manifold.

The differential geometry of Banach manifolds in infinite dimensions is
treated at length in \cite{Lang:geometry}.  A more basic question is
how to prove that certain spaces which naturally ``should'' be
Banach manifolds actually are.  This rather delicate question has been 
studied in substantial generality in the literature (see for example 
\cites{Eells:setting,Palais:global,Eliasson}): 
the hard part is always to show that
certain maps between Banach spaces are differentiable.  The key is to
consider only Banach spaces that have nice enough properties so that
certain natural classes of maps
are continuous, so that smoothness can then be proved by induction.

The next two lemmas are illustrative examples of the kinds of results 
one needs, and we'll make use of them in the next section.  First a
convenient piece of notation: if $\uU \subset \RR^m$ and $\Omega \subset \RR^n$ 
are open subsets and $\mathbf{X}(\uU,\RR^n)$ denotes some Banach space
of maps $\uU \to \RR^n$ that admits a continuous inclusion into
$C^0(\uU,\RR^n)$, then denote
$$
\mathbf{X}(\uU,\Omega) = \{ u \in \mathbf{X}(\uU,\RR^n)\ |\ 
u(\uU) \subset \Omega \}.
$$
Due to the continuous inclusion assumption, this is an open subset of
$\mathbf{X}(\uU,\RR^n)$.  We assume below for simplicity that $\Omega$ is 
convex, but this assumption is easy to remove at the cost of more cumbersome
notation; see \cite{Eliasson}*{Lemma~4.1} for a much more general version.

\begin{lemma}
\label{lemma:babySmoothness}
Suppose $\uU \subset \RR^m$ denotes an open subset, and the symbol
$\mathbf{X}$ associates to any Euclidean space $\RR^N$ a Banach space
$\mathbf{X}(\uU,\RR^N)$ consisting of bounded continuous 
maps $\uU \to \RR^N$ such that the
following hypotheses are satisfied:
\begin{itemize}
\item \textsc{($C^0$-inclusion)} The inclusion 
$\mathbf{X}(\uU,\RR^N) \hookrightarrow C^0(\uU,\RR^N)$ is continuous.
\item \textsc{(Banach algebra)} The natural bilinear pairing
$$
\mathbf{X}(\uU,\lL(\RR^n,\RR^N)) \times \mathbf{X}(\uU,\RR^n) \to
\mathbf{X}(\uU,\RR^N) : (A,u) \mapsto Au
$$
is well defined and continuous.
\item \textsc{($C^k$-continuity)}
For some integer $k \ge 0$, if $\Omega \subset \RR^n$ is any open set and
$f \in C^k(\Omega,\RR^N)$, the map
\begin{equation}
\label{eqn:Phif}
\Phi_f : \mathbf{X}(\uU,\Omega) \to \mathbf{X}(\uU,\RR^N) : u \mapsto f \circ u
\end{equation}
is well defined and continuous.
\end{itemize}
If $\Omega \subset \RR^n$ is a convex open set and
$f \in C^{k+r}(\Omega,\RR^N)$ for some $r \in \NN$, then the map
$\Phi_f$ defined in \eqref{eqn:Phif} is of class $C^r$ and has derivative
\begin{equation}
\label{eqn:theDerivative}
d\Phi_f(u)\eta = (df \circ u) \eta.
\end{equation}
\end{lemma}
\begin{remark}
In the formula \eqref{eqn:theDerivative} for the derivative we're implicitly
using both the \textsc{Banach algebra} and \textsc{$C^k$-continuity} 
hypotheses: 
the latter implies that $df \circ u$ is a map in
$\mathbf{X}(\uU,\lL(\RR^n,\RR^N))$, which the former then embeds continuously
into $\lL(\mathbf{X}(\uU,\RR^n),\mathbf{X}(\uU,\RR^N))$.
\end{remark}
\begin{proof}[Proof of Lemma~\ref{lemma:babySmoothness}]
We observe first that it suffices to prove differentiability and the
formula \eqref{eqn:theDerivative}, as $df \circ u$ is a continuous
function of $u$ and $C^r$-smoothness follows by induction.
Thus assume $r=1$ and $\eta \in \mathbf{X}(\uU,\RR^n)$ is small enough so
that $u + \eta \in \mathbf{X}(\uU,\Omega)$.  Then
\begin{equation}
\label{eqn:remainder}
\begin{split}
\Phi_f(u + \eta) &= \Phi_f(u) + [f \circ (u + \eta) - f \circ u ] =
\Phi_f(u) + \int_0^1 \frac{d}{dt} f \circ (u + t\eta) \,dt \\
&= \Phi_f(u) + \left[ \int_0^1 df \circ (u + t\eta) \,dt \right] \eta \\
&= \Phi_f(u) + (df \circ u) \eta + \left[\theta_f 
\circ (u + \eta,u)\right] \eta,
\end{split}
\end{equation}
where we've defined $\theta_f : \Omega \times \Omega \to \lL(\RR^n,\RR^N)$ by
\begin{equation}
\label{eqn:theta}
\theta_f(x,y) = \int_0^1 \left[ df((1-t) y + tx) - df(y) \right] \,dt,
\end{equation}
and observe that $\theta_f \in C^k$ since $f \in C^{k+1}$.
It follows that $\theta_f$ defines a continuous map
$$
\mathbf{X}(\uU,\Omega \times \Omega) \to \mathbf{X}(\uU,\lL(\RR^n,\RR^N)) :
(u,v) \mapsto \theta_f \circ (u,v),
$$
and in particular
$$
\lim_{\eta \to 0} \theta_f \circ (u+\eta,u) = \theta_f(u,u) = 0,
$$
where the limit is taken in the topology of
$\mathbf{X}(\uU,\lL(\RR^n,\RR^N))$.  Thus \eqref{eqn:remainder} proves
the stated formula for $d\Phi_f(u)$.
\end{proof}

We will need something slightly more general, since we'll 
also want to be able to differentiate $(f,u) \mapsto f \circ u$
with respect to~$f$.

\begin{lemma}
\label{lemma:smoothness}
Suppose $\uU$, $\Omega$ and $\mathbf{X}(\uU,\RR^n)$ are as in
Lemma~\ref{lemma:babySmoothness}, and in addition that the pairing
$T(u) f := f \circ u$ defines $T$ as a continuous map
\begin{equation}
\label{eqn:T2}
T : \mathbf{X}(\uU,\Omega) \to \lL(C^k(\Omega,\RR^N), \mathbf{X}(\uU,\RR^N) ).
\end{equation}
Then for any $r \in \NN$, the map
$$
\Psi : C^{k+r}(\Omega,\RR^N) \times \mathbf{X}(\uU,\Omega) \to 
\mathbf{X}(\uU,\RR^N) : (f,u) \mapsto f \circ u
$$
is of class $C^r$ and has derivative
$$
d\Psi(f,u)(g,\eta) = g \circ u + (df \circ u) \eta.
$$
\end{lemma}
\begin{proof}
We'll continue to write $\Phi_f = \Psi(f,\cdot)$
for each $f \in C^{k+r}(\Omega,\RR^N)$;
this is a $C^r$-smooth map $\mathbf{X}(\uU,\Omega) \to \mathbf{X}(\uU,\RR^N)$
by Lemma~\ref{lemma:babySmoothness}.  Observe that the pairing
$T(u)f = f \circ u$ of \eqref{eqn:T2} also gives a map
$$
T : \mathbf{X}(\uU,\Omega) \to \lL(C^{k+r}(\Omega,\RR^N),\mathbf{X}(\uU,\RR^N))
$$
for each integer $r \ge 0$, and we claim that this is of class~$C^r$.
The claim mostly follows already from the proof of
Lemma~\ref{lemma:babySmoothness}: expressing the remainder formula
\eqref{eqn:remainder} in new notation gives
\begin{equation}
\label{eqn:remainder2}
T(u + \eta)f = T(u)f + \left[ T_1(u) df \right] \eta +
\left[ T_2(u+\eta,u) \theta_f \right] \eta,
\end{equation}
where we've defined the related maps
\begin{equation*}
\begin{split}
T_1 : \mathbf{X}(\uU,\Omega) &\to \lL(C^{k+r-1}(\Omega,\lL(\RR^n,\RR^N)),
\mathbf{X}(\uU,\lL(\RR^n,\RR^N)) ), \\
T_2 : \mathbf{X}(\uU,\Omega\times\Omega) 
&\to \lL(C^{k+r-1}(\Omega \times \Omega,
\lL(\RR^n,\RR^N)) , \mathbf{X}(\uU,\lL(\RR^n,\RR^N)) ).
\end{split}
\end{equation*}
Note that the correspondence defined in
\eqref{eqn:theta} gives a bounded linear map
$$
C^{k+r}(\Omega,\RR^N) \to C^{k+r-1}\left(\Omega \times 
\Omega,\lL(\RR^n,\RR^N)\right) : f \mapsto \theta_f.
$$
Now arguing by induction, we can assume $T_1$ and $T_2$ are both of class
$C^{r-1}$.  Then as a family of bounded linear operators acting on~$f$, the 
pairing of $T_2(u + \eta,u)$ with $\theta_f$ goes to zero as $\eta \to 0$, and
\eqref{eqn:remainder2} implies
$$
\left[ dT(u) \eta\right] f = \left[ T_1(u) df \right] \eta,
$$
so $dT$ is of class $C^{r-1}$, proving the claim.

Next consider the derivative of the map $\Psi$ in the case $r=1$.
For any small $g \in C^{k+1}(\Omega,\RR^N)$ and
$\eta \in \mathbf{X}(\uU,\Omega)$, we compute
\begin{equation*}
\begin{split}
\Psi(f + g,u + \eta) &= \Psi(f,u) + 
\left[ T({u+\eta})(f+g) - T({u+\eta})(f) \right] +
\left[ \Phi_f(u + \eta) - \Phi_f(u) \right] \\
&= \Psi(f,u) + T(u) g + (T({u+\eta}) - T(u))g + d\Phi_f(u)\eta + o(\|\eta\|) \\
&= \Psi(f,u) + g \circ u + (df \circ u)\eta + o(\|(g,\eta)\|)  ).
\end{split}
\end{equation*}
Thus $\Psi$ is differentiable and we can write its derivative in the form
$d\Psi(f,u) = T(u) + \Psi(df,u)$.  The general result now follows easily by
induction.
\end{proof}

In the next section we'll apply this using the fact that
if $B \subset \CC$ is the open unit ball, then the space
$W^{k,p}(B)$ (for $k \ge 1$ and $p > 2$) is a
Banach algebra that embeds continuously into $C^0$, and the pairing
$(f,u) \mapsto f \circ u$ gives a continuous map
$$
C^k(\Omega,\RR^N) \times W^{k,p}(B,\Omega) \to W^{k,p}(B,\RR^N).
$$

Observe that by Lemma~\ref{lemma:babySmoothness}, the map 
$u \mapsto f \circ u$ on a suitable Banach space will be smooth if $f$ is
smooth.  Things get a bit trickier if we also consider $f$ to be a
variable in this map: e.g.~if $f$ varies arbitrarily in $C^k$ then the map
$\Psi(f,u) = f \circ u$ also has only finitely many derivatives.  This
headache is avoided if $f$ is allowed to vary only in some Banach space that
embeds continuously into $C^\infty$, for then one can apply 
Lemma~\ref{lemma:smoothness} for every $k$ and conclude that 
$\Psi$ is in $C^r$ for all~$r$.  The most obvious examples of Banach spaces
with continuous embeddings into $C^\infty$ are finite dimensional,
but we will also see an infinite-dimensional example in 
Chapter~\ref{chapter:moduli} when we 
discuss transversality and Floer's ``$C_\epsilon$ space''.

\section{Local existence of $J$-holomorphic curves}
\label{sec:nonlinear}

We shall now apply the machinery described in the previous section
to prove a local existence result
from which Theorem~\ref{thm:RiemannSurfaces} on the integrability of
Riemann surfaces follows as an easy corollary.
As usual in studying such local questions, we will consider
$J$-holomorphic maps from the unit ball $B \subset \CC$ into
$B^{2n} \subset \CC^n$, with the coordinates chosen so that
$J(0) = i$.  Let $B_r$ and $B_r^{2n}$ denote the balls of radius~$r > 0$
in $\CC$ and $\CC^n$ respectively.

In \S\ref{sec:intro} we stated the result that there
always exists a $J$-holomorphic curve tangent to any given vector at a
given point.  What we will actually prove is more general:
if $J$ is sufficiently smooth, then one can find
local $J$-holomorphic curves with specified derivatives up to some fixed
order at a point, not just the first derivatve---moreover one can also find
families of such curves that vary continuously under perturbations of~$J$.  
Some caution is in order:
it would be too much to hope that one could specify all partial derivatives 
arbitrarily, as the nonlinear
Cauchy-Riemann equation implies nontrivial relations, e.g.~$\p_t u(0) =
J(u(0)) \, \p_s u(0)$.  
What turns out to be possible is to specify the
\emph{holomorphic part} of the Taylor polynomial of $u$ at $z=0$ up to
some finite order, i.e.~the terms in the Taylor expansion that depend only
on $z$ and not on $\bar{z}$ (cf.~Equation~\eqref{eqn:Taylor}).
The relevant higher order derivatives of $u$ will thus be those of
the form $\p^k_z u(0)$.  As 
the following simple result demonstrates, trying to specify more partial 
derivatives beyond these would yield an ill-posed problem.

\begin{prop}
\label{prop:vanishingDerivatives}
Suppose $J$ is a smooth almost complex structure on $\CC^n$ with $J(0)=i$,
and $u , v : B \to \CC^n$ are a pair of $J$-holomorhic curves with
$u(0) = v(0) = 0$.  If there exists $d \in \NN$ such that
$$
\p_z^k u(0) = \p_z^k v(0)
$$
for all $k=0,\ldots,d$, then in fact $D^\alpha u(0) = D^\alpha v(0)$
for every multiindex $\alpha$ with $|\alpha| \le d$.
\end{prop}
\begin{proof}
Recall that when we used the similarity principle to prove unique continuation 
in Prop.~\ref{prop:uniqueContin},
we did so by showing that $h := u - v : B \to \CC^n$ satisfies a linear
Cauchy-Riemann type equation of the form
$$
\p_s h + \bar{J}(z) \p_t h + A(z) h = 0,
$$
where in the present situation $\bar{J} : B \to \End_\RR(\CC^n)$ 
is a smooth family
of complex structures on $\CC^n$ and $A \in C^\infty(B,\End_\RR(\CC^n))$.
Since $\p_z^k h(0) = 0$ for all $k = 0,\ldots,d$,
Lemma~\ref{lemma:higherDerivs} now implies $D^\alpha h(0) = 0$ for all
$|\alpha| \le d$.
\end{proof}

Here is the main local existence result.

\begin{thm}
\label{thm:localExistence3}
Assume $p \in (2,\infty)$, $d \ge 1$ is an integer, $m \in \NN \cup \{\infty\}$
with $m \ge d + 1$, and $J \in \jJ^m(B^{2n})$ with $J(0) = i$.
Then for sufficiently small $\epsilon > 0$, there exists a
$C^{m-d}$-smooth map
$$
\Psi : (B^{2n}_\epsilon)^{d+1} \to W^{d+1,p}(B,\CC^n)
$$
such that for each $(w_0,\ldots,w_d) \in (B^{2n}_\epsilon)^{d+1}$,
$u := \Psi(w_0,\ldots,w_d)$ is a $J$-holomorphic curve with
$$
\p^k_z u(0) = w_k
$$
for each $k= 0,\ldots,d$.
\end{thm}

\begin{exercise}
Convince yourself that Theorem~\ref{thm:localExistence3}, together with
elliptic regularity, implies that smooth
almost complex structures on a real $2$-dimensional manifold are 
always smoothly integrable, i.e.~they admit smooth local charts whose
transition maps are holomorphic.  (See also Corollary~\ref{cor:jvarying}.)
\end{exercise}

\begin{remark}
There is also an analogue of Theorem~\ref{thm:localExistence3} for local
holomorphic half-disks with totally real boundary conditions;
see \cite{Zehmisch:jets}.
\end{remark}

As with local existence of holomorphic sections, our proof of
Theorem~\ref{thm:localExistence3} will be based on the philosophy that
in a sufficiently small neighborhood, everything can be understood as
a perturbation of the standard Cauchy-Riemann equation.  To make this
precise, we will take a
closer look at the \emph{local moduli space} of $J$-holomorphic curves
that was introduced in the proof of Corollary~\ref{cor:gradBounds}.
For $p \in (2,\infty)$ and $k \ge 1$, define
$$
W^{k,p}(B,B^{2n}) = \{ u \in W^{k,p}(B,\CC^n) \ |\ u(B) \subset B^{2n} \},
$$
which is an open subset of $W^{k,p}(B,\CC^n)$ due to the continuous
embedding of $W^{k,p}$ in~$C^0$.  The space of $C^m$-smooth almost
complex structures on $B^{2n}$ will again be denoted by
$\jJ^m(B^{2n})$.  Now for $J \in \jJ^m(B^{2n})$, $p \in (2,\infty)$ and
$k \in \NN$, we define the local moduli space
$$
\Mod^{k,p}(J) = \{ u \in W^{k,p}(B,B^{2n}) \ |\ \p_s u + J(u)\p_t u = 0 \}.
$$
Observe that $\Mod^{k,p}(J)$ always contains the trivial map $u \equiv 0$.

\begin{prop}
\label{prop:BanachManifold}
Suppose $J \in \jJ^m(B^{2n})$ with $J(0) = i$ and
$m \ge k \ge 2$.  Then some
neighborhood of~$0$ in $\Mod^{k,p}(J)$ admits the structure of a
$C^{m-k+1}$-smooth Banach submanifold of $W^{k,p}(B,\CC^n)$, and its 
tangent space at~$0$ is
$$
T_0\Mod^{k,p}(J) = \{ \eta \in W^{k,p}(B,\CC^n) \ |\ \dbar\eta = 0 \}.
$$
\end{prop}

We prove this by presenting
$\Mod^{k,p}(J)$ as the zero set of a differentiable map between Banach
spaces---the tricky detail here is to determine exactly for which
values of $k$, $m$ and~$p$ the map in question is differentiable, and this
is the essential reason behind the condition $m \ge d+1$ in
Theorem~\ref{thm:localExistence3}.
For any $p \in (2,\infty)$ and $k,m \in \NN$ with $m \ge k-1$,
let $J \in \jJ^m(B^{2n})$ with $J(0) = i$ and define the nonlinear map
$$
\Phi_k : W^{k,p}(B,B^{2n}) \to W^{k-1,p}(B,\CC^n) :
u \mapsto \p_s u + J(u) \p_t u.
$$
This is well defined due to the continuous Sobolev embedding
$W^{k,p} \hookrightarrow C^{k-1}$: then $J \circ u$ is of class $C^{k-1}$
and thus defines a bounded multiplication on $\p_t u \in W^{k-1,p}$.
One can similarly show that $\Phi_k$ is continuous, though we are much
more interested in establishing conditions for it to be at least~$C^1$.

\begin{lemma}
\label{lemma:howSmooth}
If $m \ge k \ge 2$, then
$\Phi_k$ is of class $C^{m-k+1}$, and its derivative at~$0$ is
$$
d\Phi_k(0) : W^{k,p}(B,\CC^n) \to W^{k-1,p}(B,\CC^n) : 
\eta \mapsto \dbar\eta.
$$
\end{lemma}
\begin{proof}
The formula for $d\Phi_k(0)$ will follow from
Lemma~\ref{lemma:babySmoothness} once we show that $\Phi_k$ is at least~$C^1$.
The map $u \mapsto \p_s u$ is continuous and linear, thus automatically smooth,
so the nontrivial part is to show that the map $u \mapsto J(u) \p_t u$
from $W^{k,p}(B,B^{2n})$ to $W^{k-1,p}(B,\CC^n)$
is differentiable.  Since $k \ge 2$,
we can use the continuous inclusion of $W^{k,p}$ into $W^{k-1,p}$ and observe
that
$$
W^{k-1,p} \to W^{k-1,p} : u \mapsto J \circ u
$$
is of class of $C^{m-k+1}$ if $J \in C^m$, due to 
Lemma~\ref{lemma:babySmoothness}.  
Then differentiability of the map $u \mapsto J(u) \p_t u$ 
follows from the fact that $W^{k-1,p}$ is a Banach algebra.  
\end{proof}
Now we apply the crucial ingredient from the linear regularity theory:
Theorem~\ref{thm:rightInverse} implies that $d\Phi_k(0) = \dbar$ is
surjective and has a bounded right inverse.  The
implicit function theorem then gives 
$(\Phi_k)^{-1}(0)$ the structure
of a differential Banach manifold near~$0$ and identifies its tangent
space there with $\ker d\Phi_k(0) = \ker \dbar$,
so the proof of Prop.~\ref{prop:BanachManifold} is complete.

\begin{proof}[Proof of Theorem~\ref{thm:localExistence3}]
Since $m \ge d+1$, a
neighborhood of~$0$ in the local moduli space $\Mod^{d+1,p}(J)$ is a Banach 
manifold of class
$C^{m-d}$, and $T_0 \Mod^{d+1,p}(J) = \ker\dbar \subset W^{d+1,p}(B,\CC^n)$.
Due to the continuous inclusion of $W^{d+1,p}$ in $C^d$, there is a bounded
linear \emph{evaluation map}
$$
\ev_d : W^{d+1,p}(B,\CC^n) \to (\CC^n)^{d+1} : u \mapsto
(u(0),\p_z u(0),\p^2_z u(0),\ldots, \p_z^d u(0)),
$$
which restricts to the local moduli space
$$
\ev_d : \Mod^{d+1,p}(J) \to (\CC^n)^{d+1}
$$
as a $C^{m-d}$-smooth map near~$0$.  We shall use the inverse function
theorem to show that $\ev_d$ maps a neighborhood of~$0$ in $\Mod^{d+1,p}(J)$ 
onto a neighborhood of~$0$ in $(\CC^n)^{d+1}$ and admits a
$C^{m-d}$-smooth right inverse.  

To see this concretely, it will be convenient to restrict to a 
finite-dimensional submanifold of $\Mod^{d+1,p}$.  Let 
$$
\pP_d \subset W^{d+1,p}(B,\CC^n)
$$
denote the complex $n(d+1)$-dimensional vector space consisting of all 
holomorphic polynomials
with degree at most~$d$, regarded here as smooth maps $B \to \CC^n$.
Define also the closed subspace
$$
\Theta^{d+1,p}(B,\CC^n) = \im \widehat{T} \subset W^{d+1,p}(B,\CC^n),
$$
where $\widehat{T} : W^{d,p}(B,\CC^n) \to W^{d+1,p}(B,\CC^n)$ is the bounded
right inverse of $\dbar : W^{d+1,p}(B,\CC^n) \to W^{d,p}(B,\CC^n)$
provided by Theorem~\ref{thm:rightInverse}.
Note that $\Theta^{d+1,p}(B,\CC^n) \cap \pP_d = \{0\}$ since
everything in $\pP_d$ is holomorphic.  Putting these together, we define
the closed subspace
$$
\Theta\pP_d(B,\CC^n) = \Theta^{d+1,p}(B,\CC^n) \oplus
\pP_d \subset W^{d+1,p}(B,\CC^n),
$$
which contains an open subset 
$$
\Theta\pP_d(B,B^{2n}) = \{ u \in \Theta\pP_d(B,\CC^n)\ |\ 
u(B) \subset B^{2n} \}.
$$
By construction, the restriction of $\dbar : W^{d+1,p}(B,\CC^n) \to
W^{d,p}(B,\CC^n)$ to $\Theta\pP_d(B,\CC^n)$ is surjective and its kernel
is precisely~$\pP_d$.  Restricting similarly the nonlinear operator 
that was used to define $\Mod^{k,p}(J)$, we obtain a $C^{m-d}$-smooth map
$$
\widehat{\Phi} : \Theta\pP_d(B,B^{2n}) \to 
W^{d,p}(B,\CC^n) : u \mapsto \p_s u + J(u) \p_t u,
$$
whose derivative at~$0$ is surjective and has kernel~$\pP_d$, hence
$$
\widehat{\Mod}(J) := \widehat{\Phi}^{-1}(0) \subset \Mod^{d+1,p}(J)
$$
is a $C^{m-d}$-smooth finite-dimensional manifold near~$0$, with
$T_0 \widehat{\Mod}(J) = \pP_d$.  Consider now the restriction of the
evaluation map to $\widehat{\Mod}(J)$,
$$
\ev_d : \widehat{\Mod}(J) \to (\CC^n)^{d+1}.
$$
This map is linear on $W^{d+1,p}(B,\CC^n)$, thus its derivative is simply
$$
d\ev_d(0) : \pP_d \to (\CC^n)^{d+1} : \eta \mapsto \ev_d(\eta),
$$
which is the isomorphism that uniquely associates to any holomorphic
polynomial of degree~$d$ its derivatives of order~$0$ to~$d$.
Now by the inverse function theorem, the restriction of $\ev_d$ to
$\widehat{\Mod}(J)$ can be inverted on a neighborhood of~$0$, giving
rise to the desired $C^{m-d}$-smooth map~$\Psi$.
\end{proof}

Notice that one can extract from Theorem~\ref{thm:localExistence3}
\emph{parametrized families} of local $J$-holomorphic curves.
In particular, if $N \subset \CC^n$ is a sufficiently small 
submanifold of $\CC^n$, we can find a family of $J$-holomorphic disks
$\{ u_x \}_{x \in N}$ such that $u_x(0) = x$.  These vary continuously in
$W^{1,p}$, but actually if $J$ is smooth, then the regularity theorem
of \S\ref{sec:nonlinearReg} implies that they also vary continuously 
in $C^\infty$ on compact subsets.  This implies the following:

\begin{cor}
\label{cor:smoothFamily}
If $J$ is a smooth almost complex structure on $B^{2n}$,
$N \subset B^{2n}$ is a smooth submanifold passing through~$0$ and
$X$ is a smooth vector field along $N$,
then for some neighborhood $\uU \subset N$ of $0$ and some $\epsilon > 0$,
there exists a smooth family of $J$-holomorphic curves
$$
u_x : B \to \CC^n, \qquad x \in \uU
$$
such that $u_x(0) = x$ and $\p_s u_x(0) = \epsilon X(x)$.
\end{cor}
\begin{remark}
The standard meaning of the term ``smooth family'' as used in
Cor.~\ref{cor:smoothFamily} is that the map $\uU \times B \to \CC^n :
(x,z) \mapsto u_x(z)$ is smooth.  Unfortunately, smoothness in this
sense does not follow immediately from Theorem~\ref{thm:localExistence3};
the theorem rather provides smooth maps
$$
\uU_d \to W^{d,p}(B,B^{2n}) : x \mapsto u_x
$$
for arbitrarily large integers $d \ge 2$ (since $J$ is smooth), defined
on open neighborhoods $\uU_d \subset N$ whose sizes a priori depend on~$d$.
Of course more is true, as regularity guarantees that all of these maps
are actually continuous into $C^\infty(B_r,B^{2n})$ for any $r < 1$, 
but one still must
be careful in arguing that this implies a smooth family.  Since we don't
have any specific applications for this result in mind, we'll leave the
details as an exercise.  It should however be mentioned that this and
related results are occasionally used in the literature to construct
special coordinates that make certain computations easier; see for example
Exercise~\ref{EX:coordinates} below.
\end{remark}

\begin{exercise}
\label{EX:coordinates}
Use Corollary~\ref{cor:smoothFamily} to show that near any point $x_0$ in a
smooth almost complex manifold $(M,J)$, there exist smooth
coordinates $(\zeta,w) \in \CC \times \CC^{n-1}$ in which 
$J(x_0) = i$ and in general $J(\zeta,w)$ takes the block form
$$
J(\zeta,w) = \begin{pmatrix}
i & Y(\zeta,w) \\
0 & J'(\zeta,w)
\end{pmatrix},
$$
where $J'(\zeta,w)$ is a smooth
family of complex structures on $\CC^{n-1}$ and $Y(\zeta,w)$ satisfies
$i Y + Y J' = 0$.
\end{exercise}

Finally, we can generalize local existence by allowing our local
$J$-holomorphic curves to depend continuously on the 
choice of almost complex structure~$J$.  This is made possible by including
$\jJ^m(B^{2n})$ into the domain of the nonlinear operator, as it will
probably not surprise you to learn that the space of $C^m$-smooth
almost complex structures is itself a smooth Banach manifold.
For our purposes, it will suffice
to consider small perturbations of the standard complex 
structure~$i$.

By Exercise~\ref{EX:J}, the space $\jJ^m(B^{2n})$ of $C^m$-smooth almost
complex structures on $B^{2n}$ can be identified with the space of
$C^m$-smooth sections of the fiber bundle $\Aut_\RR(TB^{2n}) /
\Aut_\CC(TB^{2n})$, where we define $\Aut_\CC(TB^{2n})$ with respect to the
standard complex structure of~$\CC^n$.  One can use this fact and a version
of Lemma~\ref{lemma:babySmoothness} to show that $\jJ^m(B^{2n})$ is a
smooth Banach submanifold of the Banach space $C^m(B^{2n},\End_\RR(\CC^n))$.
We will not explicitly need
this fact for now, but we will need a single chart, for which a convenient 
choice is provided by \eqref{eqn:Cayley0}, namely for all 
$Y \in \overline{\End}_\CC(\CC^n)$ sufficiently small we can define
$J_Y \in \jJ(\CC^n)$ by
\begin{equation}
\label{eqn:YtoJ}
J_Y = \left( \1 + \frac{1}{2} i Y \right) i 
\left( \1 + \frac{1}{2} i Y \right)^{-1}.
\end{equation}
Choose $\delta > 0$
sufficiently small so that \eqref{eqn:YtoJ} is a well-defined embedding
of $\{ |Y| < \delta \}$ into $\jJ(\CC^n)$, and define the Banach space
$$
\Upsilon^m = C^m\left(B,\overline{\End}_\CC(\CC^n)\right)
$$
and open subset
$$
\Upsilon^m_\delta = \{ Y \in \Upsilon^m\ |\ \| Y \|_{C^0} < \delta \}.
$$
Then \eqref{eqn:YtoJ} defines a smooth map
\begin{equation}
\label{eqn:YtoJ1}
\Upsilon^m_\delta \to C^m(B^{2n},\End_\RR(\CC^n)) : Y \mapsto J_Y
\end{equation}
which takes $\Upsilon^m_\delta$ bijectively to a neighborhood of $i$ in
$\jJ^m(B^{2n})$.

\begin{exercise}
Verify that the map \eqref{eqn:YtoJ1} is a smooth embedding.
Lemma~\ref{lemma:babySmoothness} should be useful.
\end{exercise}

Now for integers $k , m \ge 1$ and $p \in (1,\infty)$, consider the
Banach space
$$
X^{k,p,m} = C^m(B^{2n},\End_\RR(\CC^n)) \times W^{k,p}(B,\CC^n)
$$
and subset
$$
\Mod^{k,p,m} = \{ (J,u) \in \jJ^{m}(B^{2n}) \times W^{k,p}(B,B^{2n}) \ |\ 
\p_s u + J(u)\p_t u = 0 \} \subset X^{k,p,m}.
$$
We will call this the \emph{local universal moduli space} of 
pseudoholomorphic curves.  Observe that it always contains pairs of the form
$(i,u)$ where $u : B \to B^{2n}$ is holomorphic.  
Its local structure near such a point can be
understood using the implicit function theorem: define the nonlinear map
$$
\Phi_k^m : \Upsilon^{m}_\delta \times W^{k,p}(B,B^{2n}) \to W^{k-1,p}(B,\CC^n) :
(Y,u) \mapsto \p_s u + J_Y(u) \p_t u.
$$
The zero set of this map can be identified with the space of all pairs
$(J,u) \in \Mod^{k,p,m}$ such that $J$ is within some $C^m$-small
neighborhood of~$i$, as then $J = J_Y$ for a unique $Y \in \Upsilon^m_\delta$
and $\Phi_k^m(Y,u) = 0$.
Arguing as in Prop.~\ref{lemma:howSmooth} and applying
Lemma~\ref{lemma:smoothness}, $\Phi_k^m$ is of class $C^{m-k+1}$ whenever
$m \ge k \ge 2$, and its derivative at any point of the form 
$(0,u)$ is simply
$$
d\Phi_k^m(0,u)(Y,\eta) = \dbar\eta + Y(u) \p_t u.
$$
Since $\dbar$ is surjective and has a bounded right inverse, the same
is always true of $d\Phi_k^m(0,u)$, and we conclude that any sufficiently
small neighborhood of $(i,u)$ in $\Mod^{k,p,m}$ is identified with a
$C^{m-k+1}$-smooth Banach submanifold of $X^{k,p,m}$.  Moreover, the natural
projection
$$
\pi : \Mod^{k,p,m} \to \jJ^m(B^{2n}) : (J,u) \mapsto J
$$
is differentiable, and we claim that its derivative at $(i,u)$ is also
surjective, with a bounded right inverse.  Indeed, identifying $(i,u)$
with $(0,u) \in (\Phi^m_k)^{-1}(0)$, this map takes the form
$$
d\pi(0,u)(Y,\eta) = Y,
$$
where $(Y,\eta) \in \ker d\Phi_k^m(0,u)$ and thus satisfies the
equation $\dbar \eta + Y(u) \p_t u = 0$.  Thus if $\widehat{T} : 
W^{k-1,p} \to W^{k,p}$
denotes a bounded right inverse of $\dbar$, then a bounded right inverse
of $d\pi(0,u)$ is given by the map
$$
\Upsilon^m \to \ker d\Phi_k^m(0,u) : Y \mapsto 
\left( Y , - \widehat{T}\left[ Y(u) \p_t u \right] \right).
$$
With all of this in place, one can easily use an inversion trick 
as in the proof of Theorem~\ref{thm:localExistence3} to show the following:

\begin{thm}
Suppose $u : B \to B^{2n}$ is holomorphic, i.e.~it is $i$-holomorphic
for the standard complex structure~$i$.  Then for any $p \in (2,\infty)$
and integers $m \ge k \ge 2$,
there exists a neighborhood $\uU^m_k \subset \jJ^m(B^{2n})$ of~$i$ and a
$C^{m-k+1}$-smooth map
$$
\Psi : \uU^m_k \to W^{k,p}(B,B^{2n})
$$
such that $\Psi(0) = u$ and $\Psi(J)$ is $J$-holomorphic for each
$J \in \uU^m_k$.
\end{thm}

We leave the proof as an exercise.
The following simple consequence for Riemann surfaces will come
in useful when we study compactness issues.

\begin{cor}
\label{cor:jvarying}
Suppose $j_k$ is a sequence of complex structures on a surface
$\Sigma$ that converge
in~$C^\infty$ to some complex structure~$j$, and $\varphi : (B,i)
\hookrightarrow (\Sigma,j)$ is a holomorphic embedding.  Then for
sufficiently large~$k$, there exists a sequence of holomorphic
embeddings
$$
\varphi_k : (B,i) \hookrightarrow (\Sigma,j_k)
$$
that converge in $C^\infty$ to~$\varphi$.
\end{cor}

\section{A representation formula for intersections}
\label{sec:intersections}

The main goal of this section is to prove the important fact that intersections
between distinct $J$-holomorphic curves are isolated unless the curves have
(locally) identical images.  We saw a special case of this in
\S\ref{sec:intersectionsHigher}: if $u$ and $v$ are two $J$-holomorphic
curves in an almost complex $4$-manifold that intersect at a point where
$v$ is immersed, then Theorem~\ref{thm:babyIntersections} implies that
the intersection is isolated unless $u$ maps a neighborhood
of the intersection into the image of~$v$.  It is easy to adapt the proof
of Theorem~\ref{thm:babyIntersections} and see that this fact is also true
in arbitrary dimensions, but it is much harder to understand what happens
if $u$ and $v$ both have a critical point where they intersect.
For this we will need a more
precise description of the behavior of a $J$-holomorphic curve near a 
critical point.

As a first step, it's
important to understand that $J$-holomorphic curves have well-defined
tangent spaces at every point, even the critical points.
Unless otherwise noted, throughout this section, $J$ will denote a smooth
almost complex structure on $\CC^n$ with $J(0)=i$.

\begin{prop}
\label{prop:tangent}
If $u : B \to \CC^n$ is a nonconstant $J$-holomorphic curve with $u(0) = 0$,
then there is a unique complex $1$-dimensional subspace $T_u \subset \CC^n$
and a number $k \in \NN$ such that for every $z \in B\setminus\{0\}$,
the limit
$$
\lim_{\epsilon \to 0^+} \frac{u(\epsilon z)}{\epsilon^k}
$$
exists and is a nonzero vector in $T_u$.
\end{prop}
\begin{proof}
Since $J$ is smooth, the regularity results of \S\ref{sec:nonlinearReg} imply
that $u$ is smooth, thus so is the family of complex structures defined by
$\bar{J}(z) = J(u(z))$ for $z \in B$.  Now $u$ satisfies the complex-linear 
Cauchy-Riemann type equation
$$
\p_s u + \bar{J}(z) \p_t u = 0,
$$
so by the similarity principle (see Exercise~\ref{EX:Jvaries} and 
Remark~\ref{remark:smoothness}), for sufficiently small $\delta > 0$
there is a smooth map $\Phi : B_\delta \to \End_\RR(\CC^n)$ with
$\Phi(0) = \1$, and a holomorphic map $f : B_\delta \to \CC^n$ such that
$$
u(z) = \Phi(z) f(z).
$$
By assumption $u$ is not constant, thus $f$ is not
identically zero and takes the form $f(z) = z^k g(z)$ for some $k \in \NN$
and holomorphic map $g : B_\delta \to \CC^n$ with $g(0) \ne 0$.
Then for $z \in B\setminus\{0\}$ and small $\epsilon > 0$,
$$
\frac{u(\epsilon z)}{\epsilon^k} =
\frac{\Phi(\epsilon z) \epsilon^k z^k g(\epsilon z)}{\epsilon^k} \to
z^k g(0) \in \CC g(0)
$$
as $\epsilon \to 0$.  It follows that the limit of 
$u(\epsilon z) / \epsilon^\ell$ is either zero or
infinity for all other positive integers $\ell \ne k$.
\end{proof}

\begin{defn}
\label{defn:tangent}
We will refer to the complex line $T_u \subset \CC^n$ in
Prop.~\ref{prop:tangent} as the \defin{tangent
space} to $u$ at~$0$, and its \defin{critical order} is the integer $k-1$.
\end{defn}

Here is the easiest case of the result that intersections of two
different $J$-holomorphic curves must be isolated.

\begin{exercise}
\label{EX:differentTangents}
Show that if $u , v : B \to \CC^n$ are two nonconstant
$J$-holomorphic curves with $u(0) = v(0) = 0$ but distinct tangent
spaces $T_u \ne T_v$ at~$0$, then for sufficiently small $\epsilon > 0$,
$u(B_\epsilon\setminus\{0\}) \cap v(B_\epsilon \setminus \{0\}) = \emptyset$.
\textsl{Hint: Compose $u$ and $v$ with the natural projection
$\CC^n \setminus\{0\} \to \CC P^{n-1}$.}
\end{exercise}

To understand the case of an intersection with common tangency $T_u = T_v$,
we will use the following local representation formula, which contains most
of the hard work in this discussion.

\begin{thm}
\label{thm:representation}
For any nonconstant $J$-holomorphic curve $u : B \to \CC^n$ with $u(0) = 0$,
there exist smooth coordinate changes on both the domain and target,
fixing the origin in both, so that in a neighborhood of~$0$, $u$ is
transformed into a pseudoholomorphic map $u : (B_\epsilon,\hat{\jmath}) \to
(\CC^n,\hat{J})$, where $\hat{\jmath}$ and $\hat{J}$ are smooth almost complex
structures on $B_\epsilon$ and $\CC^n$ respectively with $\hat{\jmath}(0) = i$
and $\hat{J}(0) = i$, and $u$ satisfies the formula
$$
u(z) = (z^k,\hat{u}(z)) \in \CC \times \CC^{n-1},
$$
where $k \in \NN$ is one plus the critical order of~$u$ at~$0$, and
$\hat{u} : B_\epsilon \to \CC^{n-1}$ is a smooth map whose first~$k$
derivatives at~$0$ all vanish.  In fact, $\hat{u}$ is either 
identically zero or satisfies the formula
$$
\hat{u}(z) = z^{k+\ell_u} C_u + |z|^{k+\ell_u} r_u(z)
$$
for some constants $C_u \in \CC^{n-1} \setminus \{0\}$, 
$\ell_u \in \NN$, and a function $r_u(z) \in \CC^{n-1}$
which decays to zero as $z \to 0$.  

Moreover, if
$v : B \to \CC^n$ is another nonconstant $J$-holomorphic curve with
$v(0) = 0$ and the same tangent space and critical order as $u$ at~$0$,
then the coordinates above can be chosen on $\CC^n$ so that $v$ 
(after a coordinate change on its domain) satisfies a similar representation
formula $v(z) = (z^k,\hat{v}(z))$, with either $\hat{v} \equiv 0$ or
$\hat{v}(z) = z^{k + \ell_v} C_v + |z|^{k+ \ell_v} r_v(z)$, and any two
pseudoholomorphic curves $u$ and $v$ 
written in this way are related to each other as follows:
either $\hat{u} \equiv \hat{v}$, or
$$
\hat{u}(z) - \hat{v}(z) = z^{k+\ell'} C' + |z|^{k+\ell'} r'(z),
$$
for some constants $C' \in \CC^{n-1} \setminus\{0\}$, $\ell' \in \NN$ 
and function $r'(z) \in \CC^{n-1}$ with $\lim_{z \to 0} r'(z) = 0$.
\end{thm}

\begin{exercise}
Prove Theorem~\ref{thm:representation} for the case where $J$ is integrable.
In this situation one can arrange for the coordinate changes on the domains
to be holomorphic, so $\hat{\jmath} \equiv i$.
\end{exercise}

Theorem~\ref{thm:representation} is a weak version of a deeper
result proved by Micallef and White \cite{MicallefWhite},\footnote{Our
exposition of this topic is heavily influenced by the asymptotic version
of Theorem~\ref{thm:representation}, which is a more recent result due to
R.~Siefring \cite{Siefring:asymptotics} that extends the intersection
theory of closed $J$-holomorphic curves to the punctured case.  We'll
discuss this in a later chapter.}
which provides
a similar formula in which the map $\hat{u}$ can be taken to be a
\emph{polynomial} in~$z$.  That result is
harder to prove, but it's also more than is needed for our purposes, as
the theorem above will suffice to understand everything we want to know about
intersections of holomorphic curves.  Before turning to
the proof, let us discuss some of its local applications---more such 
applications will be discussed in \S\ref{sec:multiple} and~\ref{sec:positivity}.

\begin{thm}
\label{thm:intersections}
Suppose $u , v : B \to \CC^n$ are injective smooth $J$-holomorphic curves
with $u(0) = v(0) = 0$.  Then for sufficiently small $\epsilon > 0$,
either $u = v \circ \varphi$ on $B_\epsilon$ for some holomorphic
embedding $\varphi : B_\epsilon \to B$ with $\varphi(0)=0$, or 
$$
u(B_\epsilon \setminus \{0\}) \cap v(B_\epsilon \setminus \{0\}) = \emptyset.
$$ 
\end{thm}
\begin{proof}
By Exercise~\ref{EX:differentTangents}, the second alternative holds unless
$T_u = T_v$, so assume the latter, and let $k_u , k_v$ denote the critical
orders of $u$ and~$v$ respectively, plus one.  Suppose $k_u m_u = k_v m_v = q$,
where $q \in \NN$ is the least common multiple of $k_u$ and $k_v$, hence
$m_u$ and $m_v$ are relatively prime.  Then the two curves
$$
u_0(z) := u(z^{m_u}),\qquad v_0(z) := v(z^{m_v})
$$
have the same tangent spaces and critical orders at~$0$.  We can thus use
Theorem~\ref{thm:representation} to change coordinates and rewrite these
two curves as
$$
u_0(z) = (z^q,\hat{u}_0(z)),\qquad v_0(z) = (z^q,\hat{v}_0(z)).
$$
For each $\ell = 1,\ldots,q-1$, define also the reparametrizations
$$
u_\ell(z) = (z^q,\hat{u}_\ell(z)) = u_0(e^{2\pi i \ell / q} z),
\qquad 
v_\ell(z) = (z^q,\hat{v}_\ell(z)) = v_0(e^{2\pi i \ell / q} z).
$$
Each of the differences $\hat{u}_0 - \hat{v}_\ell$ for $\ell=0,\ldots,q-1$ is
either identically zero or satisfies a formula of the form
$\hat{u}_0(z) - \hat{v}_\ell(z) = z^m C + |z|^m r(z)$,
in which case it has no zeroes in some neighborhood of~$0$.  If the
latter is true for all $\ell =0,\ldots,q-1$, then $u_0$ has no intersections
with $v_0$ near~$0$, as these correspond to
pairs $z \in B_\epsilon$ and $\ell \in \{0,\ldots,q-1\}$ for which 
$\hat{u}_0(z) = \hat{v}_\ell(z)$.  It follows then that $u$ and $v$ have no
intersections in a neighorhood of $u(0) = v(0) = 0$.

Suppose now that $\hat{u}_0 - \hat{v}_\ell \equiv 0$ for some 
$\ell \in \{0,\ldots,q-1\}$, which means
\begin{equation}
\label{eqn:twoCovers}
u(z^{m_u}) = u_0(z) = v_0(e^{2\pi i \ell / q} z) = v(e^{2\pi i \ell / k_v} z^{m_v})
\end{equation}
for all $z \in B_\epsilon$.  We finish by proving the following claim:
$m_u = m_v = 1$.  Indeed, replacing $z$ with
$e^{2\pi i / m_u } z$ in \eqref{eqn:twoCovers}, the left hand side doesn't
change, so we deduce that for all $z \in B_\epsilon$, 
$$
v(z^{m_v}) = v(e^{2\pi i m_v / m_u} z^{m_v}).
$$
Since $v$ is injective by assumption, this implies $m_v / m_u \in \ZZ$,
yet $m_u$ and $m_v$ are also relatively prime, so this can only be true
if $m_u = 1$.  Now performing the same argument again but inserting
$e^{2\pi i / m_v } z$ into \eqref{eqn:twoCovers}, we similarly deduce that
$m_v = 1$.
\end{proof}

The assumption of injectivity in the above theorem may seem like a 
serious restriction, but it is not: it turns out that on a sufficiently
small neighborhood of each point in the domain, every nontrivial
$J$-holomorphic
curve is either injective or is a branched cover of an injective curve.

\begin{thm}
\label{thm:injective}
For any nonconstant smooth $J$-holomorphic curve $u : B \to \CC^n$ with
$u(0) = 0$, there exists an injective
$J$-holomorphic curve $v : B \to \CC^n$ 
and a holomorphic map $\varphi : B_\epsilon \to B$ for some $\epsilon > 0$,
with $\varphi(0) = 0$, such that $u = v \circ \varphi$ on~$B_\epsilon$.
\end{thm}

Observe that if $\varphi'(0) \ne 0$ in the above statement then
$u$ must also be injective near~$0$; the interesting case is therefore when
$\varphi'(0) = 0$, as then $\varphi$ is locally a branched cover, mapping
a neighborhood of the origin $k$-to-$1$ to another neighborhood of the
origin for some $k \in \NN$.  It follows that $u : B_\epsilon \to \CC^n$
is then also a $k$-fold branched cover onto the image of~$v$ near~$0$.

\begin{proof}[Proof of Theorem~\ref{thm:injective}]
Using the coordinates provided by Theorem~\ref{thm:representation}, rewrite
$u$ as a pseudoholomorphic map $(B_\epsilon,j) \to (\CC^n,J)$ with
$u(z) = (z^q,\hat{u}(z))$, and define for each $\ell = 0,\ldots,q-1$,
$$
u_\ell : (B_\epsilon,j_\ell) \to (\CC^n,J) : z \mapsto 
(z^q,\hat{u}_\ell(z)) := u(e^{2 \pi i \ell / q} z).
$$
Then for $z \in B_\epsilon$, there is another point $\zeta \ne z$ with
$u(\zeta) = u(z)$ if and only if $\hat{u}(z) = \hat{u}_\ell(z)$ for some
$\ell \in \{1,\ldots,q-1\}$.  Making $\epsilon$ sufficiently small, the
representation formula for $\hat{u} - \hat{u}_\ell$ implies that
such points do not exist unless $\hat{u} \equiv \hat{u}_\ell$, so define
$$
m = \min \{ \ell \in \{1,\ldots,q \} \ |\ \hat{u} \equiv \hat{u}_\ell \}.
$$
Since $\hat{u} \equiv \hat{u}_m$ implies $\hat{u} \equiv \hat{u}_{\ell m}$
for all $\ell \in \NN$, $m$ must divide~$q$, thus we can define a
positive integer $k = q / m$.  If
$k = 1$ then $u$ is injective near~$0$ and we are done.  Otherwise,
$u$ now satisfies $u = u \circ \psi_\ell$ for all $\ell \in \ZZ_k$,
where we define the diffeomorphisms
$$
\psi_\ell : B_\epsilon \to B_\epsilon : z \mapsto e^{2\pi i \ell / k} z.
$$
This makes it possible to define a continuous map
$$
v : B_{\epsilon^k} \to \CC^n : z \mapsto u\left( \sqrt[k]{z}\right),
$$
which is injective if $\epsilon > 0$ is taken sufficiently small.

In order to view $v$ as a $J$-holomorphic curve,
we shall switch coordinates on the domain so that~$j$ becomes standard.
Observe that since $u = u \circ \psi_\ell$, pulling~$J$
back to $\dot{B}_\epsilon := B_\epsilon \setminus {0}$ through~$u$ implies 
$j = u^*J = \psi_\ell^*j$ on $\dot{B}_\epsilon$ for all $\ell \in \ZZ_k$, 
hence this holds also on 
$B_\epsilon$ by continuity.  The maps $\psi_\ell$ therefore define
a cyclic subgroup of the group of automorphisms of the Riemann surface
$(B_\epsilon,j)$.  Find a simply connected
$\ZZ_k$-invariant open neighborhood $\uU \subset B_\epsilon$ of~$0$ which
admits a holomorphic coordinate chart $\Phi : (\uU,j) \hookrightarrow (\CC,i)$.
By the Riemann mapping theorem, we can assume without loss of
generality that the image of this chart is $B$ and $\Phi(0) = 0$, 
hence the inverse $\Psi := \Phi^{-1}$ defines a holomorphic embedding
$$
\Psi : (B,i) \to (B_\epsilon,j)
$$
that maps the origin to itself and has a $\ZZ_k$-invariant image.  The maps
$$
\tilde{\psi}_\ell := \Psi^{-1} \circ \psi_\ell \circ \Psi : (B,i) \to
(B,i)
$$
for $\ell \in \ZZ_k$
now define an injective homomorphism of $\ZZ_k$ into the group of 
automorphisms of $(B,i)$ that fix~$0$.  The latter consists of rotations,
so we deduce $\tilde{\psi}_\ell(z) = e^{2\pi i \ell / k} z$.  Then the
$J$-holomorphic curve $\tilde{u} := u \circ \Psi : B \to \CC^n$ admits
the symmetry $\tilde{u} = \tilde{u} \circ \tilde{\psi}_\ell$ for all
$\ell \in \ZZ_k$, and we can thus define a new $J$-holomorphic curve on the
punctured ball $\dot{B} := B \setminus \{0\}$ by
$$
\tilde{v} : \dot{B} \to \CC^n : z \mapsto \tilde{u}(\sqrt[k]{z}).
$$
This admits a continuous extension over $B$ with $\tilde{v}(0) = 0$,
thus for all $z \in B_\epsilon$ in a sufficiently small neighborhood
of~$0$, $u$ now factors through a $k$-fold branched cover, namely
$$
u(z) = \tilde{v}\left( \left[ \Phi(z) \right]^k \right).
$$
Moreover, $\tilde{v}$ is injective, which we can see by identifying it
with the injective map $v : B_{\epsilon^k} \to \CC^n$ as follows:
consider the continuous map
$$
f : B \to B_{\epsilon^k} : z \mapsto \left[ \Psi( \sqrt[k]{z}) \right]^k,
$$
which is well defined because $\Psi(e^{2\pi i / k} z) = e^{2\pi i / k} \Psi(z)$.
This is a homeomorphism and satisfies $\tilde{v} = v \circ f$, thus
$\tilde{v}$ is injective if and only if $v$~is.

It remains only to show that the continuous map $\tilde{v} : B \to \CC^n$
is in fact smooth and thus $J$-holomorphic at~$0$.  By elliptic regularity
(Theorem~\ref{thm:regularity}),
it suffices to prove that $\tilde{v} \in W^{1,p}(B,\CC^n)$ for some $p > 2$,
i.e.~that it has a weak derivative of class $L^p$ which is defined almost
everywhere and equals the smooth map $d\tilde{v}$ on $\dot{B}$.  
Recall that $u(z) = (z^q,\hat{u}(z))$ with $q = km$, where $\hat{u}(z)
= o(|z|^q)$, thus the first $q-1$ derivatives of $u$ vanish at $z=0$,
and the same is therefore true for $\tilde{u} = u \circ \Psi$.  It follows
that there is a constant $C > 0$ such that
$$
| d\tilde{u}(z) | \le C |z|^{q - 1}
$$
for all $z \in B$, implying that for $z \in \dot{B}$,
$$
| d\tilde{v}(z) | \le \left| d\tilde{u}\left( \sqrt[k]{z} \right) \right|
\cdot \frac{1}{k} | z |^{\frac{1}{k} - 1} \le
\frac{C}{k} | z |^{\frac{1}{k}(q - 1)} |z|^{\frac{1}{k} - 1}
= \frac{C}{k} |z|^{m - 1}.
$$
Thus $d\tilde{v}$ is $C^0$-bounded on $\dot{B}$, implying it has a finite
$L^p$-norm for any $p > 2$, so the rest follows by Exercise~\ref{EX:weakly}
below.
\end{proof}

\begin{exercise}
\label{EX:weakly}
Assume $u$ is any continuous function on~$B$ which is smooth on
$\dot{B} = B \setminus \{0\}$, and its derivative $du$ on $\dot{B}$ satisfies
$\| du \|_{L^p(\dot{B})} < \infty$.  Show that $u \in W^{1,p}(B)$, and its
weak derivative equals its strong derivative almost everywhere.
\end{exercise}

We now turn to the proof of the representation formula,
Theorem~\ref{thm:representation}.  A somewhat simplified characterization
of the argument would be as follows: we need to show that for any
nonconstant $J$-holomorphic curve $u : B \to \CC^n$, assuming $J(0) = i$,
the ``leading order'' terms in its Taylor expansion about $z=0$ are
holomorphic.  Since terms in the Taylor series can always be expressed as
constant multiples of $z^k \bar{z}^\ell$, holomorphicity means the relevant
terms are actually multiples of~$z^k$, thus producing the powers of~$z$
that appear in the representation formula.  In practice, things are a bit
more complicated than this, e.g.~to keep full control over the remainders,
we will at one point use the similarity principle instead of Taylor's
theorem, but the above can be seen as a motivating principle.

\begin{proof}[Proof of Theorem~\ref{thm:representation}]
We proceed in four steps.

\textsl{Step~1: Coordinates on the target.}
Choose the coordinates on $\CC^n$ so that $J(0) = i$
and $T_u = \CC \times \{0\} \subset \CC^n$.  We can make one more 
requirement on the coordinates without loss of generality: we choose them
so that the map
$$
u_0(z) = (z,0) \in \CC \times \CC^{n-1}
$$
is $J$-holomorphic on $B_\epsilon$ for sufficiently small $\epsilon > 0$.
This is a highly nontrivial condition: the fact that it's possible follows
from the local existence result for $J$-holomorphic curves with
a fixed tangent vector, Theorem~\ref{thm:localExistence3}.

\textsl{Step~2: Coordinates on the domain.}
We next seek a coordinate
change near the origin on the domain so that $u$ becomes a map of the form
$z \mapsto (z^k,o(|z|^k))$ for some $k \in \NN$.  Applying the similarity
principle as in the proof of Prop.~\ref{prop:tangent}, we have
$u(z) = \Phi(z) f(z)$ on $B_\epsilon$ for some small $\epsilon > 0$,
a smooth map $\Phi : B_\epsilon \to \End_\RR(\CC^n)$ with $\Phi(0) = \1$
and a holomorphic map $f : B_\epsilon \to \CC^n$.  Moreover, $f(z) = z^k g(z)$
for some $k \in \NN$ (where $k-1$ is the critical order of~$u$) 
and a holomorphic map $g : B_\epsilon \to \CC^n$ with 
$g(0) \ne 0$, and our assumption on $T_u$ implies that after a complex-linear
coordinate change on the domain, we may assume $g(0) = (1,0) \in \CC \times
\CC^{n-1}$.  Thus $f(z) = (z^k g_1(z) , z^{k+1} g_2(z))$ for some holomorphic
maps $g_1 : B_\epsilon \to \CC$ and $g_2 : B_\epsilon \to \CC^{n-1}$, with
$g_1(0) = 1$.  Let us use the splitting $\CC^n = \CC \times \CC^{n-1}$ to 
write $\Phi(z)$ in block form as
$$
\Phi(z) =
\begin{pmatrix}
\alpha(z) & \beta(z) \\
\gamma(z) & \delta(z)
\end{pmatrix},
$$
so $\alpha(0)$ and $\delta(0)$ are both the identity, while $\beta(0)$ and
$\gamma(0)$ both vanish; note that all four blocks are regarded as 
\emph{real}-linear maps on complex vector spaces, i.e.~they need not commute with
multiplication by~$i$.  Now $u(z)$ takes the form 
$(u_1(z),u_2(z)) \in \CC \times \CC^{n-1}$, where
\begin{equation*}
\begin{split}
u_1(z) &= \alpha(z) z^k g_1(z) + \beta(z) z^{k+1} g_2(z),\\
u_2(z) &= \gamma(z) z^k g_1(z) + \delta(z) z^{k+1} g_2(z)).
\end{split}
\end{equation*}
We claim that after shrinking $\epsilon > 0$ further if necessary, there
exists a smooth function $\zeta : B_\epsilon \to \CC$ such that
$\zeta(0) = 0$, $d\zeta(0) = \1$ and $[\zeta(z)]^k = u_1(z)$.  Indeed,
the desired function can be written as
$$
\zeta(z) = z \sqrt[k]{\alpha(z) g_1(z) + \beta(z) z g_2(z)},
$$
which can be
defined as a smooth function for $z$ near~$0$ since the expression under
the root lies in a neighborhood of~$1$; we set $\sqrt[k]{1} = 1$.
Expressing $u$ now as a function of the new coordinate $\zeta$, we have
\begin{equation}
\label{eqn:smoothVersion}
u(\zeta) = (\zeta^k , \hat{u}(\zeta))
\end{equation}
with $\hat{u}(\zeta) = A(\zeta) \zeta^k$
for some smooth map $A(\zeta) \in \Hom_\RR(\CC,\CC^{n-1})$ with $A(0) = 0$.
Observe that since $d\zeta(0) = \1$, the new expression for $u(\zeta)$ is
pseudoholomorphic for a new complex structure $\hat{\jmath}$ on the domain 
such that $\hat{\jmath}(0) = i$.

\textsl{Step~3: The leading order term in $\hat{u} - \hat{v}$.}
This is the important part.  Using the coordinates chosen above, 
assume now that $J(0) = i$ and the two maps 
$u : (B_\epsilon,j) \to (\CC^n,J)$ and $v : (B_\epsilon,j')
\to (\CC^n,J)$ are pseudoholomorphic curves of the form
\begin{equation*}
\begin{split}
u(z) &= (z^k,\hat{u}(z)), \\
v(z) &= (z^k,\hat{v}(z)),
\end{split}
\end{equation*}
where $\hat{u}$ and $\hat{v}$ each have vanishing derivatives up to at least
order~$k$ at $z=0$.  Let 
$$
h(z) = u(z) - v(z) = (0,\hat{h}(z)),
$$
defining a map $\hat{h} : B_\epsilon \to \CC^{n-1}$.  Our main goal is to show
that the leading order term in $\hat{h}$ is a homogeneous holomorphic 
polynomial.  By unique continuation
(Theorem~\ref{thm:uniqueContin}), $h$ vanishes identically on a neighborhood
of~$0$ if and only if the derivatives $D^\ell h(0)$ of all orders vanish,
so let's assume this is not the case.  Then there is a finite positive
integer~$m$ defined by
$$
m = \min \{ \ell \in \NN \ |\ D^\ell h(0) \ne 0 \},
$$
and $m \ge k+1$ since $h(z) = o(|z|^k)$.  Now for $\epsilon > 0$,
the functions
\begin{equation*}
\begin{split}
h_\epsilon(z) := \frac{h(\epsilon z)}{\epsilon^m}
\end{split}
\end{equation*}
converge in $C^\infty$ as $\epsilon \to 0$ to a nonzero homogenous polynomial 
in $z$ and $\bar{z}$ of degree~$m$, namely the $m$th order term in the Taylor
series of $h$ about~$0$.  We claim that this polynomial is holomorphic, 
which would imply that it has the form
$$
h_0(z) = (0, z^m C)
$$
for some constant $C \in \CC^{n-1}$.

The intuitive reason for this claim should be clear: $u$ and $v$ both satisfy
nonlinear Cauchy-Riemann equations that ``converge'' to the standard one
as $z \to 0$, so their difference in the rescaled limit should also satisfy
$\dbar h_0 = 0$.  One complication in making this argument precise is that
since we've reparametrized the domains by nonholomorphic diffeomorphisms,
$u$ and $v$ are each pseudoholomorphic for \emph{different complex structures}
$j$ and $j'$ on their domains, thus it is not so straightforward to find an
appropriate PDE satisfied by $u - v$.  Of course, since both maps are 
immersed except
at~$0$, the complex structures are uniquely determined by $j = u^*J$ and 
$j' = v^*J$ on $B_\epsilon \setminus\{0\}$, which suggests that there should
be a way to reexpress the two nonlinear Cauchy-Riemann equations without
explicit reference to $j$ and~$j'$.  And there is: we only need observe that
outside of~$0$, $u$ and $v$ parametrize immersed surfaces in $\CC^n$ whose
tangent spaces are \emph{complex}, i.e.~$J$-invariant.  

This can be expressed elegantly
in the language of \defin{bivectors}: recall that a bivector is an element
of the antisymmetric tensor product bundle $\Lambda^2 T \CC^2 \to \CC^2$, 
and thus consists of a linear combination of bilinear wedge
products of the form $X \wedge Y$ for
vectors $X, Y \in T_p \CC^n$, $p \in \CC^n$, where by definition
$X \wedge Y = - Y \wedge X$.  Such a product can be thought
of intuitively as representing the oriented linear subspace in $T_p \CC^n$ 
spanned by $X$ and $Y$, with its magnitude giving the signed area of the 
corresponding parallelogram.  Let $\Aut_\RR(E)$ denote the group of
invertible real-linear smooth bundle maps on any bundle $E$.  Then there
is a natural group homomorphism
$$
\Aut_\RR(T \CC^n) \to \Aut_\RR(\Lambda^2 T\CC^n) : A \mapsto \bar{A}
$$
defined by
$$
\bar{A}(X \wedge Y) = AX \wedge AY.
$$
In particular, $J^2 = -1$ then implies $\bar{J}^2 = \1$ as an operator on
$\Lambda^2 T\CC^n$.  Now, the action of $J$ fixes the oriented subspace spanned
by $X$ and $Y$ if and only if $JX \wedge JY = c (X \wedge Y)$ for some
$c > 0$, but from $\bar{J}^2 = \1$, we deduce that $c = 1$, so the correct
condition is $JX \wedge JY = X \wedge Y$.  We conclude from this discussion
that $u : B_\epsilon \to \CC^n$ and $v : B_\epsilon \to \CC^n$ satisfy the 
first order nonlinear PDEs,
\begin{equation}
\label{eqn:quadraticCR}
\begin{split}
\p_s u \wedge \p_t u - J(u) \p_s u \wedge J(u) \p_t u &= 0, \\
\p_s v \wedge \p_t v - J(v) \p_s v \wedge J(v) \p_t v &= 0. \\
\end{split}
\end{equation}
In order to deduce the consequence for $h_0$, observe first that by the usual
interpolation trick (cf.~the proof of Prop.~\ref{prop:uniqueContin}), on a 
sufficiently small ball $B_\epsilon$ there is a smooth map
$A : B_\epsilon \to \End_\RR(\CC^n,\End_\RR(\CC^n))$ such that
$$
J\left(u(z)) - J(v(z)\right) = A(z) \left[ u(z) - v(z) \right] = A(z) h(z).
$$
Thus subtracting the second equation of \eqref{eqn:quadraticCR} from the
first gives
\begin{multline*}
\p_s u \wedge \p_t h + \p_s h \wedge \p_t v - J(u) \p_s u \wedge J(u) \p_t h
- J(u) \p_s h \wedge J(v) \p_t v \\
- J(u) \p_s u \wedge (Ah) \p_t v - (A h) \p_s v \wedge J(v) \p_t v = 0.
\end{multline*}
Replacing $z$ by $\epsilon z$ and dividing the entire expression by
$\epsilon^{k+m-2}$ now yields
\begin{equation*}
\begin{split}
0 &= \frac{\p_s u(\epsilon z)}{\epsilon^{k-1}} \wedge 
\frac{\p_t h(\epsilon z)}{\epsilon^{m-1}}
 + \frac{\p_s h(\epsilon z)}{\epsilon^{m-1}} \wedge \frac{\p_t v(\epsilon z)}{\epsilon^{k-1}} \\
& \qquad - J(u(\epsilon z)) \frac{\p_s u(\epsilon z)}{\epsilon^{k-1}} \wedge 
J(u(\epsilon z)) \frac{\p_t h(\epsilon z)}{\epsilon^{m-1}}
 - J(u(\epsilon z)) \frac{\p_s h(\epsilon z)}{\epsilon^{m-1}} \wedge 
J(v(\epsilon z)) \frac{\p_t v(\epsilon z)}{\epsilon^{k-1}}  \\
& \qquad - \epsilon^k J(u(\epsilon z)) \frac{\p_s u(\epsilon z)}{\epsilon^{k-1}} 
\wedge \left[ A(\epsilon z) \frac{h(\epsilon z)}{\epsilon^m} \right] 
\frac{\p_t v(\epsilon z)}{\epsilon^{k-1}} \\
& \qquad - \epsilon^k \left[ A(\epsilon z) \frac{h(\epsilon z)}{\epsilon^m} \right] 
\frac{\p_s v(\epsilon z)}{\epsilon^{k-1}} \wedge 
J(v(\epsilon z)) \frac{\p_t v(\epsilon z)}{\epsilon^{k-1}}.
\end{split}
\end{equation*}
We claim that every term in this expression converges in $C^\infty$
as $\epsilon \to 0$.  Indeed, the terms involving $h$ are all either
$h_\epsilon(z)$ or one of its first derivatives, so these converge respectively
to $h_0 = (0,\hat{h}_0)$, $\p_s h_0 = (0,\p_s \hat{h}_0)$ and 
$\p_t h_0 = (0,\p_t \hat{h}_0)$.  Since $\p_s u$ has vanishing derivatives
at~$0$ up until order $k-1$, $\frac{\p_s u(\epsilon z)}{\epsilon^{k-1}}$
converges to the homogenous degree~$k-1$ Taylor polynomial of $\p_s u$ at~$0$,
which is precisely the first derivative of the leading order term in $u$, 
namely $(k z^{k-1},0)$.  Likewise, $\frac{\p_t u(\epsilon z)}{\epsilon^{k-1}}
\to (i k z^{k-1},0)$, and the same goes for the first derivatives of~$v$.
Finally $J(u(\epsilon z))$ and $J(v(\epsilon z))$ both converge to~$i$,
so after the dust settles, we're left with
\begin{multline*}
(k z^{k-1},0) \wedge (0,\p_t \hat{h}_0) + 
(0,\p_s \hat{h}_0) \wedge (i k z^{k-1},0) \\ 
- (i k z^{k-1},0) \wedge (0, i \p_t \hat{h}_0) + 
(0, i \p_s \hat{h}_0) \wedge (k z^{k-1},0) = 0,
\end{multline*}
or equivalently
$$
- (k z^{k-1},0) \wedge ( 0, i \dbar \hat{h}_0) = (i k z^{k-1},0) \wedge
(0 , \dbar \hat{h}_0).
$$
This equation means that for all $z \in B_\epsilon$, if $(k z^{k-1},0)$ 
and $(0, i \dbar h_0(z))$ are linearly independent vectors in $\CC^n$, then
the oriented real subspace 
they span is the same as its image under multiplication by~$i$, 
i.e.~it is complex.  But this is manifestly untrue unless one of the vectors
vanishes, so we conclude that for all $z \in B_\epsilon \setminus \{0\}$,
$\dbar h_0(z) = 0$, and $h_0$ is thus a
holomorphic polynomial on~$B_\epsilon$.

\textsl{Step~4: Conclusion.}  It remains only to assemble the information
gathered above.  Combining Step~3 with Taylor's theorem yields the expression
$$
\hat{u}(z) - \hat{v}(z) = z^m C + |z|^m r(z),
$$
where $C \in \CC^{n-1}$ is a constant, $m > k$ is an integer and $r(z)$ is a 
remainder function such that $\lim_{z \to 0} r(z) = 0$.
The corresponding formulas for $\hat{u}$ and $\hat{v}$ individually follow
from this, because we've chosen coordinates so that 
$z \mapsto u_0(z^k) = (z^k,0)$
is also a $J$-holomorphic curve.  The degree of the leading term in each is
then simply the degree of its lowest order nonvanishing derivative at $z=0$,
and the same applies to $\hat{u} - \hat{v}$.
\end{proof}

\section{Simple curves and multiple covers}
\label{sec:multiple}

We now prove an important global consequence of the local results from
the previous section.
Recall first that if $\Sigma$ and~$\Sigma'$ are two closed, oriented
and connected surfaces, then every continuous map
$$
\varphi : \Sigma \to \Sigma'
$$
has a \defin{mapping degree} $\deg(\varphi) \in \ZZ$, most easily
defined via the homological condition that $\deg(\varphi) = k$ if
$\varphi_*[\Sigma] = k[\Sigma']$.  Equivalently, $\deg(\varphi)$ can
be defined as a signed count of points in the preimage $\varphi^{-1}(\zeta)$
of a generic point $\zeta \in \Sigma'$, cf.~\cite{Milnor:differentiable}.

\begin{exercise}
\label{EX:posDegree}
Show that if $(\Sigma,j)$ and $(\Sigma',j')$ are two closed connected
Riemann surfaces with their natural orientations, then any holomorphic
map $\varphi : (\Sigma,j) \to (\Sigma',j')$ has $\deg(\varphi) \ge 0$.
Moreover,
\begin{itemize}
\item $\deg(\varphi) = 0$ if and only if $\varphi$ is constant,
\item $\deg(\varphi) = 1$ if and only if $\varphi$ is \defin{biholomorphic},
i.e.~a holomorphic diffeomorphism with holomorphic inverse, and
\item if $\deg(\varphi) = k \ge 2$, then
$\varphi$ is a \defin{branched cover},
meaning it has at most finitely many critical points and its restriction to the 
punctured surface $\Sigma\setminus\Crit(\varphi)$ is a $k$-fold
covering map, while in a neighborhood of each critical point it
admits coordinates in which $\varphi(z) = z^\ell$ for some $\ell \in 
\{2,\ldots,k\}$.
\end{itemize}
\end{exercise}

\begin{thm}
\label{thm:multipleCovers}
Suppose $(\Sigma,j)$ is a closed connected Riemann surface,
$(M,J)$ is a smooth almost complex manifold and
$u : (\Sigma,j) \to (M,J)$ is a nonconstant $J$-holomorphic curve.
Then there exists a factorization $u = v \circ \varphi$ where
\begin{itemize}
\item
$(\Sigma',j')$ is a closed connected Riemann surface and
$v : (\Sigma',j') \to (M,J)$ a $J$-holomorphic curve that is
embedded outside a finite set of critical points and self-intersections, and
\item
$\varphi : (\Sigma,j) \to (\Sigma',j')$ is a holomorphic map of
degree~$\deg(\varphi) \ge 1$.
\end{itemize}
Moreover, $v$ is unique up to biholomorphic reparametrization.
\end{thm}
\begin{proof}
Let $\Crit(u) = \{ z \in \Sigma\ |\ du(z) = 0 \}$ denote the
set of critical points, and define $\Delta \subset \Sigma$ to be the set of
all points $z \in \Sigma$ such that there exists $\zeta \in \Sigma$ and
neighborhoods $z \in \uU_z \subset \Sigma$, $\zeta \in \uU_\zeta \subset
\Sigma$ with $u(z) = u(\zeta)$ but
$$
u(\uU_z \setminus \{z\}) \cap u(\uU_\zeta \setminus \{\zeta\}) = \emptyset.
$$
By Theorems~\ref{thm:intersections} and~\ref{thm:injective}, 
both of these sets are discrete and thus finite, and the set
$$
\dot{\Sigma}' := u\left(\Sigma \setminus (\Crit(u) \cup \Delta)\right) \subset M
$$
is a smooth submanifold of~$M$ with $J$-invariant tangent spaces, and thus
inherits a natural complex structure $j'$ such that the inclusion
$(\dot{\Sigma}',j') \hookrightarrow (M,J)$ is pseudoholomorphic.  
We shall now construct
$(\Sigma',j')$ as a compactification of $(\dot{\Sigma}',j')$, so that
$\dot{\Sigma}'$ is obtained from $\Sigma'$ by removing finitely many points.
Let
$$
\widehat{\Delta} = (\Crit(u) \cup \Delta) / \sim
$$
where two points in $\Crit(u) \cup \Delta$ are defined to be
equivalent whenever they have neighborhoods in $\Sigma$ 
with identical images under~$u$.
Then for each $[z] \in \widehat{\Delta}$, Theorem~\ref{thm:injective} 
provides an injective 
$J$-holomorphic map $u_{[z]}$ from the open unit ball $B \subset \CC$ onto
the image of a neighborhood of~$z$ under~$u$.
We define $(\Sigma',j')$ by
$$
\Sigma' = \dot{\Sigma}' \cup_\Phi \left( \bigsqcup_{[z] \in \widehat{\Delta}} B \right),
$$
where the gluing map $\Phi$ is the disjoint union of the maps
$u_{[z]}|_{B\setminus\{0\}} : B\setminus \{0\} \to \dot{\Sigma}'$ for
each $[z] \in \widehat{\Delta}$, and $j = j'$ on $\dot{\Sigma}'$ and $i$ on~$B$.
The surface $\Sigma'$ is clearly compact, and
combining the maps $u_{[z]}$ with the inclusion $\dot{\Sigma}'
\hookrightarrow M$ defines a pseudoholomorphic map
$v : (\Sigma',j') \to (M,J)$ whose restriction to the punctured surface
$\dot{\Sigma'} = \Sigma' \setminus \widehat{\Delta}$ is an embedding.
Moreover, the restriction of $u$ to $\Sigma\setminus (\Crit(u) \cup \Delta)$
defines a holomorphic map to $(\dot{\Sigma}',j')$ which extends over the
punctures to a holomorphic map $\varphi : (\Sigma,j) \to (\Sigma',j')$
such that $u = v \circ \varphi$.

We leave the uniqueness statement as an exercise for the reader.
The positivity of $\deg(\varphi)$ follows from Exercise~\ref{EX:posDegree}.
\end{proof}

\begin{defn}
\label{defn:simple}
A closed, connected and nonconstant pseudoholomorphic curve 
$u : (\Sigma,j) \to (M,J)$ is
called \defin{simple} if it does not admit any factorization
$u = v \circ \varphi$ as in Theorem~\ref{thm:multipleCovers} with
$\deg(\varphi) > 1$.  If $u$ is not simple, we say that it is
\defin{multiply covered}.
\end{defn}

With this definition in hand, the theorem above can be reformulated as follows:

\begin{cor}
\label{cor:simple}
A closed, connected and nonconstant pseudoholomorphic curve is simple
if and only if it is embedded outside of a finite (possibly empty) set of
critical points and self-intersections.
\end{cor}

\section{Positivity of intersections}
\label{sec:positivity}

We saw in \S\ref{sec:intersectionsHigher} that a $J$-holomorphic curve
and a $J$-holomorphic hypersurface (i.e.~a $J$-invariant submanifold
of real codimension two) always intersect positively.  This fact is
especially powerful in dimension four, where a $J$-holomorphic
hypersurface is simply the image of an embedded $J$-holomorphic 
curve---but we would also like to understand what happens when
two holomorphic curves intersect at a point where neither is
locally embedded.
This is made possible by the representation formula of
\S\ref{sec:intersections}, and in this section we will use it to prove two
much more powerful local results about intersections of holomorphic
curves in dimension four.  Both play major roles in applications to
symplectic $4$-manifolds and contact $3$-manifolds that we will discuss
in later chapters.

Throughout this section, $J$ denotes a smooth almost complex
structure on $\CC^2$ with $J(0)=i$.  We shall also assume that
$J$ is tamed by the standard symplectic form $\omega\std$;
since $i$ is already $\omega\std$-tame and we will only really
be concerned with a neighborhood of the origin, this condition 
does not pose a restriction in practice.

\begin{thm}
\label{thm:positivity}
Suppose $u , v : B \to \CC^2$ are $J$-holomorphic curves
with an isolated intersection $u(0) = v(0) = 0$.  Then the
local intersection index satisfies
$$
\inter(u,0 ; v,0) \ge 1,
$$
with equality if and only if the intersection is transverse.
\end{thm}

Before proving the theorem, we would also like to formulate a similar 
result for singularities of a single
curve.  Recall that by Theorem~\ref{thm:injective}, every
nonconstant $J$-holomorphic curve is locally either injective
(perhaps with isolated critical points) or a branched cover
of an injective curve.  Since a nontrivial branched cover
necessarily has infinitely many self-intersections, we restrict
in the following statement to the locally injective case.
It will be most relevant in particular to curves that are
\emph{simple} in the sense of Definition~\ref{defn:simple}.

\begin{thm}
\label{thm:critical}
Suppose $u : B \to \CC^2$ is an injective $J$-holomorphic curve
with $u(0)=0$ and an isolated critical point $du(0)=0$.  Then there exists
an integer $\delta(u,0) > 0$, depending only on the germ of~$u$
near~$0$, such that for any $\rho > 0$,
one can find a smooth map $u_\epsilon : B \to \CC^2$
satisfying the following conditions:
\begin{enumerate}
\item $u_\epsilon$ is $C^\infty$-close to $u$ and matches $u$ outside
$B_\rho$ and at~$0$;
\item $u_\epsilon$ is a \emph{symplectic immersion} with respect to
the standard symplectic structure $\omega\std$, i.e.~it
satisfies $u_\epsilon^*\omega\std > 0$;
\item $u_\epsilon$ has finitely many self-intersections and satisfies
\begin{equation}
\label{eqn:deltaLocal}
\frac{1}{2} \sum_{(z,\zeta)} \inter(u_\epsilon,z ; u_\epsilon,\zeta) = \delta(u,0),
\end{equation}
where the sum ranges over all pairs $(z,\zeta) \in B \times B$
such that $z \ne \zeta$ and $u_\epsilon(z) = u_\epsilon(\zeta)$.\footnote{Notice
that each geometric double-point $u(z) = u(\zeta)$ appears
twice in the summation over pairs $(z,\zeta)$, hence the factor
of $1/2$ in \eqref{eqn:deltaLocal}.}
\end{enumerate}
\end{thm}

\begin{remark}
Our proof will show in fact that the tangent spaces spanned by the
perturbation $u_\epsilon$ can be arranged to be uniformly close 
to $i$-complex subspaces (or equivalently $J$-complex subspaces,
since $J$ and $i$ may also be assumed uniformly close in a small
enough neighborhood of~$0$).  This implies that it is a symplectic immersion,
since the condition of being a symplectic subspace
is open.  In practice, the crucial point in applications will be 
that the complex structure on the bundle $(u_\epsilon^*T\CC^2,J)$ 
admits a homotopy supported near~$0$ to
a new complex structure for which
$\im d u_\epsilon$ becomes a complex subbundle.
\end{remark}

As a prelude to the proofs of the two theorems above, the following
exercise should provide a concrete feeling for what is involved.

\begin{exercise}
\label{EX:computeInter}
Consider the intersecting holomorphic maps $u , v : \CC \to \CC^2$ defined by
$$
u(z) = (z^3, z^5), \qquad v(z) = (z^4, z^6).
$$
\begin{enumerate}
\renewcommand{\labelenumi}{(\alph{enumi})}
\item
Show that $u$ admits a $C^\infty$-small perturbation to a
map $u_\epsilon$ such that $u_\epsilon$ and $v$ have exactly~$18$
intersections in a neighbourhood of the origin, all transverse and positive.
\item 
Show that for any neighbourhood $\uU \subset \CC$ of~$0$,
$u$ admits a $C^\infty$-small perturbation to an
\emph{immersion} $u_\epsilon$ such that
$$
\frac{1}{2}\# \{ (z,\zeta) \in \uU \times \uU\ |\ u_\epsilon(z) =
u_\epsilon(\zeta), \ z \ne \zeta \} = 10.
$$
\end{enumerate}
\end{exercise}

We now prove Theorem~\ref{thm:positivity}.  Recall from
\S\ref{sec:intersections} that even if $u$ and $v$ have
critical points at~$0$, they both have well-defined
tangent spaces and critical orders.
We first prove the theorem in the case
where the tangent spaces at the intersection are distinct.

\begin{prop}
\label{prop:differentTangents}
Under the assumptions of Theorem~\ref{thm:positivity},
suppose $u$ and $v$ have distinct tangent spaces
$T_u \ne T_v \subset \CC^2$ at the intersection, with
critical orders $k_u - 1$ and $k_v - 1$ respectively.
Then
$$
\inter(u,0 ; v,0) = k_u k_v.
$$
In particular, the intersection index is positive, and equals~$1$
if and only if the intersection is transverse.
\end{prop}
\begin{proof}
By Theorem~\ref{thm:representation}, we can smoothly change coordinates 
such that
without loss of generality, $u(z) = \left(z^{k_u} , |z|^{k_u+1} f(z)\right)$ for
some bounded function $f : B \to \CC$.  The condition of distinct
tangent spaces implies (cf.~Exercise~\ref{EX:differentTangents})
that if $\pi : \CC^2 \setminus \{0\} \to \CC P^1$
denotes the natural projection, the images of the maps
$$
\pi \circ u|_{B_\epsilon \setminus\{0\}} , \pi \circ v|_{B_\epsilon \setminus \{0\}} :
B_\epsilon \setminus \{0\} \to \CC P^1
$$
lie in arbitrarily small neighborhoods of two distinct points for small
$\epsilon > 0$.  This remains true if we replace $u$ by any of the maps
$$
u_\tau : B \to \CC^n : z \mapsto \left(z^{k_u} , \tau |z|^{k_u+1} f(z)\right)
$$
for $\tau \in [0,1]$.  Thus by homotopy invariance of the local intersection
index (Exercise~\ref{EX:independent}), $\inter(u,0 ; v,0) = \inter(u_0,0;v,0)$.
After applying the same homotopy argument in different 
coordinates adapted to $v$ and then choosing new coordinates so
that the tangent spaces of $u$ and $v$ match $\CC \times \{0\}$
and $\{0\} \times \CC$ respectively, the problem is reduced to computing
$\inter(u_0,0 ; v_0,0)$, where
$$
u_0(z) = \left( z^{k_u},0 \right), \qquad
v_0(z) = \left( 0,z^{k_v} \right).
$$
Choose $\epsilon > 0$ and perturb these maps to 
$\left(z^{k_u} + \epsilon,0 \right)$ and
$\left( 0,z^{k_v} + \epsilon \right)$ respectively.  Both are then
holomorphic for the standard complex structure on $\CC^2$ and they have
exactly $k_u k_v$ intersections, all transverse.
\end{proof}

\begin{exercise}
\label{EX:branched}
Suppose $u , v : B \to \CC^2$ are $J$-holomorphic curves with an
isolated intersection $u(0) = v(0) = 0$, and for $k,\ell \in \NN$,
define the $J$-holomorphic branched covers
$u^k , v^\ell : B \to \CC^2$ by
$$
u^k(z) := u(z^k), \qquad v^\ell(z) := v(z^\ell).
$$
Show that $\inter(u^k,0 ; v^\ell,0) = k \ell \cdot \inter(u,0 ; v,0)$.
\end{exercise}

The remaining cases of Theorem~\ref{thm:positivity} are covered by
the following result, in which the intersection can never be
transverse.

\begin{prop}
\label{prop:sameTangent}
Under the assumptions of Theorem~\ref{thm:positivity},
suppose $u$ and $v$ have identical tangent spaces
$T_u = T_v \subset \CC^2$ at the intersection, with
critical orders $k_u - 1$ and $k_v - 1$ respectively.
Then
$$
\inter(u,0 ; v,0) \ge k_u k_v + 1.
$$
\end{prop}
\begin{proof}
Since $k_u$ and $k_v$ may be different, we first replace $u$ and $v$ with
suitable branched covers so that their critical orders become the same: let
$$
m = k_u k_v \in \NN,
$$
and define $u', v' : B \to \CC^2$ by
$$
u'(z) := u(z^{k_v}), \qquad v'(z) := v(z^{k_u}),
$$
so that in particular $u'$ and $v'$ both have critical order $m-1$
at the intersection $u'(0) = v'(0)=0$, as well as matching tangent spaces.
Now by Theorem~\ref{thm:representation}, we find new choices of local
coordinates in $B$ and $\CC^2$ near~$0$ such that
$$
u'(z) = (z^m,\hat{u}(z)), \qquad v'(z) = (z^m,\hat{v}(z))
$$
for $z \in B_\rho$, with $\rho > 0$ 
and some smooth functions $\hat{u} , \hat{v} : B_\rho \to \CC$ with
vanishing derivatives up to order~$m$ at~$0$.  For each $j=0,\ldots,m-1$, there
are also $J$-holomorphic disks (in general with different complex structures
on their domains) $v_j' : B_\rho \to \CC^2$ defined by
$$
v_j'(z) := v'( e^{2\pi i j/m} z) = (z^m, \hat{v}_j(z)),
\quad \text{ where }\quad \hat{v}_j(z) = \hat{v}( e^{2 \pi i j/m} z).
$$
If $\hat{u} - \hat{v}_j$ is identically zero for some $j=0,\ldots,m-1$,
then we have
$$
u'(z) = v'( e^{2\pi i j/m} z) \quad\text{ for all $z \in B_\rho$},
$$
implying that $u'$ and $v'$ have identical images
on some neighborhood of the intersection, in which case so do $u$ and~$v$;
this is impossible since the intersection was assumed isolated.
Now
Theorem~\ref{thm:representation} gives for each $j=0,\ldots,m-1$ the formula
\begin{equation}
\label{eqn:windinguv}
\hat{u}(z) - \hat{v}_j(z) = z^{m + \ell_j} C_j + |z|^{m + \ell_j} r_j(z),
\end{equation}
where $C_j \in \CC \setminus \{0\}$, $\ell_j \in \NN$ and 
$r_j(z) \in \CC$ is a function with $r_j(z) \to 0$ as $z \to 0$.  
We can now compute $\inter(u',0 ; v',0)$ by choosing
$\epsilon \in \CC \setminus \{0\}$ close to~$0$ and defining the perturbation
$$
u'_\epsilon(z) := (z^m , \hat{u}(z) + \epsilon).
$$
This curve does not intersect $v'$ at $z=0$ since $\epsilon \ne 0$.
If $u'_\epsilon(z) = v'(\zeta)$, then $z^m = \zeta^m$, hence
$\zeta = e^{2\pi i j/m} z$ for some $j=0,\ldots,m-1$, and equality in the
second factor then implies
\begin{equation}
\label{eqn:that}
\hat{v}_j(z) - \hat{u}(z) = \epsilon.
\end{equation}
By \eqref{eqn:windinguv}, the zero of $\hat{v}_j(z) - \hat{u}(z)$ at $z=0$
has order $m + \ell_j \ge m + 1$, thus if $\epsilon \ne 0$ is sufficiently close
to~$0$
and chosen generically so that it is a regular value of $\hat{v}_j - \hat{u}$,
we conclude that \eqref{eqn:that} has exactly $m + \ell_j$ solutions
near $z=0$, all of them simple (positive or negative) zeroes 
of $\hat{v}_j - \hat{u} - \epsilon$
and thus corresponding to transverse (positive or negative) intersections of
$u'$ with $v'$.  Adding these up with the correct signs 
for all choices of $j=0,\ldots,m-1$, we conclude
$$
\inter(u',0 ; v',0) = \sum_{j=0}^{m-1} (m + \ell_j) \ge
m (m+1) = k_u k_v (k_u k_v + 1).
$$
Exercise~\ref{EX:branched} then implies $\inter(u,0 ; v,0) \ge k_u k_v + 1$.
\end{proof}

\begin{exercise}
\label{EX:noUpperBound}
Find examples to show that in the situation described
in Proposition~\ref{prop:sameTangent},
the bound $\inter(u,0 ; v,0) \ge k_u k_v + 1$ is sharp, and
there is no similar \emph{upper} bound for $\inter(u,0 ; v,0)$ in terms
of $k_u$ and~$k_v$.  \textsl{Hint: Set $J \equiv i$ and consider
holomorphic maps of the form $z \mapsto (z^k,z^{k+\ell})$.}
\end{exercise}

The proof of Theorem~\ref{thm:critical} will be similar, but there are some
additional subtleties involved in proving that the immersed perturbation
$u_\epsilon$ is \emph{symplectically} immersed---intuitively this should
be unsurprising since $\omega\std$ tames $J$ and the symplectic
subspace condition is open, but the change in tangent subspaces
cannot be understood as a $C^0$-small perturbation due to the singularity
of $du$ at~$0$.  Our strategy will be to show that the tangent spaces
spanned by $du_\epsilon$ are in fact $C^0$-close to the tangent spaces
spanned by another map which is a \emph{holomorphic} immersion.
In order to make this notion precise, we need a practical way of measuring
the ``distance'' between two subspaces of a vector space, in particular for
the case when both subspaces arise as images of injective linear maps.

\begin{defn}
\label{defn:distSubspaces}
Fix the standard Euclidean norm on $\RR^n$.  Given two subspaces
$V, W \subset \RR^n$ of the same positive dimension, define
$$
\dist(V,W) := \max_{v \in V, |v|=1} \dist(v,W) :=
\max_{v \in V, |v|=1} \min_{w \in W} |v - w|.
$$
\end{defn}
\begin{defn}
\label{defn:Inj}
The \defin{injectivity modulus} of a linear map $A : \RR^k \to \RR^n$ is
$$
\Inj(A) = \min_{v \in \RR^k \setminus \{0\}} \frac{|Av|}{|v|} \ge 0.
$$
\end{defn}
Clearly $\Inj(A) > 0$ if and only if $A$ is injective.

\begin{lemma}
\label{lemma:Inj}
For any pair of injective linear maps $A , B : \RR^k \to \RR^n$,
$$
\dist\left(\im A, \im B \right) \le \frac{\| A - B \|}{\Inj(A)}.
$$
\end{lemma}
\begin{proof}
Pick any nonzero vector $v \in \RR^n$.  Then $Av \ne 0$ since $A$ is injective,
and we have
\begin{equation*}
\begin{split}
\dist\left(\frac{Av}{|Av|} , \im B \right) &= 
\min_{w \in \RR^k} \left| A \frac{v}{|Av|} - Bw \right| \le 
\left| A \frac{v}{|Av|} - B \frac{v}{|Av|} \right| \\
&\le \| A - B \| \frac{|v|}{|Av|} \le \frac{\| A - B \|}{\Inj(A)}.
\end{split}
\end{equation*}
\end{proof}

\begin{lemma}
\label{lemma:nearlyComplex}
There exists $\epsilon > 0$ such that if $V \subset \CC^2$ is a
complex $1$-dimensional subspace, then all real $2$-dimensional
subspaces $W \subset \CC^2$ satisfying $\dist(V,W) < \epsilon$ are
$\omega\std$-symplectic.
\end{lemma}
\begin{exercise}
\label{EX:nearlyComplex}
Prove the lemma.  \textsl{Hint: $\CC P^1$ is compact.}
\end{exercise}

\begin{proof}[Proof of Theorem~\ref{thm:critical}]
By Theorem~\ref{thm:representation}, we can assume after smooth coordinate
changes near $0 \in B$ and $0 \in \CC^2$ that
$$
u(z) = (z^k, \hat{u}(z))
$$
for some integer $k \ge 2$ and a map $\hat{u} : B_\rho \to \CC$ on a
ball of some radius $\rho > 0$, such that the other branches 
$$
u_j(z) := u(e^{2\pi i j/k}z) = (z^k , \hat{u}_j(z)), \qquad
\hat{u}_j(z) := \hat{u}(e^{2\pi i j/k}z),
$$
for $j=1,\ldots,k-1$ are related by
\begin{equation}
\label{eqn:branchy}
\hat{u}_j(z) - \hat{u}(z) = z^{k + \ell_j} C_j + |z|^{k+\ell_j} r_j(z)
\end{equation}
for some $\ell_j \in \NN$, $C_j \in \CC \setminus \{0\}$ and
$r_j : B_\rho \to \CC$ with $r_j(z) \to 0$ as $z \to 0$.  Here we've
used the assumption that $u$ is injective in order to conclude that
$\hat{u}_j - \hat{u}$ is not identically zero, and by shrinking $\rho > 0$
if necessary, we can also assume $u$ is embedded on $B_\rho \setminus \{0\}$.
Fix a smooth cutoff
function $\beta : B_\rho \to [0,1]$ that equals~$1$ on $B_{\rho/2}$ 
and has compact support.  Then for $\epsilon \in \CC$ sufficiently 
close to~$0$, consider the perturbation
$$
u_\epsilon(z) := (z^k, \hat{u}(z) + \epsilon \beta(z) z),
$$
which satisfies $u_\epsilon(0) = 0$ and
is immersed if $\epsilon \ne 0$.  Since $u$ is embedded on
$B_\rho \setminus B_{\rho/2}$, we may assume
for $|\epsilon|$ sufficiently small that $u_\epsilon$
has no self-intersections outside of the region where $\beta \equiv 1$.
Then a self-intersection $u_\epsilon(z) = u_\epsilon(\zeta)$ with 
$z \ne \zeta$ occurs wherever
$\zeta = e^{2 \pi ij/k}z \ne 0$ for some $j=1,\ldots,k-1$ and
$\hat{u}(z) + \epsilon z = \hat{u}_j(z) + \epsilon e^{2\pi i j/k}z$, which by
\eqref{eqn:branchy} means
$$
z^{k + \ell_j} C_j + |z|^{k+\ell_j} r_j(z) + \epsilon \left( e^{2\pi i j/k} - 1\right) z = 0.
$$
Assume $\epsilon \in \CC \setminus \{0\}$ is chosen generically so that the
zeroes of this function are all simple (see Exercise~\ref{EX:trickySard} below).
Then each zero other than the ``trivial'' solution at $z=0$ 
represents a transverse (positive or negative) self-intersection 
of $u_\epsilon$, and the algebraic count of these (discounting the trivial solution) 
for $|\epsilon|$ sufficiently
small is $k + \ell_j - 1 \ge k$.  Adding these up for all $j=1,\ldots,k-1$, we 
obtain 
\begin{equation}
\label{eqn:deltaLocal2}
\delta(u,0) :=
\frac{1}{2}\sum_{(z,\zeta)} \inter(u_\epsilon,z ; u_\epsilon,\zeta) = 
\frac{1}{2} \sum_{j=1}^{k-1} (k + \ell_j - 1) \ge \frac{1}{2} k(k-1),
\end{equation}
which is strictly positive since $k \ge 2$.

It remains to show that $u_\epsilon$ satisfies $u_\epsilon^*\omega\std > 0$,
which is equivalent to showing that $\im du_\epsilon(z) \subset \CC^2$ is an
$\omega\std$-symplectic subspace for all~$z$.
Let us write $\hat{u}$ in the form
$$
\hat{u}(z) = z^{k+\ell} C + |z|^{k+\ell} r(z)
$$
as guaranteed by Theorem~\ref{thm:representation}, where
$C \in \CC\setminus \{0\}$, $\ell \in \NN$ and $\lim_{z \to 0} r(z) = 0$.
We shall compare $u_\epsilon$ with the holomorphic map
$$
P_\epsilon : B_\rho \to \CC^2 : z \mapsto (z^k, z^{k+\ell} C
+ \epsilon z),
$$
obtained by dropping the remainder term from $\hat{u}$.  Note that
$P_\epsilon$ is simply the degree $k+\ell$ Taylor polynomial
of $u_\epsilon$; indeed, both have the same derivatives at~$0$ up to
order $k+\ell$.  Setting $\epsilon=0$ and differentiating both,
it follows that $dP_0 : B_\rho \to
\Hom_\RR(\CC,\CC^2)$ is the degree $k+\ell-1$ Taylor polynomial
of $du_0 : B_\rho \to \Hom_\RR(\CC,\CC^2)$, thus
$$
du_0(z) = dP_0(z) + |z|^{k+\ell-1} R(z)
$$
for some function $R(z)$ with $R(z) \to 0$ as $z \to 0$.  Reintroducing the
$\epsilon$-dependent linear term, it follows that
$$
du_\epsilon(z) = dP_\epsilon(z) + |z|^{k+\ell-1} R(z)
$$
for all $\epsilon \in \CC$, where the function $R(z)$ is independent 
of~$\epsilon$ and is bounded.
Now abbreviate $A_\epsilon(z) := dP_\epsilon(z)$ and
$B_\epsilon(z) := du_\epsilon(z)$.  The Taylor formula above then gives an
estimate of the form
$$
\| A_\epsilon(z) - B_\epsilon(z) \| \le c_1 |z|^{k+\ell-1}
$$
for some constant $c_1 > 0$ independent of~$\epsilon$.  
Computing $dP_\epsilon(0)$, we find similarly
a constant $c_2 > 0$ independent of~$\epsilon$ such that
$$
| A_\epsilon(z) v | \ge c_2 |z|^{k-1} |v| \quad\text{ for all $v \in \CC$},
$$
thus $\Inj(A_\epsilon(z)) \ge c_2 |z|^{k-1}$, and
$$
\frac{\| A_\epsilon(z) - B_\epsilon(z) \|}{\Inj(A_\epsilon(z))} \le c_3 |z|^\ell
$$
for some constant $c_3 > 0$ independent of~$\epsilon$.  
Now since $P_\epsilon$ is holomorphic
(for the standard complex structure) for all $\epsilon$, $\im A_\epsilon(z) \subset \CC^2$
is always complex linear, so the above estimates imply together with
Lemmas~\ref{lemma:Inj} and~\ref{lemma:nearlyComplex} that for a sufficiently small radius
$\rho_0 > 0$, the images of $du_\epsilon(z)$ for all $z \in B_{\rho_0} \setminus \{0\}$
and $\epsilon \in B_{\rho_0}$ are $\omega\std$-symplectic.  This is also true for
$z=0$ if $\epsilon \ne 0$, since then $du_\epsilon(0) = dP_\epsilon(0)$ is
complex linear.

To conclude, fix $\rho_0 > 0$ as above
and choose $\epsilon \in \CC\setminus\{0\}$ sufficiently close to~$0$ so
that outside of $B_{\rho_0}$, $u_\epsilon$ is $C^1$-close enough to $u$ for its
tangent spaces to be $\omega\std$-symplectic (recall that $J$ is also
$\omega\std$-tame).  The previous paragraph then
implies that the tangent spaces of $u_\epsilon$ are $\omega\std$-symplectic
everywhere.
\end{proof}

\begin{exercise}
\label{EX:deltaIndependent}
Verify that the formula obtained in \eqref{eqn:deltaLocal2} for
$\delta(u,0)$ does not depend on any choices.
\end{exercise}

\begin{exercise}
\label{EX:trickySard}
Assume $f : \uU \to \CC$ is a smooth map on a domain
$\uU \subset \CC$ containing~$0$, with $f(0)=0$ and $df(0) = 0$.
Show that for almost every $\epsilon \in \CC$, the map
$f_\epsilon : \uU \to \CC : z \mapsto f(z) + \epsilon z$ has~$0$ as a
regular value.  \textsl{Hint: Use the implicit function theorem to show
that the set}
$$
X := \{ (\epsilon,z) \in \CC \times (\uU \setminus \{0\})\ |\ 
f_\epsilon(z) = 0 \}
$$
\textsl{is a smooth submanifold of $\CC^2$, and a point $(\epsilon,z) \in X$ is
regular for the projection $\pi : X \to \CC : (\epsilon,z) \mapsto \epsilon$
if and only if $z$ is a regular point of~$f_\epsilon$.  Then apply Sard's
theorem to~$\pi$.}
\end{exercise}

\begin{exercise}
\label{EX:noUpperBound2}
The proof of Theorem~\ref{thm:critical} showed that if $u : B \to \CC^2$ is
$J$-holomorphic and injective with critical order $k - 1$ at~$0$, then
$2\delta(u,0) \ge k(k-1)$.  Find examples to show that this bound is
sharp, and that there is no similar upper bound
for $\delta(u,0)$ in terms of~$k$.  \textsl{(Compare Exercise~\ref{EX:noUpperBound}.)}
\end{exercise}


\chapter{Fredholm Theory}
\label{chapter:Fredholm}

\minitoc
\vspace{12pt}

\section{Some Banach spaces and manifolds}
\label{sec:Banach}

In this chapter we begin the study of $J$-holomorphic curves in
global settings.  We will fix the following data throughout: $(\Sigma,j)$ is a
closed connected Riemann surface, and $(M,J)$ is a $2n$-dimensional manifold
with a smooth almost complex structure.  Our goal will be
to understand the local structure of the space of solutions
to the nonlinear Cauchy-Riemann equation, that is,
\begin{equation}
\label{eqn:solutionSpace}
\{ u \in C^\infty(\Sigma,M)\ |\ Tu \circ j = J \circ Tu \}.
\end{equation}
We assign to this space the natural topology defined by $C^\infty$-convergence
of maps $\Sigma \to M$.
Recall that since $J$ is smooth, elliptic regularity implies that all 
solutions of at least
class $W^{k,p}_\loc$ for some $k \in \NN$ and $p > 2$ are actually smooth, and 
the $C^\infty$-topology is equivalent to the
$W^{k,p}$-topology on the solution space.
The main result of this chapter will be that under sufficiently 
fortunate circumstances, this space is a finite-dimensional manifold, and
we will compute its dimension in terms of the given topological data.
We will put off until the next chapter the question of when such
``fortunate circumstances'' are guaranteed to exist, i.e.~when
\emph{transversality} is achieved.  It should also
be noted that in later applications to symplectic topology, the space
\eqref{eqn:solutionSpace} will not really be the one we want to consider:
it has two unnatural features, namely that it fixes an \emph{arbitrary} complex
structure on the domain, and that it may include different curves that are 
reparametrizations of each other, and thus should really be considered
``equivalent''.  We will address these issues in Chapter~\ref{chapter:moduli}, when
we give the proper definition of the moduli space of $J$-holomorphic curves.

For now, \eqref{eqn:solutionSpace} will be the space of interest, and
we sketched already in \S\ref{sec:linearization}
how to turn the study of this space
into a problem of nonlinear functional analysis.  It is time to make that
discussion precise by defining the appropriate Banach manifolds and bundles.

We must first understand how to define Sobolev spaces of sections on
vector bundles.  In general, for any smooth vector bundle $E \to \Sigma$
one can define the space $W^{k,p}_\loc(E)$ to consist of all sections
whose expressions in all choices of local coordinates and trivializations 
are of class $W^{k,p}$ on compact subsets.  One can analogously define
maps of class $W^{k,p}_\loc$ between two smooth manifolds.  When $\Sigma$
is also \emph{compact}, we define the space $W^{k,p}(E)$ to be simply
$W^{k,p}_\loc(E)$, and give it the structure of a Banach space as follows.
Choose a finite open cover $\bigcup_j \uU_j = \Sigma$, and
assume that for each
set $\uU_j \subset \Sigma$, there is a smooth chart
$\varphi_j : \uU_j \to \Omega_j$, where $\Omega_j =
\varphi_j(\uU_j) \subset \CC$, as well as a
local trivialization $\Phi_j : E|_{\uU_j} \to
\uU_j \times \CC^n$.  Then if $\{ \alpha_j : \Sigma \to [0,1] \}$ 
is a partition of unity subordinate to $\{\uU_j\}$, define for any
section $v : \Sigma \to E$,
\begin{equation}
\label{eqn:sectionNorm}
\| v \|_{W^{k,p}(E)} = \sum_j \| \pr_2 \circ \Phi_j \circ (\alpha_j v) 
 \circ \varphi_j^{-1} \|_{W^{k,p}(\Omega_j)}.
\end{equation}
This definition depends on plenty of choices, and the norm on
$W^{k,p}(E)$ is thus not canonically defined; really one should call
$W^{k,p}(E)$ a \emph{Banachable} space rather than a Banach space.  The
exercise below shows that at least the resulting \emph{topology}
on $W^{k,p}(E)$ is canonical.  In a completely analogous way, one can also
define the Banach spaces $C^k(E)$ and $C^{k,\alpha}(E)$.

\begin{exercise} \ 
\begin{enumerate}
\renewcommand{\labelenumi}{(\alph{enumi})}
\item
Show that any alternative choice of finite open covering, 
charts, trivializations
and partition of unity gives an equivalent norm on $W^{k,p}(E)$.
\textsl{Hint: Given two complete norms on the same vector space, it's enough
to show that the identity map from one to the other is continuous
(in one direction!).}
\item
Verify that your favorite embedding theorems hold: in particular,
$W^{k,p}(E)$ admits a continuous and compact embedding into 
$W^{k-1,p}(E)$ for all $p$, and $C^{k-1}(E)$ whenever $p > 2$.
\end{enumerate}
\end{exercise}

\begin{remark}
If $\Sigma$ is not compact, then even the topology of $W^{k,p}(E)$ is
not well defined without some extra choices.  We'll need to deal with this
issue later when we discuss punctured holomorphic curves.
\end{remark}

\begin{exercise}
\label{EX:classWkp}
For $kp > 2$, a vector bundle $E \to \Sigma$ is said to have a
\defin{$W^{k,p}$-smooth structure} if it admits a system of local trivializations
whose transition maps are of class $W^{k,p}$.  Show that $W^{k,p}(E)$ is also 
a well-defined Banachable space in this case,
though one cannot speak of sections of any better regularity than~$W^{k,p}$.
Why doesn't any of this make sense if $kp \le 2$?
\end{exercise}

Next we consider maps of Sobolev-type regularity between the manifolds
$\Sigma$ and $M$; we'll restrict our attention to the case $kp > 2$, so that 
all such maps are continuous.  It was already remarked
that the space $W^{k,p}_\loc(\Sigma,M)$ can be defined naturally by expressing
maps $\Sigma \to M$ in local charts, though since it isn't a vector space,
the question of precisely what structure this space has is a bit subtle.
Intuitively, we expect spaces of maps $\Sigma \to M$ to be manifolds, and
this motivates the following definition.

\begin{defn}
\label{defn:Wkp}
For any $k \in \NN$ and $p > 1$ such that $kp > 2$, choose any smooth
connection on $M$, and for any smooth map $f \in C^\infty(\Sigma,M)$,
choose a neighborhood $\uU_f$ of the zero section in $f^*TM$ such that
for all $z \in \Sigma$, the restriction of $\exp$ to $T_{f(z)}M \cap \uU_f$
is an embedding.  Then we define the space of $W^{k,p}$-smooth maps
from $\Sigma$ to $M$ by
\begin{equation*}
\begin{split}
W^{k,p}(\Sigma,M) = \{ u \in C^0(\Sigma,M) \ |\ & u = \exp_{f} \eta
\text{ for some $f \in C^\infty(\Sigma,M)$ and} \\ 
& \text{$\eta \in W^{k,p}(f^*TM)$ with $\eta(\Sigma) \subset \uU_f$} \}.
\end{split}
\end{equation*}
\end{defn}
We've not yet assigned a topology to $W^{k,p}(\Sigma,M)$, but a topology
emerges naturally from the nontrivial observation that our definition
gives rise to a smooth Banach manifold structure.  Indeed, the charts 
are the maps $\exp_f \eta \mapsto \eta$ which take subsets of
$W^{k,p}(\Sigma,M)$ into open subsets of Banach spaces, namely
$$
W^{k,p}(\uU_f) := \{ \eta \in W^{k,p}(f^*TM)\ |\ \eta(\Sigma) \subset \uU_f \}.
$$
Since the exponential map is smooth, a slight generalization
of Lemma~\ref{lemma:babySmoothness} shows that
the resulting transition maps are smooth---this depends fundamentally on
the same three properties of $W^{k,p}$ that were listed in the lemma:
it embeds into~$C^0$, it is a Banach algebra, and it behaves continuously
under composition with smooth functions.  In the same manner, one shows
that the transition maps arising from different choices of connection on $M$
are also smooth, thus the smooth structure of $W^{k,p}(\Sigma,M)$ doesn't
depend on this choice.  The complete details of these arguments
(in a very general context) are carried out in \cite{Eliasson}.  
The same paper also shows that the tangent
spaces to $W^{k,p}(\Sigma,M)$ are canonically isomorphic to exactly what
one would expect:
$$
T_u W^{k,p}(\Sigma,M) = W^{k,p}(u^*TM).
$$
Note that in general, $u^*TM \to \Sigma$ is only a bundle of class
$W^{k,p}$, but the resulting Banach space of sections is well defined
due to Exercise~\ref{EX:classWkp} above.

\begin{exercise}
\label{EX:evaluation}
Assuming $kp > 2$ as in the above discussion, show that for any chosen
point $z_0 \in \Sigma$, the natural \emph{evaluation map}
$$
W^{k,p}(\Sigma,M) \to M : u \mapsto u(z_0)
$$
is smooth.  \textsl{Hint: This depends essentially on the fact that
(1)~the exponential map on~$M$ is smooth, and (2)~for any smooth
vector bundle $E \to \Sigma$, the inclusion of $W^{k,p}$ into $C^0$
implies that
$W^{k,p}(E) \to E_{z_0} : \eta \mapsto \eta(z_0)$ defines a bounded linear
operator.}
\end{exercise}
\begin{exercise}
Show that the map $W^{k,p}(\Sigma,M) \times \Sigma \to M : (u,z) \mapsto u(z)$
is not smooth.
\end{exercise}

The definition of \emph{Banach manifold} that we have been using thus far is
absurdly general: indeed, a topological space with an atlas of
smoothly compatible charts generally need not be either Hausdorff or 
paracompact (see \cite{Lang:geometry}).  
It will be useful to note that the particular 
Banach manifolds we are considering are topologically not
nearly so exotic.

\begin{prop}
\label{prop:notExotic}
The Banach manifold $W^{k,p}(\Sigma,M)$ defined above is metrizable and
separable.
\end{prop}
\begin{proof}
Choose a smooth embedding of $M$ into $\RR^N$ for some sufficiently large
$N \in \NN$.  Using \cite{Eliasson}*{Theorem~5.3}, one can show that
this induces a smooth embedding of $W^{k,p}(\Sigma,M)$ into the
linear Banach space $W^{k,p}(\Sigma,\RR^N)$ as a smooth submanifold.
The latter is metrizable and separable, so we conclude the same
for $W^{k,p}(\Sigma,M)$.
\end{proof}

One can take these ideas further and speak of vector bundles whose fibers
are Banach spaces: a \defin{Banach space bundle} of class $C^k$ is defined
by a system of local trivializations whose 
transition maps are of class $C^k$ from open subsets of the base to the
Banach space of bounded endomorphisms $\lL(X)$ on some Banach space~$X$.  
Note that if $g : \uU \to \lL(X)$ is a transition map and $z \in \uU$,
$x \in X$, it is not enough to require continuity or smoothness of the
map $(z,x) \mapsto g(z) x$; that is a significantly weaker condition in
infinite dimensions.  We refer to \cite{Lang:geometry} for more on the
general properties of Banach space bundles.

For our purposes, it will be important to consider the Banach manifold
$$
\bB^{k,p} := W^{k,p}(\Sigma,M)
$$
with a Banach space bundle $\eE^{k-1,p} \to \bB^{k,p}$ whose fiber at
$u \in \bB^{k,p}$ is
$$
\eE^{k-1,p}_u := W^{k-1,p}(\overline{\Hom}_\CC(T\Sigma,u^*TM)).
$$
You should take a moment to convince yourself that for any
$u \in \bB^{k,p}$, it makes sense to speak of sections of class
$W^{k-1,p}$ on the bundle $\overline{\Hom}_\CC(T\Sigma,u^*TM) \to \Sigma$.  
As it turns out, the general framework of \cite{Eliasson} implies that
$\eE^{k-1,p} \to \bB^{k,p}$ admits the structure of a smooth Banach
space bundle such that
$$
\dbar_J : \bB^{k,p} \to \eE^{k-1,p} : u \mapsto Tu + J \circ Tu \circ j
$$
is a smooth section.  Note that in the last observation, we are using
the assumption that $J$ is smooth, as the question can be reduced to yet
another application of Lemma~\ref{lemma:babySmoothness}:
the section $\dbar_J$ contains the map
$W^{k,p} \to W^{k,p} : u \mapsto J \circ u$, which has only as many
derivatives as $J$ (minus some constant).  For this reason, we will assume 
whenever possible from now on that $J$ is smooth.

The zero set of $\dbar_J$ is the space of solutions \eqref{eqn:solutionSpace},
and as we already observed, the topology of this solution space will have
no dependence on $k$ or~$p$.  To show that $\dbar_J^{-1}(0)$ has a nice
structure, we want to apply the infinite-dimensional bundle version of the
implicit function theorem, which will 
apply near any point $u \in \dbar_J^{-1}(0)$ at which the linearization
$$
\mathbf{D}_u := D\dbar_J(u) : T_u\bB^{k,p} \to \eE_u^{k-1,p}
$$
is surjective and has a bounded right inverse.  Here $\mathbf{D}_u$ is the
operator we derived in \S\ref{sec:linearization}; at the time we were assuming
everything was smooth, but the result clearly extends to a bounded
linear operator
\begin{equation*}
\begin{split}
\mathbf{D}_u : W^{k,p}(u^*TM) &\to W^{k-1,p}(\overline{\Hom}_\CC(T\Sigma,u^*TM)) \\
\eta &\mapsto \nabla \eta + J(u) \circ \nabla\eta \circ j +
(\nabla_\eta J) Tu \circ j,
\end{split}
\end{equation*}
where $\nabla$ is an arbitrary symmetric connection on $M$, and
this operator must be identical to $D\dbar_J(u)$ since $C^\infty$ is
dense in all the spaces under consideration.

The condition that $\mathbf{D}_u$ have
a bounded right inverse will turn out to be trivially satisfied whenever
$\mathbf{D}_u$ is surjective,
because $\ker \mathbf{D}_u$ is finite dimensional.  This is an important 
new feature of the global setting that did not exist locally, and we will
spend the rest of this chapter proving it and computing the
dimension.  The main result can be summarized as follows.

\begin{thm}
\label{thm:main}
For any $u \in \dbar_J^{-1}(0)$, $\mathbf{D}_u$ is a Fredholm operator with index
$$
\ind(\mathbf{D}_u) = n\chi(\Sigma) + 2 \langle c_1(TM) , [u] \rangle,
$$
where $[u] := u_*[\Sigma] \in H_2(M)$ and $c_1(TM) \in H^2(M)$ 
is the first Chern
class of the complex vector bundle $(TM,J)$.
\end{thm}

Recall that a bounded linear operator $D : X \to Y$ between Banach spaces
is called \defin{Fredholm} if both $\ker D$ and $Y / \im D$ are finite
dimensional; the latter space is called the \defin{cokernel} of $D$, 
often written as $\coker D$.  The \defin{Fredholm index} of $D$
is then defined to be
$$
\ind(D) = \dim \ker(D) - \dim \coker(D).
$$
Fredholm operators have many nice things in common with linear maps on
finite-dimensional spaces.  Proofs of the following standard facts may
be found in e.g.~\cite{Taylor:PDE1}*{Appendix~A}
and~\cite{AbramovichAliprantis}*{\S 4.4}.
\begin{prop}
\label{prop:Fredholm}
Assume $X$ and $Y$ are Banach spaces, and let $\Fred(X,Y) \subset \lL(X,Y)$
denote the space of Fredholm operators from $X$ to~$Y$.
\begin{enumerate}
\item $\Fred(X,Y)$ is an open subset of $\lL(X,Y)$.
\item The map $\ind : \Fred(X,Y) \to \ZZ$ is continuous.
\item If $D \in \Fred(X,Y)$ and $K \in \lL(X,Y)$ is a compact operator, 
then $D + K \in \Fred(X,Y)$.
\item If $D \in \Fred(X,Y)$ then $\im D$ is a closed subspace of~$Y$, 
and there exists
a closed linear subspace $V \subset X$ and finite-dimensional subspace
$W \subset Y$ such that
$$
X = \ker(D) \oplus V,
\qquad
Y = \im(D) \oplus W,
$$
and $D|_V : V \to \im(D)$ is a Banach space isomorphism.
\end{enumerate}
\end{prop}
Note that the continuity of the map $\ind : \Fred(X,Y) \to \ZZ$ means it is
locally constant, thus for any continuous family of Fredholm operators
$\{D_t\}_{t \in [0,1]}$, $\ind(D_t)$ is constant.  This fact is extremely
useful for index computations, and is true despite the fact that the dimensions
of $\ker D_t$ and $Y / \im D_t$ may each change quite drastically.  
As a simple application, this
implies that for any compact operator $K$, $\ind(D + K) = \ind(D)$,
as these two are connected by the continuous family $D + tK$.
\begin{exercise}
The definition of a Fredholm operator $D : X \to Y$ often includes the 
assumption that $\im D$ is closed, but this is redundant.  Convince yourself
that for any $D \in \lL(X,Y)$, if $Y / \im D$ is finite dimensional then
$\im D$ is closed.  If you get stuck, see 
\cite{AbramovichAliprantis}*{Corollary~2.17}.
\end{exercise}

Theorem~\ref{thm:main} is of course most interesting in the case where $\mathbf{D}_u$ is
surjective, as then the implicit function theorem yields:
\begin{cor}
If $u \in \dbar_J^{-1}(0)$ and $\mathbf{D}_u$ is surjective, then a neighborhood
of $u$ in $\dbar_J^{-1}(0)$ admits the structure of a smooth finite-dimensional 
manifold, with
$$
\dim \dbar_J^{-1}(0) = n\chi(\Sigma) + 2 \langle c_1(TM) , [u] \rangle.
$$
\end{cor}

\section{Formal adjoints}
\label{sec:formal}

The Fredholm theory for the operator $\mathbf{D}_u$ fits naturally into the more
general context of Cauchy-Riemann type operators on vector bundles.
For the next three sections, we will consider an arbitrary smooth complex 
vector bundle $(E,J) \to (\Sigma,j)$ of (complex) rank~$n$, where
$(\Sigma,j)$ is a closed connected Riemann surface unless otherwise noted.  We will 
often abbreviate the first Chern number of $(E,J)$ by writing
$$
c_1(E) := \langle c_1(E,J),[\Sigma] \rangle \in \ZZ.
$$
Let $D : \Gamma(E) \to \Gamma(\overline{\Hom}_\CC(T\Sigma,E))$ denote a
(real- or complex-) linear Cauchy-Riemann type operator.  In order to
understand the properties of this operator, it will be extremely useful
to observe that it has a \defin{formal adjoint},
$$
D^* : \Gamma(\overline{\Hom}_\CC(T\Sigma,E)) \to \Gamma(E),
$$
which will turn out to have all the same nice properties of a Cauchy-Riemann
type operator.  We'll use this in the next section to understand the
cokernel of $D$, which turns out to be naturally isomorphic to the
kernel of~$D^*$.

Choose a Hermitian bundle metric 
$\langle\ ,\ \rangle$ on $E$, and let $(\ ,\ )$ denote its real part,
which is a real bundle metric that is invariant under the action of~$J$.
Choose also a Riemannian metric $g$ on $\Sigma$ that is compatible with the
conformal structure defined by~$j$; this defines a volume form $\mu_g$
on $\Sigma$, and conversely (since $\dim_\RR \Sigma = 2$), 
such a volume form uniquely determines the compatible metric~$g$
via the relation
$$
\mu_g(X,Y) = g(jX,Y).
$$
These choices naturally induce a bundle metric $(\ ,\ )_g$ on
$\overline{\Hom}_\CC(T\Sigma,E))$, and both $\Gamma(E)$ and
$\Gamma(\overline{\Hom}_\CC(T\Sigma,E))$ now inherit natural $L^2$-inner
products, defined by
$$
\langle \xi,\eta \rangle_{L^2} = \int_\Sigma (\xi,\eta)\,\mu_g,
\qquad
\langle \alpha , \beta \rangle_{L^2} = \int_\Sigma (\alpha,\beta)_g\,\mu_g
$$
for $\xi,\eta \in \Gamma(E)$ and
$\alpha, \beta \in \Gamma(\overline{\Hom}_\CC(T\Sigma,E))$.
We say that an operator 
$D^* : \Gamma(\overline{\Hom}_\CC(T\Sigma,E)) \to \Gamma(E)$ is the
\defin{formal adjoint} of $D$ if it satisfies
\begin{equation}
\label{eqn:adjoint}
\langle \alpha , D\eta \rangle_{L^2} = \langle D^*\alpha , \eta \rangle_{L^2}.
\end{equation}
for all smooth sections $\eta \in \Gamma(E)$ and
$\alpha \in \Gamma(\overline{\Hom}_\CC(T\Sigma,E))$.
The existence of such operators is a quite general phenomenon that is easy
to see locally using integration by parts: roughly speaking, 
if $D$ has the form 
$D = \dbar + A$ in some local trivialization, then we expect $D^*$ in the same 
local picture to take the form $-\p + A^T$.  
One sees also from this local expression that
$D^*$ is almost a Cauchy-Riemann type operator; to be precise, it is
\emph{conjugate} to a Cauchy-Riemann type operator.  The extra minus sign
can be removed by an appropriate bundle isomorphism, and one can always
transform $\p = \p_s - i \p_t$ into $\dbar = \p_s + i \p_t$ by reversing
the complex structure on the bundle.  Globally, the result will be the
following.

\begin{prop}
\label{prop:formalAdjoint}
For any choice of Hermitian bundle metric on $(E,J) \to (\Sigma,j)$
and Riemannian metric $g$ on $\Sigma$ compatible with~$j$, every
linear Cauchy-Riemann type operator
$D : \Gamma(E) \to \Gamma(\overline{\Hom}_\CC(T\Sigma,E))$ admits a
formal adjoint 
$$
D^* : \Gamma(\overline{\Hom}_\CC(T\Sigma,E)) \to \Gamma(E)
$$
which is conjugate to a linear Cauchy-Riemann type operator in the following
sense.  Defining a complex vector bundle $(\widehat{E},\hat{J})$
over $\Sigma$ by
$$
(\widehat{E},\hat{J}) := (\overline{\Hom}_\CC(T\Sigma,E),-J),
$$
there exist smooth real-linear vector bundle isomorphisms
$$
\Phi : \widehat{E} \to \overline{\Hom}_\CC(T\Sigma,E),
\qquad
\Psi : E \to \overline{\Hom}_\CC(T\Sigma,\widehat{E})
$$
such that $\Psi \circ D^* \circ \Phi$ is a linear Cauchy-Riemann type
operator on~$\widehat{E}$.
\end{prop}

We will prove this by deriving a global expression for $D^*$.
One can construct it by a generalization of the same procedure by which
one constructs the formal adjoint of $d$ on the algebra of differential forms,
so let us recall this first.  If $M$ is any smooth oriented manifold of real
dimension~$m$ with a Riemannian metric~$g$, let $\mu_g$ denote the induced
volume form, and use $g$ also to denote the natural extension of $g$ to
a bundle metric on each of the skew-symmetric tensor bundles 
$\Lambda^k T^*M$ for $k=0,\ldots,m$.  We will denote $\Omega^k(M) :=
\Gamma(\Lambda^k T^*M)$, i.e.~this is simply the vector space of smooth
differential $k$-forms on~$M$.  Now
for each $k = 0,\ldots,m$, 
there is a unique bundle isomorphism,
$$
* : \Lambda^k T^*M \to \Lambda^{m-k} T^*M
$$
the \defin{Hodge star operator}, which has
the property that for all $\alpha, \beta \in \Omega^k(M)$,
\begin{equation}
\label{eqn:Hodge}
g(\alpha,\beta)\,\mu_g = \alpha \wedge *\beta.
\end{equation}
One can easily show that $*$ is a bundle isometry and satisfies
$*^2 = (-1)^{k(m-k)}$.  With this, one can associate to the exterior
derivative $d : \Omega^k(M) \to \Omega^{k+1}(M)$ a formal adjoint
\begin{equation*}
\begin{split}
d^* &: \Omega^k(M) \to \Omega^{k-1}(M), \\
d^* &= (-1)^{m(k+1)+1} * d *,
\end{split}
\end{equation*}
which satisfies
$$
\int_M g(\alpha,d\beta)\,\mu_g = \int_M g(d^* \alpha,\beta)\,\mu_g
$$
for any $\alpha \in \Omega^k(M)$ and $\beta \in \Omega^{k-1}(M)$ with
compact support.  The proof of this relation is an easy exercise in
Stokes' theorem, using \eqref{eqn:Hodge}.

We can extend this discussion to bundle-valued differential forms:
given a real vector bundle $E \to M$,
let $\Omega^k(M,E) := \Gamma(\Lambda^k T^*M \otimes E)$, which is
naturally identified with the space of smooth $k$-multilinear antisymmetric
bundle maps $TM \oplus \ldots \oplus TM \to E$.  Choosing a bundle metric
$(\  ,\ )$ on~$E$, the combination of $g$ and $(\ ,\ )$ induces a
natural tensor product 
metric on $\Lambda^k T^*M \otimes E$, which we'll denote by
$(\ ,\ )_g$.  There is also an isomorphism of $E$ to its dual bundle
$E^* \to M$, defined by
$$
E \to E^* : v \mapsto \bar{v} := ( v,\cdot ),
$$
which extends naturally to an isomorphism 
$$
\Lambda^k T^*M \otimes E \to \Lambda^k T^*M \otimes E^* :
\alpha \mapsto \bar{\alpha}.
$$
There is no natural product structure on $\Lambda^* T^*M \otimes E$,
but the wedge product does define a natural pairing
$$
(\Lambda^* T^*M \otimes E^*) \oplus (\Lambda^* T^*M \otimes E) \to
\Lambda^*T^*M : (\alpha \otimes \lambda, \beta \otimes v) \mapsto
\lambda(v) \cdot \alpha \wedge \beta,
$$
as well as a fiberwise module structure,
$$
\Lambda^k T^*M \oplus (\Lambda^\ell T^*M\otimes E) \to
\Lambda^{k+\ell} T^*M \otimes E : (\alpha,\beta) \mapsto \alpha \wedge
\beta,
$$
so that in particular $\Omega^*(M,E)$ becomes an $\Omega^*(M)$-module.  

Now if $\nabla : \Gamma(E) \to \Gamma(\Hom(TM,E)) = \Omega^1(M,E)$ 
is a connection on $E \to M$, this has a natural extension to a 
\defin{covariant exterior derivative}, which is a degree~$1$
linear map $d_\nabla : \Omega^*(M,E) \to \Omega^*(M,E)$ satisfying
the graded Leibnitz rule
$$
d_\nabla (\alpha \wedge \beta) = d\alpha \wedge \beta + (-1)^k \alpha \wedge
d_\nabla\beta
$$
for all $\alpha \in \Omega^k(M)$ and $\beta \in \Omega^\ell(M,E)$.
This also has a formal adjoint $d^*_\nabla : \Omega^*(M,E) \to
\Omega^*(M,E)$, which is a linear map of degree~$-1$.  We can write it down
using a slight generalization of the Hodge star operator:
$$
* : \Lambda^k T^*M \otimes E \to \Lambda^{m-k}T^*M \otimes E :
\alpha \otimes v \mapsto *\alpha \otimes v,
$$
in other words for any $p \in M$, $\alpha \in \Lambda^k T^*_p M$ and 
$v \in E_p$, the product $\alpha v$ defines a skew-symmetric $k$-form 
on $T_p M$ with values in $E_p$, and we define $*(\alpha v)$ to be
$(*\alpha) v$.  This map has the property that for all
$\alpha,\beta \in \Omega^k(M,E)$,
$$
(\alpha,\beta)_g \,\mu_g = \bar{\alpha} \wedge *\beta,
$$
and it is then straightforward to verify that
\begin{equation}
\label{eqn:deltaNabla}
\begin{split}
d^*_\nabla &: \Omega^k(M,E) \to \Omega^{k-1}(M,E), \\
d^*_\nabla &= (-1)^{m(k+1)+1} * d_\nabla *
\end{split}
\end{equation}
has the desired property, namely that
\begin{equation}
\label{eqn:deltaAdjoint}
\int_M (\alpha,d_\nabla\beta)_g \,\mu_g = \int_M (d^*_\nabla 
\alpha,\beta)_g \,\mu_g
\end{equation}
for all $\alpha \in \Omega^k(M,E)$ and $\beta \in \Omega^{k-1}(M,E)$ with
compact support.

Let us now extend some of these constructions to
a complex vector bundle $(E,J)$ of rank~$n$ over a complex manifold
$(\Sigma,j)$ of (complex) dimension~$m$.
Here it becomes natural to split the space of bundle-valued 
$1$-forms $\Omega^1(\Sigma,E)$
into the subspaces of complex-linear and antilinear forms,
often called $(1,0)$-forms and $(0,1)$-forms respectively,
$$
\Omega^1(\Sigma,E) = \Omega^{1,0}(\Sigma,E) \oplus \Omega^{0,1}(\Sigma,E),
$$
where by definition $\Omega^{1,0}(\Sigma,E) =
\Gamma(\Hom_\CC(T\Sigma,E))$ and $\Omega^{0,1}(\Sigma,E) =
\Gamma(\overline{\Hom}_\CC(T\Sigma,E))$.  Choosing holomorphic local
coordinates $(z^1,\ldots,z^m)$ on some open subset of~$\Sigma$, all the
$(1,0)$-forms can be written on this subset as
$$
\alpha = \sum_{j=1}^m \alpha_j \,dz^j
$$
for some local sections $\alpha_j$ of~$E$, and the $(0,1)$-forms likewise
take the form
$$
\alpha = \sum_{j=1}^m \alpha_j \,d\bar{z}^j.
$$
The space of bundle-valued $k$-forms then splits into subspaces
of $(p,q)$-forms for $p + q = k$,
$$
\Omega^k(\Sigma,E) = \bigoplus_{p+q=k} \Omega^{p,q}(\Sigma,E),
$$
where any $\alpha \in \Omega^{p,q}(\Sigma,E)$ can be written locally as
a linear combination of terms of the form
$$
dz^{j_1} \wedge \ldots \wedge dz^{j_p} \wedge d\bar{z}^{k_1} \wedge
\ldots \wedge d\bar{z}^{k_q}
$$
multiplied with local sections of~$E$.  The $(p,q)$-forms are
sections of a vector bundle
$$
\Lambda^{p,q}T^*\Sigma \otimes E,
$$
which is a subbundle of $\Lambda^{p+q}T^*\Sigma \otimes E$.

As a special case, let $\Omega^{p,q}(\Sigma) := \Omega^{p,q}(\Sigma,
\Sigma \times \CC)$ denote the space of complex-valued $(p,q)$-forms.
Then the image of the exterior derivative on $\Omega^{p,q}(\Sigma)$
splits naturally:
$$
d : \Omega^{p,q}(\Sigma) \to \Omega^{p+1,q}(\Sigma) \oplus 
\Omega^{p,q+1}(\Sigma),
$$
and with respect to this splitting we can define linear operators
$$
\p : \Omega^{p,q}(\Sigma) \to \Omega^{p+1,q}(\Sigma),
\qquad
\dbar : \Omega^{p,q}(\Sigma) \to \Omega^{p,q+1}(\Sigma)
$$
such that $d = \p + \dbar$.  The restriction to $\Omega^{0,0}(\Sigma) =
C^\infty(\Sigma,\CC)$ gives (up to a factor of two)\footnote{For this section
only, we are modifying our usual definition of the operators $\dbar$ and $\p$
on $C^\infty(\Sigma,\CC)$ to include the extra factor of $1/2$.
The difference is harmless.}
the usual operators
$\p$ and $\dbar$ on smooth functions $f : \Sigma \to \CC$, namely
$$
\p f = \frac{1}{2}( df - i\,df \circ j),
\qquad
\dbar f = \frac{1}{2} ( df + i \,df \circ j).
$$
It follows now almost tautologically that $\p$ and $\dbar$ satisfy
graded Leibnitz rules,
\begin{equation*}
\begin{split}
\p (\alpha \wedge \beta) &= \p\alpha \wedge \beta + (-1)^{p+q} \alpha 
\wedge \p\beta, \\
\dbar (\alpha \wedge \beta) &= \dbar\alpha \wedge \beta + (-1)^{p+q}
\alpha \wedge \dbar\beta
\end{split}
\end{equation*}
for $\alpha \in \Omega^{p,q}(\Sigma)$ and $\beta \in \Omega^{r,s}(\Sigma)$.

Choosing a Hermitian metric on the bundle $(E,J) \to (\Sigma,j)$, we can
similarly split the derivation $d_\nabla : \Omega^k(\Sigma,E) \to
\Omega^{k+1}(\Sigma,E)$ defined by any Hermitian connection, giving rise
to complex-linear operators
\begin{equation*}
\begin{split}
\p_\nabla : \Omega^{p,q}(\Sigma,E) &\to \Omega^{p+1,q}(\Sigma,E), \\
\dbar_\nabla : \Omega^{p,q}(\Sigma,E) &\to \Omega^{p,q+1}(\Sigma,E)
\end{split}
\end{equation*}
which satisfy similar Leibnitz rules,
\begin{equation}
\label{eqn:moreLeibnitz}
\begin{split}
\p_\nabla (\alpha \wedge \beta) &= \p\alpha \wedge \beta + (-1)^{p+q} \alpha 
\wedge \p_\nabla \beta, \\
\dbar_\nabla (\alpha \wedge \beta) &= \dbar\alpha \wedge \beta + (-1)^{p+q}
\alpha \wedge \dbar_\nabla\beta
\end{split}
\end{equation}
for $\alpha \in \Omega^{p,q}(\Sigma)$ and $\beta \in \Omega^{r,s}(\Sigma,E)$.
In particular, this shows that $\dbar_\nabla : \Omega^{p,q}(\Sigma,E) \to
\Omega^{p,q+1}(\Sigma,E)$ can be regarded as a complex-linear
Cauchy-Riemann type operator on the bundle $\Lambda^{p,q}T^*\Sigma \otimes
E$, where we identify $\overline{\Hom}_\CC(T\Sigma,
\Lambda^{p,q}T^*\Sigma \otimes E)$ naturally with
$\Lambda^{p,q+1}T^*\Sigma \otimes E$.
Restricting to $\Omega^{0,0}(\Sigma,E) = \Gamma(E)$, $\dbar_\nabla :
\Gamma(E) \to \Omega^{0,1}(\Sigma,E)$ has the form
$$
\dbar_\nabla = \frac{1}{2}\left( \nabla + J \circ \nabla \circ j \right).
$$
We are now almost ready to write down the formal adjoint of this operator.
For simplicity, we restrict
to the case where $\Sigma$ has complex dimension one, since this 
is all we need.  Observe that the Hodge star then defines a
bundle isomorphism of $\Lambda^1 T^*\Sigma$ to itself, whose natural
extension to $\Lambda^1 T^*\Sigma \otimes E$ is complex-linear.

\begin{exercise}
\label{EX:Hodge} \ 
\begin{enumerate}
\renewcommand{\labelenumi}{(\alph{enumi})}
\item
Show that for any choice of local holomorphic coordinates $z = s+it$ on
$\Sigma$, $* ds = dt$ and $* dt = -ds$.  
\item
Show that for any
$\alpha \in T^*\Sigma$, $*\alpha = -\alpha \circ j$.
\item
Show that for any $\alpha \in \Lambda^{1,0}T^*\Sigma \otimes E$,
$*\alpha = -J\alpha$ and for any $\alpha \in \Lambda^{0,1}T^*\Sigma \otimes E$,
$*\alpha = J\alpha$.  In particular, $*$ respects the splitting
$\Omega^1(\Sigma,E) = \Omega^{1,0}(\Sigma,E) \oplus \Omega^{0,1}(\Sigma,E)$.
\end{enumerate}
\end{exercise}

We claim now that the formal adjoint of $\dbar_\nabla$ is defined by
a formula analogous to the operator $d^*_\nabla$
of \eqref{eqn:deltaNabla}, namely
\begin{equation}
\label{eqn:dbarStar}
\dbar_\nabla^* := - * \p_\nabla * : \Omega^{0,1}(\Sigma,E) \to
\Omega^0(\Sigma,E).
\end{equation}
In fact, this is simply the restriction of $d^*_\nabla$ to
$\Omega^{0,1}(\Sigma,E)$, as we observe that $\dbar_\nabla$ maps
$\Omega^{0,1}(\Sigma,E)$ to $\Omega^{0,2}(\Sigma,E)$, which is trivial
since $\Sigma$ has only one complex dimension.  Thus the claim follows
easily from \eqref{eqn:deltaAdjoint} and the following exercise.

\begin{exercise}
Show that $\Lambda^{1,0}T^*\Sigma \otimes E$ and $\Lambda^{0,1}T^*\Sigma
\otimes E$ are orthogonal subbundles with respect to the metric
$(\ ,\ )_g$ on $\Lambda^1 T^*\Sigma \otimes E$.
\end{exercise}

It is now easy to write down the formal adjoint of a more general
Cauchy-Riemann type operator.

\begin{proof}[Proof of Prop.~\ref{prop:formalAdjoint}]
Choosing any Hermitian connection $\nabla$ on $E$, 
Exercise~\ref{EX:Christoffel} allows us to write
$$
D = \dbar_\nabla + A,
$$
where $A : E \to \overline{\Hom}_\CC(T\Sigma,E)$ is a smooth real-linear
bundle map.  (Note that Exercise~\ref{EX:Christoffel} 
dealt only with the complex-linear
case, but the generalization to the real case is obvious.)
Extending a well-known fact from linear algebra to the context of bundles,
there is a unique smooth real-linear bundle map 
$A^T : \overline{\Hom}_\CC(T\Sigma,E) \to E$ such that
$$
(\alpha, A \eta)_g = (A^T \alpha,\eta)
$$
for all $z \in \Sigma$, $\eta \in E_z$ and $\alpha \in 
\Lambda^{0,1}T^*_z \Sigma \otimes E_z$.  Then the desired operator
$D^*$ is given by
$$
D^* = \dbar_\nabla^* + A^T.
$$
From \eqref{eqn:dbarStar}, we see that $D^*$ is conjugate to an operator
of the form
$$
D_1 = \p_\nabla + A_1 : \Omega^{0,1}(\Sigma,E) \to
\Omega^{1,1}(\Sigma,E),
$$
where $A_1 : \Lambda^{0,1}T^*\Sigma \otimes E \to \Lambda^{1,1}T^*\Sigma
\otimes E$ is some smooth bundle map, i.e.~a ``zeroth order term.''
By \eqref{eqn:moreLeibnitz}, this satisfies the Leibnitz rule,
\begin{equation}
\label{eqn:LeibnitzD1}
D_1 (f \alpha) = (\p f) \alpha + f D_1 \alpha
\end{equation}
for all smooth functions $f : \Sigma \to \CC$.  We can turn this into the
Leibnitz rule for an actual Cauchy-Riemann type operator on the bundle,
$$
(\widehat{E},\hat{J}) = (\overline{\Hom}_\CC(T\Sigma,E),-J).
$$
Indeed, the identity $\widehat{E} \to \overline{\Hom}_\CC(T\Sigma,E)$ is
then a complex-antilinear bundle isomorphism, 
and there are canonical isomorphisms
$$
\Lambda^{1,1}T^*\Sigma \otimes E = \Hom_\CC(T\Sigma,\Lambda^{0,1}T^*\Sigma
\otimes E) = \overline{\Hom}_\CC(T\Sigma,\widehat{E}),
$$
so that $D_1$ is now conjugate to an operator
$$
D_2 : \Gamma(\widehat{E}) \to \Gamma(\overline{\Hom}_\CC(T\Sigma,\widehat{E}))
$$
which satisfies $D_2 (f\beta) = (\dbar f) \beta + f D_2 \beta$ due to
\eqref{eqn:LeibnitzD1}.  
\end{proof}

\begin{exercise}
\label{EX:c1}
Show that the bundle $(\widehat{E},\hat{J})$, as defined in
Prop.~\ref{prop:formalAdjoint} satisfies
$$
c_1(\widehat{E}) = - c_1(\Lambda^{0,1}T^*\Sigma \otimes E) =
-c_1(E) - n\chi(\Sigma).
$$
\end{exercise}

\begin{remark}
It's worth noting that if $(\Sigma,j)$ is a general complex
manifold with a Hermitian vector bundle $(E,J) \to (\Sigma,j)$ and
Hermitian connection $\nabla$, the resulting complex-linear Cauchy-Riemann
type operator 
$$
\dbar_\nabla : \Gamma(E) \to \Gamma(\overline{\Hom}_\CC(T\Sigma,E))
$$
does \emph{not} necessarily define a holomorphic structure
if $\dim_\CC \Sigma \ge 2$.  It turns out that the required local existence
result for holomorphic sections is true if and only if the map
$$
\dbar_\nabla \circ \dbar_\nabla : \Gamma(E) \to \Omega^{0,2}(\Sigma,E)
$$
is zero.  It's easy to see that this condition is necessary, because if
there is a holomorphic structure, then $\dbar_\nabla$ looks like the standard
$\dbar$-operator in a local holomorphic trivialization and
$\dbar \circ \dbar = 0$ on $\Omega^*(\Sigma,E)$.  The converse is,
in some sense, a complex version of the Frobenius integrability theorem:
indeed, the corresponding statement in real differential geometry is that
vector bundles with connections locally admit flat sections if and only if
$d_\nabla \circ d_\nabla = 0$, which means the curvature vanishes.
A proof of the complex version may be found in 
\cite{DonaldsonKronheimer}*{\S~2.2.2}, and the first step in this proof
is the local existence result for the case $\dim_\CC\Sigma = 1$
(our Theorem~\ref{thm:linearExistence}).  
Observe that the integrability condition is
trivially satisfied when $\dim_\CC\Sigma = 1$, since then
$\Omega^{0,2}(\Sigma,E)$ is a trivial space.
\end{remark}

\section{The Fredholm property}
\label{sec:FredholmProperty}

For the remainder of this chapter, $(\Sigma,j)$ will be a closed Riemann surface and
$(E,J) \to (\Sigma,j)$ will be a complex vector bundle of rank~$n$ with
a real-linear Cauchy-Riemann operator~$D$.  We shall now prove the
Fredholm property for the obvious extension of $D$ to a bounded linear map
\begin{equation}
\label{eqn:D}
D : W^{k,p}(E) \to W^{k-1,p}(\overline{\Hom}_\CC(T\Sigma,E)),
\end{equation}
with $k \in \NN$ and $p \in (1,\infty)$.

\begin{thm}
\label{thm:FredholmProperty}
The operator $D$ of \eqref{eqn:D} is Fredholm, and neither $\ker D$ nor
$\ind(D)$ depends on the choice of~$k$ and~$p$.
\end{thm}

This result depends essentially on three ingredients: first, the
Calder\'on-Zygmund inequality gives an estimate for $\| \eta \|_{W^{k,p}}$
in terms of $\| D \eta \|_{W^{k-1,p}}$, 
from which we will be able to show quite
easily that $\ker D$ is finite dimensional.  The second ingredient is the
formal adjoint $D^*$ that was derived in the previous section:
since $D^*$ is also conjugate to a Cauchy-Riemann type operator,
the previous step implies that its kernel is also finite dimensional.
The final ingredient is elliptic regularity, which we can use to
identify the cokernel of $D$ with the kernel of~$D^*$.  The regularity theory
also implies that both of these kernels consist only of smooth sections,
and are thus completely independent of $k$ and~$p$.

As sketched above, the first step in proving Theorem~\ref{thm:FredholmProperty}
is an a priori estimate that follows from the linear regularity theory
of \S\ref{sec:estimates}.  In particular, the Calder\'on-Zygmund inequality
(Theorem~\ref{thm:elliptic} and Exercise~\ref{EX:CZhigher}) gives
\begin{equation}
\label{eqn:CZWkp}
\| \eta \|_{W^{k,p}} \le c \| \eta \|_{W^{k-1,p}} 
\text{ for all $\eta \in W^{k,p}_0(B,\CC^n)$}.
\end{equation}
If $\eta \in W^{k,p}(E)$ and $\{ \alpha_j \}$ is a partition of unity
subordinate to some finite open cover of $\Sigma$, then we can apply
\eqref{eqn:CZWkp} in local trivializations and charts to the sections
$\alpha_i \eta$.  In this local picture,
$D$ becomes an operator of the form $\dbar + A$, where $A$ is a smooth
family of matrices, thus locally defining a bounded linear operator from
$W^{k,p}$ to itself.  The result is the following estimate, of which
a more detailed proof may be found in \cite{McDuffSalamon:Jhol}*{Lemma~C.2.1}.

\begin{lemma}
\label{lemma:estimate}
For each $k \in \NN$ and $p \in (1,\infty)$,
there exists a constant $c > 0$ such that for every $\eta \in W^{k,p}(E)$,
$$
\| \eta \|_{W^{k,p}(E)} \le c \| D \eta \|_{W^{k-1,p}(E)} 
+ c \| \eta \|_{W^{k-1,p}(E)}.
$$
\end{lemma}

Observe that the inclusion $W^{k,p}(E) \hookrightarrow W^{k-1,p}(E)$ is
compact.  This will allow us to make use of the following general result.

\begin{prop}
Suppose $X$, $Y$ and $Z$ are Banach spaces, $A \in \lL(X,Y)$,
$K \in \lL(X,Z)$ is compact, and there is a constant $c > 0$ such that for all
$x \in X$,
\begin{equation}
\label{eqn:generalEstimate}
\| x \|_X \le c \| Ax \|_Y + c \| Kx \|_Z.
\end{equation}
Then $\ker A$ is finite dimensional and $\im A$ is closed.
\end{prop}
\begin{proof}
A vector space is finite dimensional if and only if the unit ball in that
space is a compact set, so we begin by proving the latter holds for $\ker A$.
Suppose $x_k \in \ker A$ is a bounded sequence.  Then since $K$ is a
compact operator, $K x_k$ has a convergent subsequence in $Z$, which is
therefore Cauchy.  But \eqref{eqn:generalEstimate} 
then implies that the corresponding
subsequence of $x_k$ in $X$ is also Cauchy, and thus converges.

Since we now know $\ker A$ is finite dimensional, we also know there is a
closed complement $V \subset X$ with $\ker A \oplus V = X$.  
Then the restriction
$A|_V$ has the same image as $A$, thus if $y \in \overline{\im A}$, there
is a sequence $x_k \in V$ such that $A x_k \to y$.  We claim that $x_k$
is bounded.  If not, then $A(x_k / \|x_k\|_X) \to 0$ and $K(x_k / \|x_k\|_X)$
has a convergent subsequence, so \eqref{eqn:generalEstimate} implies that
a subsequence of $x_k / \|x_k\|_X$ also converges to some $x_\infty \in V$
with $\|x_\infty\| = 1$ and $A x_\infty = 0$, a contradiction.
But now since $x_k$ is bounded, $K x_k$ also has a convergent subsequence 
and $A x_k$ converges by assumption, thus \eqref{eqn:generalEstimate}
yields also a convergent subsequence of $x_k$, whose limit $x$ satisfies
$Ax = y$.  This completes the proof that $\im A$ is closed.
\end{proof}

The above implies that every Cauchy-Riemann type operator has finite-dimensional 
kernel and closed image; operators with these two properties
are called \defin{semi-Fredholm}.  Note that by elliptic regularity,
$\ker D$ only contains smooth sections, and is thus the same space for
every $k$ and~$p$.

By Prop.~\ref{prop:formalAdjoint},
the same results obviously apply to the formal adjoint, after extending it
to a bounded linear operator
$$
D^* : W^{k,p}(\overline{\Hom}_\CC(T\Sigma,E)) \to W^{k-1,p}(E).
$$

\begin{prop}
Using the natural inclusion $W^{k,p} \hookrightarrow W^{k-1,p}$ to inject
$\ker D$ and $\ker D^*$ into $W^{k-1,p}$, there are direct sum splittings
\begin{equation*}
\begin{split}
W^{k-1,p}(\overline{\Hom}_\CC(T\Sigma,E)) &= \im D \oplus \ker D^* \\
W^{k-1,p}(E) &= \im D^* \oplus \ker D.
\end{split}
\end{equation*}
Thus the projections along $\im D$ and $\im D^*$ yield natural isomorphisms
$\coker D = \ker D^*$ and $\coker D^* = \ker D$.
\end{prop}
\begin{proof}
We will prove only the first of the two splittings, as the second is
entirely analogous.  We claim first that $\im D \cap \ker D^* = \{0\}$.
Indeed, if $\alpha \in W^{k-1,p}(\overline{\Hom}_\CC(T\Sigma,E))$ with
$D^*\alpha = 0$, then since $D^*$ is conjugate to a Cauchy-Riemann type
operator via smooth bundle isomorphisms, 
elliptic regularity implies that $\alpha$ is smooth.
Then if $\alpha = D\eta$ for some $\eta \in W^{k,p}(E)$, $\eta$ must also
be smooth, and we find
$$
0 = \langle D^*\alpha, \eta \rangle_{L^2} = \langle \alpha , 
D\eta \rangle_{L^2} = \| \alpha \|_{L^2}^2.
$$

To show that $\im D + \ker D^* = W^{k-1,p}(\overline{\Hom}_\CC(T\Sigma,E))$,
it will convenient to address the case $k=1$ first.  Note that 
$\im D + \ker D^*$ is a closed subspace since $\im D$ is closed and
$\ker D^*$ is finite dimensional.  Then if it is not all of $L^p$, there
exists a nonzero $\alpha \in L^q(\overline{\Hom}_\CC(T\Sigma,E))$, where
$\frac{1}{p} + \frac{1}{q} = 1$, such that
\begin{equation*}
\begin{split}
\langle \alpha , D\eta \rangle_{L^2} &= 0 
\text{ for all $\eta \in W^{1,p}(E)$,} \\
\langle \alpha , \beta \rangle_{L^2} &= 0
\text{ for all $\beta \in \ker D^*$.}
\end{split}
\end{equation*}
The first relation is valid in particular for all smooth $\eta$, and this
means that $\alpha$ is a weak solution of the equation $D^*\alpha = 0$, so
by regularity of weak solutions (see Corollary~\ref{cor:weakRegularity}),
$\alpha$ is smooth and belongs to $\ker D^*$.  Then we can plug $\beta =\alpha$
into the second relation and conclude $\alpha = 0$.

Now we show that $\im D + \ker D^* = W^{k-1,p}(\overline{\Hom}_\CC(T\Sigma,E))$
when $k \ge 2$.  Given $\alpha \in W^{k-1,p}(\overline{\Hom}_\CC(T\Sigma,E))$,
$\alpha$ is also of class $L^p$ and thus the previous step gives
$\eta \in W^{1,p}(E)$ and $\beta \in \ker D^*$ such that
$$
D\eta + \beta = \alpha.
$$
Then $\beta$ is smooth, and $D\eta = \alpha - \beta$ is of class
$W^{k-1,p}$, so regularity (Corollary~\ref{cor:weakRegularity} again) 
implies that $\eta \in W^{k,p}(E)$, and we are done.
\end{proof}

We are now finished with the proof of Theorem~\ref{thm:FredholmProperty},
as we have shown that both $\ker D$ and $\ker D^* \cong \coker D$ 
are finite-dimensional spaces consisting only of smooth sections, which
are thus contained in $W^{k,p}$ for all $k$ and~$p$.

\begin{exercise}
This exercise is meant to convince you that ``boundary conditions are
important.''
Recall that the Calder\'on-Zygmund inequality $\| u \|_{W^{1,p}} \le
c \| \dbar u \|_{L^p}$ is valid for smooth $\CC^n$-valued functions $u$ with
compact support in the open unit ball $B \subset \CC$.  Show that this
inequality cannot be extended to functions without compact support; in fact
there is not even any estimate of the form
$$
\| u \|_{W^{1,p}} \le c \| \dbar u \|_{L^p} + c \| u \|_{L^p}
$$
for general functions $u \in C^\infty(B) \cap W^{1,p}(B)$.  Why not?
For contrast, see Exercise~\ref{EX:disk} below.
\end{exercise}

\section{The Riemann-Roch formula and transversality criteria}

It is easy to see that the index of a Cauchy-Riemann type operator
$D : W^{k,p}(E) \to W^{k-1,p}(\overline{\Hom}_\CC(T\Sigma,E))$ depends
only on the isomorphism class of the bundle $(E,J) \to (\Sigma,j)$.
Indeed, by Exercise~\ref{EX:Christoffel}, the difference
between any two such operators $D$ and $D'$ on the same bundle defines a
smooth real-linear bundle map $A : E \to \overline{\Hom}_\CC(T\Sigma,E)$
such that
$$
D' \eta - D \eta = A \eta.
$$
We often refer to this bundle map as a ``zeroth order term.''  It defines
a bounded linear map from $W^{k,p}(E)$ to
$W^{k,p}(\overline{\Hom}_\CC(T\Sigma,E))$, which is then composed with the
compact inclusion into
$W^{k-1,p}(\overline{\Hom}_\CC(T\Sigma,E))$ and is
therefore a compact operator.  We conclude that all Cauchy-Riemann type
operators on the same bundle are compact perturbations of each 
other,\footnote{This statement is false when $\Sigma$ is not compact:
we'll see when we later discuss Cauchy-Riemann type operators on domains with
cylindrical ends that the zeroth order term is no longer compact, and the
index does depend on the behavior of this term at infinity.}
and thus have the same Fredholm index.  Since complex vector bundles over
a closed surface are classified up to isomorphism by the first Chern
number, the index will therefore depend only
on the topological type of $\Sigma$ and on $c_1(E)$.
To compute it, we can use the fact that every complex bundle admits a 
\emph{complex-linear}
Cauchy-Riemann operator (cf.~Exercise~\ref{EX:nabla}), 
and restrict our attention to the complex-linear case.  Then $E$ is a
holomorphic vector bundle, and $\ker D$ is simply the vector space 
of holomorphic sections.  We'll see below that in some important examples,
it is not hard to compute this space explicitly.  The key observation is
that one can identify holomorphic sections on vector bundles
over $\Sigma$ with complex-valued meromorphic functions on $\Sigma$ that
have prescribed poles and/or zeroes.
The problem of understanding such spaces of meromorphic functions is a
classical one, and its solution is the Riemann-Roch formula.

\begin{thm}[Riemann-Roch formula]
\label{thm:RiemannRoch}
$\ind(D) = n\chi(\Sigma) + 2 c_1(E)$.
\end{thm}

We should emphasize, especially for readers who are more accustomed to
algebraic geometry, that this is the \emph{real} index, i.e.~the difference
between $\dim\ker D$ and $\dim \coker D$ as real vector spaces---these
dimensions may indeed by odd in general since we'll be interested in cases
where $D$ is not complex-linear, but $\ind(D)$ will always be even, a
nontrivial consequence of the fact that $D$ is always 
\emph{homotopic} to a complex-linear operator.
We will later see cases (on punctured Riemann surfaces or surfaces with
boundary) where $\ind(D)$ can also be odd.

A complete proof of the Riemann-Roch formula may be found in
\cite{McDuffSalamon:Jhol}*{Appendix~C} or, from a more classical perspective,
any number of books on Riemann surfaces.  Below we will explain a proof for
the genus~$0$ case and give a heuristic argument to justify the rest.
An important feature will be the following ``transversality'' criterion,
which will also have many important applications in the study of
$J$-holomorphic curves.  It is a consequence of the identification 
$\ker D \equiv \coker D^*$, combined with the similarity principle
(recall \S\ref{sec:similarity}).

\begin{thm}
\label{thm:linearAutomatic}
Suppose $n=1$, i.e.~$(E,J) \to (\Sigma,j)$ is a complex line bundle.
\begin{itemize}
\item If $c_1(E) < 0$, then $D$ is injective.
\item If $c_1(E) > -\chi(\Sigma)$, then $D$ is surjective.
\end{itemize}
\end{thm}
\begin{proof}
The criterion for injectivity is an easy consequence of the similarity
principle, for which we don't really need to know anything about~$D$
except that it's a Cauchy-Riemann type operator.
If $E \to \Sigma$ has complex rank~$1$ and $\ker D$ contains a nontrivial
section $\eta$, then by the similarity principle, $\eta$ has only
isolated (and thus finitely many) zeroes, each of which counts with positive
order.  The count of these computes the first Chern number of $E$, thus
$c_1(E) \ge 0$, and $D$ must be injective if $c_1(E) < 0$.

The second part follows now from the observation that $D$ is surjective
if and only if $D^*$ is injective, and the latter is guaranteed by the
condition $c_1(\widehat{E}) < 0$, which by Prop.~\ref{prop:formalAdjoint}
and Exercise~\ref{EX:c1} is equivalent to $c_1(E) > - \chi(\Sigma)$.
\end{proof}

Observe that we did not need to know the index formula in order to deduce
the last result. In fact, this already gives enough
information to deduce the index formula
in the special case $\Sigma=S^2$, which will be the most important in our
applications.

\begin{proof}[Proof of Theorem~\ref{thm:RiemannRoch} in the case
$\Sigma = S^2$]
We assume first that $n=1$.  In this situation, 
at least one of the criteria $c_1(E) < 0$ or
$c_1(E) > -\chi(\Sigma) = -2$ from Theorem~\ref{thm:linearAutomatic} 
is always satisfied, hence $D$ is always injective or surjective; in fact if
$c_1(E) = -1$ it is an isomorphism.  By considering $D^*$ instead of
$D$ if necessary, we can restrict our attention to the case where
$D$ is surjective, so $\ind D = \dim\ker D$.  We will now construct for
each value of $c_1(E) \ge 0$ a ``model'' holomorphic line bundle, which is
sufficiently simple so that we can identify the space of holomorphic sections
explicitly.

For the case $c_1(E) = 0$, the model bundle is obvious: just take the 
trivial line bundle
$S^2 \times \CC \to S^2$, so the holomorphic sections are holomorphic
functions $S^2 \to \CC$, which are necessarily constant and therefore
$\dim \ker D = 2$, as it should be.  A more general model bundle can 
be defined by gluing
together two local trivializations: let $E^{(1)}$ and $E^{(2)}$ denote two 
copies of the trivial holomorphic line bundle $\CC \times \CC \to \CC$, 
and for any $k \in \ZZ$, define
$$
E_k := (E^{(1)} \sqcup E^{(2)}) / (z,v) \sim \Phi_k(z,v),
$$
where $\Phi_k : E^{(1)}|_{\CC\setminus\{0\}} \to E^{(2)}|_{\CC\setminus\{0\}}$ 
is a bundle isomorphism covering the biholomorphic map $z \mapsto 1/z$ and
defined by $\Phi_k(z,v) = (1/z, g_k(z) v)$, with
$$
g_k(z) v := \frac{1}{z^k} v.
$$
The function $g_k(z)$ is a holomorphic transition map, thus $E_k$ has a
natural holomorphic structure.  Regarding a function $f : \CC \to \CC$
as a section of $E^{(1)}$, we have
$$
\Phi_k(1 / z, f(1 / z)) = (z , z^k f(1/z)),
$$
which means that $f$ extends to a smooth
section of $E_k$ if and only if the function $g(z) = z^k f(1/z)$ extends
smoothly to $z=0$.  It follows that
$c_1(E_k) = k$, as one can choose $f(z) = 1$ for $z$ in the unit disk and
then modify $g(z) = z^k$ to a smooth function that algebraically
has $k$ zeroes at~$0$ (note that an actual modification is necessary only
if $k < 0$).
Similarly, the holomorphic sections of $E_k$ can be identified with the
entire functions $f : \CC \to \CC$ such that $z^k f(1/z)$ extends 
holomorphically to $z=0$; if $k < 0$ this implies $f \equiv 0$, and if
$k \ge 0$ it means $f(z)$ is a polynomial of degree at most~$k$,
hence $\dim \ker D = 2 + 2k$.  The proof of the index
formula for $\Sigma = S^2$ and $n=1$ is now complete.

The case $n \ge 2$ can easily be derived from the above.  It suffices to
prove that $\ind(D) = 2n + 2 c_1(E)$ for some model holomorphic bundle
of rank~$n$ with a given value of~$c_1(E)$.  Indeed, for any $k \in \ZZ$,
take $E$ to be the direct sum of $n$ holomorphic line bundles,
$$
E := E_{-1} \oplus \ldots \oplus E_{-1} \oplus E_k,
$$
which has $c_1(E) = k - (n-1)$.  By construction, the natural 
Cauchy-Riemann operator $D$ on $E$ splits into a direct sum of Cauchy-Riemann
operators on its summands, and it is an isomorphism on each of the
$E_{-1}$ factors, thus we conclude as in the line bundle case
that $D$ is injective if $k < 0$ and surjective if $k \ge 0$.  By replacing
$D$ with $D^*$ if necessary, we can now assume without loss of generality
that $k \ge 0$ and $D$ is surjective.  The space of holomorphic sections
is then simply the direct sum of the corresponding spaces for its summands,
which are trivial for $E_{-1}$ and have dimension $2 + 2k$ for $E_k$.
We therefore have
$$
\ind(D) = \dim \ker D = 2 + 2k = 2n + 2[ k - (n-1)] = n\chi(\Sigma) +
2c_1(E).
$$
\end{proof}

One should not conclude from the above proof that every Cauchy-Riemann type
operator on the sphere is either injective or surjective, which is true on
line bundles but certainly not for bundles of higher rank---above we only
used the fact that for every value of $c_1(E)$, \emph{one can construct} 
a bundle that has this property.  The proof is not so simple for general
Riemann surfaces because it is less straightforward to identify spaces of
holomorphic sections.  One lesson to be drawn from the above argument, however,
is that holomorphic sections on a line bundle with $c_1(E) = k$ can also
be regarded as holomorphic sections on some related bundle 
with $c_1(E) = k + 1$, but with an extra zero at some chosen point.  
This suggests that an increment in the value of $c_1(E)$ should
also enlarge the space of holomorphic sections by two real dimensions, because
one can add two linearly independent sections that do not vanish at
the chosen point.  What's true for line bundles in this sense is also 
true for bundles of higher rank, because one can always construct 
model bundles that are direct sums of line bundles.
We will not attempt to make this argument precise, but
it should give some motivation to believe that $\ind(D)$ scales with
$2 c_1(E)$: to be exact, there exists a constant $C = C(\Sigma,n)$ such that
$$
\ind(D) = C(\Sigma,n) + 2 c_1(E).
$$
If you believe this, then we can already deduce the general Riemann-Roch
formula by comparing $D$ with its formal adjoint.  Indeed, $D^*$ has index
$$
\ind(D^*) = C(\Sigma,n) + 2 c_1(\widehat{E}) = C(\Sigma,n) - 
2 c_1(E) - 2 n \chi(\Sigma)
$$
according to Exercise~\ref{EX:c1}, and since $\coker D = \ker D^*$
and vice versa, $\ind(D^*) = -\ind(D)$.  Thus 
adding these formulas together yields
$$
0 = 2 C(\Sigma,n) - 2 n \chi(\Sigma).
$$
We conclude $C(\Sigma,n) = n\chi(\Sigma)$, and the Riemann-Roch formula
follows.

With the index formula understood, we can derive some alternative formulations
of the transversality criteria in Theorem~\ref{thm:linearAutomatic} which
will often be useful.  First, compare the formulas for
$\ind(D)$ and $\ind(D^*)$:
\begin{equation*}
\begin{split}
\ind(D) &= \chi(\Sigma) + 2 c_1(E),\\
\ind(D^*) &= \chi(\Sigma) + 2 c_1(\widehat{E}),
\end{split}
\end{equation*}
where $\widehat{E}$ is the line bundle constructed in the proof of
Prop.~\ref{prop:formalAdjoint}.  Since $\ind(D) = -\ind(D^*)$,
subtracting the second formula from the first yields
$$
\ind(D) = c_1(E) - c_1(\widehat{E}),
$$
and thus $c_1(\widehat{E}) < 0$ if and only if $\ind(D) > c_1(E)$,
which implies by Theorem~\ref{thm:linearAutomatic} that $D^*$ is injective
and thus $D$ is surjective.  We state this as a corollary.

\begin{cor}
\label{cor:linearAutomatic}
If $n=1$ and $\ind(D) > c_1(E)$, then $D$ is surjective.
\end{cor}

\begin{exercise}
Show that another equivalent formulation of Theorem~\ref{thm:linearAutomatic}
for Cauchy-Riemann operators on complex line bundles
is the following:
\begin{itemize}
\item If $\ind(D) < \chi(\Sigma)$ then $D$ is injective.
\item If $\ind(D) > -\chi(\Sigma)$ then $D$ is surjective.
\end{itemize}
This means that for line bundles, $D$ is always surjective (or injective)
as soon as its index is large (or small) enough.  Observe that when
$\Sigma = S^2$, one of these conditions is always satisfied, but there is
always an ``interval of uncertainty'' in the higher genus case.
\end{exercise}

A different approach to the proof of Riemann-Roch, which is taken in
\cite{McDuffSalamon:Jhol}, is to cut up $E \to \Sigma$ into simpler pieces
on which the index can be computed explicitly, and then conclude the
general result by a ``linear gluing argument''.  We'll come back to this
idea in a later chapter when we discuss the generalization of the 
Riemann-Roch formula to open surfaces with cylindrical ends.
The proof in \cite{McDuffSalamon:Jhol} instead considers Cauchy-Riemann
operators on surfaces with boundary and
totally real boundary conditions: the upshot is that the problem can 
be reduced in this way to the following
exercise, in which one computes the index
for the standard Cauchy-Riemann operator on a closed disk.

\begin{exercise}
\label{EX:disk}
Let $\DD \subset \CC$ denote the closed unit disk and $E$ the trivial
bundle $\DD \times \CC \to \DD$.  For a given 
integer $\mu \in \ZZ$, define a real rank~$1$ subbundle 
$\ell_\mu \subset E|_{\p \DD}$ by
$$
(\ell_\mu)_{e^{i\theta}} = e^{i\pi \mu\theta}\RR \subset \CC.
$$
We call $\ell_\mu$ in this context a \defin{totally real} subbundle of
$\DD \times \CC$ at the boundary, and the integer $\mu$ is its \defin{Maslov index}.
Let $\dbar = \p_s + i \p_t$, and for $kp > 2$ consider the operator
$$
\dbar : W^{k,p}_{\ell_\mu}(\DD,\CC) \to W^{k-1,p}(\DD,\CC),
$$
where the domain is defined by
$$
W^{k,p}_{\ell_\mu}(\DD,\CC) = \{ \eta \in W^{k,p}(\DD,\CC)\ |\ \eta(\p\DD) \subset
\ell_\mu \}.
$$
Show that as an operator between these particular spaces,
$\ker\dbar$ has dimension $1 + \mu = \chi(\DD) + \mu$ if $\mu \ge -1$,
and $\dbar$ is injective if $\mu \le -1$.
(You may find it helpful to think in terms of Fourier series.)  
By constructing the appropriate formal adjoint of $\dbar$ in this setting
(which will also satisfy a totally real boundary condition), one can
also show that $\dbar$ is surjective if $\mu \ge -1$, and one can
similarly compute the kernel of the formal adjoint if $\mu \le -1$,
concluding that $\dbar$ is in fact Fredholm and has index
$\ind(\dbar) = \chi(\DD) + \mu$.  By considering direct sums of line bundles
with totally real boundary conditions, this generalizes easily to 
bundles of general rank $n \in \NN$ as
$$
\ind(\dbar) = n\chi(\DD) + \mu.
$$
One should think of this as another instance of the Riemann-Roch formula,
in which the Maslov index now plays the role of~$2 c_1(E)$.  The
details are carried out in \cite{McDuffSalamon:Jhol}*{Appendix~C}.
\end{exercise}

As a final remark, we note that the Fredholm theory of Cauchy-Riemann
operators gives a new proof of a local regularity result that we made
much use of in Chapter~\ref{chapter:local}: 
the standard $\dbar$-operator on the open unit ball $B \subset \CC$,
$$
\dbar : W^{k,p}(B,\CC^n) \to W^{k-1,p}(B,\CC^n)
$$
has a bounded right inverse (see Theorem~\ref{thm:rightInverse}).
This follows from our proof of
Theorem~\ref{thm:RiemannRoch} in the case $\Sigma=S^2$, because any
$f \in W^{k-1,p}(B,\CC^n)$ can be extended to a section in
$W^{k-1,p}(\overline{\Hom}_\CC(T S^2,S^2 \times \CC^n))$, and we can then
use the fact that the standard Cauchy-Riemann operator on the trivial
bundle $S^2 \times \CC^n$ is a surjective Fredholm operator, its kernel
consisting of the constant sections.  Alternatively, one can use the
fact established by Exercise~\ref{EX:disk}, that the restriction of
$\dbar$ to the domain $W^{k,p}_{\ell_0}(\DD,\CC^n)$ of functions with 
the totally real boundary condition $\eta(\p\DD) \subset \RR^n$ is
a surjective Fredholm operator with index~$n$; its kernel is again the
space of constant functions.


\chapter{Moduli Spaces}
\label{chapter:moduli}

\minitoc
\vspace{12pt}

\section{The moduli space of closed $J$-holomorphic curves}
\label{sec:moduliSpace}

In the previous chapter we considered the local structure of the space of
$J$-holomorphic maps $(\Sigma,j) \to (M,J)$ from a fixed closed 
Riemann surface
to a fixed almost complex manifold of dimension~$2n$.  From a
geometric point of view, this is not the most natural space to study:
geometrically, we prefer to picture holomorphic curves as 
$2$-dimensional submanifolds\footnote{This description is of course only
strictly correct for holomorphic curves that are embedded, which they need not be 
in general---though we'll see that in many important applications, they are.}
whose tangent spaces are invariant under the
action of~$J$.  In the symplectic context in particular, this means they 
give rise to symplectic submanifolds.  From this perspective, 
the interesting object is not the
parametrization $u$ but its image $u(\Sigma)$, thus we should regard all
reparametrizations of $u$ to be equivalent.  Moreover, the choice of
parametrization fully determines $j = u^*J$, thus one cannot choose 
$j$ in advance, but must allow it to vary over the space of all complex
structures on~$\Sigma$.  The interesting solution space is therefore
the following.

\begin{defn}
Given an almost complex manifold $(M,J)$ of real dimension $2n$,
integers $g, m \ge 0$ and a homology class $A \in H_2(M)$, we
define the \defin{moduli space of $J$-holomorphic curves in~$M$ 
with genus~$g$ and $m$ marked points representing~$A$} to be
$$
\Mod_{g,m}^A(J) = \{ (\Sigma,j,u,(z_1,\ldots,z_m)) \} / \sim,
$$
where $(\Sigma,j)$ is any closed connected Riemann surface of genus~$g$,
$u : (\Sigma,j) \to (M,J)$ is a pseudoholomorphic map with 
$[u] := u_*[\Sigma] = A$, and $(z_1,\ldots,z_m)$ is an ordered set of
distinct points in~$\Sigma$, which we'll often denote by
$$
\Theta = (z_1,\ldots,z_m).
$$
We say $(\Sigma,j,u,\Theta) \sim (\Sigma',j',u',\Theta')$ if and
only if there exists a biholomorphic diffeomorphism $\varphi :
(\Sigma,j) \to (\Sigma',j')$ such that $u = u' \circ \varphi$ and
$\varphi(\Theta) = \Theta'$ with the ordering preserved.
\end{defn}

We will often abbreviate the union of all these moduli spaces by
$$
\Mod(J) = \bigcup_{g,m,A} \Mod_{g,m}^A(J).
$$
Elements of $\Mod(J)$ are sometimes called
\defin{unparametrized} $J$-holomorphic curves, since the choice of
parametrization $u : \Sigma \to M$ is considered auxiliary.
We will nonetheless sometimes abuse the notation by writing an 
equivalence class of tuples
$[(\Sigma,j,u,\Theta)]$ simply as $(\Sigma,j,u,\Theta)$ or
$u \in \Mod(J)$ when there is no danger of confusion.  The significance of
the marked points $\Theta = (z_1,\ldots,z_m)$ is that they give rise
to a well-defined \defin{evaluation map}
\begin{equation}
\label{eqn:ev}
\ev = (\ev_1,\ldots,\ev_m) : \Mod_{g,m}^A(J) \to M \times \ldots \times M,
\end{equation}
where $\ev_i$ takes $[(\Sigma,j,u,\Theta)]$ to $u(z_i) \in M$ for each $i=1,\ldots,m$.
One can use this to find relations between the topology of $M$ and the structure
of the moduli space, which will be important in later applications to 
symplectic geometry.

A natural topology on $\Mod(J)$ can be defined via the following notion
of convergence: we say $[(\Sigma_k,j_k,u_k,\Theta_k)] \to
[(\Sigma,j,u,\Theta)]$ if for sufficiently large $k$, the sequence has
representatives of the form $(\Sigma,j_k',u_k',\Theta)$ such that
$j_k' \to j$ and $u_k' \to u$ in the $C^\infty$-topology.  In particular,
$\Sigma_k$ must be diffeomorphic to $\Sigma$ and have the same number of
marked points for sufficiently large~$k$;
observe that when this is the case, one can always choose a diffeomorphism
to fix the positions of the marked points.  In this topology,
$\Mod_{g,m}^A(J)$ and $\Mod_{g',m'}^{A'}(J)$ for distinct triples
$(g,m,A)$ and $(g',m',A')$ form distinct components of $\Mod(J)$, 
each of which may or may not be connected.

The main goal
of this chapter will be to show that under suitable hypotheses, a certain
subset of $\Mod(J)$ is a smooth finite-dimensional manifold, with various
dimensions on different components.  Its
``expected'' or \emph{virtual} dimension on the component containing a given
curve $u \in \Mod_{g,m}^A(J)$ is essentially a Fredholm index
with some correction terms, and depends on the topological 
data $g$, $m$ and~$A$.  We'll use the convenient abbreviation,
$$
c_1(A) = \langle c_1(TM,J) , A \rangle.
$$

\begin{defn}
If $\dim_\RR M = 2n$, define the \defin{virtual dimension} of the moduli space
$\Mod_{g,m}^A(J)$ to be the integer
\begin{equation}
\label{eqn:virdim}
\virdim \Mod_{g,m}^A(J) = (n-3) (2 - 2g) + 2 c_1(A) + 2m.
\end{equation}
For a curve $u \in \Mod_{g,0}^A(J)$ without marked points, this number is
also called the \defin{index} of $u$ and denoted by
\begin{equation}
\label{eqn:index}
\ind(u) := \virdim \Mod_{g,0}^A(J) = (n-3) (2 - 2g) + 2 c_1(A).
\end{equation}
\end{defn}

It is both interesting and important to consider the special case where
$M$ is a single point: then $\Mod_{g,m}^A(J)$ reduces to the
\defin{moduli space of Riemann surfaces} with genus~$g$ and $m$ marked points:
$$
\Mod_{g,m} = \{ (\Sigma,j,(z_1,\ldots,z_m)) \} / \sim,
$$
with the equivalence and topology defined the same as above (all statements
involving the map $u$ are now vacuous).  The elements
$(\Sigma,j,\Theta) \in \Mod_{g,m}$ are called \defin{pointed Riemann
surfaces}, and each comes with an \defin{automorphism group}
$$
\Aut(\Sigma,j,\Theta) = \big\{ \varphi : (\Sigma,j) \to (\Sigma,j)
\text{ biholomorphic} \ \big|\ \varphi|_{\Theta} = \Id \big\}.
$$
Similarly, a $J$-holomorphic curve $(\Sigma,j,u,\Theta) \in \Mod(J)$
has an automorphism group
$$
\Aut(u) := \Aut(\Sigma,j,\Theta,u) := 
\{ \varphi \in \Aut(\Sigma,j,\Theta)\ |\ u = u\circ \varphi \}.
$$
It turns out that in understanding the local structure of $\Mod(J)$,
a special role is played by holomorphic curves with trivial automorphism
groups.  The following simple result was proved as
Theorem~\ref{thm:multipleCovers} in Chapter~\ref{chapter:local}, and it
implies (via Exercise~\ref{EX:trivialAut}
below) that whenever any nontrivial 
holomorphic curves exist, one can also find curves with trivial
automorphism group.

\begin{prop}
\label{prop:multipleCovers}
For any closed, connected and nonconstant $J$-holomorphic curve
$u : (\Sigma,j) \to (M,J)$, there exists a factorization
$u = v \circ \varphi$ where
\begin{itemize}
\item
$v : (\Sigma',j') \to (M,J)$ is a closed $J$-holomorphic curve that is
embedded outside a finite set of critical points and self-intersections, and
\item
$\varphi : (\Sigma,j) \to (\Sigma',j')$ is a holomorphic map of
degree~$\deg(\varphi) \ge 1$.
\end{itemize}
Moreover, $v$ is unique up to biholomorphic reparametrization.  \qed
\end{prop}

\begin{defn}
The degree of $\varphi : \Sigma \to \Sigma'$ in Prop.~\ref{prop:multipleCovers}
is called the \defin{covering number} or \defin{covering multiplicity} of~$u$.
If this is~$1$, then we say $u$ is \defin{simple}.
\end{defn}

\begin{defn}
\label{defn:somewhereInj}
Given a smooth map $u : \Sigma \to M$, a point $z \in \Sigma$ is called
an \defin{injective point} for~$u$ if $du(z) : T_z \Sigma \to T_{u(z)} M$ 
is injective and
$u^{-1}(u(z)) = \{z\}$.  The map $u$ is called \defin{somewhere injective}
if it has at least one injective point.
\end{defn}

Proposition~\ref{prop:multipleCovers} implies that a closed connected
$J$-holomorphic curve is somewhere injective if and only if it is
simple.  (For a word of caution about this statement, see
Remark~\ref{remark:Achtung2} below.)  We denote by
$$
\Mod^*(J) \subset \Mod(J)
$$
the open subset consisting of all curves in $\Mod(J)$ that are somewhere 
injective.  It will also be useful to generalize this as follows: given
an open subset $\uU \subset M$, define the open subset
$$
\Mod^*_\uU(J) = \{ u \in \Mod(J)\ |\ \text{$u$ has an injective point mapped
into~$\uU$} \}.
$$

\begin{exercise}
\label{EX:trivialAut}
Show that if $u : (\Sigma,j) \to (M,J)$ is somewhere injective then $\Aut(u)$ is trivial
(for any choice of marked points).
\end{exercise}
\begin{exercise}
Show that if $u : (\Sigma,j) \to (M,J)$ has covering multiplicity~$k \in \NN$
then for any set of marked points $\Theta$, the order of
$\Aut(\Sigma,j,\Theta,u)$ is at most~$k$.
\end{exercise}

Recall that a subset $Y$ in a
complete metric space $X$ is called a \defin{Baire subset} or said to be of 
\defin{second category} if it is a countable intersection of open dense 
sets.\footnote{While this usage of the terms ``Baire subset'' and ``second category''
is considered standard among symplectic topologists, the reader should
beware that it is slightly at odds with the usage in other fields.
For instance, \cite{Royden} and other standard references define a
subset $Y \subset X$
to be of second category (or \emph{nonmeager}) if and only if it is not
of first category (or \emph{meager}), where the latter means $Y$ is a
countable union of nowhere dense sets and thus is the \emph{complement} of
what we are calling a Baire subset.  Thus it would be better in principle to
say \emph{comeager} instead of ``Baire'' or ``second category''---but 
I will not attempt to change the habits of the symplectic community
single-handedly.}
The Baire category theorem implies that such subsets are also dense, and
Baire subsets are often used to define an infinite-dimensional version of the
term ``almost everywhere,'' i.e.~they are analogous to sets whose complements
have Lebesgue measure zero.  It is common to say that a property is satisfied
for \defin{generic} choices of data if the set of all possible data contains
a Baire subset for which the property is satisfied.  

Since it is important for applications, we shall assume throughout this
chapter that~$M$ carries a symplectic structure~$\omega$, and focus our attention
on the space of \emph{$\omega$-compatible} almost complex structures
$\jJ(M,\omega)$ that was defined in \S\ref{sec:compatible}; see
Remark~\ref{remark:symplecticGeneric} below on why this is not actually
a restriction.  We will also allow the following generalization:
given $\Jfix \in \jJ(M,\omega)$ and an open subset $\uU \subset M$, define
$$
\jJ(M,\omega\,;\,\uU,\Jfix) = \{ J \in \jJ(M,\omega) \ |\ 
\text{$J = \Jfix$ on $M \setminus \uU$} \}.
$$
If $\uU$ has compact closure, then this space carries a natural 
$C^\infty$-topology and is a Fr\'{e}chet manifold.\footnote{We are not
justifying the claim that it is a Fr\'echet manifold because we will not
need to use it, but this is not hard to prove using the local charts for
$\jJ(\CC^n)$ defined in \S\ref{sec:compatible}, together with a bit of
infinite-dimensional calculus from \S\ref{sec:calculus}.
In \S\ref{subsec:dense} we will make use of certain related spaces which 
are Banach manifolds.}
In the following sections we
will prove several results which, taken together, imply the following
local structure theorem.  Note that the important special case
$\uU = M$ is allowed, and in this case the choice of $\Jfix$ is
irrelevant.

\begin{thm}
\label{thm:localStructure}
Suppose $(M,\omega)$ is a symplectic manifold without 
boundary, $\uU \subset M$ is an 
open subset with compact closure, and $\Jfix \in \jJ(M,\omega)$.
Then there exists a Baire subset $\jJ_\reg(M,\omega\,;\,\uU,\Jfix) \subset 
\jJ(M,\omega\,;\,\uU,\Jfix)$ such that for every
$J \in \jJ_\reg(M,\omega\,;\,\uU,\Jfix)$, the space 
$\Mod^*_\uU(J)$ of $J$-holomorphic curves with injective points mapped into~$\uU$ 
naturally admits the structure of a smooth finite-dimensional 
manifold, and the evaluation map on this space is smooth.
The dimension of $\Mod^*_\uU(J) \cap \Mod_{g,m}^A(J)$ for any
$g,m \ge 0$ and $A \in H_2(M)$ is precisely the
virtual dimension of $\Mod_{g,m}^A(J)$.
\end{thm}

Note that $M$ in the above statement need not be compact, but $\uU$ must
have compact closure.  In the case where $M$ is compact and
$\uU = M$, we will denote the space $\jJ_\reg(M,\omega\,;\,\uU,\Jfix)$
simply by $\jJ_\reg(M,\omega)$.

\begin{remark}
\label{remark:symplecticGeneric}
The above theorem and all other important results in this chapter
remain true if $\jJ(M,\omega)$ is replaced by the spaces of
$\omega$-tame or general almost complex structures $\jJ^\tau(M,\omega)$ or
$\jJ(M)$; in fact, the equivalence of these last two
variations is obvious since $\jJ^\tau(M,\omega)$ is an open subset
of~$\jJ(M)$.  The symplectic structure 
will play no role whatsoever in the proofs
except to make one detail slightly harder (see Lemma~\ref{lemma:choice}),
thus it will be immediate that minor alterations of the same proofs
imply the same results for tame or general almost complex structures.
\end{remark}

One of the important consequences of Theorem~\ref{thm:localStructure}
is that for generic 
choices of~$J$, every connected component of the moduli space
$\Mod^*(J)$ must have nonnegative virtual dimension, as a smooth manifold
of negative dimension is empty by definition.  Put another way,
if a somewhere injective curve of negative index exists, then one can 
always eliminate it
by a small perturbation of~$J$.  We state this as a corollary.

\begin{cor}
\label{cor:nonnegativeIndex}
If $J \in \jJ_\reg(M,\omega\,;\,\uU,\Jfix)$, then every curve
$u \in \Mod(J)$ that maps an injective point into $\uU$ satisfies
$\ind(u) \ge 0$.
\end{cor}

\begin{remark}
\label{remark:Achtung2}
By Proposition~\ref{prop:multipleCovers}, a closed $J$-holomorphic curve $u$
maps an injective point into an open set $\uU$ if and only if $u$ is simple
and intersects~$\uU$.  It should be noted however that the equivalence of
``simple'' and ``somewhere injective'' does not always hold in more
general contexts, e.g.~for holomorphic curves with totally real boundary
\cites{Lazzarini:disks,KwonOh}; in such cases, 
Corollary~\ref{cor:nonnegativeIndex} generalizes in the form stated.
\end{remark}

An important related problem is to consider the space of
$J_s$-holomorphic curves, where $\{J_s\}$ is a smooth homotopy of
almost complex structures.  Suppose $\{\omega_s\}_{s \in [0,1]}$ is a smooth
homotopy of symplectic forms on a closed manifold $M$, and given
$J_0 \in \jJ(M,\omega_0)$ and $J_1 \in \jJ(M,\omega_1)$, define
$$
\jJ(M,\{\omega_s\} \,;\, J_0,J_1)
$$
to be the space of all smooth $1$-parameter families
$\{ J_s \}_{s \in [0,1]}$ connecting $J_0$ to $J_1$ such that 
$J_s \in \jJ(M,\omega_s)$ for all~$s$.
One can similarly define the spaces $\jJ^\tau(M ,\omega \,;\, J_0,J_1)$ 
and $\jJ(M \,;\, J_0,J_1)$ of $\omega$-tame or general $1$-parameter
families respectively, or more general spaces of structures that are fixed
outside an open subset $\uU \subset M$ with compact closure (in which case
$M$ need not be closed).  All of these spaces have natural 
$C^\infty$-topologies.  Given
$\{J_s\} \in \jJ(M,\{\omega_s\} \,;\, J_0,J_1)$, we define
the ``parametrized'' moduli space,
$$
\Mod(\{ J_s \}) = \{ (s,u) \ |\ 
s \in [0,1], \ u \in \Mod(J_s) \},
$$
along with the corresponding space of somewhere injective curves
$\Mod^*(\{ J_s \})$ and the components $\Mod_{g,m}^A(\{ J_s \})$
for each $g,m \ge 0$, $A \in H_2(M)$.  These also have natural topologies,
and intuitively, we expect $\Mod^*(\{ J_s \})$ to be a manifold with
boundary $\Mod^*(J_0) \sqcup \Mod^*(J_1)$.
The only question is what should be the
proper notion of ``genericity'' to make this statement correct.  Given a
homotopy $\{ J_s \} \in \jJ(M,\{\omega_s\} \,;\, J_0,J_1)$ where 
$J_0 \in \jJ_\reg(M,\omega_0)$ and $J_1 \in \jJ_\reg(M,\omega_1)$,
it would be too much to hope that one can always perturb $\{J_s\}$ so
that $J_s \in \jJ_\reg(M,\omega_s)$ for every~$s$; by analogy with the case
of smooth Morse functions on a manifold, any two Morse functions are indeed
smoothly homotopic, but not through a family of Morse functions.
What is true however is that one
can find ``generic homotopies,'' for which $J_s \in \jJ_\reg(M,\omega_s)$
for almost every $s \in (0,1)$, and $\Mod^*(\{J_s\})$
is indeed a manifold.  We will not prove the following
result explicitly, but the proof is an easy exercise after the proof of
Theorem~\ref{thm:localStructure} is understood.  

\begin{thm}
\label{thm:genericHomotopy}
Assume $M$ is a closed manifold with a smooth $1$-parameter family
$\{\omega_s\}_{s \in [0,1]}$ of symplectic forms,
$J_0 \in \jJ_\reg(M,\omega_0)$ and $J_1 \in \jJ_\reg(M,\omega_1)$.  Then 
there exists a Baire subset $\jJ_\reg(M,\{\omega_s\} \,;\, J_0,J_1) \subset 
\jJ(M,\{\omega_s\} \,;\, J_0,J_1)$ 
such that for every $\{J_s\} \in \jJ_\reg(M,\{\omega_s\} \,;\, J_0,J_1)$, 
the parametrized
space of somewhere injective curves $\Mod^*(\{J_s\})$ admits the structure of 
a smooth finite-dimensional manifold with boundary
$$
\p \Mod^*(\{J_s\}) = \left( \{0\} \times \Mod^*(J_0) \right)
\sqcup \left( \{1\} \times \Mod^*(J_1) \right).
$$
Its dimension near any $(s,u) \in \Mod^*(\{J_s\})$ with
$u \in \Mod_{g,m}^A(J_s)$ is $\virdim \Mod_{g,m}^A(J_s) + 1$.
Moreover, for each $s \in [0,1]$ at which
$J_s \in \jJ_\reg(M,\omega_s)$, $s$ is a regular value of the natural
projection $\Mod^*(\{J_s\}) \to [0,1] : (s,u) \mapsto s$.
\end{thm}

\begin{cor}
\label{cor:nonnegativeParametrized}
For generic homotopies of compatible almost complex structures
$\{J_s\} \in \jJ_\reg(M,\{\omega_s\} \,;\, J_0,J_1)$ in the setting
of Theorem~\ref{thm:genericHomotopy}, every somewhere
injective curve $u \in \Mod(J_s)$ for any $s \in [0,1]$ satisfies
$\ind(u) \ge -1$.
\end{cor}
\begin{remark}
\label{remark:even}
The result of Corollary~\ref{cor:nonnegativeParametrized} can actually
be improved to $\ind(u) \ge 0$ due to the numerical coincidence that 
according to \eqref{eqn:index}, $\ind(u)$ is always an \emph{even} number.
This observation is sometimes quite useful in applications, but it
fails to hold in more general settings, e.g.~as
we will see in later chapters, moduli spaces of punctured holomorphic
curves in symplectic cobordisms can have odd dimension, in which case
the natural generalization of Corollary~\ref{cor:nonnegativeParametrized}
as stated above is usually the best result possible.
\end{remark}

\begin{remark}
\label{remark:adapt}
Obvious generalizations of Theorem~\ref{thm:genericHomotopy} and
Corollary~\ref{cor:nonnegativeParametrized} also hold
for $\omega$-tame or general almost complex structures, and for
structures fixed outside an open precompact subset $\uU$ (with curves
required to have injective points in~$\uU$).  This generalization
requires no significantly new ideas outside of what we will describe
in the proof of Theorem~\ref{thm:localStructure}.
\end{remark}

The intuition behind Theorems~\ref{thm:localStructure} and
\ref{thm:genericHomotopy} is roughly as follows.
As we've already seen, spaces of $J$-holomorphic curves typically can be
described, at least locally, as zero sets of sections of certain Banach space 
bundles, and we'll show in \S\ref{sec:IFT} precisely how to set up the 
appropriate section
$$
\dbar_J : \bB \to \eE
$$
whose zero set locally describes $\Mod_{g,m}^A(J)$.  The identification
between $\dbar_J^{-1}(0)$ and $\Mod_{g,m}^A(J)$ near a given curve
$u \in \Mod_{g,m}^A(J)$ will in general be locally
$k$-to-$1$, where $k$ is the order of the automorphism group $\Aut(u)$, 
and this means that even if $\dbar_J^{-1}(0)$ is a manifold, 
$\Mod_{g,m}^A(J)$ is at best an orbifold.  This is a moot point of course
if $u$ is somewhere injective, since it then has a trivial automorphism
group by Exercise~\ref{EX:trivialAut}.  Thus once the section $\dbar_J$ is
set up, the main task is to show that generic choices of $J$ make
$\dbar_J^{-1}(0)$ a manifold (at least near the somewhere injective curves), 
which means
showing that the linearization of $\dbar_J$ is always surjective.  This is
a question of \emph{transversality}, i.e.~if we regard $\dbar_J$ as an
embedding of $\bB$ into the total space $\eE$ and denote the zero section
by $\zZ \subset \eE$, then $\dbar_J^{-1}(0)$ is precisely the intersection,
$$
\dbar_J(\bB) \cap \zZ,
$$
and it will be a manifold if this intersection is everywhere transverse.
Intuitively, one expects this to be true after a generic perturbation of
$\dbar_J$, and it remains to check whether the most geometrically
natural perturbation, defined by perturbing~$J$, is ``sufficiently generic''
to achieve this.

The answer is yes and no: it turns out that perturbations of $J$ are
sufficiently generic if we only consider somewhere injective curves, but not 
for multiple covers.  It's not hard to see why transversality \emph{must}
sometimes fail: if $\tilde{u}$ is a multiple cover of $u$, then even if 
$\Mod(J)$ happens to be a manifold near~$u$, 
there are certain obvious relations between the components of $\Mod(J)$
containing $u$ and $\tilde{u}$ that will often cause the latter
to have ``the wrong'' dimension, i.e.~something other
than~$\ind(\tilde{u})$.  For example,
suppose $n=4$, so $M$ is $8$-dimensional, and for some
$J \in \jJ_\reg(M,\omega)$ there
exists a simple $J$-holomorphic sphere $u \in \Mod_{0,0}^A(J)$ with
$c_1(A) = -1$.  Then by \eqref{eqn:index}, $\ind(u) = 0$, and
Theorem~\ref{thm:localStructure} implies that the component of $\Mod^*(J)$ 
containing $u$ is a smooth
$0$-dimensional manifold, i.e.~a discrete set.  In fact, the implicit
function theorem implies much more (cf.~Theorem~\ref{thm:IFTparametrized}):
it implies that for any
other $J_\epsilon \in \jJ(M,\omega)$ sufficiently close to~$J$, there is a unique
$J_\epsilon$-holomorphic curve $u_\epsilon$ that is a small perturbation 
of~$u$.  Now for each of these curves and some $k \in \NN$, 
consider the $k$-fold cover
$$
\tilde{u}_\epsilon : S^2 \to M : z \mapsto u_\epsilon(z^k), 
$$
where as usual $S^2$ is identified with the extended complex plane, so that
$z \mapsto z^k$ defines a $k$-fold holomorphic branched cover $S^2 \to S^2$.
We have $[\tilde{u}_\epsilon] = k[u_\epsilon] = kA$, and thus
$$
\ind(\tilde{u}_\epsilon) = (n - 3)\chi(S^2) + 2 c_1(kA) = 2 - 2k,
$$
so if $k \ge 2$ then $\tilde{u}_\epsilon$ are $J_\epsilon$-holomorphic
spheres with negative index.  By construction, these cannot be
``perturbed away'': they exist for all $J_\epsilon$ sufficiently close to~$J$,
which shows that perturbations of $J$ do not suffice to make $\Mod(J)$ 
into a smooth manifold of the right dimension near~$\tilde{u}$.
In this situation it is not even
clear if $\Mod(J)$ is a manifold near $\tilde{u}$ at all---in a few lucky 
situations one might be able to prove this, but it is not true in general.

The failure of Theorems~\ref{thm:localStructure} and~\ref{thm:genericHomotopy}
for multiply covered
$J$-holomorphic curves is one of the great headaches of symplectic topology,
and the major reason why fully general definitions of the various invariants
based on counting holomorphic curves (Gromov-Witten theory, Floer homology,
Symplectic Field Theory) are often so technically difficult as to be 
controversial.  There have been many suggested approaches to the problem,
most requiring the introduction of complicated new structures,
e.g.~virtual moduli cycles, Kuranishi structures, polyfolds.
In some fortunate situations one can avoid these complications by
using topological constraints to rule out the appearance of any
multiple covers in the moduli space of interest---we'll see examples of this
in our applications, especially in dimension four.

\begin{remark}
\label{remark:compactSubset}
As indicated above, we normally will not need to assume $M$ is compact
in this discussion, but the region $\uU$ where we permit perturbations of
the almost complex structure is required to have compact closure.
This restriction is useful for various technical reasons, e.g.~it 
makes it relatively straightforward
to define Banach manifolds in which the perturbed almost complex structures
live; without this assumption, one can still do something, but it requires
considerably more care.

Here is an important class of examples where
$M$ is noncompact: suppose $M$ is a symplectic cobordism with
cylindrical ends, in which case it can be decomposed as
$$
M = ( (-\infty,0] \times V_-) \cup M_0 \cup ( [0,\infty) \times V_+),
$$
where $V_\pm$ are closed manifolds and $M_0$ is compact with
$\p M_0 = V_- \sqcup V_+$.
One can then restrict attention to a space of almost
complex structures that are \emph{fixed on the cylindrical ends}, but can
vary on the compact subset~$M_0$, and a generic subset
of this space ensures regularity for all holomorphic curves in $M$
that send an injective point to the interior of~$M_0$.
For curves
that live entirely in the cylindrical ends, one can exploit the fact that
$V_\pm$ is compact and argue separately that a generic choice of
$\RR$-invariant almost complex structure on the ends achieves transversality.
We will come back to this in a later chapter.
\end{remark}

\section{Classification of pointed Riemann surfaces}

\subsection{Automorphisms and Teichm\"uller space}
\label{sec:Teichmueller}

In order to understand the local structure of the 
moduli space of $J$-holomorphic
curves, we will first need to consider the space of pointed
Riemann surfaces, which appear as domains of such curves.  In particular,
we will need suitable local parametrizations of $\Mod_{g,m}$ near any
given complex structure on~$\Sigma$.  The discussion necessarily begins
with the following classical result, which is proved
e.g.~in \cite{FarkasKra}.

\begin{thm}[Uniformization theorem]
Every simply connected Riemann surface is biholomorphically equivalent
to either the Riemann sphere $S^2 = \CC \cup \{\infty\}$, the complex
plane $\CC$ or the upper half plane $\HH = \{ \Im z > 0 \} \subset \CC$.
\end{thm}

We will always use $i$ to denote the standard complex structure on
the Riemann sphere $S^2 = \CC \cup \{\infty\} \cong \CC P^1$ 
or the plane~$\CC$.  The pullback of $i$ via the diffeomorphism
\begin{equation}
\label{eqn:cylinder}
\RR \times S^1 \to \CC \setminus \{0\} : (s,t) \mapsto e^{2\pi(s+it)}
\end{equation}
yields a natural complex structure on the cylinder $\RR\times S^1$, which
we'll also denote by~$i$; it satisfies $i \p_s = \p_t$.

The uniformization theorem implies that every Riemann surface can be 
presented as a quotient of either $(S^2,i)$, $(\CC,i)$ or $(\HH,i)$ by some
freely acting discrete group of biholomorphic transformations.  We will be 
most interested
in the punctured surfaces $(\dot{\Sigma},j)$ where $(\Sigma,j,\Theta)$ is
a pointed Riemann surface and $\dot{\Sigma} = \Sigma\setminus\Theta$.
The only surface of this form that has $S^2$ as its universal cover is
$S^2$ itself.  It is almost as easy to see which surfaces are covered
by~$\CC$, as the only biholomorphic transformations on $(\CC,i)$ with no
fixed points are the translations, so every freely acting discrete subgroup
of $\Aut(\CC,i)$ is either trivial, a cyclic group of translations or a
lattice.  The resulting quotients are, respectively, $(\CC,i)$,
$(\RR\times S^1,i) \cong (\CC\setminus\{0\},i)$ and the unpunctured
tori $(T^2,j)$.  All other punctured Riemann surfaces have $(\HH,i)$ as
their universal cover, and not coincidentally, these are precisely the
cases in which $\chi(\Sigma\setminus \Theta) < 0$.

\begin{prop}
\label{prop:Poincare}
There exists on $(\HH,i)$ a complete Riemannian metric $g_P$ of constant
curvature~$-1$ that defines the same conformal structure as~$i$ and has
the property that all conformal transformations on $(\HH,i)$ are also
isometries of~$(\HH,g_P)$.
\end{prop}
\begin{proof}
We define $g_P$ at $z = x + iy \in \HH$ by
$$
g_P = \frac{1}{y^2} g_E,
$$
where $g_E$ is the Euclidean metric.  The conformal transformations
on $(\HH,i)$ are given by fractional linear transformations
\begin{equation*}
\begin{split}
\Aut(\HH,i) &= 
\left\{ \varphi(z) = \frac{a z + b}{c z + d} \ \Big|\ a,b,c,d \in \RR,
\quad ad - bc = 1 \right\} \bigg/ \{\pm 1\} \\
&= \SL(2,\RR) / \{\pm 1\} =: \PSL(2,\RR), \\
\end{split}
\end{equation*}
and one can check that each of these defines an isometry with respect
to~$g_P$.  One can also compute that $g_P$ has curvature~$-1$, and
the geodesics of~$g_P$ are precisely
the lines and semicircles that meet $\RR$ orthogonally, parametrized so that
they exist for all forward and backward time, thus $g_P$ is complete.
For more details on all of this, the book by Hummel \cite{Hummel} is
highly recommended.
\end{proof}
By lifting to universal covers, this implies the following.
\begin{cor}
\label{cor:Poincare}
For every pointed Riemann surface $(\Sigma,j,\Theta)$ such that
$\chi(\Sigma\setminus\Theta) < 0$, the punctured Riemann surface
$(\Sigma\setminus\Theta,j)$ admits a complete Riemannian metric $g_P$
of constant curvature~$-1$ that defines the same conformal structure as~$j$,
and has the property that all biholomorphic transformations on
$(\Sigma\setminus\Theta,j)$ are also isometries of $(\Sigma\setminus\Theta,g_P)$.
\end{cor}
The metric $g_P$ in Prop.~\ref{prop:Poincare} and Cor.~\ref{cor:Poincare}
is often called the \defin{Poincar\'{e} metric}.

The above discussion illustrates a general pattern in the study of
pointed Riemann surfaces: it divides naturally into the study
of punctured surfaces with negative Euler characteristic
and finitely many additional cases.

\begin{defn}
\label{defn:stable}
A pointed surface $(\Sigma,\Theta)$ is said to be
\defin{stable} if $\chi(\Sigma \setminus \Theta) < 0$.
\end{defn}

\begin{lemma}
\label{lemma:freeAction}
If $(\Sigma,j,\Theta)$ is a pointed Riemann surface with $\chi(\Sigma \setminus
\Theta) < 0$ and $\varphi \in \Aut(\Sigma,j,\Theta)$
is not the identity, then $\varphi$ is also not homotopic to the identity.
\end{lemma}
\begin{proof}
By assumption $\varphi \ne \Id$, thus by a simple unique continuation argument,
it has finitely many fixed points, each of which counts with positive index
since $\varphi$ is holomorphic.  The algebraic count of fixed points is thus
at least $m = \#\Theta$.  But if $\varphi$ is homotopic to~$\Id$, then this
count must equal $\chi(\Sigma)$ by the Lefschetz fixed point theorem,
contradicting the assumption $\chi(\Sigma) < \#\Theta$.
\end{proof}

The lemma implies that $\Aut(\Sigma,j,\Theta)$ is always a discrete group
when $(\Sigma,\Theta)$ is stable.  In fact more is true:

\begin{prop}
\label{prop:compactAut}
If $(\Sigma,j,\Theta)$ is a closed pointed Riemann surface with either genus
at least~$1$ or $\#\Theta \ge 3$, then $\Aut(\Sigma,j,\Theta)$ is compact.
\end{prop}

\begin{cor}
\label{cor:finite}
If $(\Sigma,\Theta)$ is stable then $\Aut(\Sigma,j,\Theta)$ is finite.
\end{cor}
Prop.~\ref{prop:compactAut} follows from the more general
Lemma~\ref{lemma:bubbling} below, which we'll use to show that
$\Mod_{g,m}$ is Hausdorff, among other things.
We should note that the corollary can be strengthened considerably,
for instance one can find a priori bounds on the order of 
$\Aut(\Sigma,j)$ in terms of the genus, 
cf.~\cite{SeppalaSorvali}*{Theorem~3.9.3}.  For our purposes, the knowledge
that $\Aut(\Sigma,j,\Theta)$ is finite will be useful enough.  As we'll
review below, automorphism groups in the non-stable cases are not discrete 
and sometimes not even compact, though they are always smooth Lie groups.

It will be convenient to have an alternative (equivalent) definition of
$\Mod_{g,m}$, the moduli space of Riemann surfaces.  Fix any smooth 
oriented closed surface $\Sigma$ with
genus~$g$ and an ordered set of distinct points $\Theta =
(z_1,\ldots,z_m) \subset \Sigma$.  Then $\Mod_{g,m}$ is homeomorphic to
the quotient
$$
\Mod(\Sigma,\Theta) := \jJ(\Sigma) / \Diff_+(\Sigma,\Theta),
$$
where $\jJ(\Sigma)$ is the space of smooth almost complex structures
on $\Sigma$ and $\Diff_+(\Sigma,\Theta)$ is the space of orientation-preserving
diffeomorphisms $\varphi : \Sigma \to \Sigma$ such that
$\varphi|_\Theta = \Id$.  Here the action of 
$\Diff_+(\Sigma,\Theta)$ on $\jJ(\Sigma)$ is defined by the pullback,
$$
\Diff_+(\Sigma,\Theta) \times \jJ(\Sigma) \to \jJ(\Sigma) :
(\varphi,j) \mapsto \varphi^*j.
$$

Informally speaking,
$\jJ(\Sigma)$ is an infinite-dimensional manifold, and we expect
$\Mod(\Sigma,\Theta)$ also to be a manifold if $\Diff_+(\Sigma,\Theta)$ acts
freely and properly.  The trouble is that in general, 
it does not: each $j \in \jJ(\Sigma)$
is preserved by the subgroup $\Aut(\Sigma,j,\Theta)$.  A solution to this
complication is suggested by Lemma~\ref{lemma:freeAction}: if we consider
not the action of all of $\Diff_+(\Sigma,\Theta)$ but only the subgroup
$$
\Diff_0(\Sigma,\Theta) = \{ \varphi \in \Diff_+(\Sigma,\Theta) \ |\ 
\text{$\varphi$ is homotopic to $\Id$} \},
$$
then at least in the stable case, the group acts freely 
on $\jJ(\Sigma)$.  We take this as motivation to study, as something of an
intermediate step, the quotient
$$
\tT(\Sigma,\Theta) := \jJ(\Sigma) / \Diff_0(\Sigma,\Theta).
$$
This is the \defin{Teichm\"uller space} of genus $g$, $m$-pointed surfaces.
It is useful mainly because its local structure is simpler than that of
$\Mod(\Sigma,\Theta)$---we'll show below that it is always a smooth 
finite-dimensional manifold, and its dimension can be computed using the
Riemann-Roch formula.  The actual moduli space
of Riemann surfaces can then be understood as the quotient of Teichm\"uller 
space by a discrete group:
$$
\Mod(\Sigma,\Theta) = \tT(\Sigma,\Theta) / M(\Sigma,\Theta),
$$
where $M(\Sigma,\Theta)$ is the \defin{mapping class group},
$$
M(\Sigma,\Theta) := \Diff_+(\Sigma,\Theta) /
\Diff_0(\Sigma,\Theta).
$$

Recall that a topological group $G$ acting continuously 
on a topological space $X$ is
said to act \defin{properly} if the map $G \times X \to X \times X :
(g,x) \mapsto (gx,x)$ is proper: this means that for any sequences
$g_n \in G$ and $x_n \in X$ such that both $x_n$ and $g_n x_n$ converge,
$g_n$ has a convergent subsequence.  This is the condition one needs in
order to show that the quotient $M / G$ is Hausdorff.  Thus for the
action of $\Diff_+(\Sigma,\Theta)$ or $\Diff_0(\Sigma,\Theta)$ on
$\jJ(\Sigma)$, we need the following compactness lemma, which also implies
Prop.~\ref{prop:compactAut}.  We'll state it for
now without proof, but will later be able to prove it using a simple
case of the ``bubbling'' arguments in the next chapter.
\begin{lemma}
\label{lemma:bubbling}
Suppose either $\Sigma$ has genus at least~$1$ or $\#\Theta \ge 3$.
If $\varphi_k \in \Diff_+(\Sigma,\Theta)$ and $j_k \in \jJ(\Sigma)$ are
sequences such that $j_k \to j$ and $\varphi_k^*j_k \to j'$ in the
$C^\infty$-topology, then $\varphi_k$ has a subsequence that converges
in $C^\infty$ to a diffeomorphism $\varphi \in \Diff_+(\Sigma,\Theta)$
with $\varphi^*j = j'$.
\end{lemma}

This implies that both $\Diff_+(\Sigma,\Theta)$ and 
$\Diff_0(\Sigma,\Theta)$ act properly on $\jJ(\Sigma)$, so
$\Mod(\Sigma,\Theta)$ and $\tT(\Sigma,\Theta)$ are both Hausdorff.
This is also trivially true in the cases $g=0$, $m \le 2$, as 
then $\Diff_0(S^2,\Theta) = \Diff_+(S^2,\Theta)$ and the
uniformization theorem implies that $\Mod(S^2,\Theta) =
\tT(S^2,\Theta)$ is a one point space.

We now examine the extent to which the discrete group $M(\Sigma,\Theta)$
does not act freely on $\tT(\Sigma,\Theta)$.

\begin{exercise}
\label{EX:stabilizer}
Show that for any stable pointed Riemann surface $(\Sigma,j,\Theta)$, the 
restriction to $\Aut(\Sigma,j,\Theta)$ of the natural quotient map
$\Diff_+(\Sigma,\Theta) \to M(\Sigma,\Theta)$ defines an isomorphism from
$\Aut(\Sigma,j,\Theta)$ to the stabilizer of $[j] \in \tT(\Sigma,\Theta)$ 
under the action of $M(\Sigma,\Theta)$.
\end{exercise}

Combining Exercise~\ref{EX:stabilizer} with Corollary~\ref{cor:finite}
above, we see that every point in Teichm\"uller space has a finite isotropy
group under the action of the mapping class group; we'll see below that
this is also true in the non-stable cases.  This gives us the best possible
picture of the local structure of $\Mod_{g,m}$: it is not a manifold
in general, but locally it looks like a quotient of Euclidean space by a
finite group action.  Hausdorff topological spaces with this kind of 
local structure are called \defin{orbifolds}.
The curious reader may consult the first section of
\cite{FukayaOno} for the definition and basic properties of orbifolds,
which we will not go into here, except to state the following local
structure result for $\Mod_{g,m}$.

\begin{thm}
$\Mod_{g,m}$ is a smooth orbifold whose isotropy subgroup at 
$(\Sigma,j,\Theta) \in \Mod_{g,m}$ is
$\Aut(\Sigma,j,\Theta)$; in particular, $\Mod_{g,m}$ is a manifold in
a neighborhood of any pointed Riemann surface $(\Sigma,j,\Theta)$ 
that has trivial automorphism group.  Its dimension is
$$
\dim \Mod_{g,m} = \begin{cases}
6g - 6 + 2m & \text{ if $2g + m \ge 3$},\\
2           & \text{ if $g=1$ and $m=0$},\\
0           & \text{ otherwise.}
\end{cases}
$$
\end{thm}

Note that the inequality $2g + m \ge 3$ is precisely the 
stability condition for 
a genus~$g$ surface with $m$ marked points.

The main piece of hard work that needs to be done now is proving that
Teichm\"uller space really is a smooth manifold of the correct dimension, and
in fact it will be useful to have local slices in $\jJ(\Sigma)$ that can
serve as charts for $\tT(\Sigma,\Theta)$.  To that end, fix a pointed
Riemann surface $(\Sigma,j,\Theta)$ and consider the
nonlinear operator
$$
\dbar_j : \bB^{1,p}_\Theta \to \eE^{0,p} : \varphi \mapsto T\varphi +
j \circ T\varphi \circ j,
$$
where $p > 2$, 
$$
\bB^{1,p}_\Theta = \big\{ \varphi \in W^{1,p}(\Sigma,\Sigma)\ \big|\ 
\varphi|_\Theta = \Id \big\},
$$
and $\eE^{0,p} \to W^{1,p}(\Sigma,\Sigma)$ is the Banach space bundle with
fibers 
$$
\eE^{0,p}_\varphi = L^p(\overline{\Hom}_\CC(T\Sigma,\varphi^*T\Sigma)).
$$
The zeroes of $\dbar_j$ are the holomorphic maps from $\Sigma$ to itself
that fix the marked points,
and in particular a neighborhood of $\Id$ in $\dbar_j^{-1}(0)$ gives a local
description of $\Aut(\Sigma,j,\Theta)$.  We have
$$
T_{\Id}\bB^{1,p}_\Theta =
W^{1,p}_{\Theta}(T\Sigma) := \{ X \in W^{1,p}(T\Sigma)\ |\ X(\Theta) = 0 \},
$$
which is a closed subspace of $W^{1,p}(T\Sigma)$ with real codimension~$2m$.  
The linearization 
$$
\mathbf{D}_{(j,\Theta)} := D\dbar_j(\Id) : W^{1,p}_\Theta(T\Sigma) \to
L^p(\overline{\End}_\CC(T\Sigma))
$$
is then the restriction to
$W^{1,p}_\Theta(T\Sigma)$ of the natural
linear Cauchy-Riemann operator defined by the holomorphic structure
of $(T\Sigma,j)$.  By Riemann-Roch, the latter has index
$\chi(\Sigma) + 2 c_1(T\Sigma) = 3\chi(\Sigma)$, thus $\mathbf{D}_{(j,\Theta)}$
has index
\begin{equation}
\label{eqn:indexDj}
\ind(\mathbf{D}_{(j,\Theta)}) = 3\chi(\Sigma) - 2m.
\end{equation}
\begin{exercise}
Show that if $A : X \to Y$ is a Fredholm operator and $X_0 \subset X$ is a
closed subspace of codimension~$N$, then $A|_{X_0}$ is also Fredholm and has
index $\ind(A) - N$.
\end{exercise}

\begin{prop}
\label{prop:inj}
If $\chi(\Sigma \setminus \Theta) < 0$ then $\mathbf{D}_{(j,\Theta)}$ is injective.
\end{prop}
\begin{proof}
By the similarity principle, any nontrivial section $X \in \ker \mathbf{D}_{(j,\Theta)}$
has finitely many zeroes, each of positive order, and there are at least
$m$ of them since $X|_\Theta = 0$.  Thus $\chi(\Sigma) =
c_1(T\Sigma) \ge m$, which contradicts the stability assumption.
\end{proof}
Observe that Prop.~\ref{prop:inj} provides an alternative proof of the
fact that $\Aut(\Sigma,j,\Theta)$ is always discrete in the stable case.

The target space of $\mathbf{D}_{(j,\Theta)}$ contains
$\Gamma(\overline{\End}_\CC(T\Sigma))$, which one can think of as the
``tangent space'' to $\jJ(\Sigma)$ at~$j$.  In particular, any smooth family
$j_t \in \jJ(\Sigma)$ with $j_0 = j$ has
$$
\p_t j_t|_{t=0} \in \Gamma(\overline{\End}_\CC(T\Sigma)).
$$
We shall now use $\mathbf{D}_{(j,\Theta)}$ to define a special class of smoothly 
parametrized families in~$\jJ(\Sigma)$.

\begin{defn}
\label{defn:TeichSlice}
For any $j \in \jJ(\Sigma)$, a \defin{Teichm\"uller slice through~$j$} is a
smooth family of almost complex structures parametrized by an injective map
$$
\oO \to \jJ(\Sigma) : \tau \mapsto j_\tau,
$$
where $\oO$ is a neighborhood of~$0$ in some finite-dimensional
Euclidean space, with $j_0 = j$ and the following transversality property.
If $T_j \tT \subset \Gamma(\overline{\End}_\CC(T\Sigma))$ denotes the vector
space of
all ``tangent vectors'' $\p_t j_{\tau(t)}|_{t=0}$ determined by smooth
paths $\tau(t) \in \oO$ through $\tau(0) = 0$, then
$$
L^p(\overline{\End}_\CC(T\Sigma)) = \im \mathbf{D}_{(j,\Theta)} \oplus T_j\tT.
$$
\end{defn}
We will typically denote a Teichm\"uller slice simply by the image
$$
\tT := \{j_\tau \ |\ \tau \in \oO \} \subset \jJ(\Sigma),
$$
and think of this as
a smoothly embedded finite-dimensional submanifold of $\jJ(\Sigma)$
whose tangent space at~$j$ is $T_j\tT$.
Note that the definition doesn't depend on~$p$; in fact, one would obtain
an equivalent definition by regarding $\mathbf{D}_{(j,\Theta)}$ as an operator from
$W^{k,p}_\Theta$ to $W^{k-1,p}$ for any $k \in \NN$ and $p > 2$.

It is easy to see that Teichm\"uller slices always exist.
Given $j \in \jJ(\Sigma)$, pick any complement of 
$\im \mathbf{D}_{(j,\Theta)}$, i.e.~a subspace $C \subset
L^p(\overline{\End}_\CC(T\Sigma))$ of dimension $\dim \coker \mathbf{D}_{(j,\Theta)}$
whose intersection with $\im \mathbf{D}_{(j,\Theta)}$ is trivial.
By approximation, we may assume every section in $C$ is smooth.
We can then choose a small
neighborhood $\oO \subset C$ of~$0$ and define the map
\begin{equation}
\label{eqn:Phi}
\oO \to \jJ(\Sigma) : y \mapsto j_y = \left( \1 + \frac{1}{2} j y\right)
j \left( \1 + \frac{1}{2} j y \right)^{-1},
\end{equation}
which has the properties $j_0 = j$ and
$\p_t j_{ty}|_{t=0} = y$,
thus it is injective if $\oO$ is sufficiently small.  This family
is a Teichm\"uller slice through~$j$.

Let $\pi_{\Theta} : \jJ(\Sigma) \to \tT(\Sigma,\Theta) : j \mapsto [j]$ 
denote the quotient projection.

\begin{thm}
\label{thm:Teichmueller}
$\tT(\Sigma,\Theta)$ admits the structure of a smooth finite-dimensional
manifold, and for any $(\Sigma,j,\Theta)$
there are natural isomorphisms
$$
T_{\Id}\Aut(\Sigma,j,\Theta) = \ker \mathbf{D}_{(j,\Theta)},
\qquad
T_{[j]}\tT(\Sigma,\Theta) = \coker \mathbf{D}_{(j,\Theta)}.
$$
In particular,
\begin{equation}
\label{eqn:TeichDimension}
\dim \tT(\Sigma,\Theta) - \dim \Aut(\Sigma,j,\Theta)  =
-\ind \mathbf{D}_{(j,\Theta)} = 6g - 6 + 2m.
\end{equation}
Moreover for any Teichm\"uller slice
$\tT \subset \jJ(\Sigma)$ through~$j$,
the projection
\begin{equation}
\label{eqn:projection}
\pi_{\Theta}|_{\tT} : \tT \to \tT(\Sigma,\Theta)
\end{equation}
is a local diffeomorphism near~$j$.
\end{thm}

We'll prove this in the next few sections separately for the non-stable and 
stable cases.  Observe that in the stable case, $\dim\Aut(\Sigma,j,\Theta) = 0$
and thus \eqref{eqn:TeichDimension} gives $6g - 6 + 2m$ as the dimension of
Teichm\"uller space.

It should be intuitively clear why $\ker \mathbf{D}_{(j,\Theta)}$ is
the same as $T_{\Id}\Aut(\Sigma,j,\Theta)$, though since $\mathbf{D}_{(j,\Theta)}$
will usually not be surjective, we still have to do something---it doesn't
follow immediately from the implicit function theorem.
The relationship between $T_{[j]}\tT(\Sigma,\Theta)$ and
$\coker \mathbf{D}_{(j,\Theta)}$ is also not difficult to understand, though here
we'll have to deal with a few analytical subtleties.
Intuitively, $T_{[j]}\tT(\Sigma,\Theta)$ should be complementary to the
tangent space at $j \in \jJ(\Sigma)$ to its orbit under the action
of $\Diff_0(\Sigma,\Theta)$.  Without worrying about the analytical
details for the moment, consider a smooth family of diffeomorphisms
$\varphi_\tau \in \Diff_0(\Sigma,\Theta)$ with $\varphi_0 = \Id$ and
$$
\p_\tau \varphi_\tau|_{\tau=0} = X,
$$
a smooth vector field that vanishes at the marked points~$\Theta$.
Then choosing a symmetric complex connection on $\Sigma$ and differentiating
the action $(\varphi_\tau,j) \mapsto \varphi_\tau^*j$, a short
computation yields
\begin{equation}
\label{eqn:tangentOrbit}
\begin{split}
\left.\frac{\p}{\p \tau} \varphi^*_\tau j\right|_{\tau=0} &=
\left. \frac{\p}{\p\tau} \left[(T\varphi_\tau)^{-1} \circ j \circ T\varphi_\tau
\right] \right|_{\tau=0} = -\nabla X \circ j + j \circ \nabla X \\
&= j (\nabla X + j \circ \nabla X \circ j).
\end{split}
\end{equation}
Note that $\nabla$ can be chosen to be the natural connection in some
local holomorphic coordinates, in which case the last expression in
parentheses above is simply the natural linear Cauchy-Riemann operator
on~$T\Sigma$ with complex structure~$j$.  Since this operator is 
complex-linear, its image is not changed by multiplication with~$j$,
and we conclude that the tangent space to the orbit is precisely
the image of $\mathbf{D}_{(j,\Theta)}$, acting on smooth vector fields that
vanish at the marked points.

\subsection{Spheres with few marked points}

A pointed surface $(\Sigma,\Theta)$ of genus~$g$ with $m$ marked points
is stable whenever $2g + m \ge 3$.
The alternative includes three cases for $g=0$, and here
uniformization tells us that
$(S^2,j)$ is equivalent to $(S^2,i)$ for every possible~$j$.  Further,
one can choose a fractional linear transformation to map up to three marked
points to any points of our choosing, thus $\Mod_{g,m}$ is a one point space
in each of these cases.  We can now easily identify the automorphism groups
for each.

\begin{itemize}
\item $g=0$, $m=0$: $(\Sigma,j) \cong (S^2,i)$, and
$\Aut(S^2,i)$ is the real $6$-dimensional group of fractional linear
transformations,
\begin{equation*}
\begin{split}
\Aut(S^2,i) &= 
\left\{ \varphi(z) = \frac{a z + b}{c z + d} \ \Big|\ a,b,c,d \in \CC,
\quad ad - bc = 1 \right\} \bigg/ \{\pm 1\} \\
&= \SL(2,\CC) / \{\pm 1\} =: \PSL(2,\CC). \\
\end{split}
\end{equation*}
These are also called the \defin{M\"obius transformations}.
\item
$g=0$, $m=1$: $(\Sigma,j,\Theta) = (S^2,i,(\infty))$ and
$$
\Aut(S^2,i,(\infty)) = \Aut(\CC,i) = \{ \varphi(z) = a z + b \ | \ 
a, b \in \CC \},
$$
a real $4$-dimensional group.
\item
$g=0$, $m=2$: $(\Sigma,j,\Theta) = (S^2,i,(0,\infty))$ and
$$
\Aut(S^2,i,(0,\infty)) = \{ \varphi(z) = a z \ | \ a \in \CC \},
$$
a real $2$-dimensional group.  Using the biholomorphic map 
\eqref{eqn:cylinder}, one can equivalently think of this as the
group of translations on the standard cylinder $(\RR\times S^1,i)$.
\end{itemize}

\begin{prop}
For each 
$(S^2,i,\Theta) \in \Mod_{0,m}$ with $m \le 2$, $\mathbf{D}_{(i,\Theta)}$
is surjective and
$\dim \ker \mathbf{D}_{(i,\Theta)} = \dim \Aut(S^2,i,\Theta)$.
\end{prop}
\begin{proof}
From \eqref{eqn:indexDj}, $\ind(\mathbf{D}_{(i,\Theta)}) = 3\chi(\Sigma) - 2m = 6 - 2m
= \dim \Aut(S^2,i,\Theta)$, so it will suffice to prove that
$\dim\ker \mathbf{D}_{(i,\Theta)}$ is not larger than $6 - 2m$.  To see this,
pick $3 - m$ distinct points $\zeta_1,\ldots,\zeta_{3-m} \in \Sigma \setminus
\Theta$ and consider the linear map
$$
\Phi : \ker \mathbf{D}_{(i,\Theta)} \to T_{\zeta_1}\Sigma \oplus \ldots \oplus
T_{\zeta_{3-m}}\Sigma : X \mapsto (X(\zeta_1),\ldots,X(\zeta_{3-m})).
$$
The right hand side is a vector space of real dimension $6 - 2m$, so the result
will follow from the claim that $\Phi$ is injective.  Indeed, if
$\eta \in \ker \mathbf{D}_{(i,\Theta)}$ and $\Phi(\eta) = 0$, then the similarity
principle implies that each zero counts positively, and the points
$\zeta_1,\ldots,\zeta_{3-m}$ combined with $\Theta$ imply 
$c_1(T S^2) \ge 3-m + m =3$, giving a contradiction unless $X \equiv 0$.
\end{proof}

By the above proposition, the implicit function theorem defines a smooth
manifold structure on $\dbar_i^{-1}(0) \subset \bB^{1,p}_\Theta$ 
near~$\Id$ and yields a natural isomorphism 
$$
T_\Id \Aut(S^2,i,\Theta) = \ker \mathbf{D}_{(i,\Theta)}.
$$

\begin{exercise}
\label{EX:autGenus0}
Show that for $m \ge 3$, $\Aut(S^2,i,\Theta)$ is always trivial
and $\Mod_{0,m}$ is a smooth manifold of real dimension $2(m-3)$.
\end{exercise}

\subsection{The torus}
\label{subsec:torus}

The remaining item on the list of non-stable pointed surfaces is the torus with
no marked points, and this is the one case where both the
automorphism groups and the Teichm\"uller space have positive dimension.
Thus we'll see that $\mathbf{D}_{(j,\Theta)}$ is neither surjective nor injective,
but fortunately the torus is a simple enough manifold so that everything
can be computed explicitly.

The universal cover
of $(T^2,j)$ is the complex plane, which implies that
$(T^2,j)$ is biholomorphically equivalent to $(\CC / \Lambda,i)$ for some
lattice $\Lambda \subset \CC$.  Without loss of generality, we can take
$\Lambda = \ZZ + \lambda \ZZ$ for some $\lambda \in \HH$.  Then choosing
a real-linear map that sends $1$ to itself and $\lambda$ to~$i$, we can
write $T^2 = \CC / (\ZZ + i\ZZ)$ and identify
$$
(\CC / \Lambda,i) \cong (T^2,j_\lambda),
$$
where $j_\lambda$ is some translation invariant complex structure on $\CC$
that is compatible with the standard orientation.  Conversely, every such
translation invariant complex structure can be obtained in this way and
descends to a complex structure on~$T^2$.

\begin{prop}
\label{prop:TeichmuellerT2}
$[j_\lambda] = [j_{\lambda'}]$ in $\tT(T^2)$ if and only if 
$\lambda = \lambda'$.
\end{prop}
\begin{proof}
If $j_\lambda = \varphi^* j_{\lambda'}$ for some $\varphi \in
\Diff_0(T^2)$, then $\varphi$ can be lifted to a diffeomorphism of
$\CC$ that (after composing with a translation) fixes the lattice 
$\ZZ + i\ZZ$.  Now composing with the linear 
map mentioned above, this gives rise to a biholomorphic map
$\psi : \CC \to \CC$ such that $\psi(0) = 0$, 
$\psi(1) = 1$ and $\psi(\lambda) = \lambda'$.
But all biholomorphic maps on $\CC$ have the form $\psi(z) = az + b$,
and the conditions at~$0$ and~$1$ imply $b=0$ and $a=1$, thus
$\lambda = \lambda'$.
\end{proof}

This shows that $\tT(T^2)$ is a smooth $2$-manifold that can be identified
naturally with the upper half plane~$\HH$, and the set of translation
invariant complex structures
$$
\tT := \{ j_\lambda \in \jJ(T^2)\ |\ \lambda \in \HH \}
$$
defines a global parametrization.  We'll see below
that it is also a Teichm\"uller slice in the sense of 
Definition~\ref{defn:TeichSlice}.

To understand the action of $M(T^2) = \Diff_+(T^2) / \Diff_0(T^2)$ 
on $\tT(T^2)$, note that every element of
$M(T^2)$ can be represented uniquely as a matrix $A \in \SL(2,\ZZ)$,
which is determined by its induced isomorphism on
$H_1(T^2) = \ZZ^2$.  Then $A^*j_\lambda$ is another translation invariant
complex structure $j_{\lambda'}$ for some $\lambda' \in \HH$, and
$$
[A] \cdot [j_\lambda] = [A^*j_\lambda] = [j_{\lambda'}].
$$
Thus the stabilizer of $[j_\lambda]$ under this action is the subgroup
$$
G_\lambda := \{ A \in \SL(2,\ZZ)\ |\ A^*j_\lambda = j_\lambda \}.
$$
This is also a subgroup of $\Aut(T^2,j_\lambda)$, and a
complementary (normal) subgroup is
formed by the intersection $\Aut(T^2,j_\lambda) \cap \Diff_0(T^2)$.

\begin{prop}
Every $\varphi \in \Aut(T^2,j_\lambda)$ that fixes $(0,0) \in T^2$ belongs
to $G_\lambda$, and every $\varphi \in \Aut(T^2,j_\lambda)
\cap \Diff_0(T^2)$ is a translation
$\varphi(z) = z+\zeta$ for some $\zeta \in T^2$.
\end{prop}
\begin{proof}
The first statement follows by a repeat of the argument used in the proof
of Prop.~\ref{prop:TeichmuellerT2} above: if $\varphi \in \Aut(T^2,j_\lambda)$
fixes $(0,0)$, then regarding it as a diffeomorphism on
$\CC / \Lambda$, it lifts to a biholomorphic map on $\CC$ which must be
of the form $\psi(z) = c z$ for $c \in \CC\setminus\{0\}$, implying that
$\varphi$ is the projection to $T^2 = \CC / \ZZ^2$ of a real-linear map
on~$\CC$ which preserves the lattice $\ZZ + i\ZZ$, and thus
$\varphi \in \SL(2,\ZZ)$.

The second statement follows because one can compose any
$\varphi \in \Aut(T^2,j_\lambda) \cap \Diff_0(T^2)$ with translations
until it fixes $(0,0)$, and conclude that the composed map is in
$\SL(2,\ZZ) \cap \Diff_0(T^2) = \{\1\}$.
\end{proof}

Denoting the translation subgroup by $T^2 \subset \Aut(T^2,j_\lambda)$,
we see now that the total
automorphism group is the semidirect product
$$
\Aut(T^2,j_\lambda) = T^2 \rtimes G_\lambda,
$$
and is thus a smooth $2$-dimensional manifold.

\begin{prop}
For each $[j_\lambda] \in \tT(T^2)$, $G_\lambda$ is finite.
\end{prop}
\begin{proof}
The claim follows from the fact that $G_\lambda$ is compact, which we
show as follows.  Choose a new real basis $(\hat{e}_1,\hat{e}_2)$ 
for $\CC = \RR^2$ such that $\hat{e}_1$ is a positive multiple of $e_1$,
$\hat{e}_2 = j_\lambda \hat{e}_1$ and the parallelogram spanned
by $\hat{e}_1$ and $\hat{e}_2$ has area~$1$.  Expressing any matrix
$A \in G_\lambda$ in this basis, $A$ now belongs to both
$\GL(1,\CC)$ and $\SL(2,\RR)$, whose intersection
$$
\GL(1,\CC) \cap \SL(2,\RR) = \U(1)
$$
is compact.
\end{proof}

By this result, $\Aut(T^2,j)$ is always compact, as was predicted by
Prop.~\ref{prop:compactAut}.  Moreover, the stabilizer of any element of
$\tT(T^2)$ under the action of $M(T^2)$ is finite, so we conclude that
$$
\Mod_{2,0} \cong \tT(T^2) / M(T^2) \cong \HH / \SL(2,\ZZ)
$$
is a smooth $2$-dimensional orbifold, and is a manifold near any
$[j_\lambda]$ for which $G_\lambda$ is trivial.  

\begin{exercise}
Show that $G_\lambda$ is trivial for all $\lambda$ in an open and
dense subset of~$\HH$.
\end{exercise}

Let us now relate the above descriptions of $\tT(T^2)$ and
$\Aut(T^2,j)$ to the natural Cauchy-Riemann operator
$$
\mathbf{D}_{j} : W^{1,p}(T T^2) \to L^p(\overline{\End}_\CC(T T^2))
$$
on $(T T^2,j)$.  After an appropriate diffeomorphism we can assume without
loss of generality that $j = j_\lambda \in \tT$ for some $\lambda \in \HH$.
Then identifying $T T^2$ with $T^2 \times \CC$ via the natural global complex
trivialization, $\mathbf{D}_j$ is equivalent to the standard Cauchy-Riemann operator
$$
\dbar = \p_s + i\p_t : W^{1,p}(T^2,\CC) \to L^p(T^2,\CC),
$$
whose kernel is the real $2$-dimensional space of constant functions,
which is precisely $T_{\Id}\Aut(T^2,j_\lambda)$ since $\Aut(T^2,j_\lambda)$
consists infinitessimally of translations.  Meanwhile, the formal adjoint
$\mathbf{D}_j^*$ is equivalent to
$$
\p = \p_s - i\p_t : W^{1,p}(T^2,\CC) \to L^p(T^2,\CC),
$$
whose kernel is again the space of constant functions, and this is precisely
$T_{j_\lambda} \tT$.

\subsection{The stable case}

Assume $2g + m \ge 3$.  We've already seen that in this case
$\Aut(\Sigma,j,\Theta)$ is finite and $\mathbf{D}_{(j,\Theta)}$ is injective, 
so Theorem~\ref{thm:Teichmueller} now
reduces to the statement that $\tT(\Sigma,\Theta)$ is a smooth manifold
whose tangent space at $[j]$ is $\coker \mathbf{D}_{(j,\Theta)}$, and local charts
are given by Teichm\"uller slices.
We argued informally above that the tangent space
at $j \in \jJ(\Sigma)$ to its orbit under $\Diff_0(\Sigma,\Theta)$ is
the image of $\mathbf{D}_{(j,\Theta)}$, which motivates the belief that
$\tT(\Sigma,\Theta)$ should locally
look like a quotient of this image, i.e.~the cokernel of $\mathbf{D}_{(j,\Theta)}$.

A naive attempt to make this precise might now
proceed by considering Banach manifold completions of $\jJ(\Sigma)$ and
$\Diff_0(\Sigma,\Theta)$ and arguing that the extension of
$$
\Phi : \Diff_0(\Sigma,\Theta) \times \jJ(\Sigma) \to \jJ(\Sigma) :
(\varphi,j) \mapsto \varphi^*j
$$
to these completions defines a smooth Banach Lie group 
action that is free and proper, so the quotient is a manifold whose
tangent space is the quotient of the relevant tangent spaces.
But this approach runs into
a subtle analytical complication: the partial derivative of the map
$\Phi$ with respect to the first factor must have the form
$$
D_1 \Phi(\Id,j) X = j \mathbf{D}_{(j,\Theta)} X,
$$
and if $j$ is not smooth, then the right hand side will always be one
step less smooth than~$j$.  Indeed, $\mathbf{D}_{(j,\Theta)}$ is in this
case a nonsmooth Cauchy-Riemann type operator, and we can see it more clearly
by redoing the computation \eqref{eqn:tangentOrbit} in smooth coordinates
that are not holomorphic: this yields a local expression of the form
$$
\left.\frac{\p}{\p \tau} \varphi^*_\tau j\right|_{\tau=0} =
j (dX + j \circ dX \circ j) + dj(X).
$$
Since this involves the first derivative of~$j$, the expression for
$D_1\Phi(\Id,j) X$ can never lie in the appropriate Banach space
completion of $T_j \jJ(\Sigma)$, but rather in a larger Banach space that
contains it.  This means that $\Phi$ is not differentiable---indeed,
this is another example (cf.~Exercise~\ref{EX:notDifferentiable}) of a
natural map between
infinite-dimensional spaces that can \emph{never} be differentiable in any
conventional
Banach space setting.  It is probably still true that one can make a
precise argument out of this idea, but it would require a significantly
different analytical framework than just smooth maps on Banach manifolds,
e.g.~one might attempt to use the category of \emph{sc-smooth} Banach
manifolds (cf.~\cite{Hofer:polyfolds}).  Another alternative, using the
correspondence between conformal structures and hyperbolic metrics on
stable Riemann surfaces, is explained in \cite{Tromba}.

Instead of trying to deal with global Banach Lie group actions, we will prove
the theorem by constructing smooth charts directly via local
Teichm\"uller slices $\tT \subset \jJ(\Sigma)$.
We will indeed need to enlarge $\Diff_0(\Sigma,\Theta)$ and $\jJ(\Sigma)$ to
Banach manifolds containing non-smooth objects, but the key observation is
that since every object in the slice $\tT$ is smooth by assumption,
the orbit of any $j \in \tT$ can still be understood as a smooth Banach
submanifold.  The following argument was explained to me 
by Dietmar Salamon, on a napkin.

\begin{proof}[Proof of Theorem~\ref{thm:Teichmueller} in the stable case]
For $k \in \NN$ and $p > 2$, let $\jJ^{k,p}(\Sigma)$ denote the space of
$W^{k,p}$-smooth almost complex structures on $\Sigma$, and for
$k \ge 2$, let
$$
\dD^{k,p}_\Theta \subset W^{k,p}_{\Theta}(\Sigma,\Sigma)
$$
denote the open subset consisting of all $\varphi\in 
W^{k,p}_\Theta(\Sigma,\Sigma)$ which are $C^1$-smooth diffeomorphisms.  Choose
$j_0 \in \jJ(\Sigma)$ and suppose $\tT \subset \jJ(\Sigma)$ is a 
Teichm\"uller slice through $j_0$.  This implies that
$T_{j_0}\tT \subset \Gamma(\overline{\End}_\CC(T\Sigma))$ is complementary
to the image of
$$
\mathbf{D}_{(j_0,\Theta)} : W^{k,p}_{\Theta}(T\Sigma) \to W^{k-1,p}
(\overline{\End}_\CC(T\Sigma))
$$
for all $k \in \NN$.

Since every $j \in \tT$ is smooth, the orbit of $j$ under the natural action of
$\dD^{k+1,p}_{\Theta}$ is in $\jJ^{k,p}(\Sigma)$; in fact the map
\begin{equation}
\label{eqn:F}
F : \dD^{k+1,p}_{\Theta} \times \tT \to \jJ^{k,p}(\Sigma) : 
(\varphi,j) \mapsto \varphi^* j
\end{equation}
is smooth and has derivative
\begin{equation*}
\begin{split}
dF(\Id,j_0) : W^{k+1,p}_{\Theta} \oplus T_{j_0}\tT &\to
W^{k,p}(\overline{\End}_\CC(T\Sigma)) \\
(X,y) &\mapsto j_0 \mathbf{D}_{(j_0,\Theta)} X + y.
\end{split}
\end{equation*}
This map is an isomorphism, thus by the inverse function theorem,
$F$ is a smooth diffeomorphism between open neighborhoods
of $(\Id,j_0) \in \dD^{k+1,p}_\Theta \times \tT$ and
$j_0 \in \jJ^{k,p}(\Sigma)$.

We claim now that after shrinking $\tT$ if necessary, the
projection $\pi_\Theta : \tT \to \tT(\Sigma,\Theta)$ is a bijection
onto a neighborhood of~$[j_0]$.  It is clearly surjective,
since every $j \in \jJ(\Sigma)$ in some neighborhood of $j_0$ is in the
image of~$F$.  To see that it is injective, we essentially use the fact
that $\Diff_0(\Sigma,\Theta)$ acts freely and properly on~$\jJ(\Sigma)$.
Indeed, we need to show that there
is no pair of sequences $j_k \ne j_k' \in \tT$ both converging to~$j_0$,
such that $j_k = \varphi_k^*j_k'$ for some $\varphi_k \in 
\Diff_0(\Sigma,\Theta)$.  If there are such sequences, then
by Lemma~\ref{lemma:bubbling}, $\varphi_k$ also has a subsequence converging
to some $\varphi \in \Diff_0(\Sigma,\Theta)$ with $\varphi^*j_0 = j_0$,
thus $\varphi = \Id$.  But then $\varphi_k$ is near the identity in
$\dD^{k+1,p}_\Theta$ for
sufficiently large~$k$, and $F(\varphi_k,j_k') = j_k$ implies
$(\varphi_k,j_k') = (\Id,j_k)$ since $F$ is locally invertible.

Finally, we show that the bijection induced by any other choice of slice 
$\tT'$ through $j_0$ to a neighborhood of $[j_0]$ in $\tT(\Sigma,\Theta)$
yields a smooth transition map
$\tT' \to \tT : j' \mapsto j$.  Indeed, this transition map must satisfy
the relation
$$
(\varphi,j) = F^{-1} \circ F'(\varphi',j')
$$
for any $\varphi,\varphi' \in \dD^{k+1,p}_{\Theta}$,
where $F' : \dD^{k+1,p}_{\Theta} \times \tT' \to \jJ^{k,p}(\Sigma)$
is the corresponding local diffeomorphism defined for $\tT'$ as in
\eqref{eqn:F}.  Explicitly then,
$j = \pr_2 \circ F^{-1} \circ F'(\Id,j')$, which is clearly a smooth map.
\end{proof}

\begin{exercise}
Using the Banach manifold charts constructed in the above proof, show that
for any $j \in \jJ(\Sigma)$ and Teichm\"uller slice $\tT$ through~$j$, 
the projection
$$
L^p(\overline{\End}_\CC(T\Sigma)) \to T_j \tT
$$
along $\im \mathbf{D}_{(j,\Theta)}$ descends to an isomorphism
$\coker \mathbf{D}_{(j,\Theta)} \to T_{[j]}\tT(\Sigma,\Theta)$ that is
independent of all choices.
\end{exercise}

\section{Fredholm regularity and the implicit function theorem}
\label{sec:IFT}

With the local structure of $\Mod_{g,m}$ understood, we now turn our
attention back to $\Mod(J)$, the moduli space of
$J$-holomorphic curves.  It is unfortunately not true that $\Mod(J)$ is
always locally a finite-dimensional manifold, nor even an orbifold.
We need an extra condition to guarantee this, called
\emph{Fredholm regularity}. To understand it, we must first set up the
appropriate version of the implicit function theorem.

The setup will be analogous to the case of $\Mod_{g,m}$ in the following
sense.  In the previous section, we analyzed $\Mod_{g,m}$ by first understanding
the Teichm\"uller space $\tT(\Sigma,\Theta)$.  The latter is a somewhat
unnatural object in that its definition depends on choices (i.e.~the surface
$\Sigma$ and marked points $\Theta \subset \Sigma$), but it has the
advantage of being a smooth finite-dimensional manifold.  Then the moduli
space $\Mod_{g,m}$ was understood as the quotient of $\tT(\Sigma,\Theta)$
by a discrete group action with finite isotropy groups: in fact,
\emph{locally} near a given $[j] \in \tT(\Sigma,\Theta)$, a neighborhood
in $\Mod_{g,m}$ looks like a quotient of $\tT(\Sigma,\Theta)$ by 
a finite group ($\Aut(\Sigma,j,\Theta)$ in the stable case), which makes
$\Mod_{g,m}$ an orbifold of the same dimension as $\tT(\Sigma,\Theta)$.

In the more general setup, we will be able to identify $\Mod_{g,m}^A(J)$
locally near a curve $(\Sigma,j,\Theta,u)$ with a quotient of the form
$$
\dbar_J^{-1}(0) / \Aut(u),
$$
where $\dbar_J$ is a generalization of the nonlinear Cauchy-Riemann operator
that we considered in Chapter~\ref{chapter:Fredholm}, using local Teichm\"uller slices to
incorporate varying complex structures on the domain.
Its zero set thus contains all $J$-holomorphic curves in some neighborhood
of~$u$, but it may also include seemingly distinct curves that are actually
equivalent in the moduli space, thus one must still divide by an appropriate
symmetry group, which locally turns out to be the finite group $\Aut(u)$.
Thus $\dbar_J^{-1}(0)$ in this context plays a role analogous
to that of Teichm\"uller space in the previous section: 
it is a somewhat unnatural object whose local structure is nonetheless very 
nice.  Unlike with Teichm\"uller space however, the nice
local structure of $\dbar_J^{-1}(0)$ doesn't come without an extra assumption,
as we need the linearization of $\dbar_J$ to be a surjective operator in
order to apply the implicit function theorem.  When this condition is
satisfied, the result will be a smooth orbifold structure for
$\Mod_{g,m}^A(J)$, with its dimension determined by
the index of the linearization of~$\dbar$.  That's the general idea; we
now proceed with the details.

Suppose $(\Sigma,j,\Theta,u) \in \Mod_{g,m}^A(J)$, and choose a
Teichm\"uller slice $\tT \subset \jJ(\Sigma)$ through~$j$.  For
any $p > 2$, denote
$$
\bB^{1,p} = W^{1,p}(\Sigma,M),
$$
and define a Banach space bundle $\eE^{0,p} \to \tT \times \bB^{1,p}$
whose fibers are
$$
\eE^{0,p}_{(j,u)} = L^p\left(\overline{\Hom}_\CC((T\Sigma,j),(u^*TM,J))\right).
$$
This bundle admits the smooth section
$$
\dbar_J : \tT \times \bB^{1,p} \to \eE^{0,p} : 
(j,u) \mapsto Tu + J \circ Tu \circ j,
$$
whose linearization at $(j,u)$ is
\begin{equation}
\label{eqn:linearization2}
\begin{split}
D\dbar_J(j,u) : T_{j}\tT \oplus W^{1,p}(u^*TM) &\to
L^p(\overline{\Hom}_\CC(T\Sigma,u^*TM)), \\
(y,\eta) &\mapsto J \circ Tu \circ y + \mathbf{D}_{u} \eta,
\end{split}
\end{equation}
where on the right hand side we take $j$ to be the complex structure on the
bundle $T\Sigma$.

\begin{defn}
\label{defn:regular}
We say that the curve $(\Sigma,j,\Theta,u) \in \Mod_{g,m}^A(J)$ is
\defin{Fredholm regular} if the linear operator $D\dbar_J(j,u)$ of
\eqref{eqn:linearization2} is surjective.
\end{defn}

The following lemma implies that our definition of Fredholm regularity 
doesn't depend on the choice of Teichm\"uller slice.  Observe that it is
also an open condition: if $D\dbar_J(j,u)$ is surjective then it will
remain surjective after small changes in $j$, $u$ and~$J$.

\begin{lemma}
The image of $D\dbar_J(j,u)$ doesn't depend on the choice of~$\tT$.
\end{lemma}
\begin{proof}
Let $\mathbf{L} = D\dbar_J(j,u)$ as in \eqref{eqn:linearization2},
and note that $T_{j}\tT$ is a
subspace of $L^p(\overline{\End}_\CC(T\Sigma))$, so $\mathbf{L}$ can be
extended to
\begin{equation*}
\begin{split}
\overline{\mathbf{L}} : L^p(\overline{\End}_\CC(T\Sigma)) \oplus
T_u \bB &\to \eE_{(j,u)}, \\
(y,\eta) &\mapsto J \circ Tu \circ y + \mathbf{D}_u \eta.
\end{split}
\end{equation*}
We claim $\im\overline{\mathbf{L}} = \im\mathbf{L}$.  Indeed, note first
that if $\mathbf{D}_{(j,\Theta)}$ denotes the natural linear Cauchy-Riemann
operator on $(T\Sigma,j)$ and $y = \mathbf{D}_{(j,\Theta)} X \in
\im \mathbf{D}_{(j,\Theta)}$ for some $X \in W^{1,p}_\Theta(T\Sigma)$, then
$$
\overline{\mathbf{L}}(y,0) = J \circ Tu \circ y = Tu(jy) =
Tu(\mathbf{D}_{(j,\Theta)}(j X))
$$
since $u$ is $J$-holomorphic and $\mathbf{D}_{(j,\Theta)}$ is 
complex-linear.  Now the following relation isn't hard to
show: for any smooth vector field $X \in \Gamma(T\Sigma)$ vanishing
on~$\Theta$,
\begin{equation}
\label{eqn:tangential}
\mathbf{D}_u (Tu(X)) = Tu(\mathbf{D}_{(j,\Theta)} X).
\end{equation}
By the density of smooth sections, this extends to all
$X \in W^{1,p}_\Theta(T\Sigma)$, and we conclude
$$
\overline{\mathbf{L}}(y,0) \in \im \mathbf{D}_u
$$
whenever $y \in \im \mathbf{D}_{(j,\Theta)}$.  Since
$L^p(\overline{\End}_\CC(T\Sigma)) = \im \mathbf{D}_{(j,\Theta)} \oplus
T_j \tT$, it follows that $\mathbf{L}$ and $\overline{\mathbf{L}}$ have the
same image.
\end{proof}
\begin{exercise}
Prove the relation \eqref{eqn:tangential} for all smooth vector fields
$X \in \Gamma(T\Sigma)$.  (Compare the proof of Lemma~\ref{lemma:holX}.)
\end{exercise}

Since $T_j\tT$ is finite dimensional and $\mathbf{D}_u$ is Fredholm,
$D\dbar_J(j,u)$ is also Fredholm and has index
\begin{equation}
\label{eqn:dimension}
\begin{split}
\ind D\dbar_J(j,u) &= \dim \tT(\Sigma,\Theta) + \ind \mathbf{D}_u \\
&= \dim \Aut(\Sigma,j,\Theta) -\ind \mathbf{D}_{(j,\Theta)} +
\ind\mathbf{D}_u \\
&= \dim \Aut(\Sigma,j,\Theta) - (3\chi(\Sigma) - 2m) +
\left( n\chi(\Sigma) + 2 c_1(u^*TM) \right) \\
&= \dim \Aut(\Sigma,j,\Theta) + \virdim \Mod_{g,m}^A(J),
\end{split}
\end{equation}
where we've applied \eqref{eqn:TeichDimension}, the Riemann-Roch formula
and the definition of the virtual dimension.

\begin{lemma}
\label{lemma:invtSlice}
For every $j \in \jJ(\Sigma)$, one can choose a Teichm\"uller slice
$\tT$ through $j$ that is invariant under the action of
$\Aut(\Sigma,j,\Theta)$.
\end{lemma}
\begin{proof}
In the case $(\Sigma,\Theta) = (T^2,\emptyset)$, one can assume after a
diffeomorphism that $j$ is translation invariant, and
$\tT$ can then be taken to be
the global Teichm\"uller slice defined in \S\ref{subsec:torus}, consisting
of all translation invariant complex structures compatible with the 
orientation.  In all other cases where $\tT(\Sigma,\Theta)$ is nontrivial,
$(\Sigma,\Theta)$ is stable, thus the group $G := \Aut(\Sigma,j,\Theta)$
is finite.  Using the construction of \eqref{eqn:Phi}, it suffices to
find a complement $C \subset L^p(\overline{\End}_\CC(T\Sigma))$ of
$\im \mathbf{D}_{(j,\Theta)}$ that is $G$-invariant, as one can then
compute that 
$$
j_{\varphi^*y} = \varphi^* j_y
$$ 
for any $\varphi \in \Aut(\Sigma,j,\Theta)$.
To start with, we observe that $\im \mathbf{D}_{(j,\Theta)}$ itself is
$G$-invariant, since $\varphi^*j = j$ also implies
$\mathbf{D}_{(j,\Theta)}(\varphi^*X) = \varphi^*(\mathbf{D}_{(j,\Theta)} X)$
for all $X \in W^{1,p}_\Theta(T\Sigma)$.
A $G$-invariant complement $C$ can then be defined as the 
$L^2$-orthogonal complement of $\im \mathbf{D}_{(j,\Theta)}$
with respect to any $G$-invariant $L^2$-inner product on 
the sections of $\overline{\End}_\CC(T\Sigma)$; such a complement
automatically contains only smooth sections due to linear regularity
for weak solutions (cf.~Corollary~\ref{cor:weakRegularity}).

Since $L^2$-inner products on $\Gamma(\overline{\End}_\CC(T\Sigma))$
arise naturally from $j$-invariant Riemannian metrics on $\Sigma$,
it suffices to find such a Riemannian metric~$g$ which is also $G$-invariant.
Recall from Corollary~\ref{cor:Poincare} that $\Sigma\setminus\Theta$
admits a complete $j$-invariant Riemannian metric $g_P$ of constant
curvature~$-1$, the \emph{Poincar\'e metric}, and it has
the convenient property that the biholomorphic transformations on
$\Sigma\setminus\Theta$ are precisely the isometries of~$g_P$.  This is
not the desired metric since it does not extend over the marked points,
but we can fix this as follows: by Exercise~\ref{EX:invtNbhd} below,
each $z \in \Theta$ admits a $G$-invariant neighborhood $\uU_z$ which can
be biholomorphically identified with the unit ball $B \subset \CC$ such that
$G$ acts by rational rotations.  Thus on $\uU_z$, the Euclidean metric in these
coordinates is also $G$-invariant, and we can interpolate this
with $g_P$ near each $z \in \Theta$ to define the desired $G$-invariant
metric on~$\Sigma$.
\end{proof}

\begin{exercise}
\label{EX:invtNbhd}
Suppose $(\Sigma,j)$ is a Riemann surface and $G$ is a finite group of
biholomorphic maps on $(\Sigma,j)$ which all fix the point $z \in \Sigma$.
Show that $z$ has a $G$-invariant neighborhood $\uU_z$ with a biholomorphic
map $\psi : (\uU_z,j) \to (B,i)$ such that for every $\varphi \in G$,
$\psi \circ \varphi \circ \psi^{-1}$ is a rational rotation.
\end{exercise}

A quick remark about the statement of the next theorem: if
$(j,u) \in \dbar_J^{-1}(0)$ and $\varphi \in \Aut(\Sigma,j,\Theta)$, then
the natural action
$$
\varphi \cdot (j,u) = (j,u\circ \varphi)
$$
preserves $\dbar_J^{-1}(0)$.  Linearizing this action at
$\Id \in \Aut(\Sigma,j,\Theta)$, we obtain a natural map of the
Lie algebra $\aut(\Sigma,j,\Theta)$ to $\ker D\dbar_J(j,u)$ of the form
$$
\aut(\Sigma,j,\Theta) \to \ker D\dbar_J(j,u) : X \mapsto (0,Tu(X)),
$$
and this is an inclusion if $u$ is not constant.  Thus in the following,
we can regard $\aut(\Sigma,j,\Theta)$ as a subspace of
$\ker D\dbar_J(j,u)$.

\begin{thm}
\label{thm:IFT2}
The open subset 
$$
\Mod_{g,m}^{A,\reg}(J) :=  \{ u \in \Mod_{g,m}^A(J)\ |\ 
\text{$u$ is Fredholm regular and not constant} \}
$$ 
naturally admits the structure of a smooth finite-dimensional orbifold with
$$
\dim \Mod_{g,m}^{A,\reg}(J) = \virdim \Mod_{g,m}^A(J).
$$
Its isotropy group at any $(\Sigma,j,\Theta,u) \in \Mod_{g,m}^{A,\reg}(J)$ is
isomorphic to $\Aut(u)$, so in particular, it is a manifold near $u$ if
$\Aut(u)$ is trivial.  There is then also a natural isomorphism
$$
T_u \Mod_{g,m}^A(J) = \ker D\dbar_J(j,u) \big/ \aut(\Sigma,j,\Theta).
$$
Moreover, the evaluation map
$\ev : \Mod_{g,m}^{A,\reg}(J) \to M^m$ is smooth.
\end{thm}
\begin{proof}
We shall prove this in the case where $2g + 3 \ge 0$ and give some hints how to
adapt the argument for the non-stable cases, 
leaving the details as an exercise.

Suppose $(\Sigma,j_0,\Theta,u_0) \in \Mod_{g,m}^A(J)$ is Fredholm regular
and $\tT$ is a Teichm\"uller slice through $j_0$ which is invariant under
the action of $\Aut(\Sigma,j_0,\Theta)$, as supplied by
Lemma~\ref{lemma:invtSlice}.  Then constructing the smooth section
$\dbar_J : \tT \times \bB^{1,p} \to \eE^{0,p}$ as described above, the
implicit function theorem gives
$$
\dbar_J^{-1}(0) \subset \tT \times \bB^{1,p}
$$
near $(j_0,u_0)$ the structure of a smooth submanifold with dimension
$\ind D\dbar_J(j_0,u_0)$.  The latter is equal to
$\virdim \Mod_{g,m}^A(J)$ by \eqref{eqn:dimension},
since $\Aut(\Sigma,j_0,\Theta)$ is in this case discrete.  Observe that
if $z_1,\ldots,z_m \in \Sigma$ denote the marked points $\Theta$, then the evaluation map
$$
\ev : \dbar_J^{-1}(0) \to M^m : (j,u) \mapsto (u(z_1),\ldots,u(z_m))
$$
is smooth as a consequence of the fact that for each $z_i$, the map
$\bB^{1,p} \to M : u \mapsto u(z_i)$ is smooth by
Exercise~\ref{EX:evaluation}.  

Since
$\Aut(\Sigma,j_0,\Theta)$ preserves $\tT$ and acts by biholomorphic maps,
it also acts on $\dbar_J^{-1}(0)$ by
$$
\Aut(\Sigma,j_0,\Theta) \times \dbar_J^{-1}(0) \to \dbar_J^{-1}(0) :
(\varphi , (j,u)) \mapsto (\varphi^*j , u \circ \varphi).
$$
Clearly any two pairs related by this action correspond to equivalent
curves in the moduli space, and we claim in fact that the resulting map
\begin{equation}
\label{eqn:homeo2}
\dbar_J^{-1}(0) \big/ \Aut(\Sigma,j_0,\Theta) \to \Mod_{g,m}^A(J)
\end{equation}
is a local 
homeomorphism onto an open neighborhood of $(\Sigma,j_0,\Theta,u_0)$.
The proof of this uses the fact that $\Diff_0(\Sigma,\Theta)$ acts
freely and properly on $\jJ(\Sigma)$.  

Indeed, to see that \eqref{eqn:homeo2} is surjective onto a neighborhood, 
suppose we have a sequence
$(\Sigma,j_k,\Theta,u_k) \in \Mod^A_{g,m}(J)$ with $j_k \to j_0$ and
$u_k \to u_0$.  Then $[j_k] \to [j_0]$ in $\tT(\Sigma,\Theta)$, so
for sufficiently large~$k$ there are unique diffeomorphisms
$\varphi_k \in \Diff_0(\Sigma,\Theta)$ such that $\varphi_k^*j_k$ is a
sequence in $\tT$ approaching $j_0$.  Now by the properness of the action
(Lemma~\ref{lemma:bubbling}),
a subsequence of $\varphi_k$ converges to an element of
$\Aut(\Sigma,j_0,\Theta)$ which is homotopic to the identity, and
therefore \emph{is} the identity since the action is also free
(Lemma~\ref{lemma:freeAction}).  It
follows that $\varphi_k \to \Id$, thus $u_k \circ \varphi_k \to u_0$ and
for large~$k$, $(\varphi_k^*j_k,u_k \circ \varphi_k)$ lies in an arbitrarily
small neighborhood of $(j_0,u_0)$ in $\dbar_J^{-1}(0)$.

We show now that \eqref{eqn:homeo2} is injective on a sufficiently
small neighborhood of $(j_0,u_0)$.
From Exercise~\ref{EX:stabilizer}, $\Aut(\Sigma,j_0,\Theta)$ is
the stabilizer of $[j_0]$ under the action of $M(\Sigma,\Theta)$ on
$\tT(\Sigma,\Theta)$, thus the natural projection
$$
\tT \big/ \Aut(\Sigma,j_0,\Theta) \to \Mod(\Sigma,\Theta) = \jJ(\Sigma) \big/
\Diff_+(\Sigma,\Theta)
$$
is a local homeomorphism near~$[j_0]$.  Then for any two elements
$(j,u)$ and $(j',u')$ of $\dbar_J^{-1}(0)$ sufficiently close to
$(j_0,u_0)$ that define equivalent holomorphic curves,
$[j] = [j'] \in \Mod(\Sigma,\Theta)$ implies that $j$ and $j'$ are
related by the action of $\Aut(\Sigma,j_0,\Theta)$, and this proves the claim.

We've shown that in a neighborhood of any regular
$(\Sigma,j_0,\Theta,u_0) \in \Mod_{g,m}^A(J)$, the moduli space admits 
an orbifold chart of the correct
dimension.  Its isotropy group at this point is the stabilizer of
$(j_0,u_0)$ under the action of $\Aut(\Sigma,j_0,\Theta)$ on
$\dbar_J^{-1}(0)$, and this is precisely $\Aut(u_0)$.  In particular,
$\Mod_{g,m}^A(J)$ is a manifold near $u_0$ if $\Aut(u_0)$ is trivial,
and the implicit function theorem identifies its tangent space at this point
with $\ker D\dbar_J(j_0,u_0)$.

It remains to show that the transition maps resulting from this construction
are smooth: the zero sets $\dbar_J^{-1}(0)$ inherit natural smooth structures
as submanifolds of $\tT \times \bB^{1,p}$, but we don't yet know that these
smooth structures are independent of all choices.  Put another away, we
need to show that for any two equivalent
curves $(\Sigma,j_0,\Theta,u_0)$ and $(\Sigma',j_0',\Theta',u_0')$ with
corresponding Teichm\"uller slices $\tT$, $\tT'$ and zero sets
$\dbar_J^{-1}(0)$, $(\dbar_J')^{-1}(0)$, there is a smooth local diffeomorphism
$$
\dbar_J^{-1}(0) \to (\dbar_J')^{-1}(0)
$$
that maps $(j_0,u_0) \mapsto (j_0',u_0')$ and maps each
$(j,u) \in \dbar_J^{-1}(0)$ smoothly to an equivalent curve
$(j',u') \in (\dbar_J')^{-1}(0)$.  Let us just consider the case where
$j_0 = j_0'$ and $u_0 = u_0'$ but the Teichm\"uller slices differ, as the
rest is an easy exercise.  For this, we can make use of the work we already
did in constructing the smooth structure of Teichm\"uller space: if
$\tT$ and $\tT'$ are two Teichm\"uller slices through $j_0$, then there is
a diffeomorphism
$$
\tT \to \tT' : j \mapsto j'
$$
such that $j' = \varphi^*_j j$ for some $\varphi_j \in \Diff_0(\Sigma,\Theta)$.
In fact, the diffeomorphism $\varphi_j$ depends smoothly
on~$j$, as we already found a formula for it in the proof of
Theorem~\ref{thm:Teichmueller}:
$$
(\varphi_j,j) = F^{-1} \circ F'(\Id,j'),
$$
where 
\begin{equation*}
\begin{split}
F : \dD^{1,p}_\Theta \times \tT &\to \jJ^{0,p}(\Sigma) : (\varphi,j) \mapsto
\varphi^*j, \\
F' : \dD^{1,p}_\Theta \times \tT' &\to \jJ^{0,p}(\Sigma) : (\varphi,j) \mapsto
\varphi^*j
\end{split}
\end{equation*}
are both smooth local diffeomorphisms near $(\Id,j_0)$.
From this formula it is clear that $\tT' \to \dD^{1,p}_\Theta : 
j' \mapsto \varphi_j$ is a smooth map, thus in light of the diffeomorphism
between $\tT$ and $\tT'$, so is $\tT \to \dD^{1,p}_\Theta : j \mapsto
\varphi_j$.  Moreover, since each $\varphi_j$ is a holomorphic map 
$(\Sigma,j') \to (\Sigma,j)$ with both $j$ and $j'$ smooth, elliptic
regularity implies that $\varphi_j$ is also smooth.  We can now define
a map
$$
\dbar^{-1}(0) \to \jJ^{0,p}(\Sigma) \times \bB^{1,p} : (j,u) \mapsto
(\varphi_j^*j, u \circ \varphi_j),
$$
whose image is clearly in $(\dbar_J')^{-1}(0)$ and thus consists only of
smooth pairs $(j',u')$ which are equivalent to $(j,u)$ in the moduli space.
Moreover, this map is smooth since $u$ is always smooth, again by elliptic
regularity.  This is the desired local diffeomorphism.

The proof is now complete for the case where $(\Sigma,\Theta)$ is stable.
Non-stable cases come in two flavors: the simpler one is the case
$g=0$, for then Teichm\"uller space is
trivial and we can fix $j=i$ on $S^2$.  Several details then simplify,
except that now $\Aut(S^2,i,\Theta)$ has positive 
dimension---nonetheless it is straightforward 
to see that \eqref{eqn:homeo2}
is still a local homeomorphism, so the only real difference in the end is the
computation of the dimension,
$$
\dim \Mod_{g,m}^A(J) = \ind D\dbar_J(j,u) - \dim \Aut(\Sigma,j,\Theta) =
\virdim \Mod_{g,m}^A(J),
$$
due to \eqref{eqn:dimension}.  In the case of $\Mod_{1,0}^A(J)$,
for which both Teichm\"uller space and the
automorphism groups have positive dimension, we can use the specific global
Teichm\"uller slice of \S\ref{subsec:torus}, and combine ideas from the
stable and genus~$0$ cases to obtain the same result and
same dimension formula in general.
\end{proof}

\begin{exercise}
Work out the details of the proof of Theorem~\ref{thm:IFT2} in the
non-stable cases.  \textsl{(For a more detailed exposition of this in a more
general context, see \cite{Wendl:automatic}*{\S 3.2}, the proof of
Theorem~0.)}
\end{exercise}

The implicit function theorem gives more than just a manifold or orbifold
structure for $\Mod_{g,m}^A(J)$: it can also be used for perturbation
arguments, in which the existence of curves in $\Mod_{g,m}^A(J)$ gives
rise to curves in $\Mod_{g,m}^A(J')$ as well, for any $J'$ sufficiently
close to~$J$.  We stated one result along these lines already,
Theorem~\ref{thm:genericHomotopy}.  For another example, assume
$$
\{ J_s \}_{s \in (-1,1)}
$$
is a smooth $1$-parameter family of almost complex structures on~$M$.

\begin{thm}
\label{thm:IFTparametrized}
Suppose $(\Sigma,j_0,\Theta,u_0) \in \Mod_{g,0}^A(J_0)$ is simple and 
Fredholm regular with $\ind(u_0) = 0$.
Then for sufficiently small $\epsilon > 0$,
$j_0$ and $u_0$ extend to a smooth family
of complex structures $j_s$ and maps $u_s : \Sigma \to M$ for
$s \in (-\epsilon,\epsilon)$ such that
$$
(\Sigma,j_s,\Theta,u_s) \in \Mod_{g,0}^A(J_s).
$$
Moreover this family is unique, in the sense that 
for any sequence $s_k \to 0$ and 
$(\Sigma,j_k',\Theta,u_k') \in \Mod_{g,0}^A(J_{s_k})$ with
$j_k' \to j_0$ and $u_k' \to u_0$ in the $C^\infty$-topology,
we have
$$
(\Sigma,j_k',\Theta,u_k') \sim (\Sigma,j_{s_k},\Theta,u_{s_k})
$$
for sufficiently large~$k$.
\end{thm}

The actual meaning of this theorem is that 
the \emph{parametrized} moduli space of holomorphic curves,
$$
\Mod_{g,0}^A(\{ J_s \}) = \{ (s,u) \ |\ 
s \in (-1,1), \ u \in \Mod_{g,0}^A(J_s) \}
$$
is a smooth $1$-dimensional manifold near $(0,u_0)$, and the latter is
also a regular point of the projection $\Mod_{g,0}^A(\{J_s\}) \to (-1,1) :
(s,u) \mapsto s$.  This is practically
automatic, as one only has to generalize the previous setup a bit:
given a Teichm\"uller slice $\tT$ through $j_0$,
redefine $\eE^{0,p}$ to be a Banach space bundle over
$\tT \times \bB^{1,p} \times (-1,1)$, with fiber
$$
\eE^{0,p}_{(j,u,s)} = L^p(\overline{\Hom}_\CC((T\Sigma,j),(u^*TM,J_s))).
$$
Then we again have a natural smooth section
$$
\dbar : \tT \times \bB^{1,p} \times (-1,1) \to \eE^{0,p} :
(j,u,s) \mapsto Tu + J_s \circ Tu \circ j,
$$
whose zero set near $(j_0,u_0,0)$ can---at least in the absence of 
automorphisms---be identified with $\Mod_{g,0}^A(\{ J_s \})$.
Thus one only has to understand the linearization of $\dbar$ at
$(j_0,u_0,0)$, which is essentially the usual $D\dbar_{J_0}(j_0,u_0)$ with one
extra dimension attached to the domain, raising its index to~$1$.  
If $u_0$ is regular then
$D\dbar_{J_0}(j_0,u_0)$ is surjective and it follows immediately that
$D\dbar(j_0,u_0,0)$ is surjective as well, so the result follows 
as usual from the implicit function theorem.  

A result
of this kind can be stated more generally for any Fredholm regular
curve with nonnegative index, and for any parametrized family of almost
complex structures.  For this reason, regular curves are also
often referred to as \defin{unobstructed}.

\section{Transversality for generic~$J$}
\label{sec:transversality}

In the previous section we proved that moduli spaces of $J$-holomorphic curves
are smooth wherever they are Fredholm regular.  Since Fredholm
regularity is in general a very difficult condition to check, in this section 
we will examine ways of ensuring regularity via generic perturbations of~$J$,
leading in particular to a proof of Theorem~\ref{thm:localStructure}.

We assume throughout this section that $(M,\omega)$ is a $2n$-dimensional
symplectic manifold without boundary, and we focus on $\omega$-compatible almost 
complex structures,
though all of our results have easily derived analogues for $\omega$-tame
or general almost complex structures
(cf.~Remark~\ref{remark:symplecticGeneric}).  We will not assume that $M$ is
compact unless specifically stated, but will fix an open subset
$\uU \subset M$ with compact closure.  
Recall from \S\ref{sec:moduliSpace} the definition of the space
$\jJ(M,\omega\,;\,\uU,\Jfix)$ of compatible almost complex structures that
are fixed outside of~$\uU$; here $\Jfix \in \jJ(M,\omega)$ is an arbitrary
choice that we assume fixed in advance (which is irrelevant if $\uU = M$).
Fix also a pair of integers $g , m \ge 0$ and a homology class $A \in H_2(M)$.

\begin{defn}
\label{defn:Jreg}
Let
$$
\jJ_\reg(M,\omega\,;\,\uU,\Jfix\,;\,g,m,A) \subset \jJ(M,\omega\,;\,\uU,\Jfix)
$$
denote the set of all $J \in \jJ(M,\omega\,;\,\uU,\Jfix)$ such that
every curve $u \in \Mod_{g,m}^A(J)$ with an injective point mapped
into~$\uU$ is Fredholm regular.  
\end{defn}

For applications involving the evaluation map $\ev : \Mod_{g,m}^A(J) \to M^m$,
it will be useful to generalize this definition given the additional data
of a smooth submanifold $Z \subset M^m$ without boundary.  The reader who
is only interested in the proof of Theorem~\ref{thm:localStructure} and not
the further applications in \S\ref{sec:evaluation} is free in the following
to ignore all references to $Z$, or assume $M=Z$, in which case all conditions
involving $Z$ will be vacuous.

\begin{defn}
\label{defn:JregZ}
Given the same data as in Definition~\ref{defn:Jreg} plus a smooth
submanifold $Z \subset M^m$ without boundary, let
$$
\jJ_\reg^Z(M,\omega\,;\,\uU,\Jfix\,;\,g,m,A) \subset \jJ(M,\omega\,;\,\uU,\Jfix)
$$
denote the set of all $J \in \jJ(M,\omega\,;\,\uU,\Jfix)$ such that
every curve $u \in \Mod_{g,m}^A(J)$ that satisfies $\ev(u) \in Z$ and maps
an injective point mapped into~$\uU$ is Fredholm regular, and the
intersection of $\ev : \Mod_{g,m}^A(J) \to M^m$ with $Z$ at~$u$ is transverse.
\end{defn}

Here is the main result of this section.

\begin{thm}
\label{thm:generic}
Given $(M,\omega)$ with the data $\uU$, $\Jfix$, $g$, $m$, $A$ and $Z$ as
described above,
$\jJ_\reg^Z(M,\omega\,;\,\uU,\Jfix\,;\,g,m,A)$ is a Baire subset of 
$\jJ(M,\omega\,;\,\uU,\Jfix)$.
\end{thm}

Taking $Z=M$, this result together with Theorem~\ref{thm:IFT2} implies
Theorem~\ref{thm:localStructure}, as we can take $\jJ_\reg(M,\omega\,;\,\uU,\Jfix)$
to be the countable intersection
$$
\jJ_\reg(M,\omega\,;\,\uU,\Jfix) := \bigcap_{g , m \ge 0,\ A \in H_2(M)}
\jJ_\reg(M,\omega\,;\,\uU,\Jfix\,;\,g,m,A).
$$
Some consequences of the case $Z \subsetneq M$ will be described in
\S\ref{sec:evaluation}.

The proof will proceed in two main steps, described in the next two
subsections.

\subsection{Regular almost complex structures are dense}
\label{subsec:dense}

In order to cut down on cumbersome notation, let us assume for the
remainder of \S\ref{sec:transversality} that the choices
$\uU \subset M$, $\Jfix \in \jJ(M,\omega)$, $g \ge 0$, $m \ge 0$, 
$A \in H_2(M)$ and $Z \subset M$ are all fixed, so we can abbreviate
$$
\jJ_\reg := \jJ_\reg^Z(M,\omega\,;\,\uU,\Jfix\,;\,g,m,A).
$$
We begin by proving a weaker version of Theorem~\ref{thm:generic}, which
nonetheless suffices for most applications.  

\begin{prop}
\label{prop:dense}
$\jJ_\reg$ is dense in $\jJ(M,\omega\,;\,\uU,\Jfix)$.
\end{prop}

Though certainly useful on its own, this statement is less beautiful
than Theorem~\ref{thm:generic} and sometimes also less convenient, as
countable intersections of dense subsets are not generally dense
(they may even be empty).
It will be the purpose of the next subsection to replace the word ``dense''
with ``Baire,'' using an essentially topological
argument originally due to Taubes.

Let us sketch the proof of Prop.~\ref{prop:dense} before getting into
the details.  One must first choose a smooth
Banach manifold of almost complex structures $\jJ_\epsilon \subset
\jJ(M,\omega\,;\,\uU,\Jfix)$ in which
to vary~$J$.  One can then define a (large) separable Banach manifold 
that contains all suitable holomorphic
curves in all the moduli spaces $\Mod_{g,m}^A(J)$
for $J \in \jJ_\epsilon$, called the \defin{universal moduli space},
$$
\univ^*(\jJ_\epsilon) = \{ (u,J) \ |\ J \in \jJ_\epsilon,\ 
u \in \Mod_{g,m}^A(J) \ \text{maps an injective point into~$\uU$} \},
$$
along with its constrained variant
$$
\univ^*_Z(\jJ_\epsilon) = \{ (u,J) \in \univ^*(\jJ_\epsilon) \ |\ 
\ev(u) \in Z \}.
$$
It takes a bit of care to make sure these spaces really are
Banach manifolds: as usual, the main task will be to prove that a certain
linear operator between Banach spaces is surjective, and
this is where the assumption of an injective point in $\uU$ will turn out
to be crucial.  It will also require the domain 
to be sufficiently large---in particular, $\jJ_\epsilon$ will have to contain
a certain set of $C^\infty_0$-perturbations of a
given~$J$, and must therefore be infinite dimensional.
Once the universal moduli space is understood, 
we have a natural smooth projection map
$$
\pi : \univ^*_Z(\jJ_\epsilon) \to \jJ_\epsilon : (u,J) \mapsto J,
$$
whose preimage $\pi^{-1}(J)$ at any $J \in \jJ_\epsilon$ is
precisely the set of all curves in $u \in \Mod_{g,m}^A(J)$ that map an injective
point into~$\uU$ and satisfy $\ev(u) \in Z$.  This will be
a smooth submanifold whenever $J$ is a regular 
value of~$\pi$, i.e.~the derivative $d\pi(u,J)$ is surjective for all
$(u,J) \in \pi^{-1}(J)$.
In finite dimensions, Sard's theorem would
tell us that this is true for almost every $J$, and in the present
situation one can apply the following infinite-dimensional version
due to Smale \cite{Smale:Sard}.

\begin{SardSmale}
Suppose $X$ and $Y$ are smooth Banach manifolds which are separable
and paracompact, and $f : X \to Y$ is a smooth map whose derivative
$df(x) : T_x X \to T_{f(x)} Y$ for every $x \in X$ is Fredholm.  
Then the regular values of $f$ form a Baire subset of~$Y$.
\end{SardSmale}
The theorem can be stated more generally for nonsmooth maps
$f \in C^k(X,Y)$ if $k$ is sufficiently large, but we will not need this.
A proof in the case where $f$ maps an open subset of a linear Banach
space to another Banach space may be found in
\cite{McDuffSalamon:Jhol}*{Appendix~A.5}.  The general case can be derived
from this, with the aid of the following exercise in general
topology (cf.~Proposition~\ref{prop:notExotic}).

\begin{exercise}
\label{EX:topologicalHorror}
Show that any Banach manifold that is both separable and paracompact admits a
\emph{countable} family of charts.
\end{exercise}

To apply the Sard-Smale theorem, we need to know that
$d\pi(u,J)$ is a Fredholm operator.  In the unconstrained case $Z=M$,
it turns out that
$d\pi(u,J)$ not only is Fredholm but has the same index and the same kernel 
as the linearization \eqref{eqn:linearization2} that defines Fredholm 
regularity, thus every regular value of $\pi$ belongs to~$\jJ_\reg$.
A similar argument works in the constrained case, and the Sard-Smale theorem
will thus imply that $\jJ_\reg$ is dense, as claimed by Prop.~\ref{prop:dense}.

In fact, the argument implies that the set of regular
almost complex structures is a Baire subset of $\jJ_\epsilon$, and
you may at this point be wondering why that doesn't already prove
Theorem~\ref{thm:generic}.  The answer is that we cannot simply
choose $\jJ_\epsilon$ to be $\jJ(M,\omega\,;\,\uU,\Jfix)$, as the latter with
its natural $C^\infty$-topology is not a Banach manifold, 
so the Sard-Smale theorem does not apply.  
We are thus forced to choose a somewhat less natural
space of varying almost complex structures, with a sufficiently 
different topology so that a Baire subset of $\jJ_\epsilon$ is not obviously
a Baire subset of $\jJ(M,\omega\,;\,\uU,\Jfix)$, though we will easily see 
that it is dense.  Extending density to ``genericity'' will require an 
additional topological argument, given in the next subsection.

We now carry out the details, starting with the definition of the
Banach manifold~$\jJ_\epsilon$.  It will be convenient to have explicit
local charts for the manifold of compatible complex structures on a vector
space, as provided by the following exercise.

\begin{exercise}
\label{EX:Tomega}
Suppose $\omega$ is a nondegenerate $2$-form on a $2n$-dimensional vector
space $V$, and $\jJ(V,\omega)$ denotes the space of all complex structures
$J$ on $V$ such that $\omega(\cdot,J\cdot)$ defines a symmetric
inner product.  Show that $\jJ(V,\omega)$ is a smooth submanifold of
$\jJ(V)$, whose tangent space at $J \in \jJ(V,\omega)$ is
$$
\overline{\End}_\CC(V,J,\omega)
:= \{ Y \in \overline{\End}_\CC(V,J)\ |\ 
\omega(v, Yw) + \omega(Yv,w) = 0 \text{ for all $v,w \in V$} \}.
$$
Show also that for any $J \in \jJ(V,\omega)$, the correspondence
\begin{equation}
\label{eqn:YtoJ2}
Y \mapsto \left( \1 + \frac{1}{2} J Y \right) J
\left( \1 + \frac{1}{2} J Y \right)^{-1}
\end{equation}
maps a neighborhood of~$0$ in $\overline{\End}_\CC(V,J,\omega)$
diffeomorphically to a neighborhood of~$J$ in $\jJ(V,\omega)$.
\textsl{Hint: Recall Corollary~\ref{cor:compatible}.}
\end{exercise}

There are two standard approaches for defining a Banach manifold
of perturbed almost complex structures:
one of them, which is treated in 
\cite{McDuffSalamon:Jhol}*{\S 3.2}, is to work in the space $\jJ^m(M,\omega)$
of $C^m$-smooth almost complex structures for sufficiently large
$m \in \NN$, and afterwards argue (using the ideas described in
\S\ref{subsec:Taubes} below) that the intersection of all the
spaces $\jJ^m_{\reg}(M,\omega)$ gives a Baire subset of~$\jJ(M,\omega)$.
The drawback of this approach is that if $J$ is not smooth, then
the Cauchy-Riemann operator will also have only finitely many derivatives:
indeed, $\dbar_J u = T u + J(u) \circ Tu \circ j$ involves the composition
map
\begin{equation}
\label{eqn:composition}
(u,J) \mapsto J \circ u
\end{equation}
which may be differentiable but is not smooth unless $J$~is
(recall Lemma~\ref{lemma:smoothness}).  This approach thus forces one to
consider Banach manifolds and maps with only finitely many derivatives,
causing an extra headache that we'd hoped to avoid after we proved
elliptic regularity in Chapter~\ref{chapter:local}.

The alternative approach is to stay within the smooth context by defining
$\jJ_\epsilon$ to be a Banach manifold that admits a continuous inclusion
into $\jJ(M,\omega)$: indeed, if $\jJ_\epsilon$ embeds continuously into
$\jJ^m(M,\omega)$ for every $m \in \NN$ and $u$ belongs to a Banach
manifold such as $W^{k,p}(\Sigma,M)$, then Lemma~\ref{lemma:smoothness} 
implies that
\eqref{eqn:composition} will be smooth.  Until now, all examples we've seen
of Banach spaces that embed continuously into $C^\infty$ have been
finite dimensional, and we would find such a space too small to
ensure the smoothness of the universal moduli space.  A suitable 
infinite-dimensional example was introduced by Floer 
\cite{Floer:action}, and has become known commonly as the ``Floer
$C_\epsilon$-space''.

Fix an arbitrary ``reference'' almost complex structure
$\Jref \in \jJ(M,\omega\,;\,\uU,\Jfix)$, and choose a sequence of positive real
numbers $\epsilon_\nu \to 0$ for integers $\nu \ge 0$.  Recall from
Exercise~\ref{EX:Tomega} the vector bundle $\overline{\End}_\CC(TM,\Jref,\omega)$,
whose smooth sections constitute what we think of as the 
``tangent space $T_{\Jref} \jJ(M,\omega)$.''
Define $C_\epsilon(\overline{\End}_\CC(TM,\Jref,\omega)\,;\,\uU)$ to be the
space of smooth sections $Y$ of
$\overline{\End}_\CC(TM,\Jref,\omega)$ with support in $\overline{\uU}$ 
for which the norm
$$
\| Y \|_\epsilon := \sum_{\nu=0}^\infty \epsilon_\nu \| Y \|_{C^\nu(\uU)}
$$
is finite.  Though it is not immediately clear whether this space
contains any nontrivial sections, it is at least
a Banach space, and it has a natural continuous
inclusion into the space of smooth sections supported in $\overline{\uU}$,
$$
C_\epsilon(\overline{\End}_\CC(TM,\Jref,\omega)\,;\,\uU) \hookrightarrow
\left\{ Y \in \Gamma(\overline{\End}_\CC(TM,\Jref,\omega))\ \big|\ 
Y|_{M\setminus \uU} \equiv 0 \right\}.
$$
One can always theoretically enlarge the space
by making the sequence $\epsilon_\nu$ converge to~$0$ faster.  As it
turns out, choosing $\epsilon_\nu$ small enough makes
$C_\epsilon(\overline{\End}_\CC(TM,\Jref,\omega)\,;\,\uU)$ into an infinite-dimensional
space that contains bump functions with small support and 
arbitrary values at any point in~$\uU$:

\begin{lemma}
\label{lemma:bump}
Suppose $\beta : B^{2n} \to [0,1]$ is a smooth function with compact support
on the unit ball $B^{2n} \subset \CC^n$ and $\beta(0)=1$.
One can choose a sequence of positive numbers $\epsilon_\nu \to 0$ such that 
for every $Y_0 \in \CC^N$ and $r > 0$, the function $Y : \CC^n \to \CC^N$
defined by
$$
Y(p) := \beta(p/r) Y_0
$$
satisfies $\sum_{\nu=0}^\infty \epsilon_\nu \| Y \|_{C^\nu} < \infty$.
\end{lemma}
\begin{proof}
Define $\epsilon_\nu > 0$ so that for $\nu \ge 1$,
$$
\epsilon_\nu = \frac{1}{\nu^\nu \| \beta \|_{C^\nu}}.
$$
Then
$$
\sum_{\nu=1}^\infty \epsilon_\nu \| Y \|_{C^\nu} \le
\sum_{\nu=1}^\infty \frac{1}{\nu^\nu \| \beta \|_{C^\nu}} \frac{\| \beta \|_{C^\nu}}{r^\nu}
= \sum_{\nu=1}^\infty \left( \frac{1 / r}{\nu} \right)^\nu < \infty.
$$
\end{proof}

\begin{exercise}[cf.~\cite{Floer:action}*{Lemma~5.1}]
Show that by choosing $\epsilon_\nu$ as in the lemma, one can arrange so that
$C_\epsilon(\overline{\End}_\CC(TM,\Jref,\omega)\,;\,\uU)$ is dense in the
space of $L^2$-sections of $\overline{\End}_\CC(TM,\Jref,\omega)$
that vanish on $M \setminus \uU$.
\end{exercise}

\begin{exercise}
\label{EX:IdontKnowHowToDoThisOne}
Prove that $C_\epsilon(\overline{\End}_\CC(TM,\Jref,\omega)\,;\,\uU)$ is
separable.
\end{exercise}

Now choose $\delta > 0$ sufficiently small so that the correspondence
\eqref{eqn:YtoJ2} with $J := \Jref$ defines an injective map
\begin{equation*}
\left\{ Y \in C_\epsilon(\overline{\End}_\CC(TM,\Jref,\omega)\,;\,\uU)\ \big|\ \| Y \|_\epsilon
< \delta \right\} \to \jJ(M,\omega\,;\,\uU,\Jfix),
\end{equation*}
and define $\jJ_\epsilon$ to be its image.
By construction, $\jJ_\epsilon$ is a smooth, separable and metrizable
Banach manifold (with only one chart),
which contains $\Jref$ and embeds continuously into $\jJ(M,\omega\,;\,\uU,\Jfix)$.
Its tangent space at any $J \in \jJ_\epsilon$ can be written naturally as
$$
T_J \jJ_\epsilon = C_\epsilon(\overline{\End}_\CC(TM,J,\omega)\,;\,\uU).
$$

As already sketched above, we now define the universal moduli space
$\univ^*(\jJ_\epsilon)$ to be the space of pairs $(u,J)$ for which
$J \in \jJ_\epsilon$ and $u \in \Mod_{g,m}^A(J)$ has an injective point
mapped into~$\uU$, and let $\univ^*_Z(\jJ_\epsilon) = \ev^{-1}(Z)$ for the
obvious extension of the evaluation map
$$
\ev : \univ^*(\jJ_\epsilon) \to M^m : (u,J) \mapsto \ev(u).
$$

\begin{prop}
\label{prop:universal}
The universal moduli space $\univ^*(\jJ_\epsilon)$ admits the structure of a
smooth, separable and metrizable Banach manifold such that the natural
projection $\pi : \univ^*(\jJ_\epsilon) \to \jJ_\epsilon : (u,J) \mapsto J$ 
and the evaluation map $\ev : \univ^*(\jJ_\epsilon) \to M^m$ are both smooth, 
and the latter is a submersion.
\end{prop}

To prove this, choose any representative $(\Sigma,j_0,\Theta,u_0)$
of an arbitrary curve $u_0 \in \Mod_{g,m}^A(J_0)$ for which 
$(u_0,J_0) \in \univ^*(\jJ_\epsilon)$, and choose a
Teichm\"uller slice $\tT$ through~$j_0$ as in \S\ref{sec:IFT}.  
A neighborhood of $(u_0,J_0)$ in
$\univ^*(\jJ_\epsilon)$ can then be described\footnote{Here we are 
restricting for the sake of
notational simplicity to the case where $(\Sigma,\Theta)$ is stable; we
leave the details of the non-stable cases as an exercise.} as the
zero set of a smooth section,
$$
\dbar : \tT \times \bB^{1,p} \times \jJ_\epsilon \to
\eE^{0,p} : (j,u,J) \mapsto Tu + J \circ Tu \circ j,
$$
where now $\eE^{0,p}$ has been extended to a Banach space bundle over
$\tT \times \bB^{1,p} \times \jJ_\epsilon$ with fiber
$$
\eE^{0,p}_{(j,u,J)} = L^p\left(\overline{\Hom}_\CC((T\Sigma,j),(u^*TM,J))\right).
$$
The linearization 
$D\dbar(j_0,u_0,J_0) : T_{j_0} \tT \oplus T_{u_0}\bB^{1,p} \oplus
T_{J_0} \jJ_\epsilon \to \eE^{0,p}_{(j_0,u_0,J_0)}$
takes the form
$$
(y,\eta,Y) \mapsto J_0 \circ Tu_0 \circ y + \mathbf{D}_{u_0} \eta +
Y \circ Tu_0 \circ j_0.
$$
The essential technical work is now contained in the following lemma.
We denote
$$
W^{1,p}_\Theta(u_0^*TM) := \left\{ \eta \in W^{1,p}(u_0^*TM) \ |\ 
\eta(\Theta) = 0 \right\},
$$
which is a closed subspace of codimension $2nm$ in $W^{1,p}(u_0^*TM)$.  

\begin{lemma}
\label{lemma:surjective}
If $u_0$ maps an injective point into~$\uU$, then the operator
\begin{equation*}
\begin{split}
\mathbf{L} : W^{1,p}_\Theta(u_0^*TM) \oplus 
C_\epsilon(\overline{\End}_\CC(TM,J_0,\omega)\,;\,\uU) &\to
L^p(\overline{\Hom}_\CC(T\Sigma,u_0^*TM)) \\
(\eta,Y) &\mapsto \mathbf{D}_{u_0} \eta +
Y \circ Tu_0 \circ j_0
\end{split}
\end{equation*}
is surjective and has a bounded right inverse.
\end{lemma}
\begin{proof}
If $\mathbf{L}$ is surjective then the existence of a bounded right
inverse follows easily since $\mathbf{D}_{u_0}$ is Fredholm.  Moreover,
the Fredholm property of $\mathbf{D}_{u_0}$ implies that
$\im \mathbf{L}$ is closed, 
thus choosing a suitable bundle metric to define the $L^2$-pairing,
it suffices (by the Hahn-Banach theorem)
to show that there is no nontrivial section
$\alpha \in L^q(\overline{\Hom}_\CC(T\Sigma,u_0^*TM))$ with 
$\frac{1}{p} + \frac{1}{q} = 1$ such that
$\langle \mathbf{L} (\eta,Y) , \alpha \rangle_{L^2} = 0$ for all
$(\eta,Y)$ in the specified domain.  This can be broken down into two 
conditions:
\begin{equation*}
\begin{split}
\langle \mathbf{D}_{u_0} \eta , \alpha \rangle_{L^2} = 0 &
\ \text{ for all $\eta \in W^{1,p}_\Theta(u_0^*TM)$, and}\\
\langle Y \circ Tu_0 \circ j_0 , \alpha \rangle_{L^2} = 0 &
\ \text{ for all $Y \in C_\epsilon(\overline{\End}_\CC(TM,J_0,\omega)\,;\,\uU)$.}
\end{split}
\end{equation*}
If such $\alpha$ exists, then the first of these two equations implies it is
a weak solution of the formal adjoint equation $\mathbf{D}_{u_0}^*\alpha = 0$
on $\Sigma \setminus \Theta$,
thus by regularity of weak solutions (Corollary~\ref{cor:weakRegularity}), 
it is smooth on $\Sigma\setminus\Theta$, and the similarity principle 
(\S\ref{sec:similarity}) implies that its zero set cannot accumulate.  The 
idea is now to choose $Y \in C_\epsilon(\overline{\End}_\CC(TM,J_0,\omega)\,;\,\uU)$ so
that the second equation implies $\alpha$ \emph{must} vanish on some nonempty
open set, yielding a contradiction.  There are two important details of
our setup that make this possible:
\begin{enumerate}
\item $u_0$ has an injective point $z_0 \in \Sigma$ with $u_0(z_0) \in \uU$;
\item $C_\epsilon(\overline{\End}_\CC(TM,J_0,\omega)\,;\,\uU)$ contains bump functions
with small support and arbitrary values at~$u_0(z_0)$.
\end{enumerate}
Indeed, since the set of injective points is open and $\alpha$ has only
isolated zeroes, we can assume without loss of generality that 
$z_0 \in \uU$ is not one of the marked points and
$\alpha(z_0) \ne 0$.  Now choose (via Lemma~\ref{lemma:bump} and
Lemma~\ref{lemma:choice} below)
$Y \in C_\epsilon(\overline{\End}_\CC(TM,J_0,\omega)\,;\,\uU)$ so that 
$\langle Y \circ Tu_0 \circ j_0 , \alpha \rangle$ is positive on a
neighborhood of $z_0$ and vanishes outside this neighborhood.
Then $\langle Y \circ Tu_0 \circ j_0 , \alpha \rangle_{L^2}$ cannot be zero,
and we have the desired contradiction.
Observe the role that somewhere injectivity plays here: 
$Tu_0 \circ j_0$ is nonzero near $z_0$ since $du_0(z_0) \ne 0$,
and since $u_0$ passes through $z_0$ only once (and the same is 
obviously true for points in a small neighborhood of~$z_0$), fixing the
value of $Y$ near $u_0(z_0)$ only affects the $L^2$-product near~$z_0$ and
nowhere else.  This is why the same proof fails for multiply covered curves.
\end{proof}

In choosing the bump function 
$Y \in C_\epsilon(\overline{\End}_\CC(TM,J_0,\omega)\,;\,\uU)$ in the above proof, 
we implicitly made use of a
simple linear algebra lemma.  This is the only point
in the argument where the symplectic structure makes any difference: it
shrinks the space of available perturbations $Y$ along $J_0$, but the lemma
below shows that this space is still large enough.  Recall that on any
symplectic vector space $(V,\omega)$ with compatible complex structure $J$,
one can choose a basis to identify $J$ with $i$ and $\omega$ with the
standard structure~$\omega\std$ (cf.~Exercise~\ref{EX:Un}).  
The linear maps $Y$ that anticommute with
$i$ and satisfy $\omega\std(Y v, w) + \omega\std(v, Yw) = 0$ for all $v,w \in V$
are then precisely the \emph{symmetric} matrices that are complex antilinear.
\begin{lemma}
\label{lemma:choice}
For any nonzero vectors $v , w \in \RR^{2n}$, there exists a symmetric
matrix $Y$ that anticommutes with $i$ and satisfies $Y v = w$.
\end{lemma}
\begin{proof}
We borrow the proof directly from \cite{McDuffSalamon:Jhol}*{Lemma~3.2.2}
and simply state a formula for $Y$:
\begin{multline*}
Y = \frac{1}{|v|^2} \left( w v^T + v w^T + i\left( wv^T + vw^T\right) i \right) \\
- \frac{1}{|v|^4} \left( \langle w,v \rangle \left( v v^T + i v v^T i\right)
- \langle w , iv \rangle \left( i v v^T - v v^T i \right) \right),
\end{multline*}
where $\langle\ ,\ \rangle$ denotes the standard real inner product on
$\RR^{2n} = \CC^n$.
\end{proof}

\begin{proof}[Conclusion of the proof of Proposition~\ref{prop:universal}]
Since $\tT$ is finite dimensional,
$W^{1,p}_\Theta(u_0^*TM) \oplus 
C_\epsilon(\overline{\End}_\CC(TM,J_0,\omega)\,;\,\uU)$ is a closed subspace
of finite codimension in $T_{j_0} \tT \oplus T_{u_0}\bB^{1,p} \oplus
T_{J_0}\jJ_\epsilon$, hence
Lemma~\ref{lemma:surjective} implies that $D\dbar(j_0,u_0,J_0)$ is also
surjective and has a bounded right inverse.  By the implicit function theorem, 
a neighborhood of $(j_0,u_0,J_0)$ in $\dbar^{-1}(0)$ is now a smooth
Banach submanifold of $\tT \times \bB^{1,p} \times \jJ_\epsilon$.
Repeating several details of the proof of
Theorem~\ref{thm:IFT2} and exploiting the fact that $\Aut(u)$ is always
trivial when $u$ is somewhere injective,
it follows also that $\univ^*(\jJ_\epsilon)$ is a smooth 
(and separable and metrizable)
Banach manifold: locally, it can be identified with
$\dbar^{-1}(0)$, and its tangent space at $(u,J)$ is
$$
T_{(u,J)} \univ^*(\jJ_\epsilon) = \ker D\dbar(j,u,J) \subset
T_j\tT \oplus W^{1,p}(u^*TM) \oplus
C_\epsilon(\overline{\End}_\CC(TM,J,\omega)\,;\,\uU).
$$
Under this local identification, the projection $\pi : \univ^*(\jJ_\epsilon) 
\to \jJ_\epsilon$ is simply the restriction to
$\dbar^{-1}(0)$ of the projection
$$
\tT \times \bB^{1,p} \times \jJ_\epsilon \to \jJ_\epsilon : 
(j,u,J) \mapsto J
$$
and is thus obviously smooth.  Writing the marked points as
$\Theta = (z_1,\ldots,z_m)$, the evaluation map is similarly the 
restriction to $\dbar^{-1}(0)$ of
$$
\tT \times \bB^{1,p} \times \jJ_\epsilon \to M^m : (j,u,J) \mapsto
(u(z_1),\ldots,u(z_m)),
$$
which is smooth by Exercise~\ref{EX:evaluation}, and its derivative 
at $(j,u,J)$ on this larger domain is the linear map
\begin{equation*}
\begin{split}
T_{j}\tT \oplus W^{1,p}(u^*TM) \oplus
C_\epsilon(\overline{\End}_\CC(TM,J,\omega)\,;\,\uU) &\to T_{u(z_1)}M \oplus
\ldots \oplus T_{u(z_m)}M,\\
(y,\eta,Y) &\mapsto (\eta(z_1),\ldots,\eta(z_m)).
\end{split}
\end{equation*}
To prove that $\ev$ is a submersion at $(u,J) \in \univ^*(\jJ_\epsilon)$, 
we therefore need to show that for any given set of tangent vectors 
$\xi_i \in T_{u(z_i)}M$ for $i=1,\ldots,m$, we can find a triple
$(y,\eta,Y) \in \ker D\dbar(j,u,J)$ such that
$\eta(z_i) = \xi_i$ for $i=1,\ldots,m$.  To see this, pick
any smooth section $\xi \in \Gamma(u^*TM)$ that satisfies
$\xi(z_i) = \xi_i$ for $i=1,\ldots,m$, then use
Lemma~\ref{lemma:surjective} to 
find $\eta \in W^{1,p}(u^*TM)$ and
$Y \in C_\epsilon(\overline{\End}_\CC(TM,J,\omega)\,;\,\uU)$ such that
$\eta$ vanishes at each of the marked points $z_1,\ldots,z_m$ and
$$
\mathbf{D}_{u} \eta + Y \circ Tu \circ j = -\mathbf{D}_{u}\xi.
$$
The desired solution is then $(0,\xi + \eta,Y)$.
The proof of Proposition~\ref{prop:universal} is now complete.
\end{proof}

To finish the proof of Proposition~\ref{prop:dense}, note first that
$\univ^*_Z(\jJ_\epsilon) := \ev^{-1}(Z) \subset \univ^*(\jJ_\epsilon)$ is
also a smooth Banach submanifold since $\ev : \univ^*(\jJ_\epsilon) \to M^m$ 
is a submersion.  
Given $(u,J) \in \univ^*_Z(\jJ_\epsilon)$ with $u$ represented by
$(\Sigma,j,\Theta,u) \in \Mod_{g,m}^A(J)$ and the marked points written as
$\Theta = (z_1,\ldots,z_m)$, identify a neighborhood of $(u,J)$ in
$\univ^*(\jJ_\epsilon)$ with $\dbar^{-1}(0)$ as in the above proof.
Then defining the finite-codimensional subspace 
$$
W^{1,p}_Z(u^*TM) = \left\{ \eta \in W^{1,p}(u^*TM)\ |\ 
(\eta(z_1),\ldots,\eta(z_m)) \in TZ \right\},
$$
the tangent space $T_{(u,J)} \univ^*_Z(\jJ_\epsilon)$ is identified with
$$
K_Z := \ker \left( D\dbar(j,u,J)\big|_{T_j\tT \oplus W^{1,p}_Z(u^*TM) \oplus T_J\jJ_\epsilon} \right),
$$
which is a finite-codimensional subspace of 
$\ker D\dbar(j,u,J) = T_{(u,J)} \univ^*(\jJ_\epsilon)$.
The smooth projection
$$
\pi_Z : \univ^*_Z(\jJ_\epsilon) \to \jJ_\epsilon : (u,J) \mapsto J
$$
then has derivative at $(u,J)$ equivalent to the linear projection 
$$
K_Z \to T_J \jJ_\epsilon : (y,\eta,Y) \mapsto Y,
$$
and this gives a natural identification of $\ker d\pi_Z(u,J)$ with the
kernel of the operator
$$
\mathbf{L}_Z := D\dbar_J(j,u)|_{T_j\tT \oplus W^{1,p}_Z(u^*TM)},
$$
where $D\dbar_J(j,u) : T_j\tT \oplus W^{1,p}(u^*TM) \to
L^p(\overline{\Hom}_\CC(T\Sigma,u^*TM))$ is the same operator that
appears in the definition of Fredholm regularity 
(see Definition~\ref{defn:regular}).
We claim that the cokernels of $d\pi_Z(u,J)$ and $\mathbf{L}_Z$
are also isomorphic, so both are Fredholm and have the same index.  This
is a special case of the following general fact from linear functional
analysis.

\begin{lemma}
\label{lemma:parametrized}
Suppose $X$, $Y$ and $Z$ are Banach spaces,
$D : X \to Z$ is a Fredholm operator, $A : Y \to Z$ is another bounded
linear operator and $L : X \oplus Y \to Z : (x,y) \mapsto Dx + Ay$ 
is surjective.  Then the projection
$$
\Pi : \ker L \to Y : (x,y) \mapsto y
$$
is Fredholm and there are natural isomorphisms $\ker \Pi = \ker D$ and
$\coker \Pi = \coker D$.
\end{lemma}
\begin{proof}
The isomorphism of the kernels is clear: it is just the restriction of the
inclusion $X \hookrightarrow X \oplus Y : x \mapsto (x,0)$ to $\ker D$.
We construct an isomorphism $\coker \Pi \to \coker D$ as follows.
Observe that
$\im \Pi$ is simply the space of all $y \in Y$ such that $Ay = -Dx$ for
any $x \in X$, hence $\im \Pi = A^{-1}(\im D)$, and 
$$
\coker \Pi = Y \big/ \im \Pi = Y \big/ A^{-1}(\im D).
$$
Now it is easy to check
that the map $A : Y \to \im A$ descends to an isomorphism
$$
A : Y \big/ A^{-1}(\im D) \to \im A \big/ (\im D \cap \im A),
$$
and similarly, the inclusion $\im A \hookrightarrow Z$ descends to an injective
homomorphism
$$
\im A \big/ (\im D \cap \im A) \to Z \big/ \im D.
$$
Since every $z \in Z$ can be written as $z = Dx + Ay$ by assumption, this
map is also surjective.
\end{proof}

We can now apply the Sard-Smale theorem and conclude that the regular
values of $\pi_Z$ form a Baire subset of $\jJ_\epsilon$, and for each
$J$ in this subset, Lemma~\ref{lemma:parametrized} implies that 
$D\dbar_J(j,u)|_{T_j\tT \oplus W^{1,p}_Z(u^*TM)}$ is surjective 
onto $L^p(\overline{\Hom}_\CC(T\Sigma,u^*TM))$ for every
representative $(\Sigma,j,\Theta,u)$ of any curve $u$ with
$(u,J) \in \univ^*_Z(\jJ_\epsilon)$.  It follows that for such a curve,
$D\dbar_J(j,u)$ is also surjective, hence $u$ is Fredholm regular and
a neighborhood of $u$ in $\Mod_{g,m}^A(J)$ is identified with the
smooth neighborhood of $(j,u)$ in $\dbar_J^{-1}(0)$.  Under this local
identification, the evaluation map on $\Mod_{g,m}^A(J)$ takes the form
$$
\ev : \dbar_J^{-1}(0) \to M^m : (j,u) \mapsto (u(z_1),\ldots,u(z_m)),
$$
and we claim that $\im d(\ev)(j,u)$ is transverse to $T_{\ev(j,u)}Z$.
To see this, observe that given an arbitrary $m$-tuple
$$
(\xi_1,\ldots,\xi_m) \in T_{u(z_1)} M \oplus \ldots \oplus T_{u(z_m)} M =
T_{\ev(u)} M^m,
$$
we can choose a smooth section
$\xi \in \Gamma(u^*TM)$ that matches $\xi_i$ at $z_i$ for $i=1,\ldots,m$,
and then appeal to the surjectivity of $D\dbar_J(j,u)$
on the restricted domain to find
$y \in T_j\tT$ and $\eta \in W^{1,p}_Z(u^*TM)$ such that
$D\dbar_J(j,u)(y,\eta) = - \mathbf{D}_u \xi$.
Then $(y,\eta+\xi) \in \ker D\dbar_J(j,u)$ and
$$
(\xi_1,\ldots,\xi_m) = d\ev(j,u)(y,\eta+\xi) - (\eta(z_1),\ldots,\eta(z_m)) \in
\im d(\ev)(j,u) + T_{\ev(j,u)} Z,
$$
proving the claim.

Since Baire subsets are also dense, 
the set of regular values contains arbitrarily good approximations to $\Jref$ in 
the $C_\epsilon$-topology, and therefore also in the $C^\infty$-topology,
and since $\Jref \in \jJ(M,\omega\,;\,\uU,\Jfix)$ was chosen arbitrarily,
this implies that $\jJ_\reg$ is dense in $\jJ(M,\omega\,;\,\uU,\Jfix)$.  
The proof of Prop.~\ref{prop:dense} is thus complete.

\subsection{Dense implies generic}
\label{subsec:Taubes}

As promised, we shall now improve Prop.~\ref{prop:dense} to the statement
that $\jJ_\reg$ is not just dense but also is a  Baire subset,
i.e.~a countable intersection of open dense subsets in 
$\jJ(M,\omega\,;\,\uU,\Jfix)$, which implies Theorem~\ref{thm:generic}.  The 
idea of this step is originally due to Taubes, and it depends on the fact that
the moduli space of somewhere injective $J$-holomorphic curves can always be 
exhausted---in a way that depends continuously on~$J$---by a countable 
collection of compact subsets.  Observe that the definition of convergence
in $\Mod_{g,m}^A(J)$ does not depend in any essential way on~$J$: thus
one can sensibly speak of a convergent sequence of curves
$u_k \in \Mod_{g,m}^A(J_k)$ where $J_k \in \jJ(M)$ are
potentially different almost complex structures.

\begin{lemma}
\label{lemma:exhaustion}
For every $J \in \jJ(M)$ and $c > 0$, there exists a subset
$$
\Mod_{g,m}^A(J,c) \subset \Mod_{g,m}^A(J)
$$ 
such that the following conditions are satisfied:
\begin{itemize}
\item Every curve in $\Mod_{g,m}^A(J)$ with an injective point mapped into $\uU$
belongs to $\Mod_{g,m}^A(J,c)$ for some $c > 0$;
\item For each $c > 0$ and any sequence $J_k \to J$ in $\jJ(M)$,
every sequence $u_k \in \Mod_{g,m}^A(J_k,c)$ has a subsequence
coverging to an element of $\Mod_{g,m}^A(J,c)$.
\end{itemize}
\end{lemma}

Postponing the proof for a moment,
we proceed to show that $\jJ_\reg$ is a Baire subset,
because it is the intersection
of a countable collection of subsets
$$
\jJ_\reg = \bigcap_{c \in \NN} \jJ_\reg^c,
$$
which are each open and dense in $\jJ(M,\omega\,;\,\uU,\Jfix)$.  We define these by
the condition that $J \in \jJ(M,\omega\,;\,\uU,\Jfix)$ belongs to
$\jJ_\reg^c$ if and only if every curve 
$u \in \Mod_{g,m}^A(J,c)$ with $\ev(u) \in Z$ is Fredholm regular and the 
evaluation map $\ev : \Mod_{g,m}^A(J,c) \to M^m$ is transverse to $Z$ at~$u$.
This set obviously
contains $\jJ_\reg$, and is therefore dense due to
Prop.~\ref{prop:dense}.  To see that it is open,
we argue by contradiction: suppose $J \in \jJ_\reg^c$ and 
$J_k \in \jJ(M,\omega\,;\,\uU,\Jfix) \setminus \jJ_\reg^c$ is a
sequence converging to~$J$. Then there is also a sequence
$u_k \in \Mod_{g,m}^A(J_k,c)$ of curves that either are not Fredholm regular or
fail to satisfy the transversality condition with respect to~$Z$.
A subsequence of $u_k$ then converges by Lemma~\ref{lemma:exhaustion} to some
$u \in \Mod_{g,m}^A(J,c)$, which must be regular and satisfy the transversality
condition since 
$J \in \jJ_\reg^c$.  But both conditions are open,
so we have a contradiction.

Theorem~\ref{thm:generic} is now established, except for the proof of
Lemma~\ref{lemma:exhaustion}.  Let us first sketch the intuition behind
this lemma.  Morally, it follows from an important
fact that we haven't yet discussed but soon will: the moduli space
$\Mod_{g,m}^A(J)$ has a natural compactification $\overline{\Mod}^A_{g,m}(J)$,
the \emph{Gromov compactification}, which is a metrizable topological space.
In fact, one can define a metric on $\overline{\Mod}^A_{g,m}(J)$ which does
not depend on $J$; in a more general context, the details of this construction
are carried out in \cite{SFTcompactness}*{Appendix~B}.  Thus if we
denote by $\overline{\Mod}_{\text{bad}}(J)$ the closed subset that consists
of the union of $\overline{\Mod}_{g,m}^A(J) \setminus \Mod_{g,m}^A(J)$
with all the curves in $\Mod_{g,m}^A(J)$ that have no injective point in~$\uU$,
one way to define $\Mod_{g,m}^A(J,c)$ would be as
$$
\Mod_{g,m}^A(J,c) = \left\{ u \in \Mod_{g,m}^A(J)\ \Big|\ 
\dist\left(u, \overline{\Mod}_{\text{bad}}(J)\right) \ge \frac{1}{c} \right\}.
$$
By Gromov's compactness theorem, any sequence $u_k \in \overline{\Mod}_{g,m}^A(J_k)$
with $J_k \to J \in \jJ(M,\omega\,;\,\uU,\Jfix)$ has a subsequence converging to an
element of $\overline{\Mod}_{g,m}^A(J)$, and since
$\Mod_{g,m}^A(J,c) \subset \overline{\Mod}_{g,m}^A(J)$ is a closed subset, the same
holds for a sequence $u_k \in \Mod_{g,m}^A(J_k,c)$ for any fixed $c > 0$.

We will not attempt to make the above sketch precise, as we do not actually
need Gromov's compactness theorem to prove the lemma---in fact, the latter
is true only for almost complex structures that are tamed by a symplectic 
form, and we don't need the symplectic structure either.  The 
following proof does however contain most of the crucial analytical
ingredients in the compactness theory of holomorphic curves.

\begin{proof}[Proof of Lemma~\ref{lemma:exhaustion}]
We'll give a proof first for the case $g=0$ and then sketch the modifications
that are necessary for higher genus.

Assume $g=0$ and $m \ge 3$, so $\Sigma = S^2$.
Any pointed Riemann surface $(\Sigma,j,\Theta)$
is then equivalent to one of the form $(S^2,i,\Theta)$ with 
$\Theta = (0,1,\infty,z_1,\ldots,z_{m-3})$ for
$$
\mathbf{z} := (z_1,\ldots,z_{m-3}) \in (S^2)^{m-3} \setminus \Delta,
$$
where we define the open subset $\Delta \subset (S^2)^{m-3}$
to consist of all tuples $(z_1,\ldots,z_{m-3})$ such that either 
$z_i \in \{0,1,\infty\}$ for some~$i$ or $z_i = z_j$ for some
$i \ne j$.  Choose metrics on $S^2$, $(S^2)^{m-3}$ and $M$, with distance
functions denoted by $\dist(\ ,\ )$.  We define $\Mod_{0,m}^A(J,c)$ to be the
set of all equivalence classes in $\Mod_{0,m}^A(J)$ which have representatives
$(S^2,i,\Theta,u)$ with $\Theta = (0,1,\infty,\mathbf{z})$ and
the following properties:
\begin{enumerate}
\item $(S^2,i,\Theta)$ is ``not close to degenerating,'' in the sense that
$\displaystyle \dist(\mathbf{z},\Delta) \ge \frac{1}{c}$;
\item $u$ is ``not close to bubbling,'' in the sense that
$| du(z) | \le c$ for all $z \in \Sigma$;
\item $u$ is ``not close to losing its injective points,'' meaning there 
exists $z_0 \in \Sigma$ such that
$$
\dist(u(z_0), M\setminus \uU) \ge \frac{1}{c}, \qquad | du(z_0)| \ge \frac{1}{c},
$$
and
$$
\inf_{z \in \Sigma\setminus\{z_0\}} \frac{\dist(u(z_0),u(z))}{\dist(z_0,z)} \ge
\frac{1}{c}.
$$
\end{enumerate}
Note that the map $u$ automatically sends an injective point into $\uU$ by
the third condition, and clearly every curve
$(S^2,i,\Theta,u)$ with this property belongs to $\Mod_{0,m}^A(J,c)$ for sufficiently
large~$c$.  Now if $J_k \to J \in \jJ(M)$ and we have a sequence
$(S^2,i,\Theta_k,u_k) \in \Mod_{0,m}^A(J_k,c)$ with $\Theta_k = 
(0,1,\infty,\mathbf{z}_k)$, we can take a
subsequence so that $\mathbf{z}_k \to \mathbf{z} \in (S^2)^{m-3}\setminus \Delta$
with $\dist(\mathbf{z},\Delta) \ge 1/c$.
Likewise, the images of the injective points of $u_k$ in $\uU$ may be
assumed to converge to a point at least distance $1/c$ away from $M \setminus \uU$
since $\overline{\uU}$ is compact.  Together with the bound $|du_k| \le c$,
this gives a uniform $C^1$-bound and thus a uniform $W^{1,p}$-bound on~$u_k$.
The regularity estimates of Chapter~\ref{chapter:local} 
(specifically Corollary~\ref{cor:gradBounds}) now
give a $C^\infty$-convergent subsequence $u_k \to u$, and
we conclude $(S^2,i,\Theta_k,u_k) \to (S^2,i,\Theta,u) \in \Mod_{0,m}^A(J,c)$,
where $\Theta := (0,1,\infty,\mathbf{z})$.

If $g = 0$ and $m < 3$, one only need modify the above argument by fixing the
marked points to be a suitable subset of $\{0,1,\infty\}$.  The
first condition in the above definition of $\Mod_{0,m}^A(J,c)$ is then
vacuous.

For $g \ge 1$, one can no longer describe variations in~$j$ purely in terms
of the marked points~$\Theta$, so we need a different trick to obtain
compactness of a sequence~$j_k$.  This requires some knowledge of the
Deligne-Mumford compactification of $\Mod_{g,m}$, which we will discuss
in a later chapter;
for now we simply summarize the main ideas.
Choose a model pointed surface $(\Sigma,\Theta)$ with genus $g$ and
$m$ marked points; if it is not stable, add enough additional marked points
to create a stable pointed surface $(\Sigma,\Theta')$, and let
$m' = \#\Theta'$.  Since $\chi(\Sigma \setminus \Theta') < 0$, for every
$j \in \jJ(\Sigma)$ there is a unique
complete hyperbolic metric $g_j$ of constant curvature~$-1$ on
$\Sigma\setminus \Theta'$ that defines the same conformal structure as~$j$.
There is also a singular pair of pants decomposition, that is, we can
fix $3g - 3 + m'$ distinct classes in $\pi_1(\Sigma\setminus\Theta')$ and
choose the unique geodesic in each of these so that they separate 
$\Sigma \setminus \Theta'$ into $-\chi(\Sigma\setminus\Theta')$ surfaces 
with the homotopy type of a twice punctured disk.  This procedure associates
to each $j \in \jJ(\Sigma)$ a set of real numbers
$$
\ell_1(j),\ldots,\ell_{3g - 3 + m'}(j) > 0,
$$
the lengths of the geodesics, which depend continuously on~$j$.
Now we define $\Mod_{g,m}^A(J,c)$ by the same scheme as with
$\Mod_{0,m}^A(J,c)$ above, but replacing the first condition by
$$
\frac{1}{c} \le \ell_i(j) \le c
\qquad
\text{ for each $i=1,\ldots,3g-3+m'$.}
$$
Now any sequence $(\Sigma,j_k,\Theta,u_k) \in \Mod_{g,m}^A(J,c)$
has a subsequence for which the lengths $\ell_i(j_k)$ converge in
$[1/c,c]$, implying that $j_k$ converges in $C^\infty$ to a complex
structure~$j$.  The rest of the argument works as before.
\end{proof}

\section{Transversality of the evaluation map}
\label{sec:evaluation}

Most applications of pseudoholomorphic curves involve the
natural evaluation map $\ev = (\ev_1,\ldots,\ev_m) : 
\Mod_{g,m}^A(J) \to M \times \ldots \times M$,
which can be used for instance to count intersections of holomorphic
curves with fixed points or submanifolds in the target.  Applications
of this type are facilitated by the following extension of 
Theorem~\ref{thm:localStructure}.

\begin{thm}
\label{thm:localStructureZ}
Assume $(M,\omega)$ is a $2n$-dimensional
symplectic manifold without boundary,
$\uU \subset M$ is an open subset with compact closure,
$\Jfix \in \jJ(M,\omega)$ and $m \in \NN$ are fixed, 
and $Z \subset M^m$ is a
smooth submanifold without boundary.  Then there exists a Baire subset
$\jJ_\reg^Z(M,\omega\,;\,\uU,\Jfix) \subset \jJ(M,\omega\,;\,\uU,\Jfix)$ 
such that for every $J \in \jJ_\reg^Z(M,\omega\,;\,\uU,\Jfix)$, the space
$\Mod^*_\uU(J ; Z) \subset \Mod(J)$ of $J$-holomorphic curves with
injective points mapped into~$\uU$ and $m$ marked points satisfying the
constraint
$$
\ev(u) \in Z
$$
is a smooth finite-dimensional manifold.  The dimension of
$\Mod^*_\uU(J;Z) \cap \Mod_{g,m}^A(J)$ for any $g \ge 0$ and $A \in H_2(M)$
is $\virdim \Mod_{g,m}^A(J) - (2nm - \dim Z)$.
\end{thm}
The theorem follows immediately from 
Theorems~\ref{thm:IFT2} and \ref{thm:generic}, as we can define
$\jJ_\reg^Z(M,\omega\,;\,\uU,\Jfix)$ as a countable intersection of the
Baire subsets provided by Theorem~\ref{thm:generic}:
$$
\jJ_\reg^Z(M,\omega\,;\,\uU,\Jfix) = \bigcap_{g \ge 0, A \in H_2(M)} 
\jJ_\reg^Z(M,\omega\,;\,\uU,\Jfix\,;\,g,m,A).
$$
We are also free to shrink $\jJ_\reg^Z(M,\omega\,;\,\uU,\Jfix)$ further by
taking its intersection with $\jJ_\reg(M,\omega\,;\,\uU,\Jfix)$, thus
ensuring without loss of generality that \emph{all} curves with 
injective points in $\uU$ are
regular, including those with $\ev(u) \not\in Z$.

\begin{example}
\label{ex:constraintPoints}
Suppose $Z$ is a single point, i.e.~pick
points $p_1,\ldots,p_m \in M$ and denote the resulting $1$-point subset
by $\mathbf{p} \in \{(p_1,\ldots,p_m)\} \subset M^m$.  
Then Theorem~\ref{thm:localStructureZ} implies
that for generic $J$, the space of closed somewhere injective
$J$-holomorphic curves $u$ with genus~$g$, in homology class $A$ and with
$m$ marked points satisfying the constraints $u(z_i) = p_i$ for $i=1,\ldots,m$
is a smooth manifold of dimension
\begin{equation*}
\begin{split}
\dim \ev^{-1}(\mathbf{p}) &=
\virdim \Mod_{g,m}^A(J) - \dim M^m \\
&= (n-3)(2-2g) + 2c_1(A) + 2m - 2nm \\
&= (n-3)(2-2g) + 2 c_1(A) - 2m(n-1).
\end{split}
\end{equation*}
\end{example}

Another simple application is the following generalization of
Corollary~\ref{cor:nonnegativeIndex}.

\begin{cor}
\label{cor:intZ}
Suppose $(M,\omega)$ is a $2n$-dimensional
symplectic manifold without boundary, $J$ is an
$\omega$-compatible almost complex structure, $\uU \subset M$ 
is an open subset with compact closure, and $Z_1,\ldots,Z_m \subset M$ is a 
pairwise disjoint finite collection of connected submanifolds without boundary.  
Then after a generic perturbation of $J$ to a new
compatible almost complex structure $J'$ matching $J$ outside~$\uU$, 
every $J'$-holomorphic curve that maps an injective
point into~$\uU$ and intersects all of the submanifolds
$Z_1,\ldots,Z_m$ satisfies
$$
\ind(u) \ge 2m(n-1) - \sum_{i=1}^m \dim Z_i.
$$
\end{cor}
\begin{proof}
Assume $J'$ is generic such that for all $g \ge 0$ and $A \in H_2(M)$,
the set of curves in $\Mod_{g,m}^A(J')$ with
injective points in~$\uU$ is a smooth manifold of
the expected dimension and the evaluation map on this space is transverse
to~$Z_1 \times \ldots \times Z_m$.  If a curve $u \in \Mod_{g,0}^A(J')$ with 
the stated properties exists, 
then by adding a marked point $z_i$ at any point where it intersects~$Z_i$
for each $i=1,\ldots,m$, we can 
regard $u$ as an element of $\ev^{-1}(Z_1 \times \ldots \times Z_m) 
\subset \Mod_{g,m}^A(J')$, proving that
the latter is nonempty and has nonnegative dimension near~$u$.  
This dimension is
\begin{equation*}
\begin{split}
0 &\le \dim \ev^{-1}(Z_1 \times \ldots \times Z_m) = 
\virdim \Mod_{g,m}^A(J') - \sum_{i=1}^m \codim Z_i \\
&= \virdim \Mod_{g,0}^A(J') + 2m - \left(2mn - \sum_{i=1}^m \dim Z_i\right) \\
&= \ind(u) + 2m (1 - n) + \sum_{i=1}^m \dim Z_i.
\end{split}
\end{equation*}
\end{proof}

\begin{remark}
\label{remark:dependsOnJ}
This seems a good moment to emphasize that the definition of the word
``generic'' in Example~\ref{ex:constraintPoints} and Corollary~\ref{cor:intZ}
\emph{depends on $Z$}, i.e.~different choices of submanifolds
$Z_1, Z_2 \subset M^m$ generally yield different Baire subsets
$\jJ_\reg^{Z_1}(M,\omega\,;\,\uU,\Jfix)$ and 
$\jJ_\reg^{Z_2}(M,\omega\,;\,\uU,\Jfix)$.  For instance, one should not get the
impression from Example~\ref{ex:constraintPoints} that a generic choice of a
single $J \in \jJ(M,\omega)$ suffices to ensure that the spaces
$$
\left\{ u \in \Mod_{g,m}^A(J)\ |\ \text{$u$ is somewhere injective and
$\ev(u) = \mathbf{p}$} \right\}
$$
are smooth manifolds of dimension $(n-3)(2-2g) + 2 c_1(A) - 2m(n-1)$ for
all $\mathbf{p} \in M^m$.  One could arrange this simultaneously for any 
\emph{countable} set of points $\mathbf{p} \in M^m$, but it is easy to see
that this cannot hold for uncountable sets in general: indeed, 
Corollary~\ref{cor:intZ} implies that for each point $\mathbf{p} \in M^m$, 
taking $J$ generic ensures that every closed somewhere injective 
$J$-holomorphic curve $u$ with $m$ marked points satisfying $\ev(u) =
\mathbf{p}$ satisfies $\ind(u) \ge 2m(n-1)$.  If one could find a $J$
such that this holds for all $\mathbf{p} \in M^m$, it would imply that 
simple $J$-holomorphic curves $u$ with $\ind(u) < 2m(n-1)$ do not exist, 
and since the choice of $m \in \NN$ in this discussion was arbitrary, the
conclusion is clearly absurd.  This illustrates the fact that an 
\emph{uncountable} intersection of Baire subsets may in general be empty.
\end{remark}

For a slightly different type of application, one can prove various results
along the lines of the statement that
generic $J$-holomorphic curves in dimension greater than four are
injective.  For example:

\begin{cor}
\label{cor:injective}
Suppose $(M,\omega)$ is a closed symplectic manifold of dimension~$2n \ge 6$.
Then for generic $J \in \jJ(M,\omega)$, every somewhere injective
$J$-ho\-lo\-mor\-phic curve $u \in \Mod^*(J)$ with $\ind(u) < 2n-4$ is 
injective.\footnote{We will strengthen this result in 
Corollary~\ref{cor:embedded} below so that
the word ``injective'' can be replaced by ``embedded''.}
\end{cor}
\begin{proof}
Choose $J$ generic so that for every $g \ge 0$ and $A \in H_2(M)$, the
evaluation map on the space of somewhere injective curves in
$\Mod_{g,2}^A(J)$ is transverse to the diagonal
$$
\Delta := \{ (p,p) \in M \times M\ |\ p \in M \}.
$$
Then for any curve $u$ that is somewhere injective but has a self-intersection
$u(z_1) = u(z_2)$ for $z_1 \ne z_2$, we can add marked points at $z_1$ and
$z_2$ and thus view $u$ as an element of $\ev^{-1}(\Delta) \subset
\Mod_{g,2}^A(J)$, proving that $\ev^{-1}(\Delta)$ is nonempty and therefore
has nonnegative dimension.  This dimension is
$$
0 \le \virdim \Mod_{g,2}^A(J) - \codim \Delta = \virdim \Mod_{g,0}^A(J) + 4 - 2n
= \ind(u) + 4 - 2n.
$$
\end{proof}

One consequence of this result is that in higher dimensions (i.e.~$2n \ge 6$),
a simple and Fredholm regular curve of index~$0$ can always have its self-intersections
perturbed away by a small change in~$J$.  No such result holds in dimension
four, and there are good topological reasons for this, as
positivity of intersections (Theorem~\ref{thm:positivity}) implies that no 
self-intersection of a simple $J$-holomorphic curve can ever be eliminated
by small perturbations.  The following exercise shows however that 
\emph{triple} intersections can generically be avoided, 
even in dimension four.

\begin{exercise}
Prove that in any closed symplectic manifold $(M,\omega)$ of dimension $2n \ge 4$,
for generic $J \in \jJ(M,\omega)$, there is no somewhere injective
$J$-holomorphic curve $u \in \Mod^*(J)$ with $\ind(u) < 4n - 6$ having three
pairwise disjoint points $z_1,z_2,z_3$ in its domain such that
$u(z_1) = u(z_2) = u(z_3)$.
\end{exercise}

Finally, we state a generalization of Theorem~\ref{thm:localStructureZ}
that is useful in defining the rational Gromov-Witten invariants of
semipositive symplectic manifolds, see
\cite{McDuffSalamon:Jhol}*{Chapters~6 and~7}.  The proof is a
straightforward modification of the proof of
Theorem~\ref{thm:localStructureZ}.

\begin{thm}
\label{thm:localStructureZZ}
Assume $(M,\omega)$, $\uU \subset M$ and $\Jfix \in \jJ(M,\omega)$ are
given as in Theorem~\ref{thm:localStructureZ},
along with finite collections of integers $g_i,m_i \ge 0$ and
homology classes $A_i \in H_2(M)$ for $i=1,\ldots,N$, and a smooth
submanifold
$$
Z \subset M^{m_1} \times \ldots \times M^{m_N}
$$
without boundary.  For any $J \in \jJ(M)$, let
$$
\Mod^*_N(J) \subset \Mod_{g_1,m_1}^{A_1}(J) \times \ldots \times
\Mod_{g_N,m_N}^{A_N}(J)
$$
denote the open subset consisting of $N$-tuples $(u_1,\ldots,u_N)$ such
that each curve $u_i : \Sigma_i \to M$ for $i=1,\ldots,N$ has an
injective point $z_i \in \Sigma_i$ with
$$
u_i(z_i) \in \uU, \quad\text{ and }\quad
u_i(z_i) \not\in \bigcup_{j \ne i} u_j(\Sigma_j).
$$
Then there exists a Baire subset 
$\jJ_\reg^Z \subset \jJ(M,\omega\,;\,\uU,\Jref)$ such that for all
$J \in \jJ_\reg^Z$, $\Mod^*_N(J)$ is a smooth manifold and the 
composite evaluation map
$$
(\ev^1,\ldots,\ev^N) : \Mod^*_N(J) \to M^{m_1} \times \ldots \times M^{m_m}
$$
is transverse to~$Z$, where $\ev^i$ denotes the evaluation map on
$\Mod_{g_i,m_i}^{A_i}(J)$ for $i=1,\ldots,N$.
\end{thm}

\begin{exercise}
Convince yourself that Theorem~\ref{thm:localStructureZZ} is true.
What can go wrong if two of the curves $u_i$ and $u_j$ for $i \ne j$
have identical images?
\end{exercise}

\section{Generic $J$-holomorphic curves are immersed}
\label{sec:immersed}

The following result demonstrates a different kind of marked point
constraint than we've seen so far.  As usual, we assume
$\uU$ is a precompact open subset in a $2n$-dimensional symplectic
manifold $(M,\omega)$ without boundary, $\Jfix \in \jJ(M,\omega)$,
$g \ge 0$ and $A \in H_2(M)$ are fixed.  

\begin{thm}
\label{thm:immersed}
Given $J \in \jJ(M)$, let
$$
\Mod_{g,\crit}^A(J) \subset \Mod_{g,1}^A(J)
$$
denote the set of curves in $\Mod_{g,1}^A(J)$ that have vanishing first
derivatives at the marked point.  Then there exists a Baire subset 
$\jJ_\reg' \subset \jJ(M,\omega\,;\,\uU,\Jfix)$
such that for every $J \in \jJ_\reg'$, the subset of 
$\Mod_{g,\crit}^A(J)$ consisting of curves with an injective point 
mapped into~$\uU$ is a
smooth manifold with dimension equal to
$\virdim \Mod_{g,0}^A(J) - (2n-2)$.
\end{thm}

\begin{cor}
\label{cor:immersed}
Suppose $(M,\omega)$ is a closed symplectic manifold of dimension~$2n \ge 4$.
Then for generic $J \in \jJ(M,\omega)$, every somewhere injective 
$J$-holomorphic curve $u \in \Mod^*(J)$ with $\ind(u) < 2n-2$ is
immersed.
\end{cor}
The proof of the corollary is analogous to that of Corollary~\ref{cor:injective}
above: if a nonimmersed curve $u \in \Mod_{g,0}^A(J)$ exists, one can add 
a marked point where $du(z)=0$ and thus view $u$ as an element of
$\Mod_{g,\crit}^A(J)$, whose dimension is given by Theorem~\ref{thm:immersed}
and must be nonnegative.  Note that unlike Corollary~\ref{cor:injective},
this gives a nontrivial result in dimension four, showing that
index~$0$ curves are generically immersed, so one can always perturb
critical points away by a small change in~$J$; Theorem~\ref{thm:critical}
indicates that in dimension four, such a perturbation produces new
self-intersections.  In higher dimensions, the above result
combines with Corollary~\ref{cor:injective} to prove:
\begin{cor}
\label{cor:embedded}
For generic $J \in \jJ(M,\omega)$ in any closed symplectic manifold $(M,\omega)$
of dimension~$2n \ge 6$, every somewhere injective 
$J$-holomorphic curve $u \in \Mod^*(J)$ with $\ind(u) < 2n-4$ is embedded.
\end{cor}

\begin{remark}
Various generalizations of Theorem~\ref{thm:immersed} and the
above corollaries can easily be proved at the cost of more cumbersome 
notation.  The general rule is that in any moduli space of somewhere injective
pseudoholomorphic curves with marked points satisfying any constraints,
imposing an additional constraint to make the curves critical at a particular
marked point decreases the dimension of the moduli space by~$2n$.
(The additional $2$ in the dimension formula of Theorem~\ref{thm:immersed}
appears because of the two dimensions gained by switching from
$\Mod_{g,0}^A(J)$ to $\Mod_{g,1}^A(J)$ before imposing the constraint.)
\end{remark}

The proof of Theorem~\ref{thm:immersed} will require a slight modification of
our previous functional analytic setup: writing down the Cauchy-Riemann 
equation on $W^{1,p}(\Sigma,M)$ will not work if we also want to impose
a pointwise constraint on derivatives, as maps in $W^{1,p}(\Sigma,M)$ are
not generally of class~$C^1$.  This problem is easy to fix by working in
$W^{2,p}(\Sigma,M)$ for any $p > 2$, which admits a continuous inclusion
into $C^1(\Sigma,M)$ due to the Sobolev embedding theorem.  
The arguments of \S\ref{sec:IFT} and \S\ref{sec:transversality} then
require only minor modifications to fit into the new setup, so we will
sketch these modifications without repeating every detail.

Recall from \S\ref{sec:Banach} that since $\Sigma$ is compact and
$\dim_\RR \Sigma = 2$,
$$
\bB^{k,p} := W^{k,p}(\Sigma,M)
$$
is a smooth Banach manifold for any $k \in \NN$ and $p > 2$,
with $W^{k,p}$-neighborhoods of smooth maps $f \in C^\infty(\Sigma,M)$
identified with neighborhoods of~$0$ in $W^{k,p}(f^*TM)$ via the
correspondence $u = \exp_f \eta$ for $\eta \in W^{k,p}(f^*TM)$.
The tangent space at $u \in \bB^{k,p}$ is
$$
T_u \bB^{k,p} = W^{k,p}(u^*TM),
$$
and the Sobolev embedding theorem implies that
there is a continuous inclusion
$$
\bB^{k,p} \hookrightarrow C^{k-1}(\Sigma,M).
$$
Recall also that for any $j \in \jJ(\Sigma)$ and $J \in \jJ(M)$, there is 
a smooth Banach space bundle 
$\eE^{k-1,p} \to \bB^{k,p}$ with fibers
$$
\eE^{k-1,p}_u := W^{k-1,p}\left(\overline{\Hom}_\CC((T\Sigma,j),(u^*TM,J))\right)
$$
and a smooth section
$$
\dbar_J : \bB^{k,p} \to \eE^{k-1,p} : u \mapsto Tu + J \circ Tu \circ j,
$$
whose zero set is the space of pseudoholomorphic maps
$(\Sigma,j) \to (M,J)$ of class~$W^{k,p}$.  Elliptic regularity implies of
course that all such maps are smooth, regardless of the values of $k$ and~$p$.
Given a Teichm\"uller slice $\tT \subset \jJ(\Sigma)$ and a Banach manifold of
Floer perturbations $\jJ_\epsilon \subset \jJ(M)$ as in \S\ref{subsec:dense},
the bundle $\eE^{k-1,p}$ has an obvious extension over the base 
$\tT \times \bB^{k,p} \times \jJ_\epsilon$, with $\dbar_J$ extending to
a smooth section
$$
\dbar : \tT \times \bB^{k,p} \times \jJ_\epsilon \to \eE^{k-1,p} :
(j,u,J) \mapsto Tu + J \circ Tu \circ j.
$$
Its linearization at any zero has the usual form restricted to the
appropriate domain and target.  One can similarly define the smooth section
\begin{equation}
\label{eqn:dOperator}
\p_J u = Tu - J \circ Tu \circ j,
\end{equation}
which for $u \in \bB^{k,p}$ takes values in the Banach space bundle whose
fiber over~$u$ is $W^{k-1,p}(\Hom_\CC(T\Sigma,u^*TM))$.  Its 
linearization takes the form
\begin{equation}
\label{eqn:lindJ}
\begin{split}
D \p_J(u) : W^{k,p}(u^*TM) &\to W^{k-1,p}(\Hom_\CC(T\Sigma,u^*TM)) \\
\eta &\mapsto \nabla \eta - J(u) \circ \nabla \eta \circ j - 
(\nabla_\eta J) \circ Tu \circ j
\end{split}
\end{equation}
for any choice of symmetric connection $\nabla$ on~$M$, and it has a
similarly obvious extension to smooth sections of Banach space bundles
over $\tT \times \bB^{k,p}$ or
$\tT \times \bB^{k,p} \times \jJ_\epsilon$.

Suppose now that
$\left(\Sigma,j_0,(z_0),u_0\right)$ represents a curve in the moduli space
$\mM_{g,\crit}^A(J)$ defined in Theorem~\ref{thm:immersed}, so in particular
$du_0(z_0) = 0$.  We shall consider the nonlinear Cauchy-Riemann operator
on the domain
$$
\bB^{2,p}_\crit := \left\{ u \in \bB^{2,p}\ |\ 
\text{$du(z_0) : T_{z_0}\Sigma \to T_{u(z_0)}M$ is complex antilinear} \right\}.
$$
Notice that since any $u \in \dbar_J^{-1}(0)$ has complex-linear derivatives,
such a map belongs to $\bB^{2,p}_\crit$ if and only if $du(z_0)=0$.
We claim that $\bB^{2,p}_\crit$ is a smooth Banach submanifold of $\bB^{2,p}$.
Indeed, define a vector bundle $\vV \to \bB^{2,p}$ with fibers
$$
\vV_u := \Hom_\CC(T_{z_0} \Sigma , T_{u(z_0)} M).
$$
It is easy to see that $\vV$ is a smooth vector bundle, as it is the pullback
of the finite-dimensional smooth vector bundle
$$
\Hom_\CC(T_{z_0}\Sigma , TM) \to M
$$
via the smooth evaluation map $\ev : \bB^{2,p} \to M : u \mapsto u(z_0)$.
Moreover, the inclusion $\bB^{2,p} \subset C^1(\Sigma,M)$ permits us
to define a smooth section
$$
\bB^{2,p} \to \vV : u \mapsto \p_J u(z_0),
$$
where $\p_J$ is the operator defined in \eqref{eqn:dOperator}.
The zero set of this section is precisely $\bB^{2,p}_\crit$, and its
linearization at a
zero $u \in \bB^{2,p}_\crit$ is simply the restriction of \eqref{eqn:lindJ} 
to the point~$z_0$, which gives the continuous linear map
\begin{equation}
\label{eqn:linPsi}
\begin{split}
W^{2,p}(u^*TM) &\to \Hom_\CC(T_{z_0}\Sigma,T_{u(z_0)}M)\\
\eta &\mapsto \left.\Big( \nabla \eta - J(u) \circ \nabla \eta \circ j - 
(\nabla_\eta J) \circ Tu \circ j\Big)\right|_{T_{z_0}\Sigma}.
\end{split}
\end{equation}

\begin{exercise}
Convince yourself that \eqref{eqn:linPsi} is surjective for any
$u \in \bB^{2,p}_\crit$.
\end{exercise}

By the exercise and the implicit function theorem, $\bB^{2,p}_\crit$ is
a smooth Banach submanifold of $\bB^{2,p}$, with codimension~$2n$.
The zero set of the restriction
$$
\dbar_J|_{\bB^{2,p}_\crit} : \bB^{2,p}_\crit \to \eE^{1,p}
$$
then consists of $J$-holomorphic maps $u : \Sigma \to M$ 
with $du(z_0)=0$, and the
linearization of this restricted section at a map
$u \in \dbar_J^{-1}(0) \cap \bB^{2,p}_\crit$ is the usual
linear Cauchy-Riemann type operator $\mathbf{D}_u$ on a restricted domain
\begin{equation}
\label{eqn:restrictedD}
\mathbf{D}_u : W^{2,p}_\crit(u^*TM) \to W^{1,p}(\overline{\Hom}_\CC(T\Sigma,u^*TM)),
\end{equation}
where we plug in $du(z_0)=0$ to \eqref{eqn:linPsi}, obtaining the space
$$
W^{2,p}_\crit(u^*TM) := \left\{ \eta \in W^{2,p}(u^*TM)\ |\ 
\text{$\nabla \eta(z_0)$ is complex antilinear} \right\}.
$$
Note that since $du(z_0)=0$, the condition defining $W^{2,p}_\crit(u^*TM)$ does not
depend on the choice of symmetric connection.  Since $W^{2,p}_\crit(u^*TM)$
has codimension $2n$ in $W^{2,p}(u^*TM)$, plugging in the index formula
from Theorem~\ref{thm:RiemannRoch} for a Cauchy-Riemann type operator
$W^{2,p}(u^*TM) \to W^{1,p}(\overline{\Hom}_\CC(T\Sigma,u^*TM))$ gives
\begin{equation}
\label{eqn:indCrit}
\ind(\mathbf{D}_u) = n \chi(\Sigma) + 2 c_1(A) - 2n.
\end{equation}

Let us call a curve
$u \in \Mod_{g,\crit}^A(J)$ \defin{Fredholm regular for $\Mod_{g,\crit}^A(J)$}
whenever the operator \eqref{eqn:restrictedD} is surjective.  
Given a Teichm\"uller slice $\tT$ through $j_0$, we can now consider the
nonlinear operator $\dbar_J$ on the finite-codimensional submanifold
$$
\left\{ (j,u) \in \tT \times \bB^{2,p}\ |\ 
\text{$du(z_0) : (T_{z_0}\Sigma,j) \to (T_{u(z_0)}M,J)$ is complex
antilinear} \right\}.
$$
With the index formula \eqref{eqn:indCrit} in hand,
a repeat of the proof of Theorem~\ref{thm:IFT2} in this context shows:

\begin{prop}
\label{prop:IFT3}
The open subset of $\Mod_{g,\crit}^A(J)$ consisting of curves that are
Fredholm regular for $\Mod_{g,\crit}^A(J)$ and have trivial automorphism
group is a smooth manifold of dimension 
$$
\virdim \Mod_{g,\crit}^A(J) := \virdim \Mod_{g,1}^A(J) - 2n = 
(n-3)(2-2g) + 2 c_1(A) + 2 - 2n.
$$
\end{prop}

It remains to show that the regularity condition is achieved for
generic~$J$.  Following the prescription of \S\ref{sec:transversality},
choose a Banach manifold $\jJ_\epsilon$ of $C_\epsilon$-smooth perturbations
of an arbitrary reference structure $\Jref \in \jJ(M,\omega\,;\,\uU,\Jfix)$,
all matching $\Jfix$ outside of~$\uU$.  Define a universal moduli space
$\univ_\crit^*(\jJ_\epsilon)$ to consist of all pairs
$(u,J)$ such that $J \in \jJ_\epsilon$, $u \in \Mod_{g,\crit}^A(J)$, and
$u$ maps an injective point into~$\uU$.

\begin{prop}
\label{prop:univ3}
$\univ_\crit^*(\jJ_\epsilon)$ admits the structure of a smooth 
(separable and metrizable) Banach manifold such that the projection
$\pi : \univ_\crit^*(\jJ_\epsilon) \to \jJ_\epsilon$ is smooth, and for
every regular value $J$ of $\pi$, every curve $u \in \Mod_{g,\crit}^A(J)$
with an injective point mapped into~$\uU$ is Fredholm regular
for~$\Mod_{g,\crit}^A(J)$.
\end{prop}

The proof is essentially the same as that of Proposition~\ref{prop:universal},
the crucial step being to establish the following analogue of
Lemma~\ref{lemma:surjective}:

\begin{lemma}
\label{lemma:surjective3}
If $u_0 : (\Sigma,j_0) \to (M,J_0)$ is a pseudoholomorphic curve that maps 
an injective point into~$\uU$ and satisfies $du_0(z_0) = 0$, then
the operator
\begin{equation*}
\begin{split}
\mathbf{L} : W^{2,p}_\crit(u_0^*TM) \oplus 
C_\epsilon(\overline{\End}_\CC(TM,J_0,\omega)\,;\,\uU) &\to
W^{1,p}(\overline{\Hom}_\CC(T\Sigma,u_0^*TM)) \\
(\eta,Y) &\mapsto \mathbf{D}_{u_0} \eta +
Y \circ Tu_0 \circ j_0
\end{split}
\end{equation*}
is surjective and has a bounded right inverse.
\end{lemma}
\begin{proof}
As in the proof of Lemma~\ref{lemma:surjective}, the Fredholm property of 
$\mathbf{D}_{u_0}$ implies that $\mathbf{L}$ has a bounded right
inverse if and only if it is surjective.  To prove surjectivity,
we can appeal to the fact that the same
operator is (by Lemma~\ref{lemma:surjective}) already known to be
surjective as a map
$$
W^{1,p}(u_0^*TM) \oplus 
C_\epsilon(\overline{\End}_\CC(TM,J_0,\omega)\,;\,\uU) \to
L^p(\overline{\Hom}_\CC(T\Sigma,u_0^*TM)).
$$
Thus for any $f \in W^{1,p}(\overline{\Hom}_\CC(T\Sigma,u_0^*TM))$,
we have $f \in L^p$ and thus find $\eta \in W^{1,p}$ and
$Y \in C_\epsilon$ with $\mathbf{D}_{u_0} \eta + Y \circ Tu_0 \circ j_0 = f$.
Since $Y$ and $u_0$ are both smooth, this implies that
$\mathbf{D}_{u_0} \eta \in W^{1,p}$, so by linear elliptic
regularity (see e.g.~Corollary~\ref{cor:weakRegularity}), $\eta \in W^{2,p}$.
The first derivative of $\eta$ is therefore well defined pointwise,
and since $du_0(z_0)=0$, restricting the relation
$\mathbf{D}_{u_0}\eta = - Y \circ Tu_0 \circ j_0 + f$ to the point $z_0$ gives
$$
\left. \nabla \eta + J_0 \circ \nabla \eta \circ j_0 \right|_{T_{z_0}\Sigma} =
f(z_0) \in \overline{\Hom}_\CC(T_{z_0}\Sigma,T_{u_0(z_0)}M),
$$
which implies $\eta \in W^{2,p}_\crit(u_0^*TM)$.
\end{proof}

The Sard-Smale theorem now implies that the space of $J \in \jJ_\epsilon$
that are regular for $\Mod_{g,\crit}^A(J)$ is a Baire subset of
$\jJ_\epsilon$ and therefore dense in $\jJ(M,\omega\,;\,\uU,\Jfix)$.
Finally, one can adapt the argument of \S\ref{subsec:Taubes} and define
an exhaustion of $\Mod_{g,\crit}^A(J)$ by compact subsets
$$
\Mod_{g,\crit}^A(J,c) := \Mod_{g,\crit}^A(J) \cap \Mod_{g,1}^A(J,c)
$$
for $c > 0$, where $\Mod_{g,1}^A(J,c)$ are defined as
in \S\ref{subsec:Taubes}.  The sets
\begin{equation*}
\begin{split}
\jJ_{\reg,c}' := \Big\{ J \in \jJ(M,\omega\,;\,\uU,\Jfix)\ |\ 
&\text{all $u \in \Mod_{g,\crit}^A(J,c)$ are}\\
&\text{Fredholm regular for $\Mod_{g,\crit}^A(J)$} \Big\}
\end{split}
\end{equation*}
are then open and dense, and the countable intersection 
$\bigcap_{c \in \NN} \jJ_{\reg,c}'$ is the desired Baire subset,
completing the proof of Theorem~\ref{thm:immersed}.

The approach outlined in this section can be taken quite a bit further,
e.g.~by working in Banach manifolds $\bB^{k,p}$ for $k > 2$, one can
also impose constraints on higher-order derivatives.  
One case that is important in applications is
to consider spaces of holomorphic curves intersecting a fixed almost complex
submanifold with prescribed orders of tangency,
see e.g.~\cite{CieliebakMohnke:transversality}*{\S 6}.  For moduli spaces of
\emph{parametrized} $J$-holomorphic curves (i.e.~without dividing out by
reparametrizations), a somewhat different and very general approach to 
higher-order constraints has been introduced by Zehmisch \cite{Zehmisch:jets}, 
using the notion of \emph{holomorphic jets}.  

Here are a few exercises to illustrate what else can be done.  They are
not necessarily easy.

\begin{exercise}
Recall that $H_2(\CC P^2)$ is generated by $[\CC P^1] \in
H_2(\CC P^2)$, with $[\CC P^1] \cdot [\CC P^1] = 1$ and, for the standard
symplectic structure $\omega\std$ and complex structure~$i$,
$c_1([\CC P^1]) = 3$ and $\langle [\omega\std] , [\CC P^1] \rangle > 0$.
For $J \in \jJ(\CC P^2,\omega\std)$, a closed $J$-holomorphic curve 
$u : \Sigma \to \CC P^2$ is said to have \defin{degree} $d \in \NN$
if $[u] = d[\CC P^1]$.  Show that for any $d \in \NN$ and any set of 
pairwise distinct points $p_1,\ldots,p_{3d-1} \in \CC P^2$, there exists
a Baire subset $\jJ_\reg \subset \jJ(\CC P^2,\omega\std)$ such that for
all $J \in \jJ_\reg$, every somewhere injective $J$-holomorphic sphere
passing through all the points $p_1,\ldots,p_{3d-1}$ has degree at least~$d$,
and if its degree is exactly~$d$, then it is immersed.
\end{exercise}

In each of the following, assume $(M,\omega)$ is a closed $2n$-dimensional
symplectic manifold, all almost complex structures are $\omega$-compatible,
and all $J$-holomorphic curves are closed and connected.

\begin{exercise}
Prove that if $\dim_\RR M = 4$, then for generic $J$, every somewhere injective
$J$-holomorphic curve with sufficiently small index 
has only transverse self-intersections.  (How small must the index be?)
\end{exercise}

\begin{exercise}
Prove that if $\dim_\RR M = 4$, then for generic $J$, any pair of
inequivalent somewhere injective $J$-holomorphic curves $u$ and $v$ with
$\ind(u) = \ind(v) = 0$ satisfies $u \pitchfork v$.
\end{exercise}

\begin{exercise}[cf.~\cite{CieliebakMohnke:transversality}*{Prop.~6.9}]
Suppose $\Sigma \subset M$ is a symplectic hypersurface, i.e.~a symplectic
submanifold of dimension $2n-2$, and $A \in H_2(M)$ satisfies $c_1(A) = 3 - n$
and $A \cdot [\Sigma] = \ell > 0$, where $\ell$ is prime.  Show that the space
$$
\{ J \in \jJ(M,\omega)\ |\ J(T\Sigma) = T\Sigma \}
$$
contains a Baire subset $\jJ_\reg$ such that for all $J \in \jJ_\reg$,
every $J$-holomorphic sphere $u : S^2 \to M$ homologous to~$A$ either is
contained in $\Sigma$ or intersects it exactly $\ell$ times, always
transversely.
\end{exercise}


\chapter{Bubbling and Nonsqueezing}
\label{chapter:nonsqueezing}

\minitoc
\vspace{12pt}

\section{Gromov's nonsqueezing theorem}
\label{sec:nonsqueezing}

In the previous chapters we have developed a large part of the technical
apparatus needed to study $J$-holomorphic curves in symplectic manifolds
of arbitrary dimension.  The only major component still missing is the
compactness theory, which we will tackle in earnest in the next chapter.
In this chapter we shall provide some extra motivation by explaining
one of the first and most famous applications
of this technical apparatus: Gromov's nonsqueezing theorem.  The proof we shall
give is essentially Gromov's original proof (see \cite{Gromov}*{0.3.A}),
and it depends on a compactness result (Theorem~\ref{thm:NSCompactness})
that is one of the simplest applications of Gromov's compactness theorem, but can
also be proved without developing the compactness theory in its full
generality.  We will explain
in \S\ref{sec:bubbling} a proof of that result using the standard method
known as ``bubbling off'' analysis, which also plays an essential role
in the more general compactness theory.

Let us first recall the statement of the theorem.  Throughout the following
discussion, we shall use the symbol $\omega\std$ to denote the standard
symplectic form on Euclidean spaces of various dimensions, as well as on
tori defined as
$$
T^{2n} = \RR^{2n} / N \ZZ^{2n}
$$
for $N > 0$.  Note that~$\omega\std$ descends to a symplectic form
on $T^{2n}$ since it is invariant under the action of $\ZZ^{2n}$ on
$\RR^{2n}$ by translations.

\begin{thm}[Gromov's ``nonsqueezing'' theorem \cite{Gromov}]
\label{thm:nonsqueezing}
For any $n \ge 2$,
there exists a symplectic embedding of $(B_r^{2n},\omega\std)$ into
$(B_R^2 \times \RR^{2n - 2},\omega\std)$ if and only if $r \le R$.
\end{thm}

The existence of the embedding when $r \le R$ is clear, so the hard part
is to show that if an embedding
$$
\iota : (B_r^{2n},\omega\std) \hookrightarrow (B_R^2 \times \RR^{2n-2},
\omega\std)
$$
exists, then we must have $r \le R$.  We shall assume $r > R$ and argue
by contradiction.  Since the theory of
$J$-holomorphic curves is generally easier to work with in \emph{closed} 
manifolds, the first step is to transform this into a problem involving
embeddings into closed symplectic manifolds.  
To that end, choose a small number
$\epsilon > 0$ and an area form~$\sigma$ on the sphere~$S^2$ such that
$$
\int_{S^2} \sigma = \pi (R + \epsilon)^2.
$$
Then there exists a symplectic embedding $(B^2_R,\omega\std) \hookrightarrow
(S^2,\sigma)$, and hence also
$$
(B^2_R \times \RR^{2n-2},\omega\std) \hookrightarrow (S^2 \times \RR^{2n-2},
\sigma \oplus \omega\std).
$$
Composing this with $\iota$ above, we may regard~$\iota$ as a symplectic
embedding
$$
\iota : (B_r^{2n},\omega\std) \hookrightarrow (S^2 \times \RR^{2n-2},
\sigma \oplus \omega\std).
$$
We can assume without loss of generality that the image
$\iota(B_r^{2n}) \subset S^2 \times \RR^{2n-2}$ is bounded: indeed, this
is obviously true for the image of a \emph{closed} ball
$\overline{B}_{r'}$ if $r' < r$, thus it can be made true for~$r$ by
shrinking~$r$ slightly but keeping the condition $r > R$.
We can then choose a number
$N > 0$ sufficiently large so that 
$\iota(B_r^{2n}) \subset S^2 \times [-N,N]^{2n-2}$.  Composing with the
natural quotient projection on the second factor,
$$
\RR^{2n-2} \to T^{2n-2} := \RR^{2n-2} / N \ZZ^{2n-2}
$$
and letting~$\omega\std$ descend to a symplectic form on $T^{2n-2}$, this
gives rise to a symplectic embedding
\begin{equation}
\iota : (B_r^{2n},\omega\std) \to (S^2 \times T^{2n-2}, \sigma \oplus \omega\std).
\end{equation}
Since $\pi_2(T^{2n-2}) = 0$, we now
obtain a contradiction if we can prove the following.

\begin{thm}
\label{thm:nonsqueezingClosed}
Suppose $(M,\omega)$ is a closed symplectic manifold of dimension $2n-2 \ge 2$
which is \emph{aspherical}, i.e.~$\pi_2(M) = 0$, $\sigma$
is an area form on~$S^2$, and there exists a symplectic embedding
$$
\iota : (B^{2n}_r,\omega\std) \hookrightarrow (S^2 \times M, \sigma
\oplus \omega).
$$
Then $\pi r^2 \le \int_{S^2} \sigma$.
\end{thm}

We will prove this as a corollary of the following two results.  The first
has its origins in the theory of minimal surfaces and is a special case
of much more general results, though it admits an easy direct proof that
we will explain in \S\ref{sec:monotonicity}.  The second will require us
to apply the technical machinery developed in the previous chapters, together
with the compactness arguments explained in \S\ref{sec:bubbling}.

\begin{thm}[monotonicity]
\label{thm:monotonicity}
Suppose $r_0 > 0$, $(\Sigma,j)$ is a Riemann surface and
$$
u : (\Sigma,j) \to (B^{2n}_{r_0},i)
$$
is a proper holomorphic map whose image contains~$0$.  Then for every
$r \in (0,r_0)$,
$$
\int_{u^{-1}(\overline{B}^{2n}_r)} u^*\omega\std \ge \pi r^2.
$$
\end{thm}

\begin{prop}
\label{prop:GWinvariant}
Given the setup of Theorem~\ref{thm:nonsqueezingClosed}, there exists
a compatible almost complex structure 
$J \in \jJ(S^2 \times M,\sigma \oplus \omega)$ with $\iota^*J = i$ on
$B_r^{2n}$ and a $J$-holomorphic sphere
$$
u : S^2 \to S^2 \times M
$$
with $[u] = [S^2 \times \{*\}] \in H_2(S^2 \times M)$ whose image 
contains~$\iota(0)$.
\end{prop}

Before discussing the proof of Proposition~\ref{prop:GWinvariant}, let us
prove the main result.  To simplify notation, denote
$$
(W,\Omega) := (S^2 \times M,\sigma \oplus \omega),
\quad\text{ and }\quad
A_0 := [S^2 \times \{*\}] \in H_2(W).
$$
Recall that in Chapter~\ref{chapter:local}, we defined the \emph{energy}
$E(u)$ of a $J$-holomorphic curve $u : \Sigma \to W$ as $\int_\Sigma u^*\Omega$,
and observed that whenever~$J$ is tamed by~$\Omega$, this is also the
(nonnegative!) area traced out by~$u$ for a natural choice of Riemannian
metric on~$W$.  For the curve $u : S^2 \to S^2 \times W$ provided by
Proposition~\ref{prop:GWinvariant}, we can find the energy by a purely
homological computation:
$$
E(u) = \int_{S^2} u^*\Omega = \langle [\Omega],[u] \rangle =
\langle [\sigma \oplus \omega],A_0 \rangle =
\langle [\sigma] , [S^2] \rangle = \int_{S^2} \sigma.
$$
Since the integrand $u^*\Omega$ is always nonnegative, this gives an
upper bound for the amount of energy~$u$ has in the image of the 
ball $B^{2n}_r$, and in this ball, we can use $\iota^{-1}$ to pull back~$u$
to a map $\iota^{-1} \circ u : u^{-1}(\iota(B_r^{2n})) \to
B_r^{2n}$ which contains~$0$ in its image and
is $i$-holomorphic since $\iota^*J = i$.  Thus combining
the above upper bound with the lower bound from
Theorem~\ref{thm:monotonicity},
we find that for any $r' \in (0,r)$,
$$
\pi (r')^2 \le \int_{u^{-1}\left(\iota(B_{r'}^{2n})\right)} 
(\iota^{-1} \circ u)^*\omega\std
= \int_{u^{-1}(\iota(B^{2n}_{r'})} u^*\Omega \le \int_{S^2} u^*\Omega =
\int_{S^2} \sigma.
$$
This proves Theorem~\ref{thm:nonsqueezingClosed}.

For the rest of this section, we discuss the truly nontrivial part of the
proof above: why does the $J$-holomorphic sphere in 
Proposition~\ref{prop:GWinvariant} exist?  This turns out to be true
not just for a specific~$J$ but also for \emph{generic}
$\Omega$-compatible almost complex structures on~$W$, and there is nothing
special about the point $\iota(0)$, as \emph{every} point in~$W$ is in the
image of some $J$-holomorphic sphere homologous to~$A_0$.
Moreover, this is also true
for a generic subset of the special class of almost complex structures 
that match the
integrable complex structure $\iota_*i$ on $\iota(B_r^{2n})$.  We will not
be able to find these $J$-holomorphic curves explicitly, as we have no
concrete knowledge about the symplectic embedding $\iota : B_r^{2n} \to W$
and thus cannot even write down an explicit expression for~$J$ having
the desired property in $\iota(B_r^{2n})$.  Instead, we argue from more
abstract principles by starting from a simpler almost complex structure,
for which the holomorphic curves are easy to classify, and then using a
deformation argument to show that the desired curves
for our more general data must also exist.  This argument can be
outlined as follows:
\begin{enumerate}
\item Find a special $J_0 \in \jJ(W,\Omega)$ for which the moduli space
$\Mod_{0,1}^{A_0}(J_0)$ of $J_0$-holomorphic
spheres homologous to $[S^2 \times \{*\}]$ and with one marked point is
easy to describe precisely: in particular, the curves in 
$\Mod_{0,1}^{A_0}(J_0)$ are all Fredholm regular, and the moduli space is a
closed $2n$-dimensional manifold diffeomorphic to~$W$, with a diffeomorphism
provided by the natural evaluation map
$$
\ev : \Mod_{0,1}^{A_0}(J_0) \to W : [(S^2,j,z,u)] \mapsto u(z).
$$
\item Choose $J_1 \in \jJ(W,\Omega)$ with the desired property 
$\iota^*J_1 = i$ and show that for a generic such choice, the moduli
space $\Mod_{0,1}^{A_0}(J_1)$ is also a smooth $2n$-dimensional manifold.
\item Choose a homotopy $\{J_t\}$ from $J_0$ to $J_1$ and show that for
a generic such choice, the resulting parametrized moduli space
$\Mod_{0,1}^{A_0}(\{J_t\})$ is a smooth $(2n+1)$-dimensional
manifold with boundary
$$
\p \Mod_{0,1}^{A_0}(\{J_t\}) = \Mod_{0,1}^{A_0}(J_0) \sqcup
\Mod_{0,1}^{A_0}(J_1).
$$
Moreover, $\Mod_{0,1}^{A_0}(\{J_t\})$ is \emph{compact}.
\item Since $\ev : \Mod_{0,1}^{A_0}(J_0) \to W$ is a diffeomorphism, its
$\ZZ_2$-mapping degree is~$1$, and the fact that $\ev$ extends naturally
over the cobordism $\Mod_{0,1}^{A_0}(\{J_t\})$ implies that its restriction
to the other boundary component $\Mod_{0,1}^{A_0}(J_1)$ also has
$\ZZ_2$-degree~$1$.  It follows that
$\ev : \Mod_{0,1}^{A_0}(J_1) \to W$ is surjective, so for every $p \in W$,
there is a $J_1$-holomorphic sphere $u : S^2 \to W$ with
$[u] = A_0$ and a point $z \in S^2$ such that $u(z) = p$.
\end{enumerate}

We carry out the details in the next several subsections.  The only
part that cannot be proved using the tools we've already developed
is the compactness of $\Mod_{0,1}^{A_0}(\{J_t\})$, which is incidentally
the only place where the assumption $\pi_2(M) = 0$ is used.  This compactness
is a deep result which we shall prove in \S\ref{sec:bubbling}.

\subsection{The moduli space for~$J_0$}
\label{subsec:J0}

Identify $S^2$ with the Riemann sphere $\CC \cup \{\infty\}$ with its
standard complex structure~$i$, choose any $J_M \in \jJ(M,\omega)$, and
define $J_0 \in \jJ(W,\Omega)$ via the natural direct sum decomposition
$T_{(z,p)}W = T_z S^2 \oplus T_p M$, that is
$$
J_0 := i \oplus J_M.
$$
Then a map $u = (u_S,u_M) : S^2 \to S^2 \times M$ is $J_0$-holomorphic
if and only if $u_S : S^2 \to S^2$ is holomorphic and
$u_M : S^2 \to M$ is $J_M$-holomorphic.  If $[u] = A_0 = [S^2 \times \{*\}]$,
then we also have
$$
[u_S] = [S^2],
\quad\text{ and }\quad
[u_M] = 0.
$$
The latter implies that $u_M$ has zero energy as a $J_M$-holomorphic curve
in~$M$, i.e.~$\int_{S^2} u_M^*\omega = \langle [\omega],[u_M] \rangle = 0$,
hence~$u_M$ is constant.  Moreover, $u_S : S^2 \to S^2$ is a holomorphic
map of degree~$1$, and thus is biholomorphic (cf.~Exercise~\ref{EX:posDegree}),
so after a reparametrization of the domain we can assume $u_S = \Id$.
It follows that the moduli space $\Mod_{0,1}^{A_0}(J_0)$ can be identified
with the following set:
$$
\Mod_{0,1}^{A_0}(J_0) = \left\{ (u_m,\zeta) \ |\ 
\text{$m \in M$ and $\zeta \in S^2$} \right\},
$$
where we define the $J_0$-holomorphic maps 
$$
u_m : S^2 \to S^2 \times M : z \mapsto (z,m).
$$
The evaluation map
$\ev : \Mod_{0,1}^{A_0}(J_0) \to S^2 \times M$ then takes the form
$$
\ev(u_m,\zeta) = (\zeta,m),
$$
and is thus clearly a diffeomorphism.  Observe that there is a natural
splitting of complex vector bundles 
\begin{equation}
\label{eqn:umSplitting}
u_m^*TW = T S^2 \oplus E_0^{(n-1)},
\end{equation}
where $E_0^{(n-1)} \to S^2$ denotes the trivial complex 
bundle of rank~$n-1$ whose fiber at every point $z \in S^2$ is $(T_m M,J_M)$.

The observations above imply that $\Mod_{0,1}^{A_0}(J_0)$ is a smooth manifold
of dimension~$2n$, and indeed, this is precisely the prediction made by
the index formula \eqref{eqn:virdim}, which gives
$$
\virdim \Mod_{0,1}^{A_0}(J_0) = 2(n - 3) + 2 c_1(A_0) + 2 = 2n
$$
after plugging in the computation
$$
c_1(A_0) = c_1(u_m^*T(S^2 \times M)) = c_1(TS^2) + c_1(E_0^{n-1}) = 2.
$$
The above does \emph{not} immediately imply that every curve in
$\Mod_{0,1}^{A_0}(J_0)$ is Fredholm regular; in general only the converse of
this statement is true.  This is something we will need to know in order
to understand the local structure of the parametrized moduli space
$\Mod_{0,1}^{A_0}(\{J_t\})$, and we must proceed with caution since
our choice of $J_0$ is definitively
\emph{non-generic}.\footnote{Even if $j \in \jJ(S^2)$ and 
$J_M \in \jJ(M,\omega)$ are chosen
generically, product structures of the form $j \oplus J_M$ on 
$S^2 \times M$ are still of a rather
special type that can never be regarded as generic.  
See Remark~\ref{remark:nongeneric} for an example of just how badly things
can potentially go wrong.}
This means that 
we cannot expect transversality to be achieved for general reasons, but
must instead check it explicitly.  This turns out to be not so hard,
simply because the curves $u_m(z) = (z,m)$ are so explicit.

\begin{lemma}
\label{lemma:productTransversality}
Every $J_0$-holomorphic sphere of the form $u_m : S^2 \to S^2 \times M :
z \mapsto (z,m)$ for $m \in M$ is Fredholm regular.
\end{lemma}
\begin{proof}
We recall from Definition~\ref{defn:regular} that $u_m$ is Fredholm
regular if and only if a certain bounded linear operator of the form
$$
D\dbar_{J_0}(i,u_m) : T_{i}\tT \oplus W^{1,p}(u_m^*TW) \to
L^p(\overline{\Hom}_\CC(T S^2,u_m^*TW))
$$
is surjective.  Here $\tT$ is a Teichm\"uller slice, which in the present
case is trivial since the Teichm\"uller space of $S^2$ with one marked
point is trivial, so we can drop this factor and simply consider the
linearized Cauchy-Riemann operator
$$
\mathbf{D}_{u_m} : W^{1,p}(u_m^*TW) \to L^p(\overline{\Hom}_\CC(T S^2,u_m^*TW)).
$$
We can make use of the natural splitting \eqref{eqn:umSplitting} to
split the domain and target of $\mathbf{D}_{u_m}$ as
$$
W^{1,p}(u_m^*TW) = W^{1,p}(TS^2) \oplus W^{1,p}(E_0^{n-1})
$$
and
$$
L^p(\overline{\Hom}_\CC(T S^2,u_m^*TW)) =
L^p(\overline{\End}_\CC(T S^2)) \oplus
L^p(\overline{\Hom}_\CC(T S^2,E_0^{n-1})).
$$
In light of the split nature of the nonlinear Cauchy-Riemann equation
for $J_0$-ho\-lo\-mor\-phic maps $u : S^2 \to S^2 \times M$, 
it then turns out that the matrix form of $\mathbf{D}_{u_m}$
with respect to these splittings is
$$
\mathbf{D}_{u_m} = \begin{pmatrix}
\mathbf{D}_i^{S^2} & 0 \\
0 & \mathbf{D}_m
\end{pmatrix},
$$
where $\mathbf{D}_i^{S^2} : W^{1,p}(TS^2) \to L^p(\overline{\End}_\CC(T S^2))$
is the natural Cauchy-Riemann operator defined by the holomorphic vector
bundle structure of $(T S^2,i)$, and 
$$
\mathbf{D}_m : W^{1,p}(T_m M) \to L^p(\overline{\Hom}_\CC(T S^2,T_m M))
$$
is the linearization of $\dbar_{J_M}$ at the constant $J_M$-holomorphic
sphere $S^2 \to M : z \mapsto m$.  Specializing \eqref{eqn:linearization}
for the case of a constant map, we see that the latter is simply the
standard Cauchy-Riemann operator on the trivial bundle $E_0^{n-1}$,
i.e.~it is the operator determined by the unique holomorphic structure on
$E_0^{n-1}$ for which the constant sections are holomorphic.
As such, this operator splits further with respect to the splitting of
$E_0^{n-1}$ into holomorphic line bundles determined by any complex basis
of $T_m M$.  This yields a presentation of $\mathbf{D}_{u_m}$ in the form
$$
\mathbf{D}_{u_m} = \begin{pmatrix}
\mathbf{D}_i^{S^2} & 0 & \cdots & 0 \\
0 & \dbar & \cdots & 0 \\
\vdots & \vdots & \ddots & \vdots \\
0 & 0 & \cdots & \dbar
\end{pmatrix},
$$
where each of the diagonal terms are complex-linear Cauchy-Riemann
type operators on line bundles, with the $\dbar$ entries in particular
denoting operators that are equivalent to the standard operator
$$
\dbar : W^{1,p}(S^2,\CC) \to L^p(\overline{\Hom}_\CC(T S^2,\CC)) :
f \mapsto df + i\, df \circ i.
$$
These operators are surjective by Theorem~\ref{thm:linearAutomatic}
since $c_1(E_0^1) = 0 > -\chi(S^2)$.
Similarly, $\mathbf{D}_i^{S^2}$ is also surjective since
$c_1(T S^2) = 2 > -\chi(S^2)$.
\end{proof}

\begin{remark}
The above is an example of a general phenomenon often called
``automatic transversality'': it refers to various situations in which
despite (or in this case even \emph{because of}) a non-generic choice of~$J$,
transversality can be achieved by reducing it to a problem involving
Cauchy-Riemann operators on line bundles and applying
Theorem~\ref{thm:linearAutomatic}.  The case above is unusually fortunate,
as it is not often possible to split a given Cauchy-Riemann operator over a
sum of line bundles in just the right way.  In dimension four, however,
arguments like this do often work out in greater generality, and we'll make 
considerable use of them in later applications to symplectic
$4$-manifolds.
\end{remark}

\begin{remark}
\label{remark:nongeneric}
The following example is meant to persuade you that no almost complex
structure of the product form $j \oplus J_M$ can be regarded as 
``generic'' by any reasonable definition.
Suppose $(\Sigma,j)$ is a closed connected Riemann surface of genus~$g$,
$\sigma$ is a compatible area form on~$\Sigma$,
$J_M \in \jJ(M,\omega)$ is as above and $J_0 = j \oplus J_M \in
\jJ(\Sigma \times M, \sigma \oplus \omega)$.  Then any
$J_0$-holomorphic curve of the form $u_m : \Sigma \to \Sigma \times M :
z \mapsto (z,m)$ for $m \in M$ has 
$$
c_1(u_m^*T(\Sigma \times M))  = 
c_1(T\Sigma) = \chi(\Sigma),
$$
so
$$
\ind(u_m) = (n-3) \chi(\Sigma) + 2 c_1([\Sigma \times \{*\}]) = 
(n-1) \chi(\Sigma),
$$
which for $n \ge 2$ is negative whenever $g \ge 2$.  Thus in this case,
a generic perturbation of~$J_0$ should eliminate such curves altogether,
but it is clear that a perturbation of the form $i \oplus J_M'$ for
$J_M' \in \jJ(M,\omega)$ will never accomplish this.  In 
Lemma~\ref{lemma:productTransversality}, we were simply lucky to be working
with genus zero.
\end{remark}

\subsection{Transversality for~$J_1$}

From now on, assume the symplectic embedding $\iota : (B_r^{2n},\omega\std) \to
(W,\Omega)$ can be extended symplectically to a neighborhood of the closure
$\overline{B}_r^{2n}$; this can always be achieved by shrinking~$r$ slightly
without violating the assumption $r > R$.  Now
consider the closed subspace of $\jJ(W,\Omega)$ defined by
$$
\jJ(W,\Omega ; \iota) := \{ J \in \jJ(W,\Omega)\ |\ 
\text{$\iota^*J = i$ on $\overline{B}_r^{2n}$} \},
$$
in other words this is the space of all $\Omega$-compatible almost complex
structures on~$W$ which match the particular integrable complex
structure $\iota_*i$ on the closed set
$\iota(\overline{B}_r^{2n})$.  

\begin{exercise}
Convince yourself that $\jJ(W,\Omega ; \iota)$ is not empty.
\textsl{Hint: It may help to recall that the usual space of 
compatible almost complex structures is always not only nonempty but also
connected, see \S\ref{sec:compatible}.}
\end{exercise}

As with~$J_0$ in the previous
subsection, the condition $\iota^*J = i$ is nongeneric in some sense, but
it turns out not to matter for our purposes:

\begin{prop}
There exists a Baire subset $\jJ_\reg(W,\Omega ; \iota) \subset
\jJ(W,\Omega ; \iota)$ such that for any $J \in \jJ_\reg(W,\Omega ; \iota)$,
all $J$-holomorphic spheres homologous to~$A_0$ are Fredholm regular, hence
$\Mod_{0,1}^{A_0}(J)$ is a smooth manifold of dimension~$2n$.
\end{prop}
\begin{proof}
We begin with the following observations:
\begin{enumerate}
\item The virtual dimension of $\Mod_{g,m}^{A}(J)$ depends in general
on~$g$, $m$ and~$A$, but not on~$J$, thus our earlier computation
$\virdim \Mod_{0,1}^{A_0}(J_0) = 2n$ also applies to $\Mod_{0,1}^{A_0}(J)$ for
any~$J$.
\item Every pseudoholomorphic curve $u : S^2 \to W$ homologous to~$A_0$
is \emph{simple}, as $A_0 = [S^2 \times \{*\}]$ is not a positive multiple
of any other homology class in $H_2(S^2 \times M)$.
\item For any $J \in \jJ(W,\Omega ; \iota)$, there is no closed nonconstant
$J$-holomorphic curve $u : \Sigma \to W$ whose image lies entirely in
$\iota(\overline{B}_r^{2n})$.  If such a curve did exist, then
$\iota^{-1} \circ u$ would be a nonconstant closed $i$-holomorphic curve
in~$\RR^{2n}$ and would thus have positive energy
$$
\int_\Sigma (\iota^{-1} \circ u)^* \omega\std > 0,
$$
but this is impossible
since $\omega\std$ vanishes on every cycle in~$\RR^{2n}$.
\end{enumerate}
The result now follows by a minor modification of the proof of
Theorem~\ref{thm:localStructure}, see also
Remark~\ref{remark:compactSubset}.  The crucial point is that the set of
perturbations allowed by $\jJ(W,\Omega ; \iota)$ is still large enough to
prove that the universal moduli space for somewhere injective curves
is smooth, because every such curve necessarily has an injective point
outside of $\iota(\overline{B}_r^{2n})$.
\end{proof}

In light of this result, we can choose
$$
J_1 \in \jJ_\reg(W,\Omega ; \iota)
$$
so that $\Mod_{0,1}^{A_0}(J_1)$ is a smooth manifold of dimension~$2n$.

\subsection{The homotopy of almost complex structures}

Denote by
$$
\jJ(W,\Omega \,;\, J_0,J_1),
$$
the space of smooth $\Omega$-compatible homotopies between
$J_0$ and $J_1$,
i.e.~this consists of all smooth $1$-parameter families
$\{ J_t \}_{t \in [0,1]}$ such that $J_t \in \jJ(W,\Omega)$ for all 
$t \in [0,1]$ and $J_t$ matches the structures chosen above for $t =0,1$.
This gives rise to the parametrized moduli space
$$
\Mod_{0,1}^{A_0}(\{ J_t \}) = \{ (u,t) \ |\ 
t \in [0,1], \ u \in \Mod(J_t) \}.
$$
The following is the fundamental input we need from the compactness
theory of holomorphic curves.  It depends on certain topological details in
the setup we've chosen, and in particular on the fact that $A_0 =
[S^2 \times \{*\}]$ is a primitive homology class and $\pi_2(M) = 0$.

\begin{prop}
\label{prop:NSCompactness}
For any $\{J_t\} \in \jJ(W,\Omega \,;\, J_0,J_1)$,
$\Mod_{0,1}^{A_0}(\{ J_t \})$ is compact.
\end{prop}

We'll come back to the proof of this in \S\ref{sec:bubbling}.  
Notice that since $\Mod_{0,1}^{A_0}(J_1)$ is naturally a closed subset of
$\Mod_{0,1}^{A_0}(\{ J_t \})$ and is already known to be a smooth manifold,
this implies that $\Mod_{0,1}^{A_0}(J_1)$ is a closed manifold.  Since
Fredholm regularity is an open condition, the same is then true for all
$\Mod_{0,1}^{A_0}(J_t)$ with~$t$ in some neighborhood of either~$0$ or~$1$,
and for $t$ in this range the natural projection
$$
\Mod_{0,1}^{A_0}(\{J_t\}) \to \RR : (u,t) \mapsto t
$$
is a submersion.  We cannot expect this to be true for all
$t \in [0,1]$, not even for a generic choice of the homotopy, but
by applying Theorem~\ref{thm:genericHomotopy} we can at least arrange for
$\Mod_{0,1}^{A_0}(\{ J_t \})$ to carry a smooth structure:

\begin{prop}
\label{prop:NSHomotopy}
There exists a Baire subset 
$$
\jJ_\reg(W,\Omega \,;\, J_0,J_1) \subset \jJ(W,\Omega \,;\, J_0,J_1)
$$ 
such that for any $\{J_t\} \in \jJ_\reg(W,\Omega \,;\, J_0,J_1)$,
$\Mod_{0,1}^{A_0}(\{ J_t \})$ is a compact smooth manifold, with boundary
$$
\p \Mod_{0,1}^{A_0}(\{J_t\}) = \Mod_{0,1}^{A_0}(J_0) \sqcup
\Mod_{0,1}^{A_0}(J_1).
$$
\end{prop}

\subsection{Conclusion of the proof}

We will now derive the desired existence result using the
$\ZZ_2$-mapping degree of the evaluation map.  Recall that in general,
if $X$ and~$Y$ are closed and connected $n$-dimensional manifolds and
$f : X \to Y$ is a continuous map, then the degree 
$\deg_2(f) \in \ZZ_2$ can be defined by the condition
$$
f_*[X] = \deg_2(f) [Y] \in H_n(Y ; \ZZ_2),
$$
where $[X] \in H_n(X ; \ZZ_2)$ and $[Y] \in H_n(Y ; \ZZ_2)$ denote the
respective fundamental classes with $\ZZ_2$-coefficients.
Equivalently, if $f$ is smooth then $\deg_2(f)$ can be defined as the 
modulo~$2$ count of points in $f^{-1}(y)$ for a regular point~$y$.

Choosing a generic homotopy $\{J_t\} \in \jJ_\reg(W,\Omega \,;\, J_0,J_1)$
as provided by Proposition~\ref{prop:NSHomotopy}, the parametrized
moduli space $\Mod_{0,1}^{A_0}(\{J_t\})$ now furnishes a smooth cobordism 
between the two closed manifolds $\Mod_{0,1}^{A_0}(J_0)$ and
$\Mod_{0,1}^{A_0}(J_1)$.\footnote{With a little more work, one can also give all
of these moduli spaces natural orientations and thus obtain an
\emph{oriented} cobordism.  This has the result that our use of the
$\ZZ_2$-mapping degree could be replaced by the integer-valued mapping
degree, but we don't need this to prove the nonsqueezing theorem.}
Consider the evaluation map
$$
\ev : \Mod_{0,1}^{A_0}(\{J_t\}) \to W : ([(S^2,j,z,u)],t) \mapsto u(z),
$$
and denote its restriction to the two boundary components by
$\ev_0 : \Mod_{0,1}^{A_0}(J_0) \to W$ and $\ev_1 : \Mod_{0,1}^{A_0}(J_1) \to W$.
As we saw in \S\ref{subsec:J0}, $\ev_0$ is a diffeomorphism, thus
$(\ev_0)_*[\Mod_{0,1}^{A_0}(J_0)] = [W] \in H_{2n}(W ; \ZZ_2)$.  It follows
that
$$
(\ev_1)_*[\Mod_{0,1}^{A_0}(J_1)] = [W] \in H_{2n}(W ; \ZZ_2)
$$
as well, hence $\deg_2(\ev_1) = 1$ and $\ev_1$ is therefore surjective.
In particular, $\ev_1^{-1}(\iota(0))$ is not empty, and this proves
Proposition~\ref{prop:GWinvariant}.

\section{Monotonicity in the integrable case}
\label{sec:monotonicity}

In this section, we consider only holomorphic curves in $\RR^{2n} = \CC^n$
with its standard complex structure~$i$ and symplectic structure~$\omega\std$.
Recall that a smooth map $u : \Sigma \to B_{r_0}$ is called \defin{proper} if
every compact set in the target has a compact preimage.  For any
$r \in (0,r_0)$, we define the compact subset
$$
\Sigma_r := u^{-1}(\overline{B}_r^{2n}) \subset \Sigma,
$$
which by Sard's theorem is a submanifold with smooth boundary for 
almost every~$r$.  Our main goal is to prove the following result,
which was previously stated as Theorem~\ref{thm:monotonicity} and was a
crucial ingredient in the proof of the nonsqueezing theorem.

\begin{thm}[monotonicity]
\label{thm:monotonicity2}
if $u : (\Sigma,j) \to (B_{r_0}^{2n},i)$ is a proper holomorphic map
whose image contains~$0$, then for every
$r \in (0,r_0)$,
$$
\int_{\Sigma_r} u^*\omega\std \ge \pi r^2.
$$
\end{thm}

This result gives a quantitative version of the statement
that a holomorphic curve cannot fit an arbitrarily small amount of area
into some fixed neighborhood of a point in its image.  
More general versions also hold for non-integrable almost complex structures
and are useful in proving a number of technical results, especially in
the compactness theory; we'll come back to this in the next chapter.
We should also mention that this kind of result is by no means unique to
the theory of holomorphic curves: monotonicity formulas are also a popular
tool in the theory of minimal surfaces 
(cf.~\cites{Lawson:minimal,Grueter:monotonicity,ColdingMinicozzi}),
and indeed, 
Theorem~\ref{thm:monotonicity2} can be regarded as a corollary of such
results after observing that whenever $J$ is \emph{compatible} with a
symplectic structure~$\omega$ and a Riemannian metric is defined by
$\omega(\cdot,J\cdot)$, $J$-holomorphic curves are also \emph{area minimizing},
cf.~\cite{McDuffSalamon:Jhol}*{Lemma~2.2.1}.  This was also the perspective
adopted by Gromov in \cite{Gromov}; see also \cite{Fish:estimates} for
some more recent results along these lines.
In order to keep the discussion
self-contained and avoid delving into the theory of minimal surfaces,
we shall instead present a direct ``contact geometric'' proof, which is
fairly simple and uses a few notions that we will find useful in our
later discussions of contact geometry.

To start with, it's easy to see from our knowledge of the local behavior
of holomorphic curves that the estimate of 
Theorem~\ref{thm:monotonicity2} holds for any given curve~$u$ whenever
$r > 0$ is sufficiently small.  Indeed, in an appropriate choice
of local coordinates on a small enough neighborhood, $u$ looks like
a small perturbation of the map
$$
B_\epsilon \to \CC \times \CC^{n-1} : z \mapsto (z^k,0),
$$
whose area is $k \pi \epsilon^2$.  (See \S\ref{sec:intersections}
for a discussion of such local representation formulas.)

The result then follows from the next statement, which explains our use
of the term ``monotonicity''.

\begin{prop}
\label{prop:monotonicity}
Given the setup of Theorem~\ref{thm:monotonicity2}, the function
$$
F(r) = \frac{1}{r^2} \int_{\Sigma_r} u^*\omega\std
$$
is nondecreasing.
\end{prop}
Note that it will suffice to prove that $F(R) > F(r)$ whenever
$0 < r < R < r_0$ and both~$r$ and~$R$ lie in the dense set of
\emph{regular} values, i.e.~those for which
the intersection of~$u$ with $\p\overline{B}_r$
is transverse.
For regular values, $\Sigma_r$ is a smooth manifold with boundary and
we can use Stokes' theorem to compute $\int_{\Sigma_r} u^*\omega\std$.
In order to uncover the dependence on $r^2$, we shall 
switch perspectives and regard~$u$ as a map into the \emph{symplectization
of the standard contact sphere}.

Label the natural coordinates on $\RR^{2n} = \CC^n$ by
$(z_1,\ldots,z_n) = (p_1 + i q_1,\ldots,p_n + iq_n)$, so the symplectic
structure has the form
$$
\omega\std = \sum_{j=1}^n dp_j \wedge dq_j.
$$
Recall from \S\ref{sec:Weinstein} that the vector field
$$
V\std := \frac{1}{2} \sum_{j=1}^n \left( p_j \frac{\p}{\p p_j} + q_j
\frac{\p}{\p q_j} \right)
$$
is a \emph{Liouville} vector field on $(\RR^{2n},\omega\std)$, meaning it
satisfies $\Lie_{V\std}\omega\std = \omega\std$.  Let $\lambda\std$ denote the
$1$-form on $\RR^{2n}$ which is $\omega$-dual to~$V\std$, i.e.
$$
\lambda\std := \omega\std(V\std,\cdot).
$$
An easy computation then produces the expression
$$
\lambda\std = \frac{1}{2} \sum_{j=1}^n \left( p_j\, dq_j - q_j\, dp_j \right),
$$
and the fact that~$V\std$ is Liouville is equivalent to the observation
that $d\lambda\std = \omega\std$.  Moreover, since $\lambda\std(V\std) =
\omega\std(V\std,V\std) = 0$, we also have
$$
\Lie_{V\std}\lambda\std = \iota_{V\std} d\lambda\std + d \iota_{V\std} \lambda\std =
\iota_{V\std}\omega\std = \lambda\std.
$$
Identify the sphere $S^{2n-1}$ with the boundary of the closed unit ball
$\overline{B}^{2n} \subset \RR^{2n}$, and define
the \defin{standard contact form} $\alpha\std$ 
on $S^{2n-1}$ as the restriction of~$\lambda\std$,
$$
\alpha\std := \lambda\std|_{T\left(\p\overline{B}^{2n}\right)}.
$$
Now consider the diffeomorphism
$$
\Phi : \RR \times S^{2n-1} \to \RR^{2n} \setminus \{0\} : (t,m) \mapsto
\varphi_{V\std}^t(m) = e^{t/2}m,
$$
where $\varphi_{V\std}^t$ denotes the flow of~$V\std$.  By
Exercise~\ref{ex:symplectization}, we have
$$
\Phi^*\lambda\std = e^t \alpha\std,
\qquad
\Phi^*\omega\std = d(e^t \alpha\std),
$$
where $t$ denotes the $\RR$-coordinate on $\RR\times S^{2n-1}$ and 
$\alpha\std$ is defined on $\RR\times S^{2n-1}$ as the pullback via
the projection $\RR\times S^{2n-1} \to S^{2n-1}$.  Define an
integrable complex structure $J_0$ on $\RR\times S^{2n-1}$ so that this
diffeomorphism is biholomorphic, i.e.
$$
J_0 := \Phi^*i.
$$
Now removing at most finitely many points from $\Sigma$ to define
$$
\dot{\Sigma} := \{ z \in \Sigma \ |\ u(z) \ne 0 \}
$$
and defining $\dot{\Sigma}_r \subset \Sigma_r$ similarly,
we obtain a $J_0$-holomorphic map 
$$
(u_\RR,u_S) := \Phi^{-1} \circ u : \dot{\Sigma} \to \RR\times S^{2n-1},
$$
so that if $r = e^{\tau/2} \in (0,r_0)$ is regular, we have
\begin{equation*}
\begin{split}
F(r) &= \frac{1}{r^2} \int_{\Sigma_r} u^*\omega\std = 
e^{-\tau} \int_{\dot{\Sigma}_r} (u_\RR,u_S)^* d(e^t \alpha\std) = 
e^{-\tau} \int_{\p\Sigma_r} (u_\RR,u_S)^*(e^t\alpha\std) \\
&= \int_{\p\Sigma_r} u_S^*\alpha\std.
\end{split}
\end{equation*}
Thus for any two regular values $0 < r < R < r_0$, we now have
$$
F(R) - F(r) = \int_{\p\Sigma_R} u_S^*\alpha\std -
\int_{\p\Sigma_r} u_S^*\alpha\std = \int_{\overline{\Sigma_R\setminus \Sigma_r}}
u_S^*d\alpha\std.
$$
Proposition~\ref{prop:monotonicity} is then immediate from the following
exercise.

\begin{exercise}
Show that the almost complex structure~$J_0 = \Phi^*i$ 
on $\RR\times S^{2n-1}$ has
the following properties:
\begin{enumerate}
\item It is invariant under the natural $\RR$-action by translation of
the first factor in $\RR\times S^{2n-1}$.
\item For any $t \in \RR$, the unique hyperplane field in 
$\{t\} \times S^{2n-1}$ preserved by $J_0$ is precisely the
\emph{contact structure} $\xi\std := \ker\alpha\std$.
\item The restriction of $J_0$ to $\xi\std$ is \emph{compatible} with the
symplectic bundle structure $d\alpha\std|_{\xi\std}$, i.e.~the pairing
$\langle X,Y \rangle := d\alpha\std(X,J_0 Y)$ defines a bundle metric
on~$\xi\std$.
\item $J_0$ maps $\p_t$ to the \emph{Reeb vector field} of~$\alpha\std$,
i.e.~the unique vector field $R_{\alpha\std}$ on $S^{2n-1}$ satisfying the
conditions
$$
d\alpha\std(R_{\alpha\std},\cdot) \equiv 0
\quad\text{ and }\quad
\alpha\std(R_{\alpha\std}) \equiv 1.
$$
\end{enumerate}
Derive from these properties the fact that for any $J_0$-holomorphic curve
$(u_\RR,u_S) : \Sigma \to \RR \times S^{2n-1}$, the integrand
$u_S^*d\alpha\std$ is nonnegative.
\end{exercise}

\section{Bubbling off}
\label{sec:bubbling}

Our goal in this section is to provide a \emph{mostly} self-contained
proof of Proposition~\ref{prop:NSCompactness}, as a
consequence of the following result.

\begin{thm}
\label{thm:NSCompactness}
Suppose $(M,\omega)$ is a closed symplectic manifold of dimension $2n-2 \ge 2$
with $\pi_2(M) = 0$, $\sigma$ is an area form on~$S^2$,
$W := S^2 \times M$, $\Omega := \sigma \oplus \omega$,
$A_0 := [S^2 \times \{*\}] \in H_2(W)$ and we have the
following sequences:
\begin{itemize}
\item $J_k \to J$ is a $C^\infty$-convergent sequence of $\Omega$-compatible
almost complex structures on $W$,
\item $u_k : (S^2,i) \to (W,J_k)$ is a sequence of
pseudoholomorphic spheres with $[u_k] = A_0$, and
\item $\zeta_k \in S^2$ is a sequence of marked points.
\end{itemize}
Then after taking a subsequence, 
there exist biholomorphic maps $\varphi_k : (S^2,i) \to (S^2,i)$
with $\varphi_k(0) = \zeta_k$ such that the reparametrized curves
$$
u_k \circ \varphi_k : S^2 \to W
$$
converge in $C^\infty$ to a $J$-holomorphic sphere
$u : (S^2,i) \to (W,J)$.
\end{thm}

To prove this, we shall introduce some of the crucial technical
tools that underlie the more general compactness results of the next
chapter.  There's only one result which we will need to take for now
as a ``black box'':

\begin{prop}[Gromov's removable singularity theorem]
\label{prop:removable}
Suppose $(M,\omega)$ is a symplectic manifold with a tame almost complex
structure~$J$, and $u : B \setminus \{0\} \to M$ is a $J$-holomorphic
curve which has finite energy $\int_{B\setminus\{0\}} u^*\omega < \infty$
and image contained in a compact subset of~$M$.  Then~$u$ extends smoothly
over~$0$ to a $J$-holomorphic curve $B \to M$.
\end{prop}

A proof may be found in the next chapter, or in 
\cites{McDuffSalamon:Jhol,Sikorav,Hummel}.

As a fundamental analytical tool for our compactness arguments, we will use
the following piece of local elliptic regularity theory that was proved in
Chapter~\ref{chapter:local} as Corollary~\ref{cor:gradBounds}:
\begin{lemma}
\label{lemma:gradBounds}
Assume $p \in (2,\infty)$ and $m \ge 1$, $J_k \in \jJ^m(B^{2n})$ is a 
sequence of almost complex structures converging in $C^m$ to $J \in 
\jJ^m(B^{2n})$, and $u_k : B \to B^{2n}$ is a sequence of $J_k$-holomorphic
curves satisfying a uniform bound $\| u_k \|_{W^{1,p}(B)} < C$.
Then $u_k$ has a subsequence converging in $W^{m+1,p}_\loc$ to a
$J$-holomorphic curve $u : B \to B^{2n}$.
\end{lemma}

In our situation, we have $J_k \to J$ in $C^m$ for all~$m$, thus we will
obtain a $C^\infty_{\text{loc}}$-convergent subsequence if we can establish
$C^1$-bounds for our maps $u_k : S^2 \to W$, since
$C^1$ embeds continuously into $W^{1,p}$.  The lemma can be applied in
a more global setting as follows.  Fix Riemannian metrics on $S^2$ and~$W$
and use these to define the norm $| du(z) | \ge 0$ of the linear map
$du(z) : T_z S^2 \to T_{u(z)} W$ for any $u \in C^1(S^2,W)$ and $z \in S^2$.
If the given sequence of $J_k$-holomorphic maps $u_k : S^2 \to W$
satisfies a uniform bound of the form
\begin{equation}
\label{eqn:globalC1bound}
| du_k(z) | < C \quad
\text{ for all $k$ and all $z \in S^2$},
\end{equation}
then since $W$ is compact, a subsequence of $u_k$ will converge in
$C^0$ to some continuous map $u : S^2 \to W$.  We can then cover both $S^2$
and~$u(S^2) \subset W$ with finitely many local coordinate charts and
apply Lemma~\ref{lemma:gradBounds}, obtaining:

\begin{lemma}
\label{lemma:globalC1bound}
Suppose $J_k \to J$ is a $C^\infty$-convergent sequence of almost complex
structures on a closed manifold~$W$ and 
$u_k : (S^2,i) \to (W,J_k)$ is a sequence of pseudoholomorphic curves
satisfying a uniform $C^1$-bound as in \eqref{eqn:globalC1bound}.
Then a subsequence of~$u_k$ converges in $C^\infty$ to a pseudoholomorphic
curve $u : (S^2,i) \to (W,J)$.
\end{lemma}
\begin{remark}
The above lemma is obviously also true if~$W$ is not compact but the
images of the curves~$u_k$ are confined to a compact subset.  This 
generalization is important for compactness results in contact geometry
and symplectic field theory, e.g.~\cite{SFTcompactness}.
\end{remark}

In most situations, one cannot expect to derive a $C^1$-bound directly
from the given data, and in the general case such a bound does
not even hold.  The strategy is however as follows: if a $C^1$-bound does
\emph{not} hold, then we can find a sequence of points $z_k \in S^2$
such that $| du_k(z_k) | \to \infty$, and by an intelligent choice of
rescalings, the restriction of $u_k$ to small neighborhoods of~$z_k$
gives rise to a sequence of holomorphic disks on expanding domains that
exhaust~$\CC$.
These disks are always nonconstant but satisfy a uniform $C^1$-bound
by construction, thus by Lemma~\ref{lemma:gradBounds} they will converge
in $C^\infty_{\text{loc}}$ to a $J$-holomorphic plane with finite energy.
Since a plane is really just a punctured sphere, this $J$-holomorphic
plane can be extended to a nonconstant holomorphic sphere, often
called a ``bubble'', and the process by which this sphere is extracted
from the original sequence is often called ``bubbling off''.  In our
situation, we will find that the existence of this bubble leads to a
contradiction and thus implies the desired $C^1$-bound on the original
sequence.  In more general settings, there is no contradiction and one
must instead find a way of organizing the information that these bubbles add
to the limit of the original sequence---this leads to the notion of
\emph{nodal} holomorphic curves, the more general objects that make up
the \emph{Gromov compactification}, to be discussed in the next chapter.

We now carry out the details of the above argument, using a 
particular type of rescaling trick that has been popularized by Hofer
and collaborators (see e.g.~\cite{HoferZehnder}*{\S 6.4}).
The results stated below all assume the setting described in
the statement of Theorem~\ref{thm:NSCompactness}: in particular,
$(W,\Omega) = (S^2 \times M,\sigma \oplus \omega)$ and $\pi_2(M) = 0$.
Notice that the curves in the sequence $u_k : S^2 \to W$ are all homologous
and thus all have the same energy
$$
E(u_k) = \int_{S^2} u_k^*\Omega = \langle [\Omega],A_0 \rangle =
\langle [\sigma],[S^2] \rangle = \int_{S^2} \sigma.
$$
For reasons that will hopefully become clear in a moment, we now give
this positive constant a special name and write
$$
\hbar := \int_{S^2} \sigma > 0.
$$
The following is then a very simple example of a general
phenomenon known as \emph{energy quantization}.

\begin{lemma}
\label{lemma:NSquantization}
For any $J \in \jJ(W,\Omega)$, every nonconstant closed
$J$-holomorphic sphere in~$W$ has energy at least~$\hbar$.
\end{lemma}
\begin{proof}
If $u = (u_S,u_M) : S^2 \to S^2 \times M$ is $J$-holomorphic and not constant,
then
\begin{equation*}
\begin{split}
0 &< E(u) = \int_{S^2} u^*\Omega = \langle [\sigma \oplus \omega],
[u_S] \times [\{*\}] + [\{*\}] \times [u_M] \rangle \\
&= \langle [\sigma],[u_S] \rangle + \langle [\omega],[u_M] \rangle.
\end{split}
\end{equation*}
Since $\pi_2(M) = 0$, the spherical homology class $[u_M] \in H_2(M)$
necessarily vanishes, so the above expression implies $E(u) =
\langle [\sigma],[u_S] \rangle$, which must be
an integer multiple of~$\hbar$.  Since it is also positive, the result follows.
\end{proof}

We next choose reparametrizations of the sequence $u_k$ so as to rule out
certain trivial possibilities, such as $u_k$ converging almost everywhere to a
constant.  Write $u_k = (u_k^S,u_k^M) : S^2 \to S^2 \times M$, and observe
that since $[u_k] = [S^2 \times \{*\}]$, $u_k^S : S^2 \to S^2$ is always
a map of degree~$1$ and hence surjective.  After taking a subsequence,
we may assume that the images of the marked points in~$S^2$ converge, i.e.
$$
u_k^S(\zeta_k) \to \zeta_\infty \in S^2.
$$
Assume without loss of generality that $\zeta_\infty$ 
is neither~$1$ nor~$\infty$; if it \emph{is} one of these,
then the remainder of our argument will require
only trivial modifications.  Now since $u_k^S$ is surjective, for
sufficiently large~$k$ we can always find biholomorphic maps
$\varphi_k : (S^2,i) \to (S^2,i)$ that have the following properties:
\begin{itemize}
\item $\varphi_k(0) = \zeta_k$,
\item $u_k^S \circ \varphi_k(1) = 1$,
\item $u_k^S \circ \varphi_k(\infty) = \infty$.
\end{itemize}
To simplify notation, let us now replace the original sequence
by these reparametrizations and thus assume without loss of generality
that the maps $u_k = (u_k^S,u_k^M) : S^2 \to S^2 \times M$ and marked
points $\zeta_k \in S^2$ satisfy
$$
\zeta_k = 0, \qquad u_k(1) \in \{1\} \times M, \qquad 
u_k(\infty) \in \{\infty\} \times M
$$
for all~$k$.

If the maps $u_k$ satisfy a uniform $C^1$-bound, then we are now finished
due to Lemma~\ref{lemma:globalC1bound}.  Thus assume the contrary, that
there is a sequence $z_k \in S^2$ with
$$
| du_k(z_k) | \to \infty,
$$
and after taking a subsequence we may assume $z_k \to z_\infty \in S^2$.
Choose a neighborhood $z_\infty \in \uU \subset S^2$ and a biholomorphic
map
$$
\varphi : (B,i) \to (\uU,i)
$$
identifying~$\uU$ with the unit ball in~$\CC$ such that
$\varphi(0) = z_\infty$, and write
$$
\tilde{u}_k = u_k \circ \varphi : (B,i) \to (W,J_k),
\qquad
\tilde{z}_k = \varphi^{-1}(z_k).
$$
We then have $| d\tilde{u}_k(\tilde{z}_k) | \to \infty$ and
$\tilde{z}_k \to 0$.

We now examine a rescaled reparametrization of the sequence $\tilde{u}_k$
on shrinking neighborhoods of~$\tilde{z}_k$.  In particular, let
$R_k := | d\tilde{u}_k(\tilde{z}_k) | \to \infty$,  pick a sequence of 
positive numbers $\epsilon_k \to 0$ which decay slowly enough so that 
$\epsilon_k R_k \to \infty$, and
consider the sequence of $J_k$-holomorphic maps
$$
v_k : (B_{\epsilon_k R_k},i) \to (W,J_k) : z \mapsto
\tilde{u}_k\left( \tilde{z}_k + \frac{z}{R_k} \right).
$$
Then 
$$
| dv_k(z) | = \frac{1}{R_k}
\left| d\tilde{u}_k\left( \tilde{z}_k + \frac{z}{R_k}
\right) \right| ,
$$ 
so in particular 
$| d v_k(0) | = \frac{1}{R_k} | d \tilde{u}_k(\tilde{z}_k) | = 1$.
To proceed further, we'd like to be able to say that $| d v_k(z) |$
satisfies a uniform bound for $z \in B_{\epsilon_k R_k}$, as then
Lemma~\ref{lemma:gradBounds} would give a subsequence converging in
$C^\infty_\text{loc}$ on~$\CC$.  Such a bound is not obvious: it would
require being able to bound $| d\tilde{u}_k(z) |$ in terms of
$| d\tilde{u}_k(\tilde{z}_k) |$ for all $z \in B_{\epsilon_k}(\tilde{z}_k)$.
While there is no reason that such a bound should 
necessarily hold for the chosen sequence, the
following topological lemma due to Hofer tells us that we can always
ensure this bound after a slight adjustment.
\begin{lemma}[Hofer]
\label{lemma:Hofer}
Suppose $(X,d)$ is a complete metric space, $g : X \to [0,\infty)$ is
continuous, $x_0 \in X$ and $\epsilon_0 > 0$.  Then there exist
$x \in X$ and $\epsilon > 0$ such that,
\begin{enumerate}
\renewcommand{\labelenumi}{(\alph{enumi})}
\item $\epsilon \le \epsilon_0$,
\item $g(x) \epsilon \ge g(x_0) \epsilon_0$,
\item $d(x,x_0) \le 2\epsilon_0$, and
\item $g(y) \le 2 g(x)$ for all $y \in \overline{B_\epsilon(x)}$.
\end{enumerate}
\end{lemma}
\begin{proof}
If there is no $x_1 \in \overline{B_{\epsilon_0}(x_0)}$ such that
$g(x_1) > 2 g(x_0)$, then we can set $x = x_0$ and $\epsilon = \epsilon_0$
and are done.  If such a point~$x_1$ does exist, then we
set $\epsilon_1 := \epsilon_0 / 2$ and repeat the above process for the
pair $(x_1,\epsilon_1)$: that is, if there is no $x_2 \in
\overline{B_{\epsilon_1}(x_1)}$ with $g(x_2) > 2 g(x_1)$, we set
$(x,\epsilon) = (x_1,\epsilon_1)$ and are finished, and 
otherwise define $\epsilon_2 = \epsilon_1 / 2$ and repeat for
$(x_2,\epsilon_2)$.  This process must eventually terminate, as otherwise
we obtain a Cauchy sequence $x_n$ with $g(x_n) \to \infty$, which is
impossible if~$X$ is complete.
\end{proof}
The upshot of the lemma is that the sequences $\epsilon_k > 0$ and
$\tilde{z}_k \in B$ can be modified slightly to have the additional
property that
\begin{equation}
\label{eqn:extraCondition}
| d \tilde{u}_k(z) | \le 2 | d \tilde{u}_k(\tilde{z}_k) | \quad
\text{ for all $z \in \overline{B_{\epsilon_k}(\tilde{z}_k)}$}.
\end{equation}
From this it follows that the rescaled sequence $v_k : B_{\epsilon_k R_k} \to W$
satisfies
$$
| dv_k(z) |  \le 2,
\qquad
| dv_k(0) | = 1,
$$ 
so we conclude from 
Lemma~\ref{lemma:gradBounds} that a subsequence of $v_k$ converges in
$C^\infty_{\text{loc}}(\CC,W)$ to a $J$-holomorphic plane
$$
v_\infty : (\CC,i) \to (W,J)
$$
which satisfies $| d v_\infty(0) | = 1$ and is thus not constant.
We claim that $v_\infty$ also has finite energy bounded by~$\hbar$.  
Indeed, for any $R > 0$, we have
$$
\int_{B_{R}} v_\infty^*\Omega = \lim_k \int_{B_R} v_k^*\Omega,
$$
while for sufficiently large~$k$,
$$
\int_{B_{R}} v_k^*\Omega \le
\int_{B_{\epsilon_k R_k}} v_k^*\Omega =
\int_{B_{\epsilon_k}(\tilde{z}_k)} \tilde{u}_k^*\Omega =
\int_{\varphi\left(B_{\epsilon_k}(\tilde{z}_k)\right)} u_k^*\Omega
\le \int_{S^2} u_k^*\Omega = \hbar.
$$
Applying the removable singularity theorem (Prop.~\ref{prop:removable}),
$v_\infty$ thus extends to a nonconstant $J$-holomorphic sphere
$$
v_\infty : (S^2,i) \to (W,J),
$$
and energy quantization (Lemma~\ref{lemma:NSquantization}) implies that its
energy is exactly~$\hbar$.  This sphere is our first real life example
of a so-called ``bubble''.

We claim next that if the above scenario happens, then for any other
sequence $z_k' \in S^2$ with $| d u_k(z_k') | \to \infty$, $z_k'$ can
only accumulate at the same point $z_\infty$ again.  Indeed, otherwise
the above procedure produces a second bubble
$v_\infty' : (S^2,i) \to (W,J)$ with energy~$\hbar$, and by inspecting the
energy estimate above, one sees that
for large~$k$, $u_k$ must have a concentration of energy close to~$\hbar$
in small neighborhoods of both $z_\infty$ and~$z_\infty'$.  That is
impossible since $E(u_k)$ is already bounded by~$\hbar$.

The above implies that on any compact subset of $S^2 \setminus \{z_\infty\}$,
$u_k$ satisfies a uniform $C^1$-bound and thus converges 
in $C^\infty_{\text{loc}}(S^2 \setminus \{ z_\infty \})$ to a
$J$-holomorphic punctured sphere
$$
u_\infty : (S^2 \setminus \{ z_\infty \},i) \to (W,J).
$$
Moreover, we have
$$
u_\infty(0) \in \{\zeta_\infty\} \times M, 
\quad u_\infty(1) \in \{1\} \times M \quad \text{ and } \quad
u_\infty(\infty) \in \{\infty\} \times M
$$
unless $z_\infty \in \{\zeta_\infty,1,\infty\}$, in which case at least 
two of these three
statements still holds.  It follows that $u_\infty$ cannot be constant,
so by Lemma~\ref{lemma:NSquantization} it has energy at least~$\hbar$.
But this again gives a contradiction if the bubble $v_\infty$ exists, 
as it implies that for large~$k$, the restrictions of~$u_k$ to some 
large subset of $S^2 \setminus \{z_\infty\}$ and some disjoint small 
neighborhood of $z_\infty$ each have energy at least slightly less 
than~$\hbar$, so that $\int_{S^2} u_k^*\Omega$ must be strictler greater
than~$\hbar$.  This contradiction excludes the bubbling scenario,
thus establishing the desired $C^1$-bound for~$u_k$ and completing
the proof of Theorem~\ref{thm:NSCompactness}.


\backmatter

\end{document}